\documentclass[reqno]{amsart}
\usepackage{amssymb}
\usepackage{amsmath}
\usepackage{esint}
\usepackage[ %dvips,
colorlinks,bookmarksopen,bookmarksnumbered,citecolor=red,urlcolor=red]{hyperref}

\usepackage[colorlinks]{hyperref}
\usepackage{paralist}

\usepackage{graphicx}

\newtheorem{theorem}{Theorem}[section]
\newtheorem{corollary}[theorem]{Corollary}

\newtheorem{defi}[theorem]{Definition}
\newenvironment{definition}{\begin{defi}\rm}{\end{defi}}
\newtheorem{ex}[theorem]{Example}

\newtheorem{lemma}[theorem]{Lemma}
\newtheorem{prob}{Problem}[section]

\newtheorem{rem}[theorem]{Remark}
\newenvironment{remark}{\begin{rem}\rm}{\end{rem}}
\newtheorem{fantoma}[theorem]{\phantom{i}}

\newcommand{\bl}{\color{blue}}

\newcommand\BibTeX{{\rmfamily B\kern-.05em \textsc{i\kern-.025em b}\kern-.08em
T\kern-.1667em\lower.7ex\hbox{E}\kern-.125emX}}

\newcommand{\comment}[1]{}

\allowdisplaybreaks

\numberwithin{equation}{section}

\begin{document}

%\runninghead{R.~Gutt, M.~Kohr, S.E.~Mikhailov, W.L.~Wendland}

\title[Mixed problem for Darcy-Forchheimer-Brinkman PDE system] %
{On the mixed problem for the semilinear Darcy-Forchheimer-Brinkman PDE system in 
Besov spaces on creased Lipschitz domains}

%\title{On the mixed Dirichlet-Neumann problem for the semilinear Darcy-Forchheimer-Brinkman system in Besov spaces\footnotemark[2]}

%\author[R.~Gutt, M.~Kohr, S.E.~Mikhailov, W.L.~Wendland]{Robert~Gutt\affil{a}, Mirela~Kohr\affil{a}\corrauth\, Sergey E.~Mikhailov\affil{b} and Wolfgang~L.~Wendland\affil{c}}

\author[R.~Gutt]{Robert~Gutt}
\address{R.~Gutt, Faculty of Mathematics and Computer Science, Babe\c{s}-Bolyai University, 1 M. Kog\u alniceanu Str., 400084 Cluj-Napoca, Romania}

\author[M.~Kohr]{Mirela~Kohr}
\address{M.~Kohr, Faculty of Mathematics and Computer Science, Babe\c{s}-Bolyai University,
1 M. Kog\u alniceanu Str., 400084 Cluj-Napoca, Romania}
%\corraddr{E-mail: mkohr@math.ubbcluj.ro}

\author[S.E.~Mikhailov]{Sergey E.~Mikhailov}
\address{S.E.~Mikhailov, Department of Mathematics, Brunel University London,
Uxbridge, UB8 3PH, United Kingdom}

\author[W.L.~Wendland]{Wolfgang~L.~Wendland}
\address{W.L.~Wendland, Institut f\"ur Angewandte Analysis und Numerische
Simulation, Universit\"at Stuttgart, Pfaffenwaldring 57, 70569
Stuttgart, Germany}

%\email{robert.gutt@ubbonline.ubbcluj.ro}

%\email{wendland@mathematik.uni-stuttgart.de}

\comment{\cgs{%The work of M. Kohr was supported by the
"Scientific Grant for Excellence in Research", GSCE--30259/2015, of the Babe\c{s}-Bolyai University
{and}
%. The work of S.E. Mikhailov and W.L. Wendland was supported by the
grant {EP/M013545/1} "Mathematical Analysis of Boundary-Domain Integral Equations for Nonlinear PDEs" from the EPSRC, UK.}}
\thanks{The work of M. Kohr was supported by the "Scientific Grant for Excellence in Research", GSCE--30259/2015, of the Babe\c{s}-Bolyai University. The work of S.E. Mikhailov and W.L. Wendland was supported by the grant EP/M013545/1DSM: "Mathematical Analysis of Boundary-Domain Integral Equations for Nonlinear PDEs" from the EPSRC, UK.}

\begin{abstract}
The purpose of this paper is to study the mixed Dirichlet-Neumann boundary value problem for the semilinear Darcy-Forchheimer-Brinkman system in $L_p$-based Besov spaces on a bounded Lipschitz domain in ${\mathbb R}^3$, with $p$ in a neighborhood of $2$. This system is obtained by adding the semilinear term $|{\bf u}|{\bf u}$ to the linear Brinkman equation. {First, we provide some results about} equivalence between the Gagliardo and non-tangential traces, as well as between the weak canonical conormal derivatives and the non-tangential conormal derivatives. Various mapping and invertibility properties of some integral operators of potential theory for the linear Brinkman system, and well posedness results for the Dirichlet and Neumann problems in $L_p$-based Besov spaces on bounded Lipschitz domains in ${\mathbb R}^n$ ($n\geq 3$) are also presented.
Then, employing integral potential operators, we show the well-posedness in $L_2$-based Sobolev spaces for the mixed problem of Dirichlet-Neumann type for the linear Brinkman system on a bounded Lipschitz domain in ${\mathbb R}^n$ $(n\geq 3)$.
Further, by using some stability results of Fredholm and invertibility properties and exploring invertibility of the associated Neumann-to-Dirichlet operator, we extend the well-posedness property to some $L_p$-based Sobolev spaces. Next we use the well-posedness result in the linear case combined with a fixed point theorem in order to show the existence and uniqueness for a mixed boundary value problem of Dirichlet and Neumann type for the semilinear Darcy-Forchheimer-Brinkman system in $L_p$-based Besov spaces, with $p\in (2-\varepsilon ,2+\varepsilon)$ and some parameter $\varepsilon >0$.
\end{abstract}

%\MOS{Primary 35J25, 42B20, 46E35; Secondary 76D, 76M}

\maketitle

\section{Introduction}
Boundary integral methods are a powerful tool to investigate linear elliptic boundary value problems
%of various type
{that} appear in {various} areas of science and engineering (see, e.g., \cite{B-H,Co,D-M,Med-CVEE,M-T}). Among many valuable contributions in the field we mention the well-posedness result of the Dirichlet problem for the Stokes system in Lipschitz domains in $\mathbb{R}^{n}$ $(n \geq 3)$ with boundary data in $L_2$-based Sobolev spaces, which have been obtained by Fabes, Kenig and Verchota in \cite{Fa-Ke-Ve} by using a layer potential analysis. Also, Mitrea and Wright \cite{M-W} obtained the well-posedness results {for}
%various boundary value problems of
Dirichlet, Neumann and transmission problems for the Stokes system on arbitrary Lipschitz domains in $\mathbb{R}^{n}$ $(n\geq 2)$, with data in Sobolev and Besov-Triebel-Lizorkin spaces. By using a boundary integral method, Mitrea and Taylor \cite{M-T} obtained well-posedness results for the Dirichlet problem for the Stokes system on arbitrary Lipschitz domains on a compact Riemannian manifold, with {boundary data in $L_2$}. {Their results} extended the results {of} \cite{Fa-Ke-Ve} from the Euclidean setting to the case of compact Riemannian manifolds. Continuing the study {of} \cite{M-T}, Dindo\u{s} and Mitrea \cite{D-M} developed a layer potential analysis to obtain existence and uniqueness results for the Poisson problem for the Stokes and Navier-Stokes systems on $C^{1}$ domains, but also on Lipschitz domains in compact Riemannian manifolds. Medkov\'{a} in \cite{Med-CVEE} studied various transmission problems for the Brinkman system.

{Due} to many practical applications, the mixed problems for elliptic boundary value problems on smooth and  {Lipschitz domains} have been {also} intensively investigated. Let us mention {that} Mitrea and Mitrea in \cite{M-M} have proved sharp well-posedness results for the Poisson problem for the Laplace operator with mixed boundary conditions of Dirichlet and Neumann type on bounded Lipschitz domains in ${\mathbb R}^3$ whose boundaries satisfy a suitable geometric condition introduced by Brown \cite{Br}, and with data in Sobolev and Besov spaces. Brown et al. \cite{B-M-M-W} have obtained the well-posedness result of the mixed Dirichlet-Neumann problem for the Stokes system on creased Lipschitz domains in ${\mathbb R}^n$ $(n\geq 3)$. In order to prove the desired well-posedness result, the authors reduced such a boundary value problem to a boundary integral equation, obtained useful Rellich-type estimates, and used the well-posedness result of the mixed Dirichlet-Neumann problem for the Lam\'{e} system that has been obtained in \cite{B-M}. Costabel and Stephan in \cite{Co-S} analyzed mixed boundary value problems in polygonal domains by using a boundary integral approach. {In \cite{{CMN-1}, Ch-Mi-Na-3}}, direct segregated systems of boundary-domain integral equations equivalent to mixed boundary value problems of Dirichlet-Neumann type {for} a scalar second-order divergent elliptic partial differential equation with a variable coefficient, were analyzed in interior and exterior domains in ${\mathbb R}^3$ (see also \cite{CMN-NMPDE-crack} for the mixed problems with cracks and \cite{MikMMAS2006} for united boundary-domain integral equations).
An interesting boundary integral equation method for a mixed boundary value problem of the biharmonic equation has been developed in \cite{C-H-W}.

Boundary integral methods combined with fixed point theorems have been focused {on} the analysis of boundary value problems for linear elliptic systems with nonlinear boundary conditions and for nonlinear elliptic systems with various (linear or nonlinear) boundary conditions.
{Recently}, the authors in \cite{K-L-W} have used a boundary integral method to obtain existence results for a nonlinear problem of Neumann-transmission type for the Stokes and Brinkman systems on Lipschitz domains in Euclidean setting and with boundary data in various $L_p$, Sobolev, or Besov spaces. The techniques of layer potential theory for the Stokes and Brinkman systems was used in \cite{K-L-W1} to analyze Poisson problems for semilinear generalized Brinkman systems on Lipschitz domains in ${\mathbb R}^n$ with Dirichlet or Robin boundary conditions and given data in Sobolev and Besov spaces. Boundary value problems of Robin type for the Brinkman and Darcy-Forchheimer-Brinkman systems in Lipschitz domains in Euclidean setting have been investigated in \cite{K-L-W2} (see also \cite{K-L-W3,K-L-W4}). An integral potential method for transmission problems with Lipschitz interface in ${\mathbb R}^3$ for the Stokes and Darcy-Forchheimer-Brinkman systems and data in weighted Sobolev spaces has been recently obtained in \cite{K-L-M-W}. Transmission problems for the Navier-Stokes and Darcy-Forchheimer-Brinkman systems in Lipschitz domains on compact Riemannian manifolds have been recently analyzed in \cite{K-M-W}. Well-posedness results for semilinear elliptic problems on Lipschitz domains in compact Riemannian manifolds have been obtained by Dindo\u{s} and Mitrea in \cite{Dindos-semi}. Let us also mention that Russo and Tartaglione in \cite{Russo-Tartaglione-2,Russo-Tartaglione-4} used a double-layer integral method in order to obtain existence results for boundary problems of Robin type for the Stokes and Navier-Stokes systems in Lipschitz domains in Euclidean setting with data in Sobolev spaces. {Maz'ya and Rossmann \cite{Ma-Ro} obtained $Lp$ estimates of solutions to mixed boundary value problems for the Stokes system in polyhedral domains. Taylor, Ott and Brown in \cite{Ta-Br} studied $Lp$-mixed Dirichlet-Neumann problem for the Laplace equation in a a bounded Lipschitz domain in ${\mathbb R}^n$ with general decomposition of the boundary.}

In this paper we analyze the mixed Dirichlet-Neumann boundary value problem for the semilinear Darcy-Forchheimer-Brinkman system in $L_p$-based Besov spaces on a bounded Lipschitz domain in ${\mathbb R}^3$, when the given boundary data belong to $L_p$ spaces, with $p$ in a neighborhood of $2$. This system is obtained by adding {the semilinear term $|{\bf u}|{\bf u}$} to the linear Brinkman equation. {First, we provide some results about} {equivalence between the Gagliardo and non-tangential traces}, {as well as between the weak canonical conormal derivatives and the non-tangential conormal derivatives. Various} {mapping} and invertibility { properties} of some integral operators of potential theory for the {linear} Brinkman system, and well posedness results for the Dirichlet and Neumann problems in $L_p$-based Besov spaces on bounded Lipschitz domains in ${\mathbb R}^n$ ($n\geq 3$) {are also presented.} Based on these results we show the well-posedness result for the mixed problem of Dirichlet-Neumann type for the Brinkman system in a bounded domain in ${\mathbb R}^n$ $(n\geq 3)$ with given data in $L_2$-based Sobolev spaces. Further, by using some stability results of Fredholm and invertibility properties, we extend the well-posedness property to the case of boundary data in $L_p$-based Sobolev spaces, with {$p\in \left(\frac{2(n-1)}{n+1}-\varepsilon, 2+\varepsilon\right)\cap (1,+\infty )$}, for some $\varepsilon>0$. The main idea for showing this property is the invertibility of an associated Neumann-to-Dirichlet operator, inspired by the approach developed by Mitrea and Mitrea in \cite{M-M}. Next we use the well-posedness result in the linear case combined with a fixed point theorem in order to show the existence and uniqueness in $L_p$-based Besov spaces for a mixed boundary value problem of Dirichlet and Neumann type for the semilinear Darcy-Forchheimer-Brinkman system in a Lipschitz domain in ${\mathbb R}^3$, when the boundary data belong to some $L_p$ spaces, with {$p\in (2-\varepsilon ,2+\varepsilon )$} and some parameter $\varepsilon > 0$. The motivation of this work is based on some practical applications,
%due to the fact that
{where}
the semilinear Darcy-Forchheimer-Brinkman system describes the motion of viscous incompressible fluids in porous media. A suggestive example is given by a sandstone reservoir filled with oil, or the convection of a viscous fluid in a porous medium located in a bounded domain, where a part of the boundary is in contact with air and the remaining part is a solid surface or the interface with another immiscible material or fluid. All these problems are well described by the Brinkman system, the semilinear Darcy-Forchheimer-Brinkman system, or by the Darcy-Forchheimer-Brinkman system, the latter of these systems containing both the nonlinear convective term $({\bf u}\cdot \nabla ){\bf u}$ and the semilinear term $|{\bf u}|{\bf u}$. For further details we refer the reader to the book by Nield and Bejan \cite{Ni-Be} (see also the theoretical and numerical approach in \cite{G-K-W,Robert}).

{It is supposed that the methods presented in this paper can be developed further, to analyze  also the nonlinear boundary-domain integro-differential equations, e.g., the ones formulated in \cite{Mikh-2, Mikh-1} for some quasi-linear boundary value problems.}

%\section{Preliminaries}
{\section{Functional setting and useful results}}
\label{sectio1}

The purpose of this section is to provide main notions and results used in this paper. We recall the definition of a bounded Lipschitz domain and give a short review of the involved {Sobolev, Bessel potential and Besov} spaces. Also we present the main properties of the layer potential operators for the Stokes and Brinkman systems in Lipschitz domains in ${\mathbb R}^n$.

{For any point $x=(x_1,x_2,\ldots ,x_n)\in {\mathbb R}^n$, %$n\geq 3$,
{we use the representation}
$x=(x',x_{n})$, where $x'\in {\mathbb R}^{n-1}$ and $x_n\in {\mathbb R}$.}
%$x=(x',x_{n})$, where $x':=(x_1,x")\in {\mathbb R}^{n-1}$ and $x":=(x_2,\ldots ,x_{n-1})\in {\mathbb R}^{n-2}$.}
{First,} we recall the definition of Lipschitz domain (cf., e.g., \cite[Definition 2.1]{M-M1}).
%\cite[Definition 2.1]{M-M}).
%\begin{definition}

\begin{defi}
\label{Lipschity domain}
A nonempty, open, bounded subset ${\Omega}$ of $\subset \mathbb{R}^{n}$ $(n\geq 3)$ is called a bounded {\em Lipschitz domain} if for any ${\bf x}\in \partial \Omega $ there exist some constants $r,h>0$ and a coordinate system in ${\mathbb R}^n$, $(y_1,\ldots ,y_n)=(y',y_n)\in \mathbb{R}^{n-1}\times \mathbb{R}$, which is isometric to the canonical one and has origin at ${\bf x}$, along with a Lipschitz function $\varphi : \mathbb{R}^{n-1} \rightarrow \mathbb{R}$, such that the following property holds. If ${\mathcal C}(r,h)$ denotes the open cylinder
%\begin{align}
$\left\{y=(y',y_{n})\in \mathbb{R}^{n-1}\times \mathbb{R}:|y'| < r, |y_{n}|< h\right\}\subseteq {\mathbb R}^n,$
%\end{align}
then
\begin{equation}
\label{Lip-b}
%\begin{split}
{\Omega}\cap {\mathcal C}(r,h)=\{y=(y',y_{n})\in \mathbb{R}^{n-1}\times \mathbb{R}:|y'|<r \mbox{ and } \varphi (y')<y_n<h\}.
%\\
%&{\mathfrak C}_{r}(x)\cap \partial{\Omega} = \{(y',y_{n}): y_{n}=\phi(y')\}\cap {\mathfrak C}_{r}(x).
%\end{split}
\end{equation}
\end{defi}
%there exists a constant $\tilde M>0$ {such} that for each $x\in \partial{\Omega}$ there exist {a coordinate system in $\mathbb{R}^{n}$ (isometric to the canonical one)}, $(x',x_{n})\in \mathbb{R}^{n-1}\times \mathbb{R}$, a constant $r>0$, a truncated cylinder $ {\mathfrak C}_{r}(x) := \left\{(y',y_{n}):|y' - x'| < r, |y_{n} - x_{n}|< 2\tilde Mr \right\}$, and a Lipschitz function $\phi : \mathbb{R}^{n-1} \rightarrow \mathbb{R}$ with $\| \nabla \phi \|_{L^{\infty}(\mathbb{R}^{n-1})} \leq \tilde M$, such that
%\begin{equation}
%\begin{split}
%&{\mathfrak C}_{r}(x)\cap {\Omega} = \{(y',y_{n}):y_{n} > \phi(y')\} \cap {\mathfrak C}_{r}(x),
%\\
%&{\mathfrak C}_{r}(x)\cap \partial{\Omega} = \{(y',y_{n}): y_{n}=\phi(y')\}\cap {\mathfrak C}_{r}(x).
%\end{split}
%\end{equation}
%\end{defi}
%\end{definition}
In view of the Definition \ref{Lipschity domain}, condition \eqref{Lip-b} implies that $\partial \Omega =\partial \overline{\Omega }$ and the characterization (cf. \cite[(2.4)-(2.6)]{M-M1})
\begin{equation}
\label{Lip-b1}
\begin{split}
&\partial {\Omega}\cap {\mathcal C}(r,h)=\{y=(y',y_{n})\in \mathbb{R}^{n-1}\times \mathbb{R}:|y'|<r \mbox{ and } y_n=\varphi (y')\},\\
&({\mathbb R}^n\setminus \overline{\Omega })\cap {\mathcal C}(r,h)=\{y=(y',y_{n})\in \mathbb{R}^{n-1}\times \mathbb{R}:|y'|<r \mbox{ and } -h<y_n<\varphi (y')\}.
\end{split}
\end{equation}

{\it Let all along the paper, {${\Omega}_{+}$} denote a bounded Lipschitz domain with a connected boundary $\partial \Omega $, and ${\Omega}_{-}:=\mathbb{R}^{n}\setminus \overline{\Omega_+}$ denote the corresponding exterior domain. Unless stated otherwise, it will be also assumed that $n\ge 3$.}

{Let} $\kappa =\kappa(\partial {\Omega})>1$ be a fixed sufficiently large constant. Then {\it the non-tangential maximal operator} of an arbitrary function $u:{\Omega}_{\pm} \rightarrow \mathbb{R}$ is defined by
\begin{equation}
\label{2.1.1}
{M}({u})(x):= \{\sup{|u(y)|: y \in {\mathfrak D}_{\pm}(x)}, \ x \in \partial {\Omega}\},
\end{equation}
where
\begin{equation}
\label{cone}
{{\mathfrak D}_{\pm}(x)\equiv {\mathfrak D}_{\kappa ;\pm}(x)}:=\{y\in {\Omega}_{\pm}: {\rm{dist}}(x,y) < \kappa \rm{dist}(y,\partial {\Omega}),\  x \in \partial {\Omega}\},
\end{equation}
are non-tangential approach cones located in ${\Omega}_{+}$ and ${\Omega}_{-}$, respectively (see, e.g., \cite{M-W}). Moreover,
\begin{equation}
\label{nt-trace}
{u^\pm_{\rm nt}}(x):= \lim_{{\mathfrak D}_{\pm}\ni y \rightarrow x}{u}(y)
\end{equation}
are the {\it non-tangential limits} of $u$ with respect to ${\Omega}_{\pm }$ at $x\in \partial {\Omega}$.
Note that if ${M}({u})\in L_p(\partial \Omega )$ for one choice of $\kappa $, where $p\in (1,\infty )$, then this property holds for arbitrary choice of $\kappa $ (see, e.g., \cite[p. 63]{Med-AAM}). For the sake of brevity, we use the notation ${\mathfrak D}_{\pm}(x)$ instead of ${\mathfrak D}_{\kappa ;\pm}(x)$.
We often need the property below (cf. \cite[page 80]{Necas2012}, \cite[Theorem 1.12]{Verchota1984}; see also \cite[Lemma 2.2]{M-M-P}).

%{\bn *** SM: (i) Include other properties of the Lipschitz domains as in the paper by Verchota, etc.
%(ii) Give a definition of Lipschitz character.

\begin{lemma}
\label{2.13D}
If $\Omega\subset {\mathbb R}^n$ is a Lipschitz domain, then there exists a sequence of $C^\infty$ domains $\Omega_j$ approximating $\Omega$ $($$\Omega_j\to{\Omega}$ as $j\to \infty$$)$ in the following sense:
\begin{itemize}
\item[$(i)$]
$\overline{\Omega}_j\subset{\Omega}$, and there exists a covering of $\partial \Omega $ with finitely many coordinate cylinders $(${\em atlas}$)$ that also form a family of coordinate cylinders for $\partial \Omega _j$, for each $j$. Moreover, {for each such cylinder ${\mathcal C}(r,h)$, if $\varphi $ and $\varphi _j$ are the corresponding Lipschitz functions whose graphs describe the boundaries of $\partial \Omega $ and $\partial \Omega _j$, respectively, in ${\mathcal C}(r,h)$}, then $\|\nabla \varphi _j\|_{L_{\infty }({\mathbb R}^{n-1})}\leq \|\nabla \varphi \|_{L_{\infty }({\mathbb R}^{n-1})}$ and $\nabla \varphi _j\to \nabla \varphi $ pointwise a.e.
\item[$(ii)$]
There exist a sequence of Lipschitz diffeomorphisms $\Phi_j:\partial\Omega\to \partial\Omega_j$ such that the Lipschitz constants of $\Phi_j$, $\Phi_j^{-1}$ are uniformly bounded in $j$.
\item[$(iii)$]
There is a constant $\kappa >0$ such that for all $j\geq 1$ and all ${\bf x}\in \partial \Omega $, we have $\Phi_j({\bf x})\in {\mathfrak D}_{+}({\bf x})\equiv {\mathfrak D}_{\kappa ;\pm}(x)$, where ${\mathfrak D}_{+}({\bf x})\equiv {\mathfrak D}_{\kappa ;\pm}(x)$ is the non-tangential approach cone with vertex at ${\bf x}$.
Moreover,
\begin{align}
%&\Phi_j({\bf x})\in {\mathfrak D}_{+}({\bf x}),\ \forall \ j\geq 1,\ \forall \ {\bf x}\in \partial \Omega ,\\
&\lim_{j\to \infty }|\Phi_j({\bf x})-{\bf x}|=0 \mbox{ \rm{uniformly in} } {\bf x}\in \partial \Omega ,\\
&\lim_{j\to \infty }\boldsymbol \nu^{(j)}(\Phi_j({\bf x}))=\boldsymbol \nu ({\bf x}) \mbox{ \rm{for a.e.} } {\bf x}\in \partial \Omega , {\mbox{ \rm{and in every space} } L_p(\partial \Omega ),\, p\in (1,\infty )},
\end{align}
where $\boldsymbol \nu^{(j)}$ is the outward unit normal to $\partial \Omega _j$, and $\boldsymbol \nu $ is the outward unit normal to $\partial \Omega $.
\item[$(iv)$]
There exist some positive functions $\omega_j:\partial \Omega \to {\mathbb R}$ $($the Jacobian related to $\Phi_j$, $j\in {\mathbb N}$$)$ bounded away from zero and infinity uniformly in $j$, such that, for any measurable set $A\subset \partial \Omega $, $\int _A\omega _jd\sigma =\int _{\Phi_j(A)}d\sigma _j$. In addition, $\lim _{j\to \infty }\omega _j=1$ a.e. on $\partial \Omega $ and in every space $L_p(\partial \Omega )$, $p\in (1,\infty )$.
%\item[$(v)$]
%There exists a real-valued, smooth, compactly supported vector field $\boldsymbol \Theta $ in ${\mathbb R}^n$ such that for all $j$ and ${\bf x}\in \partial \Omega $ we have $\langle \boldsymbol \Theta (\Phi_j({\bf x})),\boldsymbol \nu (\Phi_j({\bf x}))\rangle \geq C>0$, where $C$ depends only on $\Theta $ and the Lipschitz character of $\Omega$.
\end{itemize}
\end{lemma}
%Moreover, the Lipschitz constants of $\Phi_j$ and $\Phi_j^{-1}$ are uniformly bounded in $j$.
%In addition, there exist some positive functions $\omega_j:\partial \Omega \to {\mathbb R}$ $($the Jacobian related to $\Phi_j$, $j\in {\mathbb N}$$)$ bounded away from zero and infinity uniformly in $j$, such that, for any measurable set $A\subset \partial \Omega $, $\int _A\omega _jd\sigma =\int _{\Phi_j(A)}d\sigma _j$. Also, $\lim _{j\to \infty }\omega _j=1$ a.e. on $\partial \Omega $ and in every space $L_p(\partial \Omega )$, $p\in (1,\infty )$. %$($cf. \cite[p. 80]{Necas2012}, \cite[Theorem\, 1.12 (iv)]{Verchota1984}$)$.

{Lemma \ref{2.13D} implies that the Lipschitz characters of the domains $\Omega _j$ are uniformly controlled by the Lipschitz character of $\Omega $.}
The meaning of {Lipschitz character} of a Lipschitz domain is given below (cf., e.g., \cite[p. 22]{M-M1}).
\begin{definition}\label{LCh}
Let $\Omega\subset {\mathbb R}^n$ be a Lipschitz domain. Let $\{{\mathcal C}_k(r_k,h_k):1\leq k\leq N\}$ $($with associated Lipschitz functions $\{\varphi _k:1\leq k\leq N\}$$)$ be an atlas for $\partial \Omega $, i.e., a finite collection of cylinders covering the boundary $\partial \Omega $.
Having fixed such an atlas of $\partial \Omega $, the {{\it Lipschitz character}} of $\Omega$ is defined as the set consisting of the numbers $N$, $\max \{\|\nabla \varphi _k\|_{L_{\infty }({\mathbb R}^{n-1})}:1\leq k\leq N\}$, $\min \{r_k:1\leq k\leq N\}$, and $\min \{h_k:1\leq k\leq N\}$.
\end{definition}
%{\rd The existence of a sequence of approximating domains $\{\Omega_j\}_{j\in {\mathbb N}}$ as in Lemma \ref{2.13D} follows from \cite[Theorem\, 1.12]{Verchota1984}.}

\subsection{\bf Sobolev and Besov spaces and related results}
%Let ${\Omega}\subset \mathbb{R}^{n}$ $(n\geq 3)$ be a bounded Lipschitz domain, and set ${\Omega}_{+}:={\Omega}$ and ${\Omega}_{-}:= \mathbb{R}^{n}\setminus \overline{{\Omega}}$.
In this subsection we assume $n\ge 2$. We denote by ${\mathcal D}({\mathbb R}^n):=C^{\infty }_{{\rm{comp}}}({\mathbb R}^n)$ the space of infinitely differentiable functions with compact support in ${\mathbb R}^n$ and by ${\mathcal D}({\mathbb R}^n,{\mathbb R}^n):=C^{\infty }_{{\rm{comp}}}({\mathbb R}^n,{\mathbb R}^n)$ the space of infinitely differentiable vector-valued functions with compact support in ${\mathbb R}^n$.
Also, let
${\mathcal E}({\Omega}_{\pm }):=C^{\infty }({\Omega}_{\pm })$ denote the space of infinitely differentiable functions and let ${\mathcal D}({\Omega}_{\pm }):=C^{\infty }_{{\rm{comp}}}({\Omega}_{\pm })$ be the space of infinitely differentiable functions with compact support in ${\Omega}_{\pm }$, equipped with the inductive limit topology. Let\footnote{If $X$ is a topological space, then $X'$ denotes its dual.} ${\mathcal E}'({\mathbb R}^n)$ and ${\mathcal D}'({\mathbb R}^n)$ be the duals of ${\mathcal E}({\mathbb R}^n)$ and ${\mathcal D}({\mathbb R}^n)$, respectively, i.e., the spaces of distributions on ${\mathbb R}^n$. The spaces ${\mathcal E}'({\Omega}_{\pm })$ and ${\mathcal D}'({\Omega}_{\pm })$ can be similarly defined.

Let ${\mathcal F}$ denote the Fourier transform defined on the space of tempered distributions to itself, and ${\mathcal F}^{-1}$ be its inverse. For $p \in (1, \infty)$, $L_p({\mathbb R}^n)$ is the Lebesgue space of (equivalence classes of) measurable, $p^{th}$ integrable functions on ${\mathbb R}^n$, and $L_{\infty }({\mathbb R}^n)$ is the space of (equivalence classes of) essentially bounded measurable functions on ${\mathbb R}^n$. For $s\in {\mathbb R}$, the $L_p$-based {Bessel potential} spaces  $H_p^{s}(\mathbb{R}^{n})$ and $H_p^{s}(\mathbb{R}^{n},{\mathbb R}^n)$ are defined by
\begin{align}
\label{bessel-potential}
&H_p^{s}(\mathbb{R}^{n}):=\{{f:\,}(\mathbb{I} - \triangle )^{\frac{s}{2}}{f\in }L_p(\mathbb{R}^{n}) \}
=\{{f: J^s f\in} L_p(\mathbb{R}^{n})\},\\
&H_p^{s}(\mathbb{R}^{n}, \mathbb{R}^{n}):= \left\{{\bf \tilde f}=(f_{1}, f_{2},\ldots , f_{n}) : f_{i} \in H_p^{s}(\mathbb{R}^{n}),\ j = 1,\ldots, n \right\},
\end{align}
{where $J^s:{\mathcal S}'({\mathbb R}^n)\to {\mathcal S}'({\mathbb R}^n)$ is the Bessel potential operator of order $s$ defined by $J^s f =\mathcal F^{-1} ({\rho}^s\mathcal F f)$ with
\begin{align}\label{Jhat}
{\rho}(\xi)=(1+|\xi|^2)^{\frac{1}{2}}
\end{align}
(see, e.g., \cite[Chapter 3]{Lean}).} Note that $H_p^{s}(\mathbb{R}^{n})$ is a Banach space with respect to the norm
\begin{equation}{\label{E17}}
\|f\|_{H_p^{s}(\mathbb{R}^{n})}
{= \|J^s f \|_{L_p(\mathbb{R}^{n})} }
= \|\mathcal{F}^{-1}{({\rho}^{s}} \mathcal{F}f{)} \|_{L_p(\mathbb{R}^{n})}.
\end{equation}
%For ${s\ge0}$, the spaces $H_p^{s}(\mathbb{R}^{n})$ coincide with the Sobolev-Slobodetskii spaces $W_p^s({\mathbb R}^n)$ (see, e.g., \cite[Chapter 4]{H-W}). In particular, for $s$ integer, such a space is a Sobolev space.
For integer ${s\ge0}$, the spaces $H_p^{s}(\mathbb{R}^{n})$ coincide with the Sobolev spaces $W_p^s({\mathbb R}^n)$.

The {Bessel potential} spaces $H_p^{s}({\Omega})$ and $\widetilde H_p^s({\Omega})$ are defined by
\begin{align}
\label{Sobolev-1}
&H_p^s({\Omega}):=\{f\in {\mathcal D}'({\Omega}):\exists \ F\in H_p^s({\mathbb R}^n)
\mbox{ such that } F|_{{\Omega}}=f\},\\
\label{Sobolev-2}
&\widetilde H_p^s({\Omega}):=\left\{f\in H_p^s({\mathbb R}^n):{\rm{supp}}\
f\subseteq \overline{{\Omega}}\right\},
\end{align}
and the {Bessel potential} spaces $H_p^s({\Omega},{\mathbb R}^n)$ and $\widetilde H_p^s({\Omega},{\mathbb R}^n)$ are defined as the spaces of vector-valued functions (distributions) whose components belong to the spaces $H_p^s({\Omega})$ and $\widetilde H_p^s({\Omega})$, respectively (see, e.g., \cite{Lean}). For any $s\in {\mathbb R}$, $C^{\infty }(\overline{{\Omega}})$ is dense in $H_p^s({\Omega})$ and the following duality relations hold (see \cite[Proposition 2.9]{J-K1}, \cite[(1.9)]{Fa-Me-Mi}, \cite[(4.14)]{M-T2})
\begin{equation}
\label{duality-spaces}
\left(H_p^s({\Omega})\right)'=\widetilde H_{p'}^{-s}({\Omega}),\ \ H_{p'}^{-s}({\Omega})=\left(\widetilde H_{p}^{s}({\Omega})\right)'.
\end{equation}
{Here and further on $p,p'\in (1,\infty)$ are related as}
%\begin{align}
%\label{dual-exp}
$\dfrac{1}{p}+\dfrac{1}{p'} = 1.$
%\end{align}

Replacing ${\Omega}$ by ${\Omega}_{-}$ in \eqref{Sobolev-1} and \eqref{Sobolev-2}, one obtains the {Bessel potential} spaces $H_p^s({\Omega}_{-})$, $\widetilde H_p^s({\Omega}_{-})$.

{For $p\in (1,\infty)$ and $s\in ({-1},1)$, the boundary {Bessel potential} space $H_p^{s}(\partial {\Omega})$ can be defined by using the space $H_p^{s}(\mathbb{R}^{n-1})$, a partition of unity and pull-pack. In addition, $H_{p'}^{-s}(\partial {\Omega})=\left(H_p^s(\partial {\Omega})\right)'$.}
{We can also equivalently define $H^0_p(\partial {\Omega})= L_p(\partial {\Omega})$ as} the Lebesgue space of measurable, $p^{th}$ power integrable functions on $\partial {\Omega}$. In addition, {$H^1_p(\partial {\Omega})$ coincides, with equivalent norm, with the Sobolev space}
\begin{align}
\label{L1-p}
&W^1_p(\partial {\Omega}):=\left\{f\in L_p(\partial {\Omega}):\|f\|_{W^1_p(\partial {\Omega})}<\infty \right\},\nonumber\\ 
&\|f\|_{W^1_p(\partial {\Omega})}:=\|f\|_{L_p(\partial {\Omega})}+\|\nabla _{{\rm{tan}}}f\| _{L_p(\partial {\Omega})}.
\end{align}
Here the weak tangential gradient of a function $f$ locally integrable on $\partial\Omega$ is  $\nabla _{{\rm{tan}}}f:=\left(\nu _k\partial_{\tau _{kj}}f\right)_{1\leq j\leq n}$, where  $\partial_{\tau _{kj}}f$ is defined in the weak form as (cf. e.g., \cite[(2.9)]{M-W})
$\langle \partial_{\tau _{kj}}f,\phi\rangle_{\partial\Omega}:=-\langle f,\partial_{\tau _{kj}}\phi\rangle_{\partial\Omega}$ for any $\phi\in\mathcal D(\mathbb R^n)$ with
$\partial _{\tau _{kj}}\phi :=\nu _k\left(\partial _j\phi \right)|_{\partial {\Omega}}-\nu _j\left(\partial _k\phi \right)|_{\partial {\Omega}},\ j,k=1,\ldots ,n,$
and $\boldsymbol \nu =(\nu _1,\ldots ,\nu _n)$ is the outward unit normal to ${\Omega}$, which exists at almost every point on $ \partial {\Omega}$.
If $f$ is defined and smooth enough in the vicinity of $\partial\Omega$, then by integrating by parts it is possible to show that the weak definition coincides with the strong one, given by $\partial _{\tau _{kj}}f :=\nu _k\left(\partial _j f \right)|_{\partial {\Omega}}-\nu _j\left(\partial _k f \right)|_{\partial {\Omega}}$.
\comment{
Let us also set
\begin{align}
\label{duap-L1-p}
L_{-1}^p(\partial {\Omega}):=\left(L_1^{p'}(\partial {\Omega})\right)',
\end{align}
where $p,p'\in (1,\infty )$ satisfy \eqref{dual-exp}.
} %\comment end

\comment{
***SM: the following paragraph is already covered by the remark below \eqref{real-int} about \eqref{complex-int}.

Let $X$ and $Y$ be two Banach spaces and let $[X,Y]_{\theta}$ denote their complex interpolation. A main property of the Sobolev spaces $H^{\pm s}_p({\Omega})$, $\widetilde H^{\pm s}_p({\Omega})$ is that they are complex interpolation spaces for $s>0$ and $p\in (1,\infty )$. Hence, for $p_{1}, p_{2}\in (1,\infty )$, $s_{1}, s_{2}\in (0,\infty )$ and $\theta \in (0,1)$, we have that
\begin{equation}
[H^{s_1}_{p_1}({\Omega}), H^{s_2}_{p_2}({\Omega})]_{\theta} = H_p^{s}({\Omega}),\ \
[H^{-s_1}_{p_1}({\Omega}), H^{-s_2}_{p_2}({\Omega}) ]_{\theta} = H_p^{-s}({\Omega}),
\end{equation}
where $s=(1-\theta )s_{1}+\theta s_{2}$ and $\frac{1}{p}=\frac{1-\theta}{p_{1}}+\frac{\theta}{p_{2}}$ (see, e.g., \cite[Theorem 13.6.3]{Agr}, \cite{Triebel}).
} %\comment end

Now, for $s\in {\mathbb R}$ and $p,q\in (1,\infty )$, denote by $B_{p,q}^{s}({\mathbb R}^n)$ the scale of Besov spaces in ${\mathbb R}^n$, see Appendix A. Similar to \eqref{Sobolev-1} and \eqref{Sobolev-2}, the Besov spaces $B_{p,q}^{s}({\Omega})$ and $B_{p,q}^{s}({\Omega},{\mathbb R}^n)$ are defined by
\begin{align}
&B_{p,q}^{s}({\Omega}):=\{f\in {\mathcal D}'({\Omega}):\exists \ F\in B_{p,q}^{s}({\mathbb R}^n)
\mbox{ such that } F|_{{\Omega}}=f\},\\
&B_{p,q}^{s}({\Omega},\mathbb{R}^{n}):=\big\{{\bf f}=(f_{1}, f_{2}, \ldots , f_{n}) : f_{i} \in B_{p,q}^{s}({\Omega}), j = 1,\ldots, n \big\},\\
&\widetilde B_{p,q}^{s}({\Omega}):=\left\{f\in B_{p,q}^{s}({\mathbb R}^n):{\rm{supp}}\
f\subseteq \overline{{\Omega}}\right\}.
\end{align}

{For $s\in [0,1]$ and $p,q\in (1,\infty )$, the {Sobolev and Besov spaces $H_{p}^{s}({\partial\Omega})$ and $B_{p,q}^{s}({\partial\Omega})$ on the boundary $\partial\Omega$} can be defined by using the spaces
$H_{p}^{s}({\mathbb R}^{n-1})$ and $B_{p,q}^{s}({\mathbb R}^{n-1})$, a partition of unity and the pull-backs of the local parametrization of $\partial \Omega $. In addition, we note that $H_{p}^{-s}(\partial\Omega )=\left({H}_{p'}^s(\partial\Omega )\right)'$ and $B_{p,q}^{-s}=\left(B_{p',q'}^s(\partial\Omega  )\right)'$, where $p',q'\in (1,\infty )$ such that $\frac{1}{p}+\frac{1}{p'}=1$ and $\frac{1}{q}+\frac{1}{q'}=1$ (for further details about boundary Sobolev and Besov spaces see, e.g., \cite[p. 35]{M-W}).}

%{For further properties of Sobolev and Besov spaces on bounded Lipschitz domains and Lipschitz boundaries, we refer to \cite{J-K1,Lean,M-T2,M-W,Triebel}.}

%{Spaces $B_{p,q}^{s}({\partial\Omega})$ and $H_{p}^{s}({\partial\Omega})$ have to be defined.}

A useful result for the problems we are going to investigate in this paper is the following trace lemma (see \cite[Chapter VIII, Theorems 1,2]{Jonsson-Wallin1984},  \cite[Theorem 3.1]{J-K1} and also \cite[Lemma 3.6]{Co} for the case $p=2$ and a discussion on the critical smoothness index $s=1$).
\comment{
\begin{lemma}
\label{lem 1.5}
Assume that ${\Omega}_{+}:={\Omega}\subset {\mathbb R}^n$
is a bounded Lipschitz domain with connected boundary $\partial {\Omega}$ and let ${\Omega}_{-}:={\mathbb R}^n\setminus \overline{{\Omega}}$ be the corresponding exterior domain. Let $p\in (1, \infty)$ and $s\in (0,1)$ be given.
Then there exist linear and continuous Gagliardo trace operators
%\footnote{The trace operator defined on {Bessel potential} spaces of vector fields on the domain ${{\Omega}}_{\pm }$ is also denoted by ${\gamma}_{\pm }$.}
${\gamma}_{\pm }:H_p^{s+ \frac{1}{p}}({\Omega})\to {B_{p,p}^{s}}(\partial {\Omega})$ such that ${\gamma}_{\pm } f=f|_{{{\partial {\Omega}}}}$ for any $f\in C^{\infty }(\overline{{\Omega}}_{\pm })$.
These operators are surjective and have $($non-unique$)$ linear and continuous right inverse operators ${\gamma}^{-1}_{\pm }:{B_{p,p}^{s}}(\partial {\Omega})\to H_p^{s+ \frac{1}{p}}({\Omega}).$
\end{lemma}
The result below is an extension of Lemma \ref{lem 1.5} to Besov spaces (see, e.g., \cite[Theorem 2.5.2]{M-W}).
} %\comment end
\begin{lemma}
\label{trace-lemma-Besov}
Assume that ${\Omega}\subset {\mathbb R}^n$ is a bounded Lipschitz domain with connected boundary $\partial {\Omega}$ and let ${\Omega}_{-}:={\mathbb R}^n\setminus \overline{{\Omega}}$ be the corresponding exterior domain.
Let $p,q\in (1,\infty)$ and $s\in (0,1)$. Then there exist linear and continuous Gagliardo trace  operators ${\gamma}_{\pm }:H_p^{s+ \frac{1}{p}}({\Omega}_\pm )\to {B_{p,p}^{s}}(\partial {\Omega})$ and ${\gamma}_{\pm}:B_{p,q}^{s+\frac{1}{p}}({\Omega_\pm})\to B_{p,q}^{s}(\partial {\Omega})$, respectively, such that ${\gamma}_{\pm}f=f|_{{{\partial {\Omega}}}}$ for any $f\in C^{\infty }(\overline{{\Omega}}_\pm )$. These operators are surjective and have $($non-unique$)$ linear and continuous right inverse operators
${\gamma}^{-1}_{\pm }:{B_{p,p}^{s}}(\partial {\Omega})\to H_p^{s+ \frac{1}{p}}({\Omega}_\pm )$ and
${\gamma}^{-1}_{\pm}:B_{p,q}^{s}(\partial {\Omega})\to B_{p,q}^{s+\frac{1}{p}}({\Omega_\pm})$, respectively.
\end{lemma}
{Lemma \ref{trace-lemma-Besov} holds also for vector-valued and matrix-valued functions $f$. If $f$ is such that $\gamma_+f=\gamma_-f$, we will often write $\gamma f$.}

We have the following trace equivalence assertion.
\begin{theorem}
\label{trace-equivalence-L}
Assume that ${\Omega}\subset {\mathbb R}^n$ is a bounded Lipschitz domain with connected boundary $\partial {\Omega}$ and let ${\Omega}_{-}:={\mathbb R}^n\setminus \overline{{\Omega}}$ be the corresponding exterior domain.
Let
$p,q\in (1,\infty)$, and let ${u}\in {B}_{p,q}^{s+\frac{1}{p}}(\Omega_\pm)$ or
${u}\in {H}_{p}^{s+\frac{1}{p}}(\Omega_\pm)$ for some $s>0$. Then the Gagliardo trace $\gamma_+{u}$ is well defined on $\partial\Omega$ and, moreover,
\begin{itemize}
\item[$(i)$] if
the pointwise non-tangential trace ${u}_{\rm nt}^\pm$ exists a.e. on $\partial\Omega$, then
${u}_{\rm nt}^\pm=\gamma_\pm{u}$;
\item[$(ii)$]
%Moreover,
if the pointwise non-tangential trace ${u}_{\rm nt}^\pm$ exists a.e. on $\partial\Omega$ and $s\in(0,1)$ then ${u}_{\rm nt}^\pm=\gamma_\pm{u}\in B_{p,q}^{s}(\partial{\Omega})$;
\item[$(iii)$] if
${u}_{\rm nt}^\pm\in H_{p}^s(\partial {\Omega})$ for some $s\in (0,1]$, then
$\gamma_\pm{u}\in  H_{p}^{s}(\partial {\Omega})$ as well.
\end{itemize}
\end{theorem}
\begin{proof}
Item (i) for $0<s<1$ is implied by Theorem 8.7(iii) in \cite{BMMM2014}, while for $s\ge 1$ the equality $\gamma_\pm{u}={u}_{\rm nt}^\pm$ still applies by an imbedding argument.
Item (ii) and (iii) follow from item (i) and the well known imbedding $\gamma_\pm{u}\in B_{p,q}^{s}(\partial{\Omega})$ for $s\in(0,1)$.
\hfill\end{proof}

Further on, $\langle \cdot, \cdot \rangle _{\Omega'}$ will denote the dual form between corresponding dual spaces defined on a set $\Omega'$.
For further details about Sobolev, {Bessel potential and Besov} spaces, we refer the reader to, e.g., \cite{Agr, H-W, Lean, Triebel, Triebel-1}.

\subsection{\bf The Brinkman system and conormal derivatives in Bessel-potential and Besov spaces}\label{BsBpBs}
In this subsection we also assume $n\ge 2$.
For a couple $({\bf u},\pi)$, and a real number $\alpha\ge 0$, let us consider the linear Brinkman system (in the incompressible case)
\begin{align}
\label{BrSyss}
{\mathcal L}_{\alpha}({\textbf u},\pi) ={\mathbf f},\
{\rm{div}}\ {\bf u} = 0,
\end{align}
where the Brinkman operator is defined as
\begin{equation}
\label{conormal}
{\mathcal L}_{\alpha}({\textbf u},\pi):= \triangle {\textbf u} - \alpha{\textbf u}- \nabla \pi.
\end{equation}
When $\alpha=0$, the Brinkman operator becomes the Stokes operator.

Now, for $({\bf u},\pi)\in C^1(\overline{{\Omega}}_{\pm},{\mathbb R}^n)\times C^0(\overline{{\Omega}}_{\pm})$, such that ${\rm{div}}\ {\bf u}={0}$ in ${\Omega}_{\pm }$, we define the classical conormal derivatives (tractions) for the Brinkman (or the Stokes) system, $\mathbf t_{\alpha}^{{\rm c}\pm}({\bf u},\pi )$, by using the well-known formula
\begin{align}
\label{2.37-}
\mathbf t^{{\rm c}\pm}({\bf u},\pi ):=\left({\gamma}_{\pm }
%\left(-\pi{\mathbb I}+2{\mathbb E}({\bf u})\right)
\boldsymbol{\sigma}({\bf u},\pi)\right)
\boldsymbol \nu ,
\end{align}
where
\begin{align}
\label{classical-stress}
\boldsymbol{\sigma}({\bf u},\pi):=-\pi{\mathbb I}+2{\mathbb E}({\bf u})
\end{align}
is the stress tensor,
${\mathbb E}({\bf u})$ is the {strain rate {tensor} (symmetric part of $\nabla {\bf u}$)}, and $\boldsymbol \nu{=\boldsymbol \nu^+}$ is the outward unit normal to ${\Omega}_{ +}$, defined a.e. on $\partial {\Omega}$. Then for any function $\boldsymbol \varphi \in {\mathcal D}({\mathbb R}^n,{\mathbb R}^n)$ we obtain by {integrating} by parts the first Green identity,
\begin{align}
\label{special-case-1}
{\pm}\left\langle \mathbf t^{{\rm c}\pm}({\bf u},\pi ),
\boldsymbol \varphi \right\rangle _{_{\!\partial {{\Omega}}}}= &2\langle {\mathbb E}({\bf
u}),{\mathbb E}(\boldsymbol \varphi )\rangle _{{{\Omega}_{\pm}}}+\alpha \langle {\bf
u},\boldsymbol \varphi \rangle _{{{\Omega}_{\pm}}}-\langle \pi,{\rm{div}}\ \boldsymbol \varphi \rangle _{{{\Omega}_{\pm}}}+\left\langle {\mathcal L}_{\alpha }({\bf u},\pi ),\boldsymbol \varphi\right\rangle _{{{{\Omega}_{\pm}}}}.
\end{align}

If the non-tangential traces of the stress tensor, $\boldsymbol{\sigma}^\pm_{\rm{nt}}({\bf u},\pi)$ and the normal vector $\boldsymbol \nu$ exist at a boundary point, then the non-tangential conormal derivatives are defined at this point as
\begin{align}
\label{2.37nt}
\mathbf t_{\rm{nt}}^\pm({\bf u},\pi):=\boldsymbol{\sigma}^\pm_{\rm{nt}}\,
\boldsymbol \nu .
\end{align}

For {$s\in\mathbb R$} and $p,q\in (1,\infty )$, we consider the spaces
\begin{align}
&H_{p;{\rm{div}}}^{s}(\Omega _\pm ,\mathbb{R}^{n})=\left\{{\bf u}_\pm \in H_p^{s}(\Omega _\pm, \mathbb{R}^{n}):{\rm{div}}\ {\bf u} = 0 \mbox{ in } \Omega _\pm\right\},\\
&{{B}_{p,q,\rm div}^{s}({\Omega_\pm},\mathbb{R}^{n}):=\left\{{\bf u}_\pm \in {B}_{p,q}^{s}(\Omega _\pm, \mathbb{R}^{n}):{\rm{div}}\ {\bf u} = 0 \mbox{ in } \Omega _\pm\right\}.}
\end{align}
We need also the following spaces (cf. \cite[Definition\, 3.3]{Mikh}).
\begin{defi}\label{2.5}
Let $\Omega$ be a Lipschitz domain $($bounded or unbounded$)$.
For $s\in\mathbb R$, $p,q\in (1,\infty )$ and $t\ge -1/p'$, let us {consider} the following spaces equipped with the corresponding graphic norms:
\begin{align*}
\mathfrak{H}_{p,\rm div}^{s+\frac{1}{p},t}({\Omega},\mathcal{L}_{\alpha}):=&
\Big\{({\textbf u},\pi)\in H^{s+\frac{1}{p}}_{p}({\Omega}, \mathbb{R}^{n})\times H_{p}^{s+\frac{1}{p}-1}({\Omega}):\\
&{\mathcal L}_{\alpha}({\textbf u},\pi)=\tilde{\bf f}|_\Omega,\ \tilde{\bf f}\in \widetilde H^t_p({\Omega},\mathbb{R}^{n}) \ \mbox{ and } \ {\rm{div}} \ {\bf u} = 0 \ \mbox{ in } \ {\Omega} \Big\},\\
\|({\textbf u},\pi)\|^2_{\mathfrak{H}_{p,\rm div}^{s+\frac{1}{p},t}({\Omega},\mathcal{L}_{\alpha})}
:=&\|{\textbf u}\|^2_{H^{s+\frac{1}{p}}_{p}({\Omega}, \mathbb{R}^{n})}
+\|\pi\|^2_{H_{p}^{s+\frac{1}{p}-1}({\Omega})}
+\|\tilde{\bf f}\|^2_{\widetilde H^t_p({\Omega},\mathbb{R}^{n})},\\
\mathfrak{B}_{p,q,\rm div}^{s+\frac{1}{p},t}({\Omega},\mathcal{L}_{\alpha}):=&
\Big\{({\textbf u},\pi)\in B^{s+\frac{1}{p}}_{p,q}({\Omega}, \mathbb{R}^{n})\times B_{p,q}^{s+\frac{1}{p}-1}({\Omega}):\\
&{\mathcal L}_{\alpha}({\textbf u},\pi)=\tilde{\bf f}|_\Omega,\ \tilde{\bf f}\in
\widetilde B_{p,q}^t({\Omega},\mathbb{R}^{n}) \ \mbox{ and } \ {\rm{div}} \ {\bf u} = 0 \ \mbox{ in } \ {\Omega} \Big\},\\
\|({\textbf u},\pi)\|^2_{\mathfrak{B}_{p,q,\rm div}^{s+\frac{1}{p},t}({\Omega},\mathcal{L}_{\alpha})}
:=&\|{\textbf u}\|^2_{B^{s+\frac{1}{p}}_{p,q}({\Omega}, \mathbb{R}^{n})}
+\|\pi\|^2_{B_{p,q}^{s+\frac{1}{p}-1}({\Omega})}
+\|\tilde{\bf f}\|^2_{\widetilde B^t_{p,q}({\Omega},\mathbb{R}^{n})},
\end{align*}
where ${\mathcal L}_{\alpha}({\textbf u},\pi)$ is defined in \eqref{conormal}.
\end{defi}
If $t_1>t_2$, the following continuous embeddings hold, $\mathfrak{H}_{p,\rm div}^{s+\frac{1}{p},t_1}({\Omega},\mathcal{L}_{\alpha})\hookrightarrow \mathfrak{H}_{p,\rm div}^{s+\frac{1}{p},t_2}({\Omega},\mathcal{L}_{\alpha})$, $\mathfrak{B}_{p,q,\rm div}^{s+\frac{1}{p},t_1}({\Omega},\mathcal{L}_{\alpha})\hookrightarrow \mathfrak{B}_{p,q,\rm div}^{s+\frac{1}{p},t_2}({\Omega},\mathcal{L}_{\alpha})$.

Let ${\mathcal D}_{{\rm{div}}}(\overline\Omega,{\mathbb R}^n):=\left\{{\bf v}\in {\mathcal D}(\overline\Omega,{\mathbb R}^n):{\rm{div}}\ {\bf v}=0\ \mbox{ in }\ \Omega \right\}$. Similar to \cite[Theorem 6.9]{Mikh-3}, one can prove the following assertion.
\begin{theorem}
\label{Dens}
If  $\Omega$ is a Lipschitz domain $($bounded or unbounded$)$ or $\Omega=\mathbb R^n$, $\alpha\ge 0$, $p,q\in (1,\infty )$, $s\in\mathbb R$ and $t>-\frac{1}{p'}$, then ${\mathcal D}_{{\rm{div}}}(\overline\Omega,{\mathbb R}^n)\times \mathcal D(\overline\Omega)$ is dense in $\mathfrak{H}_{p}^{s+\frac{1}{p},t}({\Omega},\mathcal{L}_{\alpha})$ and in $\mathfrak{B}_{p,q}^{s+\frac{1}{p},t}({\Omega},\mathcal{L}_{\alpha})$.
\end{theorem}

Let $p,q\in (1,\infty)$.
%and further on let $p',q'\in (1,\infty )$  be such that $\frac{1}{p}+\frac{1}{p'}= 1$, $\frac{1}{q}+\frac{1}{q'}= 1$.
\comment{
Let {$\mathring E_\pm$} be the operator of extension of functions from $H_p^t(\Omega _{\pm })$  by zero on $\mathbb R^n\setminus \Omega_{\pm }$.
% Let $1<p<\infty$.
Following the proof of Theorem 2.16 in \cite{Mikh}, let us define the operator $\widetilde E_\pm$ on $H_p^t(\Omega_{\pm})$, for $0\le t<\frac{1}{p}$ as
$
\widetilde E_\pm:=\mathring E_\pm ,
$
and for $-\frac{1}{p'}<t<0$ as
\begin{align}
\label{ext-2}
\langle\widetilde E_\pm h,v\rangle_{\Omega _\pm}:=\langle h, \widetilde E_\pm v\rangle_{\Omega _\pm}={ \langle h, \mathring E_\pm v\rangle_{\Omega _\pm}},\
h\in H_p^t(\Omega _{\pm }),\, v\in H_{p'}^{-t}(\Omega _{\pm }).
\end{align}
Then, evidently ${\widetilde E}_\pm: H_p^t(\Omega_\pm)\to \widetilde H_p^t(\Omega_\pm)$, $-1/p'<t<1/p$, is a bounded linear {extension} operator.
} % Comment end
Let {$\mathring E_\pm$} be the operator of extension of functions defined on $\Omega _{\pm }$ by zero on $\mathbb R^n\setminus \Omega_{\pm }$.
%Let $1<p<\infty$.
Following the proof of Theorem 2.16 in \cite{Mikh}, let us define the operator $\widetilde E_\pm$ on $H_{p}^t(\Omega_{\pm})$ and $B_{p,q}^t(\Omega_{\pm})$ {as
$
\widetilde E_\pm:=\mathring E_\pm
$
for $0\le t<\frac{1}{p}$, and  as
\begin{align*}
%\label{ext-2B}
\langle\widetilde E_\pm h,v\rangle_{\Omega _\pm}:=\langle h, \widetilde E_\pm v\rangle_{\Omega _\pm}={ \langle h, \mathring E_\pm v\rangle_{\Omega _\pm}},\quad \mbox{when }-\frac{1}{p'}<t<0,
\end{align*}
{for all \
$h\in H_{p}^t(\Omega _{\pm }),\, v\in H_{p'}^{-t}(\Omega _{\pm })$, or for all
$h\in B_{p,q}^t(\Omega _{\pm }),\, v\in B_{p',q'}^{-t}(\Omega _{\pm })$, respectively.}
Then, {for $-1/p'<t<1/p$, evidently
$$
{\widetilde E}_\pm: H_{p}^t(\Omega_\pm)\to \widetilde H_{p}^t(\Omega_\pm),\quad
{\widetilde E}_\pm: B_{p,q}^t(\Omega_\pm)\to \widetilde B_{p,q}^t(\Omega_\pm)
$$
are bounded linear {extension} operators.}
Similar definition and properties hold also for vector fields.

Analogously to the corresponding definition for Petrovskii-elliptic systems in \cite[Definition 3.6]{Mikh}, we can introduce an operator $\tilde {\mathcal L}_{\alpha}$ as follows.
\begin{defi}
\label{Dce}
Let $\Omega$ be a Lipschitz domain $($bounded or unbounded$)$, $p,q\in (1,\infty)$, $s\in\mathbb R$, $t\ge -1/p'$. The operator $\tilde {\mathcal L}_{\alpha}$ mapping
\begin{itemize}
\item[$(i)$]
functions $({\textbf u},\pi)\in\mathfrak{H}_{p,\rm div}^{s+\frac{1}{p},t}({\Omega},\mathcal{L}_{\alpha})$ to the extension of the distribution ${\mathcal L}_{\alpha}({\textbf u},\pi)\in H^t_p({\Omega},\mathbb{R}^{n})$ to $\widetilde H^t_p({\Omega},\mathbb{R}^{n})$
\end{itemize}
or
\begin{itemize}
\item[$(ii)$]
functions $({\textbf u},\pi)\in \mathfrak{B}_{p,q,\rm div}^{s+\frac{1}{p},t}(\Omega;{\mathcal L}_{\alpha})$ to the extension of the distribution ${\mathcal L}_{\alpha}({\textbf u},\pi)\in B^t_{p,q}({\Omega},\mathbb{R}^{n})$ to $\widetilde B^t_{p,q}({\Omega},\mathbb{R}^{n})$,
\end{itemize}
will be called {\em the canonical extension} of the operator ${\mathcal L}_{\alpha}$.
\end{defi}

\begin{remark}
Similar to the paragraph following Definition 3.3 in \cite{Mikh}, one can prove that the canonical extensions mentioned in Definition~\ref{Dce} {exist and} are unique.
If  $p,q\in (1,\infty)$, $s\in\mathbb R$, $t\ge -1/p'$, then
\begin{align*}
&\|\tilde {\mathcal L}_{\alpha}({\textbf u},\pi)\|_{\widetilde H^t_p({\Omega},\mathbb{R}^{n})}\le
\|({\textbf u},\pi)\|_{\mathfrak{H}_{p,\rm div}^{s+\frac{1}{p},t}({\Omega},\mathcal{L}_{\alpha})}, \quad\|\tilde {\mathcal L}_{\alpha}({\textbf u},\pi)\|_{\widetilde B^t_{p,q}({\Omega},\mathbb{R}^{n})}\le
\|({\textbf u},\pi)\|_{\mathfrak{B}_{p,q,\rm div}^{s+\frac{1}{p},t}({\Omega},\mathcal{L}_{\alpha})}
%\nonumber
\end{align*}
by definition of the spaces
$\mathfrak{H}_{p,\rm div}^{s+\frac{1}{p},t}({\Omega},\mathcal{L}_{\alpha})$
and
$\mathfrak{B}_{p,q,\rm div}^{s+\frac{1}{p},t}({\Omega},\mathcal{L}_{\alpha})$. Hence the linear operators
$\tilde {\mathcal L}_{\alpha}: \mathfrak{H}_{p,\rm div}^{s+\frac{1}{p},t}({\Omega},\mathcal{L}_{\alpha})\to \widetilde H^t_p({\Omega},\mathbb{R}^{n})$
and
$\tilde {\mathcal L}_{\alpha}: \mathfrak{B}_{p,q,\rm div}^{s+\frac{1}{p},t}({\Omega},\mathcal{L}_{\alpha})\to \widetilde B^t_{p,q}({\Omega},\mathbb{R}^{n})$
are continuous.
Moreover, if $-1/p'<t<1/p$, and $\Omega$ is a Lipschitz domain $($bounded or unbounded$)$, then we have the representation $\tilde {\mathcal L}_{\alpha}:=\widetilde E^+{\mathcal L}_{\alpha}$, {or $\tilde {\mathcal L}_{\alpha}:=\widetilde E^-{\mathcal L}_{\alpha}$, respectively}, cf. \cite[Remark 3.7]{Mikh}.
\end{remark}

Formula \eqref{special-case-1} suggests the following definition of the {\em canonical} conormal derivative in the setting of Besov spaces, cf.,
\cite[Lemma 3.2]{Co},
\cite[Lemma 2.2]{K-L-W1},
%\cite[Lemma 2.2]{K-L-W4},
\cite[Definition 3.8, Theorem 3.9]{Mikh}, \cite[Definition 6.5, Theorem 6.6]{Mikh-3},
%\cite[Proposition 3.6]{M-M-W},
%\cite[Theorem 10.4.1]{M-W}).
\cite[Proposition 10.2.1]{M-W}).
}
\begin{defi}
\label{lem 1.6D}
%Let ${\Omega}\subset \mathbb{R}^{n}$ be a bounded Lipschitz domain with the boundary $ \partial {\Omega}$.
{Let} $\alpha \ge 0$, $s\in (0,1)$, $p,q\in (1,\infty)$.
Then the {\em canonical} conormal derivative {operators $\mathbf t_{\alpha}^\pm$ are} defined on any
$({\textbf u},\pi )\in \mathfrak{H}_{p,\rm div}^{s+\frac{1}{p},-\frac{1}{p'}}({\Omega_\pm},\mathcal {L}_{\alpha})$, or
$({\textbf u},\pi )\in \mathfrak{B}_{p,q,\rm div}^{s+\frac{1}{p},-\frac{1}{p'}}({\Omega_\pm},\mathcal {L}_{\alpha})$,
in the weak sense, by the formula
\begin{align}
%\label{conormal-derivative}
\pm\langle \mathbf t_{\alpha}^\pm({\bf u},\pi),\boldsymbol \varphi \rangle_{\partial {\Omega_\pm}}:=2\left\langle {\widetilde E_\pm}\mathbb{E}(\textbf{u}),\mathbb{E}({\gamma}^{-1}_\pm\boldsymbol \varphi)\right\rangle_{{\Omega_\pm}}\!+\!\alpha \langle {\widetilde E}_\pm{\textbf u},{\gamma}^{-1}_\pm\boldsymbol \varphi \rangle _{{\Omega_\pm}}\nonumber\\
-\!\left\langle {\widetilde E}_\pm\pi,{\rm{div}}({\gamma}^{-1}_\pm\boldsymbol \varphi)\right\rangle _{{\Omega_\pm}}\!
+\!\langle \tilde{\mathcal L}_{\alpha}({\textbf u},\pi),{\gamma}^{-1}_\pm\boldsymbol \varphi \rangle_{{\Omega_\pm}},
\label{conormal-can-def}\\
\forall \ \boldsymbol \varphi \in B_{p',p'}^{1-s}(\partial {\Omega},\mathbb{R}^{n}),\ \mbox{ or }\
\forall \ \boldsymbol \varphi \in B_{p',q'}^{1-s}(\partial {\Omega},\mathbb{R}^{n}),\ \mbox{respectively}.\nonumber
\end{align}
\end{defi}
Note that the {\em canonical} conormal derivative operators introduced in Definition~\ref{lem 1.6D} are different from the {\em generalized} conormal derivative operator, cf. \cite[Lemma 2.2]{K-L-W4},
\cite[Definition 3.1, Theorem 3.2]{Mikh}, \cite[Definition 5.2, Theorem 5.3]{Mikh-3}.
{Similar to \cite[Theorem 3.9]{Mikh}, one can prove the following assertion.}
\begin{lemma}
\label{lem 1.6}
Under the hypothesis of Definition \ref{lem 1.6D}, the {canonical} conormal derivative operators
\begin{align*}
\mathbf t_{\alpha}^\pm:\mathfrak{H}_{p,\rm div}^{s+\frac{1}{p},-\frac{1}{p'}}(\Omega_\pm,\mathcal {L}_{\alpha})\to B^{s-1}_{p,p}(\partial {\Omega},\mathbb{R}^{n}), \
\mathbf t_{\alpha}^\pm:\mathfrak{B}_{p,q,\rm div}^{s+\frac{1}{p},-\frac{1}{p'}}(\Omega_\pm,\mathcal {L}_{\alpha})\to B^{s-1}_{p,q}(\partial {\Omega},\mathbb{R}^{n}),
\end{align*}
are linear, bounded and independent of the choice of the operators ${\gamma}^{-1}_\pm$. In addition, the following first Green identity holds
\begin{align}
\label{Green formula}
\pm\langle \mathbf t_{\alpha}^\pm({\bf u},\pi),{\gamma}_+{\bf w} \rangle_{\partial {\Omega}}
=2\left\langle \widetilde E_\pm\mathbb{E}({\bf u}),\mathbb{E}({\textbf w})\right\rangle _{\Omega_\pm}
+\alpha \left\langle\widetilde E_\pm{\bf u},{\bf w}\right\rangle _{\Omega_\pm}
-\left\langle\widetilde E_\pm\pi ,{\rm{div}}\ {\bf w}\right\rangle _{\Omega_\pm}
+\left\langle \tilde{\mathcal L}_{\alpha}({\bf u},\pi ),{\bf w}\right\rangle _{\Omega_\pm}
\end{align}
for all $({\textbf u},\pi)\in \mathfrak{H}_{p,\rm div}^{s+\frac{1}{p},-\frac{1}{p'}}(\Omega_\pm,\mathcal {L}_{\alpha})$,
${\textbf w} \in H_{p'}^{1+\frac{1}{p'}-s}(\Omega_\pm, \mathbb{R}^{n})$
and all\\
$({\textbf u},\pi)\in \mathfrak{B}_{p,q,\rm div}^{s+\frac{1}{p},-\frac{1}{p'}}(\Omega_\pm,\mathcal {L}_{\alpha})$,
${\textbf w} \in B_{p',q'}^{1+\frac{1}{p'}-s}(\Omega_\pm, \mathbb{R}^{n})$,
and  the following second Green identity holds
\begin{align}
\label{Green-formula2}
\pm\left(\langle \mathbf t_{\alpha}^\pm({\bf u},\pi),{\gamma}_+{\bf v} \rangle_{\partial {\Omega}}
-\langle \mathbf t_{\alpha}^\pm({\bf v},q),{\gamma}_+{\bf u} \rangle_{\partial {\Omega}}\right)
=\left\langle \tilde{\mathcal L}_{\alpha}({\bf u},\pi ),{\bf v}\right\rangle _{\Omega_\pm}
-\left\langle \tilde{\mathcal L}_{\alpha}({\bf v},q ),{\bf u}\right\rangle _{\Omega_\pm}
\end{align}
for all $({\textbf u},\pi)\in \mathfrak{H}_{p,\rm div}^{s+\frac{1}{p},-\frac{1}{p'}}(\Omega_\pm,\mathcal {L}_{\alpha})$,
$({\textbf v},q) \in \mathfrak{H}_{p',\rm div}^{1+\frac{1}{p'}-s,-\frac{1}{p}}(\Omega_\pm, \mathbb{R}^{n})$
and all
$({\textbf u},\pi)\in \mathfrak{B}_{p,q,\rm div}^{s+\frac{1}{p},-\frac{1}{p'}}(\Omega_\pm,\mathcal {L}_{\alpha})$,
$({\textbf v},q) \in \mathfrak{B}_{p',q'}^{1+\frac{1}{p'}-s,-\frac{1}{p}}(\Omega_\pm, \mathbb{R}^{n})$.
\end{lemma}
%\begin{remark}
%If $({\bf u},\pi )\in C^{\infty }({\Omega},{\mathbb R}^n)\times C^{\infty }({\Omega})$ satisfies the homogeneous Brinkman system and ${\gamma}_+{\bf u}\in H^1_2(\partial {\Omega},{\mathbb R}^n)$ and $\mathbf t_{\alpha}^{+}({\bf u},\pi )\in L_2(\partial {\Omega},{\mathbb R}^n)$, then formula \eqref{Green formula} holds true (with ${\bf \tilde f}=0$).
%\end{remark}

\comment{
Consider the Dirichlet problem for the {homogeneous} Stokes system,
\begin{align}
\label{Stokes-homogeneous1}
&\triangle {\bf u}-\nabla \pi ={\bf 0},\ \ {\rm{div}}\ {\bf u}= 0 \ \mbox{ in } \ {\Omega},\\
\label{DirCond-Stokes}
& {\bf u}^+_{nt}={\bf h}_0 \ \mbox{ on } \ {\partial\Omega}.
\end{align}
Then we have the following assertion (cf. \cite[Theorems 9.2.2 and 9.2.5]{M-W} and \cite[Theorem 7.1]{M-T}).
\begin{lemma}
\label{M-H-D-1}
Let ${\Omega}\subset \mathbb{R}^{n}$ be a bounded Lipschitz domain and $p\in (1,\infty )$.
\begin{itemize}
\item[$(i)$]
Let $ \mathbf h_0\in L_{p;\nu}(\partial {\Omega},\mathbb{R}^{n})$. Then there exists $\varepsilon=\varepsilon(\partial\Omega)>0$ such that when
%\begin{align}
$
2-\varepsilon<p<\begin{cases}\infty\ \mbox{ if }\ n=2,3,\\
\frac{2(n-1)}{n-3}+\varepsilon \ \mbox{ if }\ n\ge 4
\end{cases},
$
%\end{align}
the interior Dirichlet problem \eqref{Stokes-homogeneous1}-\eqref{DirCond-Stokes} has a solution such that
${M}(\textbf{u})\in L_p(\partial {\Omega})$, which is unique up to an arbitrary additive constant in $\pi$.
Moreover, there exists a constant $C>0$ such that
%\begin{equation*}
$
\|M({\bf u})\|_{L_p(\partial {\Omega})}\le C\| \mathbf h_0\|_{L_p(\partial {\Omega},\mathbb{R}^{n})}.
$
%\end{equation*}
\item[$(ii)$]
Let $ \mathbf h_0\in H_{p;\nu}^1(\partial {\Omega},\mathbb{R}^{n})$. Then there exists $\varepsilon=\varepsilon(\partial\Omega)>0$ such that when
%\begin{align}
$
 \frac{2(n-1)}{n+1}-\varepsilon<p<2+\varepsilon,
$
%\end{align}
the interior Dirichlet problem \eqref{Stokes-homogeneous1}-\eqref{DirCond-Stokes} has a solution such that
${M}(\nabla\textbf{u}),{M}(\pi)\in L_p(\partial {\Omega})$, which is unique up to an arbitrary additive constant in $\pi$.
Moreover, there exists a constant $C>0$ such that
\begin{equation*}
\|M({\bf u})\|_{L_p(\partial {\Omega})}+\|M(\nabla{\bf u})\|_{L_p(\partial {\Omega})}
+\|M(\pi)\|_{L_p(\partial {\Omega})/\mathbb R}\le C\| \mathbf h_0\|_{H^1_p(\partial {\Omega},\mathbb{R}^{n})}.
\end{equation*}
\item[$(iii)$] Let $0\le s\le 1$ and $ \mathbf h_0\in H_{p;\nu}^s(\partial {\Omega},\mathbb{R}^{n})$. Then there exists $\varepsilon=\varepsilon(\partial\Omega)>0$ such that,
when $2-\varepsilon<p<2+\varepsilon$, for the solution from item $($i$)$ we have $\textbf{u}\in B_{p,p^*}^{s+\frac{1}{p}}({\Omega},\mathbb{R}^{n})$ and
%\begin{equation*}
$
\|{\bf u}\|_{B_{p,p^*}^{s+\frac{1}{p}}({\Omega},\mathbb{R}^{n})}\le C\| \mathbf h_0\|_{H^s_p(\partial {\Omega},\mathbb{R}^{n})},
$
where  $p^*:=\max \{p,2\}.$
%\end{equation*}
\end{itemize}
\end{lemma}
}
\begin{remark}
\label{CSD}
Similar to \cite[Remark 2.6]{K-L-M-W}, we note that by exploiting arguments analogous to those of the proof of Theorem 3.10 and the paragraph following it in \cite{Mikh}, one can see that the canonical conormal derivatives on $\partial\Omega$ can be equivalently defined as
$
{\bf t}_{\alpha}^{\pm}({\bf u},\pi)
=r_{_{\partial\Omega}}{\bf t}_{\alpha}^{\prime\pm}({\bf u},\pi).
$
Here ${\bf t}_{\alpha}^{\prime\pm}({\bf u},\pi)$ is defined by the dual form like \eqref{conormal-can-def} but only on Lipschitz subsets $\Omega'_{\pm}\subset\Omega_{\pm}$ such that $\partial \Omega \subset \partial \Omega _{\pm }'$ and closure of $\Omega _{\pm }\setminus \overline{\Omega _{\pm }'}$ coincides with $\Omega_\pm \setminus \Omega _{\pm }'$ (i.e.,  $\Omega'_{\pm}$ are some layers near the boundary $\partial\Omega$).
Moreover, such a definition is well applicable to the functions
$({\textbf u},\pi)$ from $\mathfrak{H}_{p,\rm div}^{s+\frac{1}{p},-\frac{1}{p'}}(\Omega'_\pm,\mathcal {L}_{\alpha})$ or $\mathfrak{B}_{p,q,\rm div}^{s+\frac{1}{p},-\frac{1}{p'}}(\Omega'_\pm,\mathcal {L}_{\alpha})$ that are not obliged to belong to $\mathfrak{H}_{p,\rm div}^{s+\frac{1}{p},-\frac{1}{p'}}(\Omega_\pm,\mathcal {L}_{\alpha})$ or
$\mathfrak{B}_{p,q,\rm div}^{s+\frac{1}{p},-\frac{1}{p'}}(\Omega_\pm,\mathcal {L}_{\alpha})$, respectively.
It is particularly useful for the functions $({\textbf u},\pi)$ that belong to $\mathfrak{H}_{p,\rm div}^{s+\frac{1}{p},-\frac{1}{p'}}(\overline{\Omega}_-,\mathcal {L}_{\alpha})$
or
$\mathfrak{B}_{p,q,\rm div}^{s+\frac{1}{p},-\frac{1}{p'}}(\overline{\Omega}_-,\mathcal {L}_{\alpha})$
 only locally.
\end{remark}

Now we prove the equivalence between canonical and non-tangential conormal derivatives (as well as classical conormal derivative, when appropriate).
\begin{theorem}
\label{2.13}
\label{trace-equivalence}
{Let $n\ge 2${,}
%and ${\Omega}\subset \mathbb{R}^{n}$ be a bounded Lipschitz domain,
$\alpha \ge 0$, and $p,q\in (1,\infty)$.}
\begin{itemize}
\item[$(i)$] Let $s>1$ and
$({\bf u},\pi)\in {B}_{p,q,\rm div}^{s+\frac{1}{p}}({\Omega_\pm},\mathbb{R}^{n})\times {B}_{p,q}^{s-1+\frac{1}{p}}({\Omega_\pm})$ 
or\\ 
$({\bf u},\pi)\in {H}_{p,\rm div}^{s+\frac{1}{p}}({\Omega_\pm},\mathbb{R}^{n})\times {H}_{q}^{s-1+\frac{1}{p}}({\Omega_\pm})$.
Then the classical conormal derivative $\mathbf t^{{\rm c}\pm}({\bf u},\pi)$ and the canonical conormal derivative $\mathbf t_{\alpha}^\pm({\bf u},\pi)$ are well defined and
$\mathbf t^\pm_\alpha({\bf u},\pi)=\mathbf t^{{\rm c}\pm}({\bf u},\pi)\in L_{p}(\partial {\Omega}, \mathbb{R}^{n})$.

If, moreover, the non-tangential trace of the stress tensor,
$\boldsymbol{\sigma}^\pm_{\rm{nt}}({\bf u},\pi)$, exists a.e. on $\partial\Omega$, then the non-tangential conormal derivative, defined by \eqref{2.37nt}, also exists a.e. on $\partial\Omega$ and $\mathbf t_{\rm{nt}}^\pm({\bf u},\pi)=\mathbf t^\pm_\alpha({\bf u},\pi)=\mathbf t^{{\rm c}\pm}({\bf u},\pi)\in L_{p}(\partial {\Omega}, \mathbb{R}^{n})$.
\item[$(ii)$]
Let $0<s\le 1$, $({\bf u},\pi)\in \mathfrak{B}_{p,q,\rm div}^{s+\frac{1}{p},t}({ \Omega_\pm},\mathcal{L}_{\alpha})$
{or $({\bf u},\pi)\in \mathfrak{H}_{p,\rm div}^{s+\frac{1}{p},t}({ \Omega_\pm},\mathcal{L}_{\alpha})$}, for some $t>-\frac{1}{p'}$. Let also assume that the non-tangential maximal function $M(\boldsymbol{\sigma}({\bf u},\pi))$ and the non-tangential trace of the stress tensor, $\boldsymbol{\sigma}^\pm_{\rm{nt}}({\bf u},\pi)$, {exist and are finite a.e.} on $\partial\Omega$ and belong to the space $L_{p}(\partial {\Omega}, \mathbb{R}^{n\times n})$.
Then $\mathbf t^\pm_\alpha({\bf u},\pi)=\mathbf t_{\rm{nt}}^\pm({\bf u},\pi)\in L_{p}(\partial {\Omega}, \mathbb{R}^{n})$.
\end{itemize}
\end{theorem}
\begin{proof}
We will give a proof in the case of {a bounded domain $\Omega_+$ and the Besov spaces}. For an unbounded domain $\Omega_-$ {and the Bessel potential spaces} the arguments are the same.

(i) Let $({\bf u},\pi)\in {B}_{p,q,\rm div}^{s+\frac{1}{p}}({\Omega_+}, \mathbb{R}^{n})\times {B}_{p,q}^{s-1+\frac{1}{p}}({\Omega_+})$ for some $p,q\in (1,\infty)$ and $s>1$. Evidently, the stress tensor $\boldsymbol{\sigma}({\bf u},\pi)$ belongs to ${B}_{p,q}^{s-1+\frac{1}{p}}(\Omega, \mathbb{R}^{n\times n})$, which for $1<s<2$ implies that  $\gamma_{-}\boldsymbol{\sigma}({\bf u},\pi)\in  {B}_{p,q}^{s-1}(\partial {\Omega}, \mathbb{R}^{n\times n})\subset L_p(\partial {\Omega}, \mathbb{R}^{n\times n})$.
Taking into account that the unit normal vector to the boundary, $\boldsymbol \nu$, belongs to $L_{\infty}(\partial {\Omega}, \mathbb{R}^{n})$, we obtain by \eqref{2.37-} that
$\mathbf t^{{\rm c}+}({\bf u},\pi)\in L_{p}(\partial {\Omega}, \mathbb{R}^{n})$.

On the other hand, the inclusion $({\bf u},\pi)\in {B}_{p,q,\rm div}^{s+\frac{1}{p}}({\Omega_+})\times {B}_{p,q}^{s-1+\frac{1}{p}}({\Omega_+})$ for $p,q\in (1,\infty)$ and $s>1$ implies that
$({\textbf u},\pi)\in \mathfrak{B}_{p,q,\rm div}^{s+\frac{1}{p},t}({\Omega_+},\mathcal{L}_{\alpha})$
for  $t\in(-1/p', s-1-1/p')$ and thus the canonical conormal derivative $\mathbf t^+_\alpha({\bf u},\pi)$ is well defined and belongs to $B^{s'-1}_{p,q}(\partial {\Omega},\mathbb{R}^{n})$ for any $s'\in (0,1)$.
For $1<s<2$, the proof that $\mathbf t^+_\alpha({\bf u},\pi)=\mathbf t_{\alpha}^{{\rm c}+}({\bf u},\pi)\in L_{p}(\partial {\Omega}, \mathbb{R}^{n})$ is similar to \cite[Corollary 3.14]{Mikh} (with evident modification to $L_p$-based spaces), while for $s\ge 2$ the {relation} $\mathbf t^+_\alpha({\bf u},\pi)=\mathbf t^{{\rm c}+}({\bf u},\pi)\in L_{p}(\partial {\Omega}, \mathbb{R}^{n})$ still {stays} by imbedding.

If, in addition, the non-tangential trace of the stress,
$\boldsymbol{\sigma}^+_{\rm nt}({\bf u},\pi)$, exists a.e. on $\partial\Omega$, then $\boldsymbol{\sigma}^+_{\rm nt}({\bf u},\pi)=\gamma^+\boldsymbol{\sigma}({\bf u},\pi)$ by Theorem \ref{trace-equivalence-L}(i) implying that $\mathbf t_{\rm{nt}}^+({\bf u},\pi)=\mathbf t^+_\alpha({\bf u},\pi)=\mathbf t^{{\rm c}+}({\bf u},\pi)\in L_{p}(\partial {\Omega}, \mathbb{R}^{n})$.

(ii) Let $0<s< 1$ first, and the case $s=1$ will follow by inclusion. Under the other hypotheses of item (ii), the canonical conormal derivative, $\mathbf t^+_\alpha({\bf u},\pi)$, is well defined on the boundary $\partial\Omega$ and belongs to  ${B}_{p,q}^{s-1}({\partial\Omega},\mathbb R^n)$.
Let $\{\Omega _j\}_{j\geq 1}$ be a sequence of sub-domains in $\Omega _{+}$ that converge to $\Omega _{+}$ in the sense of Lemma \ref{2.13D}, with the corresponding notations $\Phi_j$, $\nu^{(j)}$ and  $\omega_j$ also introduced there.

Similar to the proof of Lemma 3.15 in \cite{Mikh}, one can now prove that
the canonical conormal derivative on $\partial\Omega$ is a limit of the canonical conormal derivatives on $\partial\Omega_j$, i.e., $\langle \mathbf t^+_{\alpha,\partial \Omega}({\bf u},\pi),\gamma_{_{\partial\Omega +}}{\bf w} \rangle_{\partial\Omega}
=\lim_{j\to\infty}\langle \mathbf t^{+}_{\alpha,\partial \Omega_j}({\bf u},\pi),\gamma_{_{\partial\Omega_j}}{\bf w} \rangle_{\partial\Omega_j}$
for any ${\textbf w} \in B_{p',q'}^{1+\frac{1}{p'}-s}({\Omega_+}, \mathbb{R}^{n})$.

The inclusion $({\bf u},\pi)\in \mathfrak{B}_{p,q,\rm div}^{s+\frac{1}{p},t}({ \Omega_+},\mathcal{L}_{\alpha})$ means that the couple $({\bf u},\pi)$ satisfies the elliptic Brinkman PDE system \eqref{BrSyss} with a right hand side $\mathbf f\in {B}_{p,q}^{t}({{\Omega_+}, \mathbb{R}^{n}})$, which implies that $({\bf u},\pi)\in {B}_{p,q,\rm div}^{t+2}({\Omega_j})\times {B}_{p,q}^{t+1}({\Omega_j})$.\\
Then $\gamma_{\partial\Omega_j}\boldsymbol{\sigma}({\bf u},\pi)\in{B}_{p,q}^{t+1-\frac{1}{p}}(\partial \Omega_j, \mathbb{R}^{n\times n})\subset L_p({\partial \Omega_j},\mathbb{R}^{n\times n})$ and
$ %\mathbf t_{{\rm{nt}}\partial \Omega_j}^+({\bf u},\pi)=
\mathbf t^+_{\alpha,\partial \Omega_j}({\bf u},\pi)
=\mathbf t_{\partial \Omega_j}^{{\rm c}+}({\bf u},\pi)
=\gamma^+_{\partial\Omega_j}\boldsymbol{\sigma}({\bf u},\pi){\boldsymbol\nu}
%|_{\partial\Omega_j}
\in L_{p}(\partial {\Omega_j}, \mathbb{R}^{n})$
by item (i).

On the other hand, for a.e. point $x\in\partial\Omega$ {the non-tangential function $M(\boldsymbol{\sigma}({\bf u},\pi))(x)$ exists and is finite}, which particularly implies that $\boldsymbol{\sigma}({\bf u},\pi)$ is well defined and bounded in {the approach cones} ${\mathfrak D}_{+}(x)$.
We can consider  $\boldsymbol{\sigma}({\bf u},\pi)(x)$ as strictly defined (by its limit mean values $\lim_{r\to 0}\fint_{B(x,r)}\boldsymbol{\sigma}({\bf u},\pi)(\xi)d\xi$ in the sense of Jonnson \& Wallin \cite[p.15]{Jonsson-Wallin1984}, see also \cite[Theorem 8.7]{BMMM2014});
then $\gamma_{\partial\Omega_j}\boldsymbol{\sigma}({\bf u},\pi)(y)=\boldsymbol{\sigma}({\bf u},\pi)(y)$ and hence
$\mathbf t^+_{\alpha,\partial \Omega_j}({\bf u},\pi)(y)
=\mathbf t_{\partial \Omega_j}^{{\rm c}+}({\bf u},\pi)(y)
=\boldsymbol{\sigma}({\bf u},\pi)(y)\cdot\boldsymbol{\nu}_{j}(y)$ for $y\in {\mathfrak D}_{+}(x)\cap \partial\Omega_j$.
In addition $\mathbf t^+_{\alpha,\partial \Omega_j}({\bf u},\pi)(\Phi_j(x))
=\mathbf t_{\partial \Omega_j}^{{\rm c}+}({\bf u},\pi)(\Phi_j(x))
=\boldsymbol{\sigma}({\bf u},\pi)(\Phi_j(x))\cdot\boldsymbol{\nu}(\Phi_j(x))$ tends to $\boldsymbol{\sigma}^+_{\rm nt}({\bf u},\pi)(x)\cdot\boldsymbol{\nu}(x)
=\mathbf t_{{\rm{nt}},\partial \Omega}^+({\bf u},\pi)(x)$ as $j\to\infty$ for a.e. $x\in \partial\Omega$, for which $\boldsymbol{\sigma}^+_{\rm nt}({\bf u},\pi)(x)$ does exist.

Let us now prove that $\mathbf t_{\partial \Omega_j}^{{\rm c}+}({\bf u},\pi)(\Phi_j(x))$ converges to $\mathbf t_{{\rm{nt}},\partial \Omega}^+({\bf u},\pi)(x)$ not only point-wise for a.e. $x\in \partial\Omega$ but also in the weak sense, i.e.,
$\lim_{j\to\infty}\langle \mathbf t^{{\rm c}+}_{\partial \Omega_j}({\bf u},\pi),
\gamma_{_\partial\Omega_j}{\bf w} \rangle_{\partial\Omega_j}=
\langle \mathbf t_{{\rm{nt}},\partial \Omega}^+({\bf u},\pi),
\gamma_{_{\partial\Omega+}}{\bf w} \rangle_{\partial\Omega}
$
for any ${\textbf w} \in B_{p',q'}^{1+\frac{1}{p'}-s}({\Omega_+}, \mathbb{R}^{n})$.
%Let $\omega_j:\partial\Omega\to \mathbb R_+$ be the transformation Jacobian (related with $\Phi_j$) bounded away from zero and infinity uniformly in $k$; then $\omega_j\to 1$ pointwise a.e. and in $L_q(\partial\Omega$ for any $q\in[1,\infty)$, cf. \cite[Theorem 1.12(iv)]{Verchota1984}.
We have
\begin{multline}\label{34}
|\langle \mathbf t^{{\rm c}+}_{\partial \Omega_j}({\bf u},\pi),
\gamma_{_{\partial\Omega_j}}{\bf w} \rangle_{\partial\Omega_j}-
\langle \mathbf t_{{\rm{nt}},\partial \Omega}^+({\bf u},\pi),
\gamma_{_{\partial\Omega+}}{\bf w} \rangle_{\partial\Omega}|\\
=
|\langle \mathbf t^{{\rm c}+}_{\partial \Omega_j}({\bf u},\pi)\circ\Phi_j,
\omega_j\gamma_{_{\partial\Omega_j}}{\bf w}\circ\Phi_j \rangle_{\partial\Omega}-
\langle \mathbf t_{{\rm{nt}},\partial \Omega}^+({\bf u},\pi),
\gamma_{_{\partial\Omega+}}{\bf w} \rangle_{\partial\Omega}|\\
\le |\langle \mathbf t^{{\rm c}+}_{\partial \Omega_j}({\bf u},\pi)\circ\Phi_j
-\mathbf t_{{\rm{nt}},\partial \Omega}^+({\bf u},\pi),
\omega_j\gamma_{_{\partial\Omega_j}}{\bf w}\circ\Phi_j \rangle_{\partial\Omega}|\\
+|\langle \mathbf t_{{\rm{nt}},\partial \Omega}^+({\bf u},\pi),
(\omega_j-1)\gamma_{_{\partial\Omega_j}}{\bf w}\circ\Phi_j \rangle_{\partial\Omega}|
+|\langle \mathbf t_{{\rm{nt}},\partial \Omega}^+({\bf u},\pi),
\gamma_{_{\partial\Omega_j}}{\bf w}\circ\Phi_j -
\gamma_{_{\partial\Omega+}}{\bf w} \rangle_{\partial\Omega}|.
\end{multline}
Let us prove that the summands in the right hand side of \eqref{34} tend to zero as $j\to\infty$. To this end, we use the inequality
\begin{multline}\label{35}
|\langle \mathbf t^{{\rm c}+}_{\partial \Omega_j}({\bf u},\pi)\circ\Phi_j
\!-\!\mathbf t_{{\rm{nt}},\partial \Omega}^+({\bf u},\pi),
\omega_j\gamma_{_{\partial\Omega_j}}{\bf w}\circ\Phi_j \rangle_{\partial\Omega}|
\!\\
\le
\| \mathbf t^{{\rm c}+}_{\partial \Omega_j}({\bf u},\pi)\circ\Phi_j
\!-\mathbf t_{{\rm{nt}},\partial \Omega}^+({\bf u},\pi)\|_{L_p(\partial\Omega)}\
\|\omega_j\gamma_{_{\partial\Omega_j}}{\bf w}\circ\Phi_j \|_{L_{p'}(\partial\Omega)}.
\end{multline}
We have,
\begin{align}\label{36}
|\mathbf t^{{\rm c}+}_{\partial \Omega_j}({\bf u},\pi)(\Phi_j(x))
-\mathbf t_{{\rm{nt}},\partial \Omega}^+({\bf u},\pi)(x)|
\le M(\boldsymbol{\sigma}({\bf u},\pi))(x)+|\mathbf t_{{\rm{nt}},\partial \Omega}^+({\bf u},\pi)(x)|,
\end{align}
the both terms in the right hand side of \eqref{36} belong to $L_p(\partial\Omega)$ and $\mathbf t^{{\rm c}+}_{\partial \Omega_j}({\bf u},\pi)\circ\Phi_j
-\mathbf t_{{\rm{nt}},\partial \Omega}^+({\bf u},\pi)\to 0$ pointwise a.e. on $\partial\Omega$.
Then the Lebesgue dominated convergence theorem implies that the first multiplier in the right hand side of \eqref{35} tends to zero.
Since $\gamma_{_{\partial\Omega_j}}{\bf w}\in B_{p',q'}^{1-s}({\partial\Omega_j}, \mathbb{R}^{n})\subset L_{p'}^{1-s}({\partial\Omega_j}, \mathbb{R}^{n})$ and
$\gamma_{_{\partial\Omega_j}}{\bf w}\circ\Phi_j \to \gamma_{_{\partial\Omega+}}{\bf w}$ (cf. \cite[Chapter 2, Theorem 4.5]{Necas2012}), the second multiplier in the right hand side of \eqref{35} is bounded and hence the whole right hand side of \eqref{35} tends to zero.
The second summand in the right hand side of \eqref{34} tends to zero since $\omega_j\to 1$, and the third, again, because $\gamma_{_{\partial\Omega_j}}{\bf w}\circ\Phi_j \to \gamma_{_{\partial\Omega+}}{\bf w}$.

Combining this with the previous argument, we obtain,
\begin{align*}
\langle \mathbf t^+_{\alpha,\partial \Omega}({\bf u},\pi),\gamma_{_{\partial\Omega +}}{\bf w} \rangle_{\partial\Omega}
=\lim_{j\to\infty}\langle \mathbf t^{{\rm c}+}_{\partial \Omega_j}({\bf u},\pi),\gamma_{_{\partial\Omega_j}}{\bf w} \rangle_{\partial\Omega_j}=
\langle \mathbf t_{{\rm{nt}},\partial \Omega}^+({\bf u},\pi),
\gamma_{_{\partial\Omega+}}{\bf w} \rangle_{\partial\Omega}\quad
\forall\ {\textbf w} \in B_{p',q'}^{1+\frac{1}{p'}-s}({\Omega_+}, \mathbb{R}^{n})
\end{align*}
Taking $\mathbf w=\gamma_+^{-1}\boldsymbol\varphi$,
this gives
$\langle \mathbf t^+_{\alpha,\partial \Omega}({\bf u},\pi), \boldsymbol\varphi\rangle_{\partial\Omega}
=\langle \mathbf t_{{\rm{nt}},\partial \Omega}^+({\bf u},\pi), \boldsymbol\varphi \rangle_{\partial\Omega}$
for any
$\boldsymbol\varphi \in B_{p',q'}^{1-s}(\partial {\Omega},\mathbb{R}^{n})$, i.e.,
$\mathbf t^+_\alpha({\bf u},\pi)=\mathbf t_{\rm{nt}}^+({\bf u},\pi)$, and since
$\mathbf t_{\rm{nt}}^+({\bf u},\pi)=\boldsymbol{\sigma}^+_{\rm nt}({\bf u},\pi)\,\boldsymbol \nu \in  L_{p}(\partial {\Omega}, \mathbb{R}^{n})$, this completes the proof of item (ii) for $0<s<1$, while for $s=1$ the statement follows by inclusion.
\hfill\end{proof}

\begin{remark}
\label{2.13loc}
Due to Remark~\ref{CSD}, Theorem \ref{2.13} will still valid for $\Omega_-$ if the functions belong to the corresponding spaces only locally, i.e., if
$({\bf u},\pi)\in {B}_{p,q,\rm div,loc}^{s+\frac{1}{p}}(\overline{\Omega}_-,\mathbb{R}^{n})\times {B}_{p,q,\rm loc}^{s-1+\frac{1}{p}}(\overline{\Omega}_-)$
in item (i) and
$({\bf u},\pi)\in \mathfrak{B}_{p,q,\rm div,loc}^{s+\frac{1}{p},t}(\overline{\Omega}_-,\mathcal{L}_{\alpha})$
in item (ii).
\end{remark}

\section{Integral potentials for the Brinkman system}

This section is devoted to the main properties of Newtonian and layer potentials for the Brinkman system.
%{In this section we assume that $n\ge 3$.}

\subsection{\bf Newtonian {potential} for the Brinkman system}

Let $\alpha >0$ be a {constant}. Let us denote by $\mathcal{G}^{\alpha}$ and $\Pi $ \textit{the fundamental velocity tensor} and \textit{the fundamental pressure vector} for the Brinkman system in $\mathbb{R}^{n}$ ($n\geq 3$), with the components (see, e.g., \cite[(3.6)]{McCracken1981},
\cite[Section 3.2.1]{12}, \cite[(2.14)]{24})
\begin{equation}
\label{E41}
\mathcal{G}^{\alpha}_{jk}({\bf x})=\frac{1}{\tilde\omega_n} \left\{ \frac{\delta_{jk}}{|{\bf x}|^{n-2}}  A_{1}(\alpha |{\bf x}|)+ \frac{x_{j} x_{k}}{|{\bf x}|^{n}} A_{2}(\alpha |{\bf x}|) \right\},\ \
{\Pi_k}({\bf x})= \frac{1}{\tilde\omega_n}\frac{x_{k}}{|{\bf x}|^{n}}
\end{equation}
where $A_{1}(z)$ and $A_{2}(z)$ are defined by
\begin{align}
\label{E41-new}
A_{1}(z):= \frac{\big(\frac{z}{2}\big)^{\frac{n}{2} - 1} K_{\frac{n}{2} - 1}(z)}{\Gamma \big(\frac{n}{2}\big)}+2\frac{\big(\frac{z}{2}\big)^{\frac{n}{2}} K_{\frac{n}{2}}(z)}{\Gamma \big(\frac{n}{2}\big)z^{2}}-\dfrac{1}{z^{2}},\
A_{2}(z):= \frac{n}{z^{2}}-4\frac{\big(\frac{z}{2}\big)^{\frac{n}{2} +1} K_{\frac{n}{2} +1}(z)}{\Gamma \big(\frac{n}{2}\big)z^{2}},
\end{align}
$K_{{\varkappa} }$ is the Bessel function of the second kind and order ${\varkappa} \ge 0$, $\Gamma $ is the Gamma function, and $\tilde\omega_n$ is the area of the unit sphere in $\mathbb{R}^{n}$. The fundamental solution of the Stokes system, $(\mathcal{G},\Pi )$, which corresponds to $\alpha =0$, is given by (see, e.g., \cite[(1.12)]{24})
\begin{equation}
\label{E41-0}
\mathcal{G}_{jk}({\bf x})=\frac{1}{2\tilde\omega_n}\left\{\frac{1}{n-2}\frac{\delta_{jk}}{|{\bf x}|^{n-2}}+\frac{x_{j} x_{k}}{|{\bf x}|^{n}}\right\},\ \
{\Pi_k}({\bf x})= \frac{1}{\tilde\omega_n}\frac{x_{k}}{|{\bf x}|^{n}}.
\end{equation}

Next we use the notations ${\mathcal G}^{\alpha }({\bf x},{\bf y})={\mathcal G}^{\alpha }({\bf x}-{\bf y})$ and $ {\Pi}({\bf x},{\bf y})={ \Pi}({\bf x}-{\bf y})$.
Then
\begin{equation}
\label{2.2.1}
(\triangle _{{\bf x}}-\alpha \mathbb{I})\mathcal{G}^{\alpha}({\bf x},{\bf y}) - \nabla_{{\bf x}}{\Pi}({\bf x},{\bf y})={-\delta_{{\bf y}}({\bf x}){\mathbb I}}, \ \ {\rm{div}}_{{\bf x}}\mathcal{G}^{\alpha }({\bf x},{\bf y})= 0, \ \forall  \ {\bf y}\in \mathbb{R}^{n},
\end{equation}
where $\delta_{{{\bf x}}}$ is the Dirac distribution with mass in ${\bf y}$, and the subscript ${\bf x}$ added to a differential operator refers to the action of that operator with respect to the variable ${\bf x}$.

\textit{The fundamental stress tensor} $\textbf{S}^{\alpha}$ has the components
\begin{equation}
\label{stress-tensor-alpha}
S^{\alpha}_{ij\ell }({\bf x},{\bf y})=-{\Pi}_{j}({\bf x},{\bf y}) \delta_{i\ell }+\dfrac{\partial \mathcal{G}^{\alpha}_{ij}({\bf x},{\bf y})}{\partial x_{\ell}} + \dfrac{\partial \mathcal{G}^{\alpha}_{\ell j}({\bf x},{\bf y}) }{\partial x_{i}},
\end{equation}
where $\delta_{jk} $ is the Kronecker symbol. Let $\boldsymbol\Lambda^{\alpha}$ be \textit{the fundamental pressure tensor} with components $\Lambda^{\alpha}_{jk}$.
Then {for fixed $i$ and $k$, the pair $(S^{\alpha}_{ijk}, \Lambda^{\alpha}_{ik})$ satisfies} the Brinkman system in $\mathbb{R}^{n}$ {if ${\bf x}\neq {\bf y}$}, i.e.,
\begin{equation}
\label{2.2.1'}
\left\{
\begin{array}{ll}
\triangle _{{\bf x}}S^{\alpha}_{ijk}({\bf x},{\bf y})-\alpha S^{\alpha}_{ijk}({\bf x},{\bf y})-\dfrac{\partial \Lambda^{\alpha}_{ik}({\bf y},{\bf x})}{\partial x_{j}}=0,\\
\dfrac{\partial S^{\alpha}_{ijk}({\bf x},{\bf y})}{\partial x_{j}} = 0
\end{array}
\right.
\end{equation}
%{The components $S^{\alpha}_{ijk}$ and $\Lambda^{\alpha}_{jk}$ are given explicitly in \cite[(1.16), (2.18)]{24}, but we omit their presentation for the sake of brevity.
{The components $\Lambda^{\alpha}_{jk}({\bf x},{\bf y})$ are given by (see, e.g., \cite[(2.18)]{24})
\begin{align}
\label{dl-pressure}
\Lambda^{\alpha}_{ik}({\bf x},{\bf y})=\frac{1}{\omega _n}\left\{-(y_i-x_i)\frac{2n(y_k-x_k)}{|{\bf y}-{\bf x}|^{n+2}}+\frac{2\delta _{ik}}{|{\bf y}-{\bf x}|^n}-\alpha \frac{1}{(n-2)}\frac{1}{|{\bf y}-{\bf x}|^{n-2}}\delta _{ik}\right\}.
\end{align}

For $\alpha =0$, we use the notations $S_{ijk}:=S^{0}_{ijk}$ and $\Lambda_{ik}:=\Lambda^{0}_{ik}$.}

Let $*$ denote the convolution product.
Let us consider the velocity and pressure Newtonian potential operators for the Brinkman system,
\begin{align}
\label{Newtonian-Brinkman-vp}
&\left({\mathbf N}_{\alpha;{\mathbb R}^n}\boldsymbol\varphi\right)({\bf x})
:=-\left({\mathcal G}^{\alpha }* \boldsymbol\varphi\right)({\bf x})
=-\big\langle {\mathcal G}^{\alpha }({\bf x},\cdot),\boldsymbol\varphi
\big\rangle_{_{\!{\mathbb R^n}}},\quad
%\nonumber\\
%&
\left({\mathcal Q}_{\alpha;{\mathbb R}^n}\boldsymbol\varphi\right)({\bf x})
=\left({\mathcal Q}_{{\mathbb R}^n}\boldsymbol\varphi\right)({\bf x})
:=-\left(\boldsymbol{\Pi}*\boldsymbol\varphi\right)({\bf x})
=-\big\langle\boldsymbol{\Pi}({\bf x},\cdot),\boldsymbol\varphi\big\rangle _{_{\!{ \mathbb R^n}}},
\end{align}
where the fundamental tensor ${\mathcal G}^{\alpha }$ is presented through its components in \eqref{E41}.
Note that the Fourier transform of ${\mathcal G}^{\alpha }$-components {is given} by
\begin{align}
\label{Fourier}
\widehat{\mathcal G}^{\alpha }_{kj}(\xi )=\frac{(2\pi )^{-\frac{n}{2}}}{|\xi |^2+\alpha }\left(\delta _{kj}-\frac{\xi _k\xi _j}{|\xi |^2}\right).
\end{align}
Then we have the following property (cf. \cite[Theorem 3.10]{McCracken1981} in the case $n=3$, $s=0$).
\begin{lemma}
\label{Newtonian-potential}
Let
$\alpha > 0$. Then for all $p,q\in (1,\infty )$ and $s\in {\mathbb R}$ the following  linear operators are continuous
\begin{align}
\label{volume-potential-C}
{\mathbf N}_{\alpha ;{\mathbb R}^n}&:H^{s}_p({\mathbb R}^n,{\mathbb R}^n)\to H^{s+2}_p({\mathbb R}^n,{\mathbb R}^n), %\quad
\\
\label{ss-2}
%{\gr\mathbf N}_{\alpha ;{\mathbb R}^n} &:B^{s}_{p,q}({\mathbb R}^n,{\mathbb R}^n)\to B^{s+2}_{p,q}({\mathbb R}^n,{\mathbb R}^n)
{\mathbf N}_{\alpha ;{\mathbb R}^n}&:B^{s}_{p,q}({\mathbb R}^n,{\mathbb R}^n)\to B^{s+2}_{p,q}({\mathbb R}^n,{\mathbb R}^n),\\
\label{volume-potential-pressure}
{\mathcal Q}_{{\mathbb R}^n}&:H^{s}_p({\mathbb R}^n,{\mathbb R}^n)\to H^{s+1}_{p\rm,loc}({\mathbb R}^n),\\
\label{ss-2-pressure}
{\mathcal Q}_{{\mathbb R}^n}&:B^{s}_{p,q}({\mathbb R}^n,{\mathbb R}^n)\to B^{s+1}_{p,q\rm,loc}({\mathbb R}^n).
\end{align}
%are continuous for any $p,q\in (1,\infty )$ and $s\in\mathbb R$.
%are continuous .
\end{lemma}
\begin{proof}
Let $\boldsymbol \varphi \in H^{s}_p({\mathbb R}^n,{\mathbb R}^n)$. By \eqref{E17},
\begin{align}
\label{c1}
\|{\mathbf N}_{\alpha;{\mathbb R}^n}\boldsymbol \varphi\|_{H^{s+2}_p({\mathbb R}^n,{\mathbb R}^n)}
=\left\|{\mathcal F}^{-1}\left({\rho}\, ^{s+2}\mathcal F({\mathbf N}_{\alpha;{\mathbb R}^n}\boldsymbol \varphi)\right)\right\|_{{L_p({\mathbb R}^n,{\mathbb R}^n)}},
\end{align}
where $\rho$ is the weight function given by \eqref{Jhat}. In addition, we note that
\begin{align}
\label{c2}
{\mathcal F}\left({\mathbf N}_{\alpha;{\mathbb R}^n}\boldsymbol \varphi  \right)={\mathcal F}\left({\mathcal G}^\alpha *{\boldsymbol \varphi }\right)=
\widehat{\mathcal G}^\alpha \widehat {\boldsymbol \varphi }
\end{align}
and hence by \eqref{c1},
\begin{align}
\label{c3}
\|{\mathbf N}_{\alpha;{\mathbb R}^n}\boldsymbol \varphi\|_{H^{s+2}_p({\mathbb R}^n,{\mathbb R}^n)}
=\left\|{\mathcal F}^{-1}\left({\rho}\, ^{s+2}
\widehat{\mathcal G}^\alpha \widehat {\boldsymbol \varphi }\right)\right\|_{{L_p({\mathbb R}^n,{\mathbb R}^n)}}
=\left\|{\mathcal F}^{-1}(\widehat m\,\mathcal F(J^s{\boldsymbol\varphi})) \right\|_{{L_p({\mathbb R}^n,{\mathbb R}^n)}}.
\end{align}
In view of \eqref{Fourier}, the matrix-function $\widehat m:={\rho}\, ^2\widehat{\mathcal G}^\alpha$ has the components
$${\widehat m_{kj}(\xi )=(2\pi )^{-\frac{n}{2}}}\displaystyle\frac{1+|\xi |^2}{|\xi |^2+\alpha }\left(\delta _{kj}-\frac{\xi _k\xi _j}{|\xi |^2}\right),\ k,j=1,\ldots ,n,$$ and is smooth everywhere except the origin and uniformly bounded in $\mathbb R^n\times\mathbb R^n$. Hence it is a Fourier multiplier in $L_p({\mathbb R}^n)$ (cf. Theorem 2 in Appendix of \cite{Mikhlin}), i.e., there exists a constant $M>0$, (which depends on $p$ but is independent of $\boldsymbol \varphi )$ such that
$$
\|{\mathbf N}_{\alpha;{\mathbb R}^n}\boldsymbol \varphi\|_{H^{s+2}_p({\mathbb R}^n,{\mathbb R}^n)}
\leq M\left\|J^s{\boldsymbol\varphi} \right\|_{{L_p({\mathbb R}^n,{\mathbb R}^n)}}
= M\|\boldsymbol \varphi \|_{H^{s}_p({\mathbb R}^n,{\mathbb R}^n)}.
$$
and thus
%\begin{align}
%\label{c5}
$\left\|{\mathbf N}_{\alpha;{\mathbb R}^n}\right\|_{H^{s}_p({\mathbb R}^n,{\mathbb R}^n)\to H^{s+2}_p({\mathbb R}^n,{\mathbb R}^n)}\leq M,$
%\end{align}
while operator \eqref{volume-potential-C} is continuous.
\comment{
In addition, the {continuity} of {operator} \eqref{volume-potential-C} (for all $p,q\in (1,\infty )$ and $s\in {\mathbb R}$)
%formula \eqref{real-int} and Lemma \ref{Prop1.4},
implies that the operator {on the interpolation spaces,}
\begin{align}
\label{volume-potential-1}
{\mathbf N}_{\alpha ;{\mathbb R}^n}:\left(H^{s_1}_{p}({\mathbb R}^n,{\mathbb R}^n),H^{s_2}_{p}({\mathbb R}^n,{\mathbb R}^n)\right)_{\theta ,q}\to
\left(H^{s_1+2}_{p}({\mathbb R}^n,{\mathbb R}^n),H^{s_2+2}_{p}({\mathbb R}^n,{\mathbb R}^n)\right)_{\theta ,q},
\end{align}
%for all $\theta \in (0,1)$ and $p,q\in (1,\infty )$. Since
%$$\left(H^{s}_{p}({\mathbb R}^n,{\mathbb R}^n),H^{-s-2}_{p}({\mathbb R}^n,{\mathbb R}^n)\right)_{\theta ,q}=B^r_{p,q}({\mathbb R}^n,{\mathbb R}^n),\ \left(H^{s+2}_{p}({\mathbb R}^n,{\mathbb R}^n),H^{-s}_{p}({\mathbb R}^n,{\mathbb R}^n)\right)_{\theta ,q}=B^{r+2}_{p,q}({\mathbb R}^n,{\mathbb R}^n),$$ where $r\in {\mathbb R}$ is given by $r=(1-2\theta )s-2\theta $ (see \eqref{real-int}), we deduce that {operator} \eqref{ss-2} is continuous, as asserted.
is also continuous, for all $\theta \in (0,1)$, $p,q\in (1,\infty )$ and $s_1\neq s_2$ (see, e.g., \cite{M-W}).
} %\comment end

Moreover, by formula \eqref{real-int} we have the interpolation property
\begin{align}
\label{volume-potential-1a}
\left(H^{s_1}_{p}({\mathbb R}^n,{\mathbb R}^n),H^{s_2}_{p}({\mathbb R}^n,{\mathbb R}^n)\right)_{\theta ,q}=B^{s}_{p,q}({\mathbb R}^n,{\mathbb R}^n),\ \left(H^{s_1+2}_{p}({\mathbb R}^n,{\mathbb R}^n),H^{s_2+2}_{p}({\mathbb R}^n,{\mathbb R}^n)\right)_{\theta ,q}=B^{s+2}_{p,q}({\mathbb R}^n,{\mathbb R}^n),
\end{align}
where $s=(1-\theta )s_1+\theta s_2$.
Then by continuity of operator \eqref{volume-potential-C}, we obtain that operator \eqref{ss-2} is also continuous for $p,q\in (1,\infty )$ and any $s\in {\mathbb R}$.

Let us now show the continuity of operators \eqref{volume-potential-pressure} and \eqref{ss-2-pressure}. To this end, we note that the pressure Newtonian potential operator for the Brinkman system {coincides with the one for the Stokes system and for any $\boldsymbol\varphi \in {\mathcal D}({\mathbb R}^n,{\mathbb R}^n)$} can be written as
\begin{align}
\label{Laplace-Newtonian-potential-0}
{\mathcal Q}_{{\mathbb R}^n}\boldsymbol\varphi=\mathrm{div}\, {\mathcal N}_{\triangle ;{\mathbb R}^n}\boldsymbol\varphi,
\end{align}
where
\begin{align}
\label{Laplace-Newtonian-potential}
\left({\mathcal N}_{\triangle ;{\mathbb R}^n}\boldsymbol\varphi\right)({\bf x}):=-\left({\mathcal G}_\triangle *\boldsymbol\varphi \right)({\bf x}),
\end{align}
and ${\mathcal G}_\triangle ({\bf x},{\bf y}):=-\displaystyle\frac{1}{(n-2)\omega _n}\frac{1}{|{\bf x}-{\bf y}|^{n-2}}$ is the fundamental solution of the Laplace equation in ${\mathbb R}^n$. Therefore, the mapping properties of the pressure Newtonian potential are provided by those of the harmonic Newtonian potential ${\mathcal N}_{\triangle ;{\mathbb R}^n}$. Since ${\mathcal N}_{\triangle ;{\mathbb R}^n}$ is a pseudodifferential operator of order $-2$ in ${\mathbb R}^n${,
%(see, e.g., the proof of Lemma 5.6.6 in \cite{H-W}),
the following operator is continuous,}
\begin{align}
\label{Laplace-Newtonian-potential-1}
{\mathcal N}_{\triangle ;{\mathbb R}^n}:H^{s}_p({\mathbb R}^n)\to H^{s+2}_{p\rm,loc}({\mathbb R}^n),\ \forall \  s\in {\mathbb R},\, p\in (1,\infty ).
\end{align}
Then by {\eqref{Laplace-Newtonian-potential-0} and \eqref{Laplace-Newtonian-potential-1}} we deduce the {continuity} property of the pressure Newtonian potential operator in \eqref{volume-potential-pressure}. By using an interpolation argument as for \eqref{ss-2}, we also obtain {continuity of operator} \eqref{ss-2-pressure}.
\hfill\hfill\end{proof}

%Now let ${\Omega}\subset \mathbb{R}^{n}$ $(n\geq 3)$ be a bounded Lipschitz domain with connected boundary $\partial {\Omega}$, and set ${\Omega}_{+}:={\Omega}$ and ${\Omega}_{-}:={\mathbb R}^n\setminus \overline{{\Omega}}$.
{Let} $\alpha \geq 0$ and $p\in (1,\infty )$ be given. {The Newtonian velocity and pressure potential operators of the Brinkman system in {Lipschitz domains $\Omega_\pm$} are defined as
\begin{align}
\label{Newtonian-1a-Lp}
&%{\gr\mathbf N}_{\alpha;{\Omega}}:L_p({\Omega},\mathbb{R}^{n})\to H_{p}^{2}({\Omega},\mathbb{R}^{n}),\ \
{\mathbf N}_{\alpha;{\Omega}}=r_{\Omega}{\mathbf N}_{\alpha ;{\mathbb R}^{n}}\mathring E_\pm
%\quad
%{\gr\mathbf N}_{\alpha;\mathbb R^{n}}{\bf w}:=-\left\langle \mathcal {G}^{\alpha}(\cdot,\cdot), {\bf w}\right\rangle _{\mathbb R^{n}}
\quad\mbox{and}\quad
%\\
%\label{Newtonian-1b-Lp}
%& %{\mathcal Q}_{\Omega_\pm}:L_p({\Omega_\pm},\mathbb{R}^{n})\to {H^{1}_{p}}({\Omega_\pm}),\ \
{\mathcal Q}_{\Omega_\pm}=r_{\Omega_\pm}{\mathcal Q}_{{\mathbb R}^{n}}\mathring E_\pm.
%\quad
%{\mathcal Q}_{\mathbb R^{n}}{\bf w}:=-\left\langle \Pi(\cdot,\cdot), {\bf w}\right\rangle _{\mathbb R^{n}}.
\end{align}
Recall that
${\mathring E}_\pm$ is the operator of extension of vector fields defined in ${\Omega_\pm}$ by zero on ${\mathbb R}^{n}\setminus {\Omega_\pm}$, and $r_{\Omega_\pm}$ is the restriction operator from ${\mathbb R}^{n}$ to ${\Omega_\pm}$.
The operators ${\mathring E_\pm}:L_p({\Omega_\pm},\mathbb{R}^{n})\to L_p({\mathbb R}^{n},\mathbb{R}^{n})$ and $r_{\Omega_\pm}:H_p^{2}({\mathbb R}^{n},\mathbb{R}^{n})\to H_p^{2}({\Omega_\pm},\mathbb{R}^{n})$ are linear and continuous.
In addition, the volume potential operator ${\mathbf N}_{\alpha ;{\mathbb R}^{n}}:L_p({\mathbb R}^{ n},\mathbb{R}^{n})\to H_p^{2}({\mathbb R}^{n},\mathbb{R}^{n})$ is linear and continuous as well, for any $p\in (1,\infty )$ (cf., e.g., \cite[Theorem 3.10]{McCracken1981}, \cite[Lemma 1.3]{Deuring} {and Lemma~\ref{Newtonian-potential}}). Therefore, the velocity Newtonian potential {operators}
\begin{align}
\label{Newtonian-1a-Lpm}
{\mathbf N}_{\alpha;{\Omega_\pm}}:L_p({\Omega_\pm},\mathbb{R}^{n})\to H_{p}^{2}({\Omega_\pm},\mathbb{R}^{n}),\quad p\in (1,\infty ),
\end{align}
are continuous operators. A similar argument yields the continuity of the Newtonian pressure potential operators
\begin{align}
\label{Newtonian-1b-Lpm}
{\mathcal Q}_{\Omega_+}:L_p({\Omega_+},\mathbb{R}^{n})\to H^{1}_{p}({\Omega_+}),\quad
{\mathcal Q}_{\Omega_-}:L_p({\Omega_-},\mathbb{R}^{n})\to H^{1}_{p,\rm loc}({\Omega_-}),\quad
p\in (1,\infty ).
\end{align}
%because $\Pi$ is the fundamental pressure vector for both, the Brinkman system and the Stokes system, while the Newtonian pressure potential operator for the Stokes system is continuous (see, e.g., \cite{M-W}).
Next,
%we use the continuity property of the Newtonian potential operator \eqref{Newtonian-1a-Lp}. In addition,
in view of \eqref{Z2}, \eqref{Z1} and the first inclusion in \eqref{embed-3} we obtain the inclusions
\begin{align}
\label{Nlp-1}
H^{2}_p({\mathbb R}^{n},{\mathbb R}^{n})=W^{2}_p({\mathbb R}^{n},{\mathbb R}^{n})\hookrightarrow W^{1+\frac{1}{p}}_p({\mathbb R}^{n},{\mathbb R}^{n})=B_{p,p}^{1+\frac{1}{p}}({\mathbb R^n},\mathbb{R}^{n})\hookrightarrow B_{p,p^*}^{1+\frac{1}{p}}({\mathbb R^n},\mathbb{R}^{n}), \quad \forall\ p\ge 1,\ p^*=\max\{p,2\},
\end{align}
which are continuous. Then relations \eqref{Newtonian-1a-Lpm} and \eqref{Nlp-1}
%-\eqref{Nlp-3}
imply also the continuity of the velocity Newtonian potential operator
\begin{align}
\label{Newtonian-1}
&{\mathbf N}_{\alpha;{\Omega}_\pm }:L_p({\Omega_\pm},\mathbb{R}^{n})\to B_{p,p^*}^{1+\frac{1}{p}}({\Omega_\pm},\mathbb{R}^{n}),
\ \ p\in (1,\infty).
\end{align}
%for any $p\in (2-\varepsilon ,2+\varepsilon )$.
A similar argument yields the continuity property of the pressure Newtonian potential operator
\begin{align}
\label{Newtonian-1p}
{\mathcal Q}_{\alpha;{\Omega_+}}:L_p({\Omega_+},\mathbb{R}^{n})\to B_{p,p^*}^{\frac{1}{p}}({\Omega_+}),\quad
{{\mathcal Q}_{\alpha;{\Omega_-}}:L_p({\Omega_-},\mathbb{R}^{n})\to B_{p,p^*,\rm loc}^{\frac{1}{p}}({\overline{\Omega}_-}),}
\quad
p\in (1,\infty).
\end{align}
In addition, due to \eqref{Newtonian-1a-Lp}, we have the relations
\begin{align}
\label{Newtonian-2-Lp}
&\triangle {\mathbf N}_{\alpha ;{\Omega_\pm}}{\mathbf f}-\alpha {\mathbf N}_{\alpha ;{\Omega_\pm}}{\mathbf f}-\nabla {\mathcal Q}_{\Omega_\pm}{\mathbf f}={\mathbf f},\
{\rm{div}}\ {\mathbf N}_{\alpha ;{\Omega_\pm}}{\mathbf f}=0\ \mbox{ in }\ {\Omega}_\pm.
\end{align}
This leads us to the following assertion.
\begin{corollary}
\label{C2.16}
{Let}
%$n \geq 3$,
$\alpha > 0$, $p\in (1,\infty)$, and $p^*=\max\{p,2\}$. Then {the Brinkman Newtonian potentials satisfy equations \eqref{Newtonian-2-Lp} and} the  following operators are continuous
\begin{align}
({\mathbf N}_{\alpha;\Omega_+}, \mathcal Q_{\Omega_+}):L_p({\Omega_+},\mathbb{R}^{n})
\to \mathfrak{H}_{p,\rm div}^{2,0}(\Omega_+,\mathcal{L}_{\alpha}),& \quad
({\mathbf N}_{\alpha;\Omega_-}, \mathcal Q_{\Omega_-}):L_p({\Omega_-},\mathbb{R}^{n})
\to \mathfrak{H}_{p,\rm div,loc}^{2,0}(\overline{\Omega}_-,\mathcal{L}_{\alpha}),\\
({\mathbf N}_{\alpha;\Omega_+}, \mathcal Q_{\Omega_+}):L_p({\Omega_+},\mathbb{R}^{n})
\to \mathfrak{B}_{p,p^*,\rm div}^{2,0}(\Omega_+,\mathcal{L}_{\alpha}),& \quad
({\mathbf N}_{\alpha;\Omega_-}, \mathcal Q_{\Omega_-}):L_p({\Omega_-},\mathbb{R}^{n})
\to \mathfrak{B}_{p,p^*,\rm div, loc}^{2,0}(\overline{\Omega}_-,\mathcal{L}_{\alpha}).
\end{align}
\end{corollary}
\begin{rem}
\label{gammaN}
Let $\mathbf f_{\pm}\in L_p(\Omega_{\pm},\mathbb{R}^{n})$ for some $p\in (1,\infty)$, and $p^*=\max\{p,2\}$. Then Corollary $\ref{C2.16}$, Lemmas $\ref{trace-lemma-Besov}$, $\ref{lem 1.6}$ and Remark $\ref{CSD}$ imply that
\begin{align}
\label{Newtonian-3-Lp0}
&{\gamma}_\pm\left({\mathbf N}_{\alpha ;\Omega_\pm}{\mathbf f}_\pm\right)
\in B^s_{p,p^*;\boldsymbol\nu}(\partial {\Omega},{\mathbb R}^{n}),\quad
{\mathbf t}_\alpha^\pm\left({\mathbf N}_{\alpha ;\Omega_\pm}{\mathbf f}_\pm,{\mathcal Q}_{\Omega_\pm}{\mathbf f}_\pm\right)
\in B^{s-1}_{p,p^*}(\partial {\Omega},{\mathbb R}^{n}),\quad \forall \ s\in (0,1).
\end{align}

Moreover, due to \eqref{Newtonian-1a-Lpm}, the first equality in \eqref{Nlp-1}, Theorem $\ref{2.13}$, and  \cite[Theorem 5]{Buffa}, these inclusions can be improved to the following ones
\begin{align}
\label{Newtonian-3-Lp}
&{\gamma}_\pm\left({\mathbf N}_{\alpha ;\Omega_\pm}{\mathbf f}_\pm\right)
\in H^1_{p;\boldsymbol\nu}(\partial\Omega,{\mathbb R}^{n}),\
{\mathbf t}_\alpha^\pm\left({\mathbf N}_{\alpha ;\Omega_\pm}{\mathbf f}_\pm,
{\mathcal Q}_{\Omega_\pm}{\mathbf f}_\pm\right)=
{\mathbf t}^{c\pm}\left({\mathbf N}_{\alpha ;\Omega_\pm}{\mathbf f}_\pm,
{\mathcal Q}_{\Omega_\pm}{\mathbf f}_\pm\right)
\in L_p(\partial {\Omega},{\mathbb R}^{n}).
\end{align}
%{*** We should show that there exist the non-tangential limits of ${\mathbf N}_{\alpha ;\Omega_\pm}{\mathbf f}_\pm$, $\nabla {\mathbf N}_{\alpha ;\Omega_\pm}{\mathbf f}_\pm$ and ${\mathcal Q}_{\Omega_\pm}{\mathbf f}_\pm$ at almost all points of $\partial \Omega $. ***}
\end{rem}
In \eqref{Newtonian-3-Lp0}, \eqref{Newtonian-3-Lp} and further on, the following space notations are used for $p\in (1,\infty )$, $q\in(1,\infty]$, $s\in (0,1]$, and the outward unit normal $\boldsymbol\nu$ to the Lipschitz domain $\Omega _+\subset {\mathbb R}^n$,
\begin{align}
&L_{p;\boldsymbol \nu}(\partial {\Omega}, \mathbb{R}^{n}):=\left\{{\bf v}\in L_p(\partial {\Omega}, \mathbb{R}^{n}):
\int _{\partial {\Omega}}{\bf v}\cdot \boldsymbol \nu d\sigma =0\right\},\
H_{p;\boldsymbol \nu}^{s}(\partial {\Omega}, \mathbb{R}^{n}):=
\left\{{\bf v}\in H_p^{s}(\partial {\Omega}, \mathbb{R}^{n}):
\int _{\partial {\Omega}}{\bf v}\cdot \boldsymbol \nu d\sigma =0\right\},\nonumber\\
&B_{p,q;\boldsymbol \nu}^{s}(\partial {\Omega}, \mathbb{R}^{n}):=
\left\{{\bf v}\in B_{p,q}^{s}(\partial {\Omega}, \mathbb{R}^{n}):
\int _{\partial {\Omega}}{\bf v}\cdot \boldsymbol \nu d\sigma =0\right\}.
\label{p-nu}
\end{align}
\subsection{\bf Layer potentials for the Brinkman system}
For a given density $\textbf{g}\in L_p(\partial {\Omega}, \mathbb{R}^{n})$, the \textit{velocity single-layer potential} for the Brinkman system, ${\textbf{V}}_{\alpha }{\bf g}$, and the corresponding \textit{pressure single-layer potential}, $Q^{s}{\bf g}$, are given by
\begin{equation}
\label{single-layer}
(\textbf{V}_{\alpha }\textbf{g})({\bf x}):=\langle \mathcal{G}^{\alpha}({\bf x},\cdot), {\bf g} \rangle_{\partial {\Omega}}, \
(\mathcal Q^{s}\textbf{g})({\bf x}):= \langle {\Pi}({\bf x},\cdot), {\bf g}\rangle_{\partial {\Omega}}, \ {\bf x}\in \mathbb{R}^{n}\setminus \partial {\Omega}.
\end{equation}

Let ${\textbf h}\in H^1_p(\partial {\Omega}, \mathbb{R}^{n})$ be a given density. Then the  \textit{velocity double-layer potential}, ${\bf W}_{\alpha ;\partial {\Omega}}{\bf h}$, and the corresponding \textit{pressure double-layer potential}, $Q^{d}_{\alpha;\partial {\Omega}}\textbf{h}$, are defined by
\begin{align}
\label{dl-velocity}
&({\bf W}_{\alpha }{\textbf h})_{j}({\bf x}):= \int_{\partial {\Omega}} S^{\alpha}_{ij\ell }({\bf x},{\bf y})\nu _{\ell}({\bf y}) h_{i}({\bf y})d\sigma _{\bf y}, \ %\ \forall \ {\bf x}\in \mathbb{R}^{n}\setminus \partial {\Omega},\
%\label{dl-pressureQ}
({\mathcal Q}^{d}_{\alpha }{\bf h})({\bf x}):= \int_{\partial {\Omega}} \Lambda^{\alpha }_{j\ell}({\bf x},{\bf y})\nu _{\ell}({\bf y})h_{j}({\bf y})d\sigma _{\bf y},\ \ \forall \ {\bf x}\in \mathbb{R}^{n}\setminus \partial {\Omega},
\end{align}
where $\nu _\ell $, $\ell =1,\ldots ,n$, are the components of the outward unit normal $\boldsymbol \nu $ to ${\Omega}_+$, which is defined a.e. (with respect to the surface measure $\sigma $) on $\partial {\Omega}$. {Note that the definition of the double layer potential in \cite[(3.9)]{Shen} {differs from} definition \eqref{dl-velocity} due to different conormal derivatives used  {in \cite[(1.14)]{Shen} and in formula \eqref{classical-stress} of our paper.}

The single- and double-layer potentials can be also defined {for any} ${\bf g}\in B^{s-1}_{p,q}(\partial {\Omega}, {\mathbb R}^{n})$ and ${\textbf h}\in B_{p,q}^{s}(\partial {\Omega},\mathbb{R}^{n})$, respectively, where $s\in (0,1)$ and $p,q\in (1,\infty )$.
For $\alpha = 0$ (i.e., for the Stokes system) we use the notations ${\textbf V}\textbf{g}, \mathcal Q^{s}\textbf{g}, {\textbf W}{\textbf h}$ and ${\mathcal Q^{d}}{\textbf h}$ for the corresponding single- and double-layer potentials.

In view of equations (\ref{2.2.1}) and (\ref{2.2.1'}), the pairs $(\textbf{V}_{\alpha }{\bf g}, Q^{s}\textbf{g})$ and $(\textbf{W}^{s}_{\alpha }\textbf{h},Q^{d}_{\alpha }\textbf{h})$ satisfy the homogeneous Brinkman system in ${\Omega}_{\pm }$,
\begin{align}
\label{sl}
&(\triangle - \alpha\mathbb{I}){\textbf V}_{\alpha }{\bf g} - \nabla \mathcal Q^{s}\textbf{g}={\bf 0},\ \ {\rm{div}}{\bf V}_{\alpha }{\bf g} = 0 \ \mbox{ in } \ \mathbb{R}^{n}\setminus \partial {\Omega},
\\
\label{dl0}
&(\triangle - \alpha\mathbb{I}){\textbf W}_{\alpha }{\textbf h} - \nabla \mathcal Q^{s}{\textbf h}={\bf 0},\ \ {\rm{div}}{\bf W}_{\alpha }\textbf{h}=0 \ \mbox{ in } \ \mathbb{R}^{n}\setminus \partial {\Omega}.
\end{align}

The direct value of the double layer potential ${\bf W}_{\alpha; \partial {\Omega}} \textbf{h}$ on the boundary is defined in terms of Cauchy principal value by
\begin{equation}
({\bf K}_{\alpha }\textbf{h})_{k}({\bf x}):={\rm{p.v.}}\int\limits_{\partial {\Omega}} S^{\alpha}_{jk\ell }({{\bf y},{\bf x}})\nu _{\ell }({\bf y})h_{j}({\bf y})d \sigma _{\bf y} \ \mbox{ a.e. } \ {\bf x}\in \partial {\Omega}.
\end{equation}

%Let us also mention the following useful property of the {Brinkman single layer potential} (cf., e.g., \cite[Proposition 4.2.3]{M-W} and \cite[Proposition 2.68]{M-M1} in the case $\alpha =0$).
\begin{lemma}
\label{nontangential-sl}
Let $\Omega _+\!\subset \!{\mathbb R}^n$ $(n\geq 3)$ be a bounded Lipschitz domain with connected boundary $\partial \Omega $ and let $\Omega _{-}:={\mathbb R}^n\setminus \overline{\Omega }_+$.
{Let $\alpha \geq 0$ and $p\in (1,\infty )$.}
% be given constants.
There exist some constants {$C_i>\!0$, $i=1,\ldots ,4$, depending only on $p$, $\alpha$ and the Lipschitz character of $\Omega _+$}, such that the following properties hold:
%\begin{itemize}
%\item[$(i)$]
\begin{align}
\label{7ms}
&\|M\left(\nabla {\bf V}_{\alpha }{\bf g}\right)\|_{L_p(\partial \Omega )}
+\|M\left({\bf V}_{\alpha }{\bf g}\right)\|_{L_p(\partial \Omega )}
+\|M\left({\mathcal Q}^{s}{\bf g}\right)\|_{L_p(\partial \Omega )}
\le C_1\|{\bf g}\|_{L_p(\partial \Omega ,{\mathbb R}^n)},
\ \ \forall \ {\bf g}\in L_p(\partial \Omega ,{\mathbb R}^n),\\
\label{7ms-sl-0}
&{\|M\left({\bf V}_{\alpha }{\bf g}\right)\|_{L_p(\partial \Omega )}
\le C_2\|{\bf g}\|_{H_p^{-1}(\partial \Omega ,{\mathbb R}^n)},
\ \ \forall \ {\bf g}\in H_p^{-1}(\partial \Omega ,{\mathbb R}^n)},\\
\label{7ms-dl-0}
&{\|M\left({\bf W}_{\alpha }{\bf h}\right)\|_{L_p(\partial \Omega )}
\le C_3\|{\bf h}\|_{L_p(\partial \Omega ,{\mathbb R}^n)},
\ \ \forall \ {\bf h}\in L_p(\partial \Omega ,{\mathbb R}^n)},\\
\label{7ms-dl}
&{\|M\left(\nabla {\bf W}_{\alpha }{\bf h}\right)\|_{L_p(\partial \Omega )}
+\|M\left({\bf W}_{\alpha }{\bf h}\right)\|_{L_p(\partial \Omega )}
+\|M\big({\mathcal Q}_{\alpha }^{d}{\bf h}\big)\|_{L_p(\partial \Omega )}
\le C_4\|{\bf h}\|_{H_p^1(\partial \Omega ,{\mathbb R}^n)},
\ \ \forall \ {\bf h}\in H_p^1(\partial \Omega ,{\mathbb R}^n)}.
\end{align}
%\item[$(ii)$]
%For any ${\bf g}\in L_p(\partial \Omega ,{\mathbb R}^n)$, there exist

Moreover, the following estimates hold for the non-tangential traces that exist
%of ${\bf V}_{\alpha }{\bf g}$, $\nabla {\bf V}_{\alpha }{\bf g}$ and ${\mathcal Q}^{s}{\bf g}$
at almost all points of $\partial \Omega $:
\begin{align}
\label{ntV}
&\|({\bf V}_{\alpha }{\bf g})^\pm _{\rm nt}\|_{L_p(\partial \Omega ,{\mathbb R}^n)},\,
\|(\nabla {\bf V}_{\alpha }{\bf g})^\pm _{\rm nt}\|_{L_p(\partial \Omega ,{\mathbb R}^n)},\,
\|({\mathcal Q}^{s}{\bf g})^\pm _{\rm nt}\|_{L_p(\partial \Omega ,{\mathbb R}^n)}
\le C_1\|{\bf g}\|_{L_p(\partial \Omega ,{\mathbb R}^n)},
\ \ \forall \ {\bf g}\in L_p(\partial \Omega ,{\mathbb R}^n),
\\
%\end{align}
%for some constant $C'_1=C'_1(\partial \Omega ,p,\alpha )>0$.
%\item[$(iii)$]
\comment{For any ${\bf g}\in H^{-1}_p(\partial \Omega ,{\mathbb R}^n)$, there exist the non-tangential limits of ${\bf V}_{\alpha }{\bf g}$ at almost all points of $\partial \Omega $, and
\begin{align}
\label{7ms-sl-00}
 \|({\bf V}_{\alpha }{\bf g})^\pm _{\rm nt}\|_{L_p(\partial \Omega )}\leq C_2 \|{\bf g}\|_{H^{-1}_p(\partial \Omega ,{\mathbb R}^n)}.
\end{align}
%\item[$(iv)$]
}
%For any ${\bf g}\in H^{-1}_p(\partial \Omega ,{\mathbb R}^n)$, there exist the nontangential limits of ${\bf V}_{\alpha }{\bf g}$ at almost all points of $\partial \Omega $, and
%\begin{align}
\label{7ms-sl-00}
&\|({\bf V}_{\alpha }{\bf g})^\pm _{\rm nt}\|_{L_p(\partial \Omega )}\leq C_2 \|{\bf g}\|_{H^{-1}_p(\partial \Omega ,{\mathbb R}^n)},
\ \ \forall \ {\bf g}\in H_p^{-1}(\partial \Omega ,{\mathbb R}^n),
\\
%\end{align}
%\item[$(iv)$]
%{For any ${\bf h}\in L_p(\partial \Omega ,{\mathbb R}^n)$, there exists the nontangential limits of
%${\bf W}_{\alpha }{\bf h}$ at almost all points of $\partial \Omega $ and}
%and
%\begin{align}
\label{ntW}
&\|({\bf W}_{\alpha }{\bf g})^\pm _{\rm nt}\|_{L_p(\partial \Omega ,{\mathbb R}^n)}
\le C_3\|{\bf h}\|_{L_p(\partial \Omega ,{\mathbb R}^n)},
\ \ \forall \ {\bf h}\in L_p(\partial \Omega ,{\mathbb R}^n),
\\
%\label{7ms-dl}
%\end{align}
%{for some constant $C'_2=C'_2(\partial \Omega ,p,\alpha )>0$.}
%\item[$(v)$]
%{For any ${\bf h}\in H_p^1(\partial \Omega ,{\mathbb R}^n)$, there exist the nontangential limits of ${\bf W}_{\alpha }{\bf h}$, $\nabla {\bf W}_{\alpha }{\bf h}$ and ${\mathcal Q}^d_{\alpha }{\bf h}$ at almost all points of $\partial \Omega $ and}
%\begin{align}
\label{ntW1}
&\|({\bf W}_{\alpha }{\bf h})^\pm _{\rm nt}\|_{L_p(\partial \Omega ,{\mathbb R}^n)},\,
\|(\nabla {\bf W}_{\alpha }{\bf h})^\pm _{\rm nt}\|_{L_p(\partial \Omega ,{\mathbb R}^n)},\,
\|({\mathcal Q}_{\alpha }^{d}{\bf h})^\pm _{\rm nt}\|_{L_p(\partial \Omega ,{\mathbb R}^n)}
\le C_4\|{\bf h}\|_{H_p^1(\partial \Omega ,{\mathbb R}^n)},
\ \ \forall \ {\bf h}\in H_p^1(\partial \Omega ,{\mathbb R}^n).
\end{align}
%\end{itemize}
\end{lemma}
\begin{proof}
In the case $\alpha =0$, inequalities \eqref{7ms}-\eqref{7ms-dl} follow from
\cite[Propositions 4.2.3 and 4.2.8]{M-W}.
%In addition, inequality \eqref{7ms-dl} follows from relation (4.56) in Proposition 4.2.3 and from Proposition 4.2.8 in \cite{M-W}.

In the case $\alpha >0$, Inequality \eqref{7ms} has been obtained in \cite[Lemma 3.2]{Shen}. In addition, inequality \eqref{7ms-sl-0} follows {by} the same arguments {as in the proof of} its counterpart in the case $\alpha =0$ (cf. \cite[(4.61)]{M-W}). Indeed, if ${\bf g}\in H_p^{-1}(\partial \Omega ,{\mathbb R}^n)$, then there exist ${\bf g}_0=(g_{0;1},\ldots ,g_{0;n}),\, {\bf g}_{r\ell }=(g_{r\ell ;1},\ldots ,g_{r\ell ;n})\in L_p(\partial \Omega ,{\mathbb R}^n)$, $r,\ell =1,\ldots ,n$, such that
\begin{align}
\label{estimate-Lp}
g_k=g_{0;k}+\sum_{r,\ell =1}^n\partial _{\tau _{r\ell }}g_{r\ell ;k},\
\|g_{0;k}\|_{L_p(\partial \Omega )}+\sum _{r,\ell =1}^{n}\|g_{r\ell ;k}\|_{L_p(\partial \Omega )}\leq 2\|g_k\|_{H_p^{-1}(\partial \Omega )},\ \ k=1,\ldots ,n,
\end{align}
(cf. \cite[Corollary 2.1.2 and relation (4.65)]{M-W}),
where $\partial _{\tau _{r\ell }}=\nu _r\partial _\ell -\nu _\ell \partial _r$ are the tangential derivative operators. Hence, integrating by parts,
\begin{align}
\label{7ms-sl-001}
({\bf V}_\alpha {\bf g})_j({\bf x})=\int _{\partial \Omega }{\mathcal G}_{jk}^\alpha ({\bf x}-{\bf y})g_{0;k}({\bf y})d\sigma _{\bf y}-\sum _{k=1}^n\sum_{r,\ell =1}^n\int_{\partial \Omega }\left(\partial _{\tau _{r\ell }}\left({\mathcal G}_{jk}^\alpha ({\bf x}-{\bf y})\right)\right)g_{r\ell ;k}({\bf y})d\sigma _{{\bf y}},\, \forall \ {\bf x}\in {\mathbb R}^n\setminus \partial \Omega
\end{align}
(cf. \cite[(4.66)]{M-W} for $\alpha =0$). Inequality \eqref{7ms-sl-0} immediately follows from equality \eqref{7ms-sl-001} and the estimates in \eqref{7ms} and \eqref{estimate-Lp}.

%{ [***SM: I do not see how inequality \eqref{7ms-sl-0} in terms of the norm of $\mathbf g$ is obtained.
%It looks like relation \eqref{7ms-sl-001} gives an inequality only in terms of the norms of $\mathbf g_{0;k}$ and $\mathbf g_{rl;k}$.]}

%while inequality \eqref{7ms-dl-0} has been obtained in \cite[Theorem 3.5]{Shen} (see \cite[(4.56), Proposition 4.2.8]{M-W} for $\alpha =0$).

Let us now show inequality \eqref{7ms-dl-0} for $\alpha >0$ {(note that its analogue for a differently defined double layer potential in place of $\mathbf W_\alpha$ was given in \cite[Theorem 3.5]{Shen}). First,} we note that Lemma 4.1 in \cite{Med-CVEE-16} (see also \cite[Theorem 2.5]{Shen}) {implies} that there exists a constant $c_\alpha = c_\alpha (\Omega _{+},\alpha )>0$ such that
\begin{align}
\label{7ms-c}
|\nabla {\mathcal G}^\alpha ({\bf x},{\bf y})-\nabla {\mathcal G}({\bf x},{\bf y})|\leq c_\alpha |{\bf x}-{\bf y}|^{2-n},\ \ \forall \ {\bf x},{\bf y}\in \overline{\Omega }_{+},\, {\bf x}\neq {\bf y}.
\end{align}
Then, in view of formula \eqref{stress-tensor-alpha} and equality $\Pi ^\alpha =\Pi $, there exists a constant $C_5=C_5(\Omega _{+},\alpha )>0$ such that
\begin{align}
\label{7ms-c1}
|S_{ijk}^{\alpha }({\bf y},{\bf x})-S_{ijk}({\bf y},{\bf x})|&\leq \Big|\frac{\partial \mathcal{G}^{\alpha}_{ij}({\bf y},{\bf x})}{\partial y_{k}}-\frac{\partial \mathcal{G}_{ij}({\bf y},{\bf x})}{\partial y_{k}}\Big|
+\Big|\frac{\partial \mathcal{G}^{\alpha}_{kj}({\bf y},{\bf x})}{\partial y_{i}}-\frac{\partial \mathcal{G}({\bf y},{\bf x})}{\partial y_{i}}\Big|\leq C_5|{\bf x}-{\bf y}|^{2-n},\ \ \forall \ {\bf x},{\bf y}\in \overline{\Omega }_{+},\, {\bf x}\neq {\bf y}.
\end{align}
Inequality \eqref{7ms-c1} and \cite[Proposition 1]{Med-AAM} (applied to the integral operator ${\bf W}_{\alpha }-{\bf W}$ whose kernel is $\left({\bf S}^\alpha ({\bf y},{\bf x})-{\bf S}({\bf y},{\bf x})\right)\boldsymbol \nu ({\bf y})$) show that there exists a constant $C_6=C_6(\partial \Omega ,p, \alpha )>0$ such that
\begin{align}
\label{7ms-f}
\|M\left(\left({\bf W}_{\alpha }-{\bf W}\right){\bf h}\right)\|_{L_p(\partial \Omega )}\leq
C_6\|{\bf h}\|_{L_p(\partial \Omega ,{\mathbb R}^n)},\ \forall \ {\bf h}\in L_p(\partial \Omega ,{\mathbb R}^n).
\end{align}
Moreover, by \cite[(4.56)]{M-W}, there exists a constant $C_{7}=C_{7}(\partial \Omega ,p)>0$ such that
\begin{align}
\label{7ms-g}
\|M\left({\bf W}{\bf h}\right)\|_{L_p(\partial \Omega )}\leq
C_{7}\|{\bf h}\|_{L_p(\partial \Omega ,{\mathbb R}^n)},\ \forall \ {\bf h}\in L_p(\partial \Omega ,{\mathbb R}^n),
\end{align}
and then, by \eqref{7ms-f} and \eqref{7ms-g}, we obtain
%the estimate
%\begin{align*}
%%\label{7ms-h}
%\|M\left({\bf W}_{\alpha }{\bf h}\right)\|_{L_p(\partial \Omega )}\leq
%C_{8}\|{\bf h}\|_{L_p(\partial \Omega ,{\mathbb R}^n)},\ \forall \ {\bf h}\in L_p(\partial \Omega ,{\mathbb R}^n),
%\end{align*}
%with some constant $C_{8}=C_{8}(\partial \Omega ,p,\alpha )>0$. Thus, we have proved
{inequality \eqref{7ms-dl-0}}.

Let us now show inequality \eqref{7ms-dl} for $\alpha >0$. According to the second formula in \eqref{dl-velocity} and formula \eqref{dl-pressure} the kernel of the Brinkman double-layer pressure potential operator ${\mathcal Q}_{\alpha }^d$ is given by
\begin{align}
\label{Lambdanu}
\Lambda^{\alpha}_{jk}({\bf x},{\bf y})\nu _{k}({\bf y})=\frac{1}{\tilde\omega _n}\left\{-\frac{2n(y_j-x_j)(y_k-x_k)\nu _{k}({\bf y})}{|{\bf y}-{\bf x}|^{n+2}}+\frac{2\nu _{j}({\bf y})}{|{\bf y}-{\bf x}|^n}-\alpha \frac{1}{(n-2)}\frac{1}{|{\bf y}-{\bf x}|^{n-2}}\nu _{j}({\bf y})\right\}.
\end{align}
For $\alpha =0$, \eqref{Lambdanu} reduces to the kernel of the Stokes double-layer pressure potential operator ${\mathcal Q}^d$. Therefore, %there exists a constant $C_{4}\equiv C_{4}(\Omega _{+},\alpha )>0$ such that the following inequality holds
\begin{align}
\label{7ms-dl-2}
|\Lambda^{\alpha}_{jk}({\bf x},{\bf y})\nu _{k}({\bf y})-\Lambda _{jk}({\bf x},{\bf y})\nu _{k}({\bf y})|\leq
% C_{4}|{\bf x}-{\bf y}|^{2-n},
\frac{\alpha }{\tilde\omega _n(n-2)}\frac{1}{|{\bf y}-{\bf x}|^{n-2}},
\ \ \forall \ {\bf x}\in \overline{\Omega }_{+},\, {\bf y}\in {\partial\Omega },\ {\bf x}\neq {\bf y}.
\end{align}
Then according to \cite[Proposition 1]{Med-AAM} applied to the operator ${\mathcal Q}_{\alpha }^d-{\mathcal Q}^d$, there exists a constant $C_8=C_{8}(\partial \Omega ,p,\alpha )$ such that
\begin{align}
\label{7ms-dl-3}
\left\|M\left(\left({\mathcal Q}_{\alpha }^d-{\mathcal Q}^d\right){\bf h}\right)\right\|_{L_p(\partial \Omega )}\leq
C_{8}\|{\bf h}\|_{L_p(\partial \Omega ,{\mathbb R}^n)},\ \forall \ {\bf h}\in H_p^1(\partial \Omega ,{\mathbb R}^n).
\end{align}

In view of \cite[Proposition 4.2.8]{M-W}, the Stokes double-layer pressure potential operator ${\mathcal Q}^d$ satisfies the inequality
\begin{align}
\label{7ms-dl-4}
\left\|M\left({\mathcal Q}^d{\bf h}\right)\right\|_{L_p(\partial \Omega )}\leq
C_{9}\|{\bf h}\|_{H_p^1(\partial \Omega ,{\mathbb R}^n)},\ \forall \ {\bf h}\in H_p^1(\partial \Omega ,{\mathbb R}^n),
\end{align}
with a constant $C_{9}\equiv C_{9}(\partial \Omega ,p)>0$. Then by \eqref{7ms-dl-3} and \eqref{7ms-dl-4} there exists a constant $C_{10}\equiv C_{10}(\partial \Omega ,p,\alpha )>0$ such that
\begin{align}
\label{7ms-dl-5}
{\left\|M\left({\mathcal Q}_{\alpha }^d{\bf h}\right)\right\|_{L_p(\partial \Omega )}\leq
C_{10}\|{\bf h}\|_{H_p^1(\partial \Omega ,{\mathbb R}^n)},\ \forall \ {\bf h}\in H_p^1(\partial \Omega ,{\mathbb R}^n).}
\end{align}

Next, we show that there exists a constant $c_3=c_3(\Omega_+,p,\alpha )>0$ such that
\begin{align}
\label{7ms-dl-7}
\|M\left(\nabla {\bf W}_{\alpha }{\bf h}\right)\|_{L_p(\partial \Omega )}\leq c_3\|{\bf h}\|_{H_p^1(\partial \Omega ,{\mathbb R}^n)},\ \forall \ {\bf h}\in H_p^1(\partial \Omega ,{\mathbb R}^n).
\end{align}
To this end, we use expressions \eqref{dl-velocity} and \eqref{stress-tensor-alpha} for the Brinkman double layer potential ${\bf W}_{\alpha }{\bf h}$ to obtain for any ${\bf h}\in H_p^1(\partial \Omega ,{\mathbb R}^n)$,
\begin{align}
\label{1a}
\partial_r\left({\bf W}_{\alpha }{\bf h}\right)_j({\bf x})&=-\int_{\partial \Omega }\left\{\nu _{\ell }({\bf y})\left(\partial _r \partial _\ell {\mathcal G}_{jk}^{\alpha }\right)({\bf y}-{\bf x})+\nu _{\ell }({\bf y})\left(\partial _r \partial _j{\mathcal G}_{\ell k}^{\alpha }\right)({\bf y}-{\bf x})-\nu _j({\bf y})\left(\partial _r {\Pi }_{k}\right)({\bf y}-{\bf x})\right\}h_k({\bf y})d\sigma _{\bf y}\nonumber\\
&=-\int_{\partial \Omega }\left\{\partial _{\tau _{\ell r}({\bf y})}\left(\partial _\ell {\mathcal G}_{jk}^{\alpha } \right)({\bf y}-{\bf x})+\partial _{\tau _{\ell r}({\bf y})}\left(\partial _j{\mathcal G}_{\ell k}^{\alpha } \right)({\bf y}-{\bf x})-\partial _{\tau _{jr}({\bf y})}\Pi _{k}({\bf y}-{\bf x})
\right\}h_k({\bf y})d\sigma _{\bf y}\nonumber\\
&\quad -\int_{\partial \Omega }\left\{\nu _r({\bf y})\triangle {\mathcal G}_{jk}^{\alpha }({\bf y}-{\bf x})+\nu _r({\bf y})\left(\partial _\ell \partial _j{\mathcal G}_{\ell k}^{\alpha }\right)({\bf y}-{\bf x})-\nu _r({\bf y})\left(\partial _j\Pi _k\right)({\bf y}-{\bf x})\right\}h_k({\bf y})d\sigma _{{\bf y}}\nonumber\\
&=\int_{\partial \Omega }\left\{\left(\partial _\ell {\mathcal G}_{jk}^{\alpha }\right)({\bf y}-{\bf x})\left(\partial _{\tau _{\ell r}}h_k\right)({\bf y})+\left(\partial _j{\mathcal G}_{\ell k}^{\alpha }\right)({\bf y}-{\bf x})\left(\partial _{\tau _{\ell r}}h_k\right)({\bf y})-\Pi _k({\bf y}-{\bf x})\left(\partial _{\tau _{jr}}h_k\right)({\bf y})\right\}d\sigma _{{\bf y}}\nonumber\\
&\quad -\alpha \int_{\partial \Omega }\nu _r({\bf y}){\mathcal G}_{jk}^{\alpha }({\bf y}-{\bf x})h_k({\bf y})d\sigma _{{\bf y}},\ \ j,r=1,\ldots ,n,
\end{align}
where
%$\partial _{\tau _{jk}}=\nu _j\partial _k-\nu _k\partial _j$ are the tangential derivative operators, and
$\partial _j:=\dfrac{\partial }{\partial x_j}$.
We also employed the following integration by parts formula, which holds for any $p\in (1,\infty )$ (cf. \cite[(2.16)]{M-W}),
\begin{align}
\label{int-parts}
\int_{\partial \Omega }f\left(\partial _{\tau _{jk}}g\right)d\sigma =\int_{\partial \Omega }\left(\partial _{\tau _{kj}}f\right)g d\sigma ,\ \ \forall \ f\in H_p^1(\partial \Omega),\ \forall \ g\in H_{p'}^1(\partial \Omega),
\end{align}
where $\frac{1}{p}+\frac{1}{p'}=1$.
The last integral in \eqref{1a} follows from equations \eqref{2.2.1}, which, in particular, yield that
\begin{equation}
\label{2.2.1p}
(\triangle _{{\bf y}}-\alpha \mathbb{I})\mathcal{G}^{\alpha}({\bf y}-{\bf x}) - \nabla_{{\bf y}}{\Pi}({\bf y}-{\bf x})=0, \ \ {\rm{div}}_{{\bf y}}\mathcal{G}^{\alpha }({\bf y}-{\bf x})= 0, \ \forall  \ {\bf x}\in {\mathbb R}^n\setminus \partial \Omega, \ {\bf y}\in \partial \Omega .
\end{equation}
In the case $\alpha =0$, formula \eqref{1a} has been obtained in \cite[(4.84)]{M-W}.

Now, from formula \eqref{1a} and its counterpart corresponding to $\alpha =0$, we obtain for all $j,r=1,\ldots ,n$,
%the following derivative formula of the Brinkman double layer potential
\begin{align}
\label{7ms-dl-6}
\partial_r\left({\bf W}_{\alpha }{\bf h}\right)_j=\partial_r\left({\bf W}{\bf h}\right)_j
+ %\sum_{\ell =1}^n
\partial _\ell \left(\left({\bf V}_{\alpha }-{\bf V}\right)\left(\partial _{\tau _{\ell r}}{\bf h}\right)\right)_j
+ %\sum_{\ell =1}^n
\partial _j\left(\left({\bf V}_{\alpha }-{\bf V}\right)\left(\partial _{\tau _{\ell r}}{\bf h}\right)\right)_\ell -\alpha \left({\bf V}_{\alpha }\left(\nu _r{\bf h}\right)\right)_j,\ \forall \ {\bf h}\in H_p^1(\partial \Omega ,{\mathbb R}^n).
\end{align}
Further, by using estimate (4.86) in \cite[Proposition 4.2.8]{M-W} for the Stokes double layer potential, ${\bf W}{\bf h}$, property \eqref{7ms} for the Brinkman and Stokes single layer potentials involved in formula \eqref{7ms-dl-6}, and continuity of the tangential derivative operators $\partial _{\tau _{jk}}:H_p^1(\partial {\Omega})\to L_p(\partial \Omega )$, we obtain inequality \eqref{7ms-dl-7}, as asserted (see also \cite[(3.35)]{K-Medk-W}).

Finally, inequalities \eqref{7ms-dl-0}, \eqref{7ms-dl-5} and \eqref{7ms-dl-7} imply inequality \eqref{7ms-dl}.

For any $n\geq 3$ and $\ell \geq 0$, there exists a constant $C=C(n,\ell ,\alpha )>0$ such that the inequality (cf. \cite[Theorem 2.4]{Shen})
\begin{align}
\label{Shen-1}
\left|\nabla _{{\bf x}}^\ell {\mathcal G}^{\alpha }({\bf x})\right|\leq \frac{C}{\left(1+\alpha |{\bf x}|^2\right)|{\bf x}|^{n-2+\ell }},
\end{align}
holds and implies that
%\begin{align}
%\label{Shen-2}
$
\left|{\mathcal G}^{\alpha }({\bf x}-{\bf y})\right|%&\leq \frac{C_0}{\left(1+\alpha |{\bf x}-{\bf y}|^2\right)|{\bf x}-{\bf y}|^{n-2}}\nonumber\\
\leq C_0|{\bf x}-{\bf y}|^{2-n},
$
%\end{align}
with some constant $C_0=C_0(n,\alpha )>0$.
Then in view of \cite[Proposition 1]{Med-AAM}, for any ${\bf g}\in L_p(\partial \Omega ,{\mathbb R}^n)$  there exist the non-tangential limits of the Brinkman single layer potential ${\bf V}_{\alpha }{\bf g}$ at almost all points of $\partial \Omega $.
Moreover, the existence of the non-tangential limits of $\nabla {\bf V}_{\alpha }{\bf g}$ at almost all points of $\partial \Omega $ follows immediately from \cite[Lemma 3.3]{Shen}. For ${\mathcal Q}^s{\bf g}$ such a result is valid since the Brinkman pressure single layer potential coincides with the Stokes pressure single layer potential, for which the result is well known, cf., e.g., \cite[Proposition 4.2.2]{M-W} and \cite[Lemma 3.3]{Shen}.

\comment{If ${\bf g}\in H^{-1}_p(\partial \Omega ,{\mathbb R}^n)$ then the existence of the non-tangential limits of ${\bf V}_{\alpha }{\bf g}$ a.e. on $\partial \Omega $ and inequality \eqref{7ms-sl-00} follow from formula \eqref{7ms-sl-001} and the statement from item (ii).}

If ${\bf g}\in H^{-1}_p(\partial \Omega ,{\mathbb R}^n)$ then the existence of the non-tangential limits of ${\bf V}_{\alpha }{\bf g}$ a.e. on $\partial \Omega $
%and inequality \eqref{7ms-sl-00}
follows from formula \eqref{7ms-sl-001} and the corresponding statement for the existence of non-tangential limits for a single layer potential and the gradient a single layer potential with a density in $L_p(\partial \Omega ,{\mathbb R}^n)$.

Now let ${\bf h}\in L_p(\partial \Omega ,{\mathbb R}^n)$.
Then the existence of the non-tangential limits of the Brinkman double layer potential ${\bf W}_{\alpha }{\bf h}$ at almost all points of $\partial \Omega $ follows easily from the case $\alpha =0$.
%and Theorem 2.5 in \cite{Shen}.
Indeed, estimate \eqref{7ms-c1}
%for $\nabla _{{\bf x}}\left\{{\mathcal G}^{\alpha }({\bf x}-{\bf y})-{\mathcal G}^{\alpha }({\bf x}-{\bf y})\right\}$ in this theorem
and \cite[Proposition 1]{Med-AAM} imply that the difference
\begin{multline}
\label{Shen-3}
\left({\bf W}_{\alpha }{\bf h}\right)_j({\bf x})-\left({\bf W}{\bf h}\right)_j({\bf x})
=\int_{\partial \Omega }\left(S_{ijk}^{\alpha }({\bf y}-{\bf x})-S_{ijk}({\bf y}-{\bf x})\right)\nu _k({\bf y})h_i({\bf y})d\sigma _{{\bf y}}\\
=\int_{\partial \Omega }\left\{\left(\frac{\partial {\mathcal G}^{\alpha }_{ij}({\bf y}-{\bf x})}{\partial y_k}-\frac{\partial {\mathcal G}_{ij}({\bf y}-{\bf x})}{\partial y_k}\right)+\left(\frac{\partial {\mathcal G}^{\alpha }_{kj}({\bf y}-{\bf x})}{\partial y_i}-\frac{\partial {\mathcal G}_{kj}({\bf y}-{\bf x})}{\partial y_i}\right)\right\}\nu _k({\bf y})h_i({\bf y})d\sigma _{{\bf y}},
\ {\bf x}\in \Omega_\pm 
\end{multline}
has non-tangential limits $\left(\left({\bf W}_{\alpha }{\bf h}\right)_j-\left({\bf W}{\bf h}\right)_j\right)_{\rm{nt}}^{\pm }({\bf x}_0)$ at almost all points ${\bf x}_0\in \partial \Omega $. On the other hand, according to \cite[Proposition 4.2.2]{M-W} there exist the {non-tangential} limits of the Stokes double layer potential ${\bf W}{\bf h}$ at almost all points ${\bf x}_0$ of $\partial \Omega $.
Therefore, the {non-tangential} limits of the Brinkman double layer potential ${\bf W}_{\alpha }{\bf h}$ exist as well at almost all points ${\bf x}$ of $\partial \Omega $.

Now let ${\bf h}\in H_p^1(\partial \Omega ,{\mathbb R}^n)$.
Then the existence of the non-tangential limits of $\nabla {\bf W}_{\alpha }{\bf h}$ at almost all points of $\partial \Omega $ follows from their existence in the case $\alpha =0$ (cf. \cite[(4.91)]{M-W}), formula \eqref{7ms-dl-6}, and the statement for the existence of non-tangential limits for a single layer potential and the gradient a single layer potential with a density in $L_p(\partial \Omega ,{\mathbb R}^n)$, while the existence of non-tangential limits of ${\mathcal Q}_{\alpha }{\bf h}$ a.e. on $\partial \Omega $ is provided by the corresponding result in the case $\alpha =0$ (cf. \cite[(4.85)]{M-W}) and \cite[Proposition 1]{Med-AAM} applied to the {complementary} term $\left({\mathcal Q}_{\alpha }^d-{\mathcal Q}^d\right){\bf h}=\alpha V_{\triangle }({\bf h}\cdot \boldsymbol \nu )$, which by \eqref{Lambdanu} is the Laplace single layer potential with density $\alpha {\bf h}\cdot \boldsymbol \nu \in L_p(\partial \Omega )$.

Finally, note that inequalities \eqref{ntV}-\eqref{ntW1} follow from inequalities \eqref{7ms}-\eqref{7ms-dl} and the estimate $\|f_{\rm{nt}}^{\pm} \|_{L_p(\partial \Omega )}\leq \|M(f)\|_{L_p(\partial \Omega )}$, whenever $f$ has the property that both $f_{\rm{nt}}^{\pm}$ and $M(f)$ exist a.e on $\partial \Omega $ (see \cite[Remark 9]{Choe-Kim}).
\hfill\end{proof}

The mapping properties of layer potential operators for the {\em Stokes} system {(i.e., for $\alpha=0$)} in Bessel-potential and Besov spaces on bounded Lipschitz domains, as well as their jump relations across a Lipschitz boundary, are well known,
%and we refer the reader to,
cf., e.g., \cite{Fa-Ke-Ve}, \cite{H-W}, \cite[Theorem 10.5.3]{M-W}, \cite[Theorem 3.1, Proposition 3.3]{M-T}.
The main properties of layer potential {operators} for the Brinkman system are collected below (some of them are also available in \cite[Proposition 3.4]{D-M}, \cite[Lemma 3.4]{K-L-M-W},
\cite[Lemma 3.1]{K-L-W}, \cite[Theorem 3.1]{M-T}, \cite[Theorems 3.4 and 3.5]{Shen}).
\begin{theorem}
\label{layer-potential-properties}
Let ${\Omega}_{+}\subset {\mathbb R}^n$ {$(n\geq 3)$} be a bounded Lipschitz domain with connected boundary $\partial {\Omega}$ and let $\Omega _{-}:={\mathbb R}^n\setminus \overline{\Omega }_+$.
Let %$n\geq 3$,
$p,q\in (1,\infty )$, $\alpha >0$, and $p^*:=\max\{p,2\}$. Let $t\ge -\frac{1}{p'}$ be arbitrary, where $\frac{1}{p}+\frac{1}{p'}=1$.
\begin{itemize}
\item[$(i)$] Then the following operators are linear and continuous,
\begin{align}
\label{ss-1}
&{\bf V}_{\alpha }|_{{\Omega}_{+}}:L_p(\partial {\Omega},{\mathbb R}^n)
\to {B}^{1+\frac{1}{p}}_{p,p^*;{\rm div}}({\Omega} _{+},{\mathbb R}^n),\
\mathcal Q^s|_{{\Omega} _{+}}:L_p(\partial {\Omega},{\mathbb R}^n)
\to {B}_{p,p^*}^{\frac{1}{p}}({\Omega} _{+}),\\
\label{fr1}
&\left({\bf V}_\alpha|_{\Omega_+}, \mathcal Q^s|_{\Omega_+}\right):L_p(\partial {\Omega},{\mathbb R}^n)\to \mathfrak{B}_{p,p^*;{\rm div}}^{1+\frac{1}{p},t}(\Omega_+,\mathcal{L}_{\alpha}),\\
\label{ss-1-1}
&{\bf V}_{\alpha }|_{{\Omega}_{+}}:H^{-1}_p(\partial {\Omega},{\mathbb R}^n)
\to {B}^{\frac{1}{p}}_{p,p^*;{\rm div}}({\Omega} _{+},{\mathbb R}^n),\
\mathcal Q^s|_{{\Omega} _{+}}:H^{-1}_p(\partial {\Omega},{\mathbb R}^n)
\to {B}_{p,p^*}^{-1+\frac{1}{p}}({\Omega} _{+}),\\
\label{fr1-1}
&\left({\bf V}_\alpha|_{\Omega_+}, \mathcal Q^s|_{\Omega_+}\right):H^{-1}_p(\partial {\Omega},{\mathbb R}^n)\to \mathfrak{B}_{p,p^*;{\rm div}}^{\frac{1}{p},t}(\Omega_+,\mathcal{L}_{\alpha}),\\
\label{ds-1}
&{\bf W}_{\alpha }|_{{\Omega}_{+}}\!:\!H^1_p(\partial {\Omega},{\mathbb R}^n)
\to {B}^{1+\frac{1}{p}}_{p,p^*;{\rm div}}({\Omega} _{+},{\mathbb R}^n),\
{\mathcal Q}_{\alpha }^d\big|_{{\Omega} _{+}}\!:\!H^1_p(\partial {\Omega},{\mathbb R}^n)\to {B}_{p,p^*}^{\frac{1}{p}}({\Omega} _{+}),\\
\label{fr2}
&\left({\bf W}_\alpha|_{\Omega_+}, \mathcal Q_\alpha^d|_{\Omega_+}\right):H^1_p(\partial {\Omega},{\mathbb R}^n)
\to \mathfrak{B}_{p,p^*;{\rm div}}^{1+\frac{1}{p},t}(\Omega_+,\mathcal{L}_{\alpha}).
\\
\label{ds-1-0}
&{\bf W}_{\alpha }|_{{\Omega}_{+}}\!:\!L_p(\partial {\Omega},{\mathbb R}^n)
\to {B}^{\frac{1}{p}}_{p,p^*;{\rm div}}({\Omega} _{+},{\mathbb R}^n),\
{\mathcal Q}_{\alpha }^d\big|_{{\Omega} _{+}}\!:\!L_p(\partial {\Omega},{\mathbb R}^n)\to {B}_{p,p^*}^{\frac{1}{p}-1}({\Omega} _{+}),\\
\label{fr2-0}
&\left({\bf W}_\alpha|_{\Omega_+}, \mathcal Q_\alpha^d|_{\Omega_+}\right):L_p(\partial {\Omega},{\mathbb R}^n)
\to \mathfrak{B}_{p,p^*;{\rm div}}^{\frac{1}{p},t}(\Omega_+,\mathcal{L}_{\alpha}).
\end{align}
\item[$(ii)$]
{Moreover, the} following operators are also linear and continuous for $s\in (0,1)$,
\begin{align}
\label{exterior-weight}
&{\bf V}_{\alpha }:B_{p,q}^{s-1}(\partial {\Omega},{\mathbb R}^n)
\to B^{s+\frac{1}{p}}_{p,q;{\rm div}}({\mathbb R}^n,{\mathbb R}^n),\
\mathcal Q^s:B_{p,q}^{s-1}(\partial {\Omega},{\mathbb R}^n)
\to {B_{p,q;{\rm{loc}}}^{s+\frac{1}{p}-1}({\mathbb R}^n)},\\
\label{ss-s1}
&{\bf V}_{\alpha }|_{{\Omega}_{+}}:B_{p,q}^{s-1}(\partial {\Omega},{\mathbb R}^n)
\to B^{s+\frac{1}{p}}_{p,q;{\rm div}}({\Omega} _{+},{\mathbb R}^n),\
\left(\mathcal Q^s\right)|_{{\Omega} _{+}}\!:\!B_{p,q}^{s-1}(\partial {\Omega},{\mathbb R}^n)\to B_{p,q}^{s+\frac{1}{p}-1}({\Omega} _{+}),\\
\label{fr3}
&\left({\bf V}_\alpha|_{\Omega_+}, \mathcal Q^s|_{\Omega_+}\right):B_{p,q}^{s-1}(\partial {\Omega},{\mathbb R}^n)
\to \mathfrak{B}_{p,q,\rm div}^{s+\frac{1}{p},t}(\Omega_+,\mathcal{L}_{\alpha}),\\
\label{ds-s1}
&{\bf W}_{\alpha }|_{{\Omega}_{+}}:B_{p,q}^{s}(\partial {\Omega},{\mathbb R}^n)
\to B^{s+\frac{1}{p}}_{p,q;{\rm div}}({\Omega} _{+},{\mathbb R}^n),\
{\mathcal Q}_{\alpha }^d|_{{\Omega} _{+}}\!:\!B_{p,q}^{s}(\partial {\Omega},{\mathbb R}^n)\to B_{p,q}^{s+\frac{1}{p}-1}({\Omega} _{+}),\\
\label{fr4}
&\left({\bf W}_\alpha|_{\Omega_+},\mathcal Q^d_\alpha|_{\Omega_+}\right):B_{p,q}^{s}(\partial {\Omega},{\mathbb R}^n)
\to \mathfrak{B}_{p,q;{\rm div}}^{s+\frac{1}{p},t}(\Omega_+,\mathcal{L}_{\alpha}),\\
\label{exterior-weight-b}
&{\bf V}_{\alpha }|_{{\Omega}_{-}}:B_{p,q}^{s-1}(\partial {\Omega},{\mathbb R}^n)
\to B^{s+\frac{1}{p}}_{p,q;{\rm div}}(\Omega_-,{\mathbb R}^n),\
\mathcal Q^s|_{{\Omega} _{-}}:B_{p,q}^{s-1}(\partial {\Omega},{\mathbb R}^n)\to {B_{p,q;{\rm{loc}}}^{s+\frac{1}{p}-1}(\overline{\Omega}_-)},\\
\label{fr5}
&\left({\bf V}_\alpha|_{\Omega_-}, \mathcal Q^s|_{\Omega_-}\right):B_{p,q}^{s-1}(\partial {\Omega},{\mathbb R}^n)
\to \mathfrak{B}_{p,q;{\rm div};{\rm loc}}^{s+\frac{1}{p},t}(\overline{\Omega}_-,\mathcal{L}_{\alpha}),\\
\label{exterior-weight-s}
&{{\bf W}_{\alpha }|_{{\Omega}_{-}}}:B_{p,q}^{s}(\partial {\Omega},{\mathbb R}^n)\to
B_{p,q;{\rm div};{\rm{loc}}}^{s+\frac{1}{p}}(\overline{\Omega}_{-},{\mathbb R}^n),\
\mathcal Q^d|_{{\Omega} _{-}}:B_{p,q}^{s}(\partial {\Omega},{\mathbb R}^n)\to
{B_{p,q;{\rm{loc}}}^{s+\frac{1}{p}-1}(\overline{\Omega}_{-})},\\
\label{fr6}
&\left({\bf W}_\alpha|_{\Omega_-}, \mathcal Q^d_\alpha|_{\Omega_-}\right):B_{p,q}^{s}(\partial {\Omega},{\mathbb R}^n)
\to \mathfrak{B}_{p,q,{\rm div};{\rm loc}}^{s+\frac{1}{p},t}(\overline{\Omega}_-,\mathcal{L}_{\alpha}).
\end{align}
\item[$(iii)$]
%Let ${\bf h}\in H^1_p(\partial {\Omega},{\mathbb R}^n)$ and ${\bf g}\in L_p(\partial {\Omega},{\mathbb R}^n)$.
%Then
The following relations hold a.e. on $\partial\Omega $,
\begin{align}
\label{68}
&\big({\bf V}_{\alpha }{\bf g}\big)^+_{\rm nt}
=\big({\bf V}_{\alpha }{\bf g}\big)^-_{\rm nt}
=:{\mathcal V}_{\alpha }{\bf g},
\quad\forall\ {\bf g}\in H^{-1}_p(\partial {\Omega},{\mathbb R}^n);\\
\label{68-s1}
&\frac{1}{2}{\bf h}+{({\bf W}_{\alpha }{\bf h})^+_{\rm nt}}
=-\frac{1}{2}{\bf h}+{({\bf W}_{\alpha }{\bf h})^-_{\rm nt}}
=:{\bf K}_{\alpha }{\bf h},
\quad\forall\ {\bf h}\in L_p(\partial {\Omega},{\mathbb R}^n);\\
\label{70aaa}
-&\frac{1}{2}{\bf g}+{\mathbf t_{\rm nt}^{+}}\left({\bf V}_{\alpha }{\bf g},{\mathcal Q}^{s}{\bf g}\right)
=\frac{1}{2}{\bf g}+{\mathbf t_{\rm nt}^{-}}\left({\bf V}_{\alpha }{\bf g},{\mathcal Q}^{s}{\bf g}\right)
=:{\bf K}_{\alpha }^*{\bf g},
\quad\forall\ {\bf g}\in L_p(\partial {\Omega},{\mathbb R}^n);\\
\label{70aaaa}
&{\mathbf t_{\rm nt}^{+}}\big({\bf W}_{\alpha }{\bf h},{\mathcal Q}_{\alpha }^d{\bf h}\big)
={\mathbf t_{\rm nt}^{-}}\big({\bf W}_{\alpha }{\bf h},{\mathcal Q}_{\alpha }^d{\bf h}\big)
=:{\bf D}_{\alpha }{\bf h},
\quad\forall\ {\bf h}\in H^{1}_p(\partial {\Omega},{\mathbb R}^n);
\end{align}
where ${\bf K}_{\alpha }^*$ is the transpose of ${\bf K}_{\alpha ;\partial {\Omega}}$, and the following boundary integral operators are linear and bounded,
\begin{align}
\label{ss-s2}
&{\mathcal V}_{\alpha }:L_p(\partial {\Omega},{\mathbb R}^n)\to H^1_p(\partial {\Omega},{\mathbb R}^n),\
{\bf K}_{\alpha }:H^1_p(\partial {\Omega},{\mathbb R}^n)\to H^1_p(\partial {\Omega},{\mathbb R}^n),\\
\label{ss-s2-1}
&{\mathcal V}_{\alpha }:H^{-1}_p(\partial {\Omega},{\mathbb R}^n)\to L_p(\partial {\Omega},{\mathbb R}^n),\
{\bf K}_{\alpha }:L_p(\partial {\Omega},{\mathbb R}^n)\to L_p(\partial {\Omega},{\mathbb R}^n),\\
\label{ds-s2}
&{\bf K}_{\alpha }^*:L_p(\partial {\Omega},{\mathbb R}^n)\to L_p(\partial {\Omega},{\mathbb R}^n),\ {\bf D}_{\alpha }:H^1_p(\partial {\Omega},{\mathbb R}^n)\to L_p(\partial {\Omega},{\mathbb R}^n).
\end{align}
For ${\bf h}\in B_{p,q}^{s}(\partial {\Omega},{\mathbb R}^n)$ and ${\bf g}\in B^{s-1}_{p,q}(\partial {\Omega},{\mathbb R}^n)$, {$s\in (0,1)$, the following relations hold a.e. on $\partial\Omega $,
\begin{align}
\label{68s}
&{\gamma}_{+}\big({\bf V}_{\alpha }{\bf g}\big)
={\gamma}_{-}\big({\bf V}_{\alpha }{\bf g}\big)
=:{\mathcal V}_{\alpha }{\bf g},\\
%{{\gamma}_{\pm}({\bf V}_{\alpha }{\bf g})=({\bf V}_{\alpha }{\bf g})^\pm_{\rm nt}},\\
\label{68-s1s}
&\frac{1}{2}{\bf h}+{\gamma}_{+}({\bf W}_{\alpha }{\bf h})=
-\frac{1}{2}{\bf h}+{\gamma}_{-}({\bf W}_{\alpha }{\bf h})
=:{\bf K}_{\alpha }{\bf h},\\
%{ {\gamma}_{\pm}({\bf W}_{\alpha }{\bf h})=({\bf W}_{\alpha }{\bf h})^\pm_{\rm nt}},\\
\label{70aaas}
-&\frac{1}{2}{\bf g}+\mathbf t_\alpha^{+}\left({\bf V}_{\alpha }{\bf g},{\mathcal Q}^{s}{\bf g}\right)
=\frac{1}{2}{\bf g}+\mathbf t_\alpha^{-}\left({\bf V}_{\alpha }{\bf g},\mathcal Q_{\partial {\Omega}}^{s}{\bf g}\right)
=:{\bf K}_{\alpha }^*{\bf g},\\
\label{70aaaas}
&\mathbf t_\alpha^{+}\big({\bf W}_{\alpha }{\bf h},{\mathcal Q}_{\alpha }^d{\bf h}\big)
=\mathbf t_\alpha^{-}\big({\bf W})_{\alpha }{\bf h},{\mathcal Q}_{\alpha ;\partial {\Omega}}^d{\bf h}\big)
=:{\bf D}_{\alpha }{\bf h},
\end{align}
and the following operators are linear and continuous,}
\begin{align}
\label{ss-s2-sm}
&{\mathcal V}_{\alpha }:B_{p,q}^{s-1}(\partial {\Omega},{\mathbb R}^n)\to B_{p,q}^{s}(\partial {\Omega},{\mathbb R}^n),\
{\bf K}_{\alpha }:B_{p,q}^{s}(\partial {\Omega},{\mathbb R}^n)\to B_{p,q}^{s}(\partial {\Omega},{\mathbb R}^n),\\
\label{ds-s2-sm}
&{\bf K}_{\alpha }^*:B_{p,q}^{s-1}(\partial {\Omega},{\mathbb R}^n)\to B_{p,q}^{s-1}(\partial {\Omega},{\mathbb R}^n),\ {\bf D}_{\alpha }:B_{p,q}^{s}(\partial {\Omega},{\mathbb R}^n)\to B_{p,q}^{s-1}(\partial {\Omega},{\mathbb R}^n).
\end{align}
\end{itemize}
\end{theorem}
\begin{proof}
{(i)} {First of all, we remark that all range spaces of the velocity vector-valued layer potential operators in \eqref{ss-1}-\eqref{fr6} are
%{subspaces of
divergence-free
%vector fields}
due to the second relations in \eqref{sl}-\eqref{dl0}.}
\comment%
{Recall that for $\mathbf{f}\in \mathcal D(\mathbb R^n,\mathbb R^n)$, the velocity and pressure Newtonian potentials are defined by formula \eqref{Newtonian-Brinkman-vp}, i.e.,
\begin{align}
\label{NoT}
\big({\mathcal N}_{{\alpha ; \mathbb R^n}}\boldsymbol{\varphi}\big)({\bf x})\!:=\!
-\big\langle {\mathcal G}^{\alpha }({\bf x}, \cdot),\boldsymbol{\varphi}
\big\rangle_{_{\!{\mathbb R^n}}}\!=\!
-\int_{\mathbb R^n} {\mathcal G}^{\alpha }({\bf x},{\bf y})\boldsymbol{\varphi}({\bf y})\, d{\bf y},\,
%\label{QoT}
\big(\mathcal Q_{{\mathbb R^n}}\boldsymbol{\varphi}\big)({\bf x})\!:=\!
-\big\langle{\Pi}({\bf x},\cdot),\boldsymbol{\varphi}\big\rangle _{_{\!{ \mathbb R^n}}}\!=\!
-\int_{\mathbb R^n}{\Pi}({\bf x}, {\bf y})\boldsymbol{\varphi}({\bf y})\, d{\bf y},\ {\bf x} \in {\mathbb R}^n.
\end{align}
}
%{Also note that by Lemma~\ref{trace-lemma-Besov}} the trace operator ${\gamma}:B^{1-s+\frac{1}{p'}}_{p',q'}({\mathbb R}^n,{\mathbb R}^n)\to B^{1-s}_{p',q'}(\partial {\Omega},{\mathbb R}^n)$ and { thus} its adjoint, ${\gamma}':B^{s-1}_{p,q}(\partial {\Omega},{\mathbb R}^n)\to B^{s-2+\frac{1}{p}}_{p,q}({\mathbb R}^n,{\mathbb R}^n)$, are continuous operators for all $s\in (0,1)$, $p,p'\in (1,\infty )$ and $q,q'\in (1,\infty )$,
%such that $\frac{1}{p}+\frac{1}{p'}=1$ and $\frac{1}{q}+\frac{1}{q'}=1$
%.
Further, let us note that by \eqref{single-layer} and \eqref{Newtonian-Brinkman-vp} the single layer potential can be presented as (cf. \cite[(4.1)]{Co}),
\begin{align}\label{VNg}
{\bf V}_{\alpha }\mathbf g=\langle \gamma\mathcal{G}^{\alpha}({\bf x},\cdot), {\bf g} \rangle_{\partial {\Omega}}
=\langle \mathcal{G}^{\alpha}({\bf x},\cdot), \gamma'{\bf g} \rangle_{\mathbb R^n}
={\mathbf N}_{\alpha ;{\mathbb R}^n}\circ {\gamma}'\mathbf g
\end{align}
for any $\mathbf g\in B^{s-1}_{p,q}(\partial {\Omega},{\mathbb R}^n)$, $p,q\in (1,\infty)$ and $s\in (0,1)$.
Here the operator
{${\gamma}':B^{s-1}_{p,q}(\partial {\Omega},{\mathbb R}^n)\to {B_{p,q;{\rm{comp}}}^{s-1-\frac{1}{p'}}({\mathbb R}^n},{\mathbb R}^n)$}
is adjoint to the trace operator
{${\gamma}:{B_{p',q';{\rm{loc}}}^{1-s+\frac{1}{p'}}({\mathbb R}^n,{\mathbb R}^n)}\to B^{1-s}_{p',q'}(\partial {\Omega},{\mathbb R}^n)$}
and they both are continues due to Lemma~\ref{trace-lemma-Besov}.

Next, we show the continuity of the first operator in \eqref{ss-1} in the case $\alpha >0$ (i.e., for the {\em Brinkman} system). To this end,  we split the Brinkman single-layer potential operator into two operators, as
%\begin{equation}
%\label{sl-2s}
${\bf V}_{{\alpha }}={\bf V}+{\bf V}_{\alpha ;0},$
%\end{equation}
where ${\bf V}_{{\alpha ;0}}$ is the complementary single-layer potential operator, i.e.,
%${\bf V}_{{\alpha ;0}}:={\bf V}_{{\alpha }}-{\bf V}$. This operator can be written as
\begin{equation}
\label{sl-2}
{\bf V}_{{\alpha ;0}}:={\bf V}_{{\alpha }}-{\bf V}={\bf N}_{{\alpha ;0;\mathbb R^n}}\circ {\gamma}'\circ \iota ,
\end{equation}
where the imbedding operator {$\iota :L_p(\partial\Omega,\mathbb R^n)\hookrightarrow B^{s-1}_{p,p^*}(\partial\Omega,\mathbb R^n)$}
%\end{equation}
is continuous for any $s\in (0,1)$ and $p\in (1,\infty )$. In addition, ${\bf N}_{{\alpha ;0;\mathbb R^n}}:={\mathbf N}_{{\alpha ;\mathbb R^n}}-{\mathbf N}_{{0;\mathbb R^n}}$ is a pseudodifferential operator of order $-4$ with the kernel ${\mathcal G}^{\alpha ;0}:={\mathcal G}^{\alpha }-{\mathcal G}$ (see formula (2.27) in \cite{K-L-W}), and hence the linear operator
\begin{align}
\label{dl}
&{\bf N}_{{\alpha ;0;\mathbb R^n}}:{B^{s-1-\frac{1}{p'}}_{p,p^*;{\rm{comp}}}({\mathbb R}^n,{\mathbb R}^n)}\to {B^{s+3-\frac{1}{p'}}_{p,p^*;{\rm{loc}}}({\mathbb R}^n,{\mathbb R}^n)}
\end{align}
is continuous for any $s\in (0,1)$ and $p\in (1,\infty )$, where $\frac{1}{p'}=1-\frac{1}{p}$, and { ${B^{s-1-\frac{1}{p'}}_{p,p^*;{\rm{comp}}}({\mathbb R}^n,{\mathbb R}^n)}$ is the space of distributions in ${B^{s-1-\frac{1}{p'}}_{p,p^*}({\mathbb R}^n,{\mathbb R}^n)}$ with compact supports.}
%The trace operator
%${\gamma}:{B_{p',{p^*}';{\rm{loc}}}^{1-s+\frac{1}{p'}}({\mathbb R}^n\otimes {\mathbb R}^n)}\to
%B^{1-s}_{p',{p^*}'}(\partial {\Omega},{\mathbb R}^n)$ and {thus} its adjoint
%${\gamma}':B^{s-1}_{p,p^*}(\partial {\Omega},{\mathbb R}^n\otimes {\mathbb R}^n)\to
%{B_{p,p^*;{\rm{comp}}}^{s-1-\frac{1}{p'}}({\mathbb R}^n,{\mathbb R}^n\otimes {\mathbb R}^n)}$ are continuous ({cf. Lemma~\ref{trace-lemma-Besov}}).
Then formula \eqref{sl-2} and the continuity of the involved operators imply that the operators
\begin{align}
{\bf V}_{{\alpha ;0}}:L_p(\partial {\Omega},{\mathbb R}^n)\to B^{s+2+\frac{1}{p}}_{p,p^*;{\rm{loc}}}({\mathbb R}^n,{\mathbb R}^n),\
\left({\bf V}_{{\alpha ;0}}\right)|_{{\Omega}_{+}}:L_p(\partial {\Omega},{\mathbb R}^n)\to B_{p,p^*}^{s+2+\frac{1}{p}}({\Omega}_{+},{\mathbb R}^n)\nonumber
\end{align}
are continuous as well. Now, the continuity of the embedding
$B^{s+2+\frac{1}{p}}_{p,p^*}({\Omega}_{+},{\mathbb R}^n)\hookrightarrow
B^{1+\frac{1}{p}}_{p,p^*}({\Omega}_{+},{\mathbb R}^n)$ {for any $s\in (0,1)$} shows that
\begin{align}
\label{sl-compl-Lp}
{\bf V}_{{\alpha ;0}}:L_p(\partial {\Omega},{\mathbb R}^n)\to B^{1+\frac{1}{p}}_{p,p^*}({\Omega}_{+},{\mathbb R}^n)
\end{align}
is a continuous operator, even compact.

Moreover, %the continuity of the embeddings $L_{p}(\partial \Omega ,{\mathbb R}^n)\hookrightarrow B^{s-1}_{p,p^*}(\partial \Omega ,{\mathbb R}^{n})$ and $B^{s+\frac{1}{p}}_{p,p^*}({\Omega}_{+},{\mathbb R}^n)\hookrightarrow L_{p}({\Omega}_{+},{\mathbb R}^n)$ for any $s\in (0,1)$ and the first mapping property in \eqref{ss-s1} (for $\alpha =0$) imply that
the Stokes single layer potential operator ${\bf V}:L_{p}(\partial \Omega ,{\mathbb R}^n)\to L_{p}({\Omega}_{+},{\mathbb R}^{n})$ is continuous (cf., e.g., the mapping property (10.73) in \cite{M-W} and the continuity of the embeddings $L_{p}(\partial \Omega ,{\mathbb R}^n)\hookrightarrow B^{s-1}_{p,p^*}(\partial \Omega ,{\mathbb R}^{n})$ and $B^{s+\frac{1}{p}}_{p,p^*}({\Omega}_{+},{\mathbb R}^n)\hookrightarrow L_{p}({\Omega}_{+},{\mathbb R}^n)$ for any $s\in (0,1)$).

On the other hand, the kernel $\nabla {\mathcal G}$ of the integral operator $\nabla {\bf V}$ satisfies the relations
\begin{align}
\label{1ms}
\nabla {\mathcal G}\in C^{\infty }({\mathbb R}^n\setminus \{0\}),\ (\nabla {\mathcal G})(-{\bf x})
=-(\nabla {\mathcal G})({\bf x}),\ (\nabla {\mathcal G})(\lambda {\bf x})
={\lambda^{-(n-1)}(\nabla {\mathcal G})({\bf x})},\ \forall \ \lambda >0.
\end{align}
Then, in view of \cite[Proposition 2.68]{M-M1}, there exists a constant $C_0\equiv C_0(\Omega _{+},p)>0$ such that
\begin{align}
\label{7ms-1}
&{\|\nabla {\bf V}{\bf g}\|_{B_{p,p^*}^{\frac{1}{p}}(\Omega _{+},{\mathbb R}^{n\times n})}\leq C_0\|{\bf g}\|_{L_p(\partial \Omega ,{\mathbb R}^n)},\ \ \forall \ {\bf g}\in L_p(\partial \Omega ,{\mathbb R}^n).}
\end{align}
Consequently, there exists a constant $\mathfrak{C}\equiv \mathfrak{C}(\Omega _{+},p)>0$ such that
\begin{align}
\label{continuity-sl-p}
\|{\bf V}{\bf g}\|_{B^{1+\frac{1}{p}}_{p,p^*}({\Omega}_{+},{\mathbb R}^{n})}=\|{\bf V}{\bf g}\|_{L_{p}({\Omega}_{+},{\mathbb R}^{n})}+\|\nabla {\bf V}{\bf g}\|_{B^{\frac{1}{p}}_{p,p^*}({\Omega}_{+},{\mathbb R}^{n\times n})}\leq \mathfrak{C}\|{\bf g}\|_{L_{p}({\Omega}_{+},{\mathbb R}^{n})},\ \forall \ {\bf g}\in L_{p}({\Omega}_{+},{\mathbb R}^{n}),
\end{align}
which shows that the Stokes single layer potential operator
%${\bf V}:L_{p}(\partial \Omega ,{\mathbb R}^n)\to B^{1+\frac{1}{p}}_{p,p^*}({\Omega}_{+},{\mathbb R}^{n})$
\begin{align}
\label{Stokes-sl-p}
{\bf V}:L_p(\partial \Omega ,{\mathbb R}^n)\to B_{p,p^*}^{1+\frac{1}{p}}(\Omega _+, {\mathbb R}^n)
\end{align}
is also continuous {(cf., e.g., \cite[Theorem 7.1, (3.33)]{M-T}, see also \cite{Fa-Ke-Ve} for $p=2$)}.
This mapping property and the continuity of operator \eqref{sl-compl-Lp} show that the Brinkman single layer operator ${\bf V}_{\alpha }:L_{p}(\partial \Omega ,{\mathbb R}^n)\to B^{1+\frac{1}{p}}_{p,p^*}({\Omega}_{+},{\mathbb R}^{n})$ is continuous, as well. %Thus, we have proved the continuity of the first operator in \eqref{ss-1}.

{{Let us show} the continuity of the second operator in \eqref{ss-1}. To this end, we note that the Stokes single layer pressure potential ${\mathcal Q}^s{\bf f}$ with a density ${\bf f}=(f_1,\ldots ,f_n)\in L_p(\partial \Omega ,{\mathbb R}^n)$ can be written as
\begin{align}
\label{Laplace-Newtonian-potential-0s}
\left({\mathcal Q}^s{\bf f}\right)({\bf x})=\left(\mathrm{div}\, V_{\triangle }{\bf f}\right)({\bf x}),\ \ \forall \ {\bf x}\in {\mathbb R}^n\setminus \partial \Omega ,
\end{align}
where $V_{\triangle }g$ is the harmonic single layer potential with density $g\in L_p(\partial \Omega )$, given by
\begin{align}
\label{Laplace-sl-p}
(V_{\triangle }g)({\bf x}):=-\frac{1}{(n-2)\tilde\omega _n}\int _{\partial \Omega }\frac{1}{|{\bf x}-{\bf y}|^{n-2}}g({\bf y})d\sigma _{{\bf y}},\ \ {\bf x}\in {\mathbb R}^n\setminus \partial \Omega .
\end{align}
Then the continuity of the single layer pressure potential potential operator $\mathcal Q^s:L_p(\partial {\Omega},{\mathbb R}^n)\to {B}_{p,p^*}^{\frac{1}{p}}({\Omega} _{+})$ for any $p\in (1,\infty )$ is a direct consequence of Proposition 4.23 in \cite{M-M-W1}. Note that Proposition 2.68 in \cite{M-M1} applies as well, and shows the desired continuity of the single layer pressure potential operator in \eqref{ss-1} {(see also \cite[Theorem 3.1, (3.30)]{M-T})}.
%Moreover, the continuity of the Brinkman single layer potential operator ${\mathcal Q}_{\alpha }^s:L_p(\partial {\Omega},{\mathbb R}^n)\to {B}_{p,p^*}^{\frac{1}{p}}({\Omega}_+)$ for any $p\in (1,+\infty )$ follows from the equality between ${\mathcal Q}_{\alpha }^s$ and ${\mathcal Q}^s$.
Thus, we have proved the {continuity} of the operators in \eqref{ss-1}.}

Continuity of the first operator in \eqref{ss-1-1} follows from the continuity of operators involved in the right hand side of equality \eqref{7ms-sl-001}.
Continuity of the second operator in \eqref{ss-1-1} follows from equality \eqref{Laplace-Newtonian-potential-0s}, which is valid also for any ${\bf f}\in H^{-1}_p(\partial \Omega ,{\mathbb R}^n)$, and by the continuity of the harmonic single layer potential operator $V_{\triangle }$ from $H^{-1}_p(\partial \Omega )$ to ${B}_{p,p^*}^{\frac{1}{p}}({\Omega} _{+})$.
Indeed, for any $f\in H_p^{-1}(\partial \Omega )$ there exist $f_{0},f_{r\ell }\in L_p(\partial \Omega )$, $r,\ell =1,\ldots n$, such that $f=f_{0}+\sum_{r,\ell =1}^n\partial _{\tau _{r\ell }}f_{r\ell }$ ({see} \eqref{estimate-Lp}). Then by using the integration by parts formula \eqref{int-parts}, we obtain that
\begin{align}
\label{7ms-sl-001-Laplace}
(V_{\triangle }f)({\bf x})
=\int _{\partial \Omega }{\mathcal G}_{\triangle }({\bf x}-{\bf y})f_{0}({\bf y})d\sigma _{\bf y}
-\sum_{r,\ell =1}^n\int_{\partial \Omega }\left(\partial _{\tau _{rs,{\bf y}}}{\mathcal G}_{\triangle}({\bf x},{\bf y})\right)f_{r\ell }({\bf y})d\sigma _{{\bf y}},\,
\forall \ {\bf x}\in {\mathbb R}^n\setminus \partial \Omega ,
\end{align}
where ${\mathcal G}_{\triangle }({\bf x},{\bf y})$ is the fundamental solution of the Laplace equation in ${\mathbb R}^n$ ($n\geq 3$). By using again \cite[Proposition 2.68]{M-M1} (see also \eqref{7ms-1}) and the continuity of the Laplace single layer potential operator
%\begin{align}
%\label{Laplace-single-layer}
$
V_{\triangle }:L_p(\partial \Omega )\to B_{p,p^*}^{1+\frac{1}{p}}(\Omega _+)
$
%\end{align}
(see, e.g., \cite[Proposition 4.23]{M-M-W1} and property (3.49) in \cite[Proposition 3.3]{M-T}), there exists a constant $C_0$ such that
\begin{align}
\label{continuity-sl-Laplace}
\|V_{\triangle }f\|_{B^{1+\frac{1}{p}}_{p,p^*}({\Omega}_{+})}=\|V_{\triangle } f\|_{L_{p}({\Omega}_{+})}+\|\nabla V_{\triangle }f\|_{B^{\frac{1}{p}}_{p,p^*}({\Omega}_{+},{\mathbb R}^{n})}\leq C_0\|f\|_{L_{p}({\Omega}_{+})},\ \forall \ f\in L_{p}({\Omega}_{+}).
\end{align}

Thus, the operator %$V_{\triangle }:L_{p}({\Omega}_{+})\to L_{p}({\Omega}_{+})$ and
$\nabla V_{\triangle }:L_{p}(\partial\Omega)\to B^{\frac{1}{p}}_{p,p^*}({\Omega}_{+},{\mathbb R}^{n})$ is also continuous.
Finally, by {continuity of this operator and of the operator
$V_{\triangle }:L_p(\partial \Omega )\to B_{p,p^*}^{1+\frac{1}{p}}(\Omega _+)$
and also by the second relation in \eqref{estimate-Lp}, we obtain from \eqref{7ms-sl-001-Laplace} continuity of the operator
$V_\triangle :H^{-1}_p(\partial \Omega )\to {B}_{p,p^*}^{\frac{1}{p}}({\Omega} _{+})$
and, accordingly, continuity} of the second operator in \eqref{ss-1-1}.

Let us now show the continuity of the first operator in \eqref{ds-1}. To this end, we notice that the Brinkman double-layer potential operator can be written as
%\begin{equation}
%\label{dl-2s}
${\bf W}_{{\alpha }}={\bf W}+{\bf W}_{\alpha ;0},$
%\end{equation}
where ${\bf W}_{{\alpha ;0}}$ is the complementary double layer potential operator, i.e.,
%${\bf W}_{{\alpha ;0}}:={\bf W}_{{\alpha }}-{\bf W}$. This operator can be written as
\begin{equation}
\label{dl-2}
{\bf W}_{{\alpha ;0}}:={\bf W}_{{\alpha }}-{\bf W}={\mathbb K}_{{\alpha ;0}}\circ {\gamma}'\circ \mathfrak N
\end{equation}
(see \cite[{Eq.} (3.31)]{K-L-W}), where the operator $\mathfrak N:H_p^1(\partial\Omega,\mathbb R^n)\to L_p(\partial\Omega,\mathbb R^n\otimes\mathbb R^n)\hookrightarrow B^{-s}_{p,p^*}(\partial\Omega,\mathbb R^n\otimes\mathbb R^n),\
\mathfrak N{\bf h}(x):={\boldsymbol \nu}(x)\otimes {\bf h}(x)$,
%\end{equation}
is continuous for any $s\in (0,1)$. In addition, ${\mathbb K}_{{\alpha ;0}}$ is a pseudodifferential operator of order $-3$ with the kernel ${\bf S}^{\alpha ;0}:={\bf S}^{\alpha }-{\bf S}$ (cf., e.g., \cite[(2.27)]{K-L-W}), and hence the operator
\begin{align}
\label{dl}
&{\mathbb K}_{{\alpha ;0}}:{B^{-1-s+\frac{1}{p}}_{p,p^*;{\rm{comp}}}({\mathbb R}^n,{\mathbb R}^n\otimes {\mathbb R}^n)}\to {B^{2-s+\frac{1}{p}}_{p,p^*;{\rm{loc}}}({\mathbb R}^n,{\mathbb R}^n)},\nonumber\\
&\left({\mathbb K}_{{\alpha ;0}}{\mathbb T}\right)_j({\bf x}):=\left\langle \big({S}^{\alpha }_{ji\ell }-{S}_{ji\ell }\big)(\cdot ,{\bf x}),T_{_{i\ell }}\right\rangle _{\mathbb R^n}, \ \forall\ {\mathbb T}\in B^{-1-s+\frac{1}{p}}_{p,p^*;{\rm{comp}}}({\mathbb R}^n,{\mathbb R}^n\otimes {\mathbb R}^n),
\end{align}
is also linear and continuous for any $s\in (0,1)$, where ${B^{-1-s+\frac{1}{p}}_{p,p^*;{\rm{comp}}}({\mathbb R}^n,{\mathbb R}^n\otimes {\mathbb R}^n)}$ is the space of all distributions in $B^{-1-s+\frac{1}{p}}_{p,p^*}({\mathbb R}^n,{\mathbb R}^n\otimes {\mathbb R}^n)$ having compact support in ${\mathbb R}^n$. In addition, the trace operator
${\gamma}:{B_{p',{p^*}';{\rm{loc}}}^{s+\frac{1}{p'}}({\mathbb R}^n\otimes {\mathbb R}^n)}\to
B^s_{p',{p^*}'}(\partial {\Omega},{\mathbb R}^n\otimes {\mathbb R}^n)$
(acting on matrix valued functions) and its adjoint
${\gamma}':B^{-s}_{p,p^*}(\partial {\Omega},{\mathbb R}^n\otimes {\mathbb R}^n)\to
{B_{p,p^*;{\rm{comp}}}^{-s-1+\frac{1}{p}}({\mathbb R}^n,{\mathbb R}^n\otimes {\mathbb R}^n)}$ are continuous (see the proof of \cite[Theorem 1]{Co}). Then formula \eqref{dl-2} and the continuity of the involved operators imply that the operators
\begin{align}
{\bf W}_{{\alpha ;0}}:H_p^1(\partial {\Omega},{\mathbb R}^n)\to B^{2-s+\frac{1}{p}}_{p,p^*;{\rm{loc}}}({\mathbb R}^n,{\mathbb R}^n),\
\left({\bf W}_{{\alpha ;0}}\right)|_{{\Omega}_{+}}:H_p^1(\partial {\Omega},{\mathbb R}^n)\to B_{p,p^*}^{2-s+\frac{1}{p}}({\Omega}_{+},{\mathbb R}^n)\nonumber
\end{align}
are continuous as well. Now, the continuity of the embedding
$B^{2+\frac{1}{p}-s}_{p,p^*}({\Omega}_{+},{\mathbb R}^n)\hookrightarrow
B^{1+\frac{1}{p}}_{p,p^*}({\Omega}_{+},{\mathbb R}^n)$ {for any $s\in (0,1)$} shows that
\begin{align}
\label{dl-compl-H1p}
{\bf W}_{{\alpha ;0}}:H_p^1(\partial {\Omega},{\mathbb R}^n)\to B^{1+\frac{1}{p}}_{p,p^*}({\Omega}_{+},{\mathbb R}^n)
\end{align}
is a continuous operator, even compact.
{Let us now show that the Stokes double-layer potential operator
\begin{align}
\label{dl-Stokes-H1p}
{\bf W}:H_p^1(\partial {\Omega},{\mathbb R}^n)\to B^{1+\frac{1}{p}}_{p,p^*}({\Omega}_{+},{\mathbb R}^n)
\end{align}
is continuous as well. {In the setting of Riemannian manifolds and {for double layer potentials for second order elliptic equations}, this continuity property follows from \cite[Theorem 8.5]{M-T2}, but we will provide a direct proof here in the context of Euclidean setting}. To this end, we use the following characterization of the space $H_p^1(\partial {\Omega})$
\begin{align}
\label{H1p}
h\in H_p^1(\partial {\Omega}) \Longleftrightarrow h\in L_p(\partial \Omega ),\ \partial _{\tau _{jk}}h\in L_p(\partial \Omega ),\ \ j,k=1,\ldots ,n
\end{align}
%where
%$\partial _{\tau _{jk}}:=\nu _j\partial _k-\nu _k\partial _j$ are the tangential derivative operators, and $\partial _j:=\dfrac{\partial }{\partial x_j}$.
(cf., e.g., \cite[(2.11)]{M-W}), and recall that the tangential derivative operators $\partial _{\tau _{jk}}:H_p^1(\partial {\Omega})\to L_p(\partial \Omega )$ are continuous. In addition,
%for every $g\in L_p(\partial \Omega )$,
consider the operator $V_{jk}$ defined as
% ${V}_{jk}g:{\mathbb R}^n\setminus {\partial \Omega }\to {\mathbb R}$, where
\begin{align}
\label{s-Stokes}
\left({V}_{jk}g\right)({\bf x}):=\int_{\partial \Omega }{\mathcal G}_{jk}({\bf x}-{\bf y})g({\bf y})d\sigma _{\bf y},\ \ {\bf x}\in {\mathbb R}^n\setminus {\partial \Omega }.
\end{align}
We have proved that the Stokes single layer potential operator \eqref{Stokes-sl-p}
is continuous for any $p\in (1,\infty )$ {(see also \cite[Theorem 3.1, (3.33)]{M-T})}. Consequently, the operators
\begin{align}
\label{s-Stokes-1}
V_{jk}:L_p(\partial \Omega )\to B_{p,p^*}^{1+\frac{1}{p}}(\Omega _+)
\end{align}
are continuous as well, for all $j,k=1,\ldots ,n$. Recall that the operator
%$V_{\triangle }$ from $L_p(\partial \Omega )$ to $B_{p,p^*}^{1+\frac{1}{p}}(\Omega _+)$
$V_{\triangle }:L_p(\partial \Omega )\to B_{p,p^*}^{1+\frac{1}{p}}(\Omega _+)$
is also linear and {continuous.}
%(see \eqref{Laplace-single-layer}).
Finally, we mention the following formula (cf. \cite[(4.84)]{M-W})
\begin{align}
\label{dl-sl}
\partial _r\left({\bf W}{\bf h}\right)_j=-\partial _\ell {V}_{jk}\left(\partial _{\tau _{\ell r}}h_k\right)-\partial _j{V}_{\ell k}\left(\partial _{\tau _{\ell r}}h_k\right)-\partial _k V_{\triangle }\left(\partial _{\tau _{jr}}h_k\right)\ \mbox{ in }\ {\mathbb R}^n\setminus \partial \Omega,
\end{align}
which holds for every ${\bf h}\in H_p^1(\partial \Omega ,{\mathbb R}^n)$ and $j,r=1,\ldots ,n$, where $h_j$ is the $j$-th component of ${\bf h}$. Then by using the continuity of {operator \eqref{s-Stokes-1} and} %\eqref{Laplace-single-layer}, and
properties \eqref{H1p} and \eqref{dl-sl}, we deduce that the operators
\begin{align}
\partial _r\left({\bf W}\right)_j:H_p^1(\partial \Omega ,{\mathbb R}^n)\to B_{p,p^*}^{\frac{1}{p}}(\Omega _+), \ \ r,j=1,\ldots ,n
\end{align}
are continuous. {By} \cite[Proposition 10.5.1, (10.68)]{M-W}, the operator ${\bf W}:H_p^1(\partial \Omega ,{\mathbb R}^n)\to L_p(\Omega _+,{\mathbb R}^n)$ is also continuous (as its range is a subspace of the space $H_{p}^{s+\frac{1}{p}}(\Omega _+,{\mathbb R}^n)$ for any $s\in (0,1)$, $H_p^1(\partial \Omega ,{\mathbb R}^n)\hookrightarrow B_{p,p}^{s}(\partial \Omega ,{\mathbb R}^n)$ (due to formula \eqref{real-int}), and $B_{p,p}^{s+\frac{1}{p}}(\Omega _+,{\mathbb R}^n)\hookrightarrow L_{p}(\Omega _+,{\mathbb R}^n)$).
Consequently, the Stokes double layer potential operator
${\bf W}:H_p^1(\partial {\Omega},{\mathbb R}^n)\to B^{1+\frac{1}{p}}_{p,p^*}({\Omega}_{+},{\mathbb R}^n)$ is continuous, as asserted.
This mapping property combined with the continuity of {operator \eqref{dl-compl-H1p} implies
%the continuity of the Brinkman double layer potential operator ${\bf W}_{\alpha }:H_p^1(\partial {\Omega},{\mathbb R}^n)\to B^{1+\frac{1}{p}}_{p,p^*}({\Omega}_{+},{\mathbb R}^n)$, i.e.,
the continuity} of the first operator in \eqref{ds-1}.

Continuity of the second operator in \eqref{ds-1} follows from similar {arguments}. {To this end, let us mention the useful formula ${\mathcal Q}^d{\bf g}={\rm{div}}(W_{\triangle }{\bf g})$,
where the harmonic double layer potential operator $W_{\triangle }:H_p^1(\partial \Omega )\to B_{p,p^*}^{1+\frac{1}{p}}(\Omega _{+})$ is continuous (cf., e.g., \cite[Proposition 4.23, (2.120), (4.96)]{M-M-W1}).
Thus, the continuity of the Stokes double layer pressure potential operator ${\mathcal Q}^d:H_p^1(\partial {\Omega},{\mathbb R}^n)\to B^{\frac{1}{p}}_{p,p^*}({\Omega}_{+})$ immediately follows.
This property and continuity of the complementary double layer potential operator ${\mathcal Q}_{\alpha ;0}^d:={\mathcal Q}_{\alpha }^d-{\mathcal Q}^d:H_p^1(\partial {\Omega},{\mathbb R}^n)\to B^{\frac{1}{p}}_{p,p^*}({\Omega}_{+})$,
where (cf. \cite[(3.10)]{Shen})
\begin{align}
\label{pressure-dl-B}
{\mathcal Q}_{\alpha;0}^d{\bf h}=\alpha V_{\triangle }({\bf h}\cdot \boldsymbol \nu ),
\end{align}
yield the continuity of the Brinkman double layer pressure potential operator ${\mathcal Q}_{\alpha }^d={\mathcal Q}^d+{\mathcal Q}_{\alpha ;0}^d:H_p^1(\partial {\Omega},{\mathbb R}^n)\to B^{\frac{1}{p}}_{p,p^*}({\Omega}_{+})$}.}

Continuity of the first operator in \eqref{ds-1-0} for the case $\alpha =0$ is an immediate consequence of \cite[Proposition 2.68]{M-M1} applied to the integral operator whose kernel is given by the fundamental stress tensor ${\bf S}^0$.
%Thus, the operator ${\bf W}$ from $L_p(\partial \Omega ,{\mathbb R}^n)$ to $B_{p,p^*}^{\frac{1}{p}}(\partial \Omega ,{\mathbb R}^n)$ is continuous.
Moreover, by using again formulas \eqref{dl-2} and \eqref{dl} we can see that the operator
${\bf W}_{\alpha ;0}: L_p(\partial {\Omega},{\mathbb R}^n) \to B^{\frac{1}{p}}_{p,p^*}({\Omega}_{+},{\mathbb R}^n)$
is continuous. Therefore, for $\alpha >0$ the first operator in \eqref{ds-1-0} is continuous as well.
To prove continuity of the second operator in \eqref{ds-1-0}, we again use the representation
${\mathcal Q}^d{\bf g}={\rm{div}}(W_{\triangle }{\bf g})$,
and continuity of the harmonic double layer potential operator $W_{\triangle }:L_p(\partial \Omega )\to B_{p,p^*}^{\frac{1}{p}}(\Omega _{+})$, e.g., again  by \cite[Proposition 2.68]{M-M1},
along with continuity of the complementary double layer potential operator
$\mathcal Q_{\alpha ;0}^d:L_p(\partial {\Omega},{\mathbb R}^n)\to B^{\frac{1}{p}-1}_{p,p^*}({\Omega}_{+})$.
%presented as in \eqref{pressure-dl-B}.

Mapping properties \eqref{fr1}{, \eqref{fr1-1}} and \eqref{fr2} are implied by the ones just above them and by the first relations in \eqref{sl}-\eqref{dl0}.

${(ii)}$ Now, relation \eqref{VNg},
%${\bf V}_{\alpha }={\gr\mathbf N}_{\alpha ;{\mathbb R}^n}\circ {\gamma}'$ (see also \cite[(4.1)]{Co}),
continuity of the operator ${\gamma}':B^{s-1}_{p,q}(\partial {\Omega},{\mathbb R}^n)\to B^{s-2+\frac{1}{p}}_{p,q}({\mathbb R}^n,{\mathbb R}^n)$ (cf. Lemma~\ref{trace-lemma-Besov}),
%which is the adjoint of the trace operator ${\gamma}:B^{1-s+\frac{1}{p'}}_{p',q'}({\mathbb R}^n,{\mathbb R}^n)\to B^{1-s}_{p',q'}(\partial {\Omega},{\mathbb R}^n)$,
and continuity of the Newtonian potential operator ${\mathbf N}_{\alpha ;{\mathbb R}^n}:B^{s-2+\frac{1}{p}}_{p,q}({\mathbb R}^n,{\mathbb R}^n)\to B^{s+\frac{1}{p}}_{p,q}({\mathbb R}^n,{\mathbb R}^n)$ (see \eqref{ss-2}) imply
%that the operator ${\bf V}_{\alpha }:B^{s-1}_{p,q}(\partial {\Omega},{\mathbb R}^n)\to B^{s+\frac{1}{p}}_{p,q}({\mathbb R}^n,{\mathbb R}^n)$ is also continuous. Hence we have proved
the continuity of
the first operator in \eqref{exterior-weight} and thus of the first operators in \eqref{ss-s1} and \eqref{exterior-weight-b}. {Continuity} of the second operator in \eqref{exterior-weight} follows {by} similar arguments based on the equalities
$ %{\mathcal Q}_{\alpha }^s=
{\mathcal Q}^s={\mathcal Q}_{{\mathbb R}^n}\circ \gamma '$, and implies also {continuity} of the second operators in \eqref{ss-s1} and \eqref{exterior-weight-b} ({cf.} \cite[Proposition 10.5.1]{M-W}).

Further, let us mention that relations \eqref{dl-2} and \eqref{dl} imply that the operator
${\bf W}_{{\alpha ;0}}:B^s_{p,q}(\partial {\Omega},{\mathbb R}^n)\to B^{s+\frac{1}{p}}_{p,q}({\Omega}_+,{\mathbb R}^n)$ is continuous for all $p\in (1,+\infty )$ and $s\in (0,1)$. This mapping property combined with the continuity of the Stokes double-layer potential operator
${\bf W}|_{{\Omega}_{+}}:B_{p,q}^{s}(\partial {\Omega}_+,{\mathbb R}^n)\to B^{s+\frac{1}{p}}_{p,q}({\Omega} _{+},{\mathbb R}^n)$ (see \cite[Proposition 10.5.1]{M-W}) implies the continuity of the first operator in \eqref{ds-s1}. The continuity of the second operator in \eqref{ds-s1} can be similarly {obtained}. Other mapping properties of layer potentials mentioned in \eqref{exterior-weight} and \eqref{exterior-weight-s}, follow with similar arguments to those for \eqref{ss-1} and \eqref{ds-1}. We omit the details for the sake of brevity (see also the proof of \cite[Lemma 3.4]{K-L-M-W}).

$(iii)$ Equality \eqref{68}  for ${\bf g}\in L_p(\partial {\Omega},{\mathbb R}^n)$ can be obtained by using inequality \eqref{Shen-1} and \cite[Proposition 1]{Med-AAM} (see also \cite[Theorems 3.4]{Shen}).
Since $\big({\bf V}_{\alpha }{\bf g}\big)^+_{\rm nt}$
and $\big({\bf V}_{\alpha }{\bf g}\big)^-_{\rm nt}$ are well defined for
${\bf g}\in H^{-1}_p(\partial {\Omega},{\mathbb R}^n)$ due to Lemma~\ref{nontangential-sl}(iii), inequality \eqref{7ms-sl-00} and the density argument then imply equality  \eqref{68} also for ${\bf g}\in H^{-1}_p(\partial {\Omega},{\mathbb R}^n)$.
Formulas \eqref{68-s1} and \eqref{70aaa} follow {by} using arguments similar to those for the trace formulas (3.11) and (3.18) in \cite{Shen}. To this end, {we first prove the formulas}
\begin{align}
\label{f-1-a}
\left(\partial _j\left(V_{ik}^{\alpha }g\right)\right)\big|_{{\rm{nt}}}^{\pm }({\bf x})=\pm \frac{1}{2}\nu _j({\bf x})\left(\delta _{ik}-\nu _i({\bf x})\nu _k({\bf x})\right)g({\bf x})+
{\rm{p.v.}}\int_{\partial \Omega }\partial _j{\mathcal G}_{ik}^{\alpha }({\bf x}-{\bf y})g({\bf y})d\sigma _{{\bf y}}\ \mbox{ a.a. }\ {\bf x}\in \partial \Omega
\end{align}
{for any $g\in L_p(\partial \Omega )$ and all $i,k=1,\ldots ,n$,}
where the function $V_{ik}^{\alpha }g$ is defined as in \eqref{s-Stokes} with ${\mathcal G}_{jk}^{\alpha }$ instead of ${\mathcal G}_{jk}$.
Indeed, formula \eqref{f-1-a} has been proved in \cite[(4.50)]{M-W} in the case $\alpha =0$.
Moreover, the estimate \cite[(2.27)]{Shen} of the kernel $\nabla _{{\bf x}}{\mathcal G}_{jk}^{\alpha }({\bf x})-\nabla _{{\bf x}}{\mathcal G}_{jk}({\bf x})$ and \cite[Proposition 1]{Med-AAM} imply that there exist the {non-tangential} limits of the {complementary} potential $\partial _jV_{ik}^{\alpha }g-\partial _jV_{ik}g$ at almost all points of $\partial \Omega $, and
\begin{align}
\label{f-1-b}
\left(\partial _j\left(V_{ik}^{\alpha }g\right)-\partial _j\left(V_{ik}g\right)\right)|_{{\rm{nt}}}^{\pm }({\bf x})={\rm{p.v.}}\int_{\partial \Omega }\left(\partial _j{\mathcal G}_{ik}^{\alpha }-\partial _j{\mathcal G}_{ik}\right)({\bf x}-{\bf y})g({\bf y})d\sigma _{{\bf y}}\ \mbox{ a.a. }\
{\bf x}\in \partial \Omega ,
\end{align}
which implies \eqref{f-1-a} also for $\alpha\not=0$.
%immediately follows from its counterpart in the case $\alpha =0$ and formula \eqref{f-1-b}.
Moreover, formula \eqref{f-1-a} yields for any ${\bf f}\in L_p(\partial \Omega ,{\mathbb R}^n)$ that
\begin{align}
\label{68-s1-proof1}
\left(\partial _j({\bf V}_{\alpha }{\bf f})\right)\big|_{{\rm{nt}}}^{\pm }({\bf x})=\pm \frac{1}{2}\nu _j({\bf x})\left\{{\bf f}({\bf x})-f_k({\bf x})\nu _k({\bf x})\boldsymbol \nu ({\bf x})\right\}+{\rm{p.v.}}\int_{\partial \Omega }\partial _j{\mathcal G}^{\alpha }({\bf x}-{\bf y}){\bf f}({\bf y})d\sigma _{{\bf y}}\ \mbox{ a.a. }\ {\bf x}\in \partial \Omega
\end{align}
(cf. \cite[(4.54)]{M-W} for $\alpha =0$ and \cite[Lemma 3.3]{Shen} for $\alpha >0$).

In addition,
\begin{align}
\label{68-s1-proof2}
\left({\mathcal Q}^s{\bf f}\right)\big|_{{\rm{nt}}}^{\pm }({\bf x})=\mp\frac{1}{2}\nu_k({\bf x})f_k({\bf x})+{\rm{p.v.}}\int_{\partial \Omega }\Pi _k({\bf x}-{\bf y})f_k({\bf y})d\sigma _{{\bf y}} \ \mbox{ a.a. }\ {\bf x}\in \partial \Omega
\end{align}
(cf. \cite[(4.42)]{M-W}, \cite[Lemma 3.3]{Shen}). {Then formulas \eqref{68-s1} and \eqref{70aaa} follow from formulas \eqref{classical-stress}, \eqref{2.37nt}, \eqref{stress-tensor-alpha}, \eqref{dl-velocity}, \eqref{68-s1-proof1} and \eqref{68-s1-proof2}}.

%{[***SM: More arguments are needed about \eqref{68-s1}.]}

Formula \eqref{70aaaa} follows from formula \eqref{7ms-dl-6} and \eqref{pressure-dl-B}
together with \cite[Proposition 4.2.9]{M-W} (i.e., the counterpart of the trace formula \eqref{70aaaa} corresponding to the case $\alpha =0$).
%with similar arguments to those for Proposition 4.2.9 in \cite{M-W} corresponding to the case $\alpha =0$.}

{Continuity of operators} \eqref{fr3}, \eqref{fr4}, \eqref{fr5}, \eqref{fr6} {is} implied by the continuity of the operators just above them and by the first relations in \eqref{sl} and \eqref{dl0}.

Now, we note that formula ${\mathcal V}_{\alpha }={\mathcal V}+{\mathcal V}_{\alpha ;0}$, continuity of the Stokes single layer operator ${\mathcal V}:L_p(\partial \Omega ,{\mathbb R}^n)\to H_p^1(\partial \Omega ,{\mathbb R}^n)$ (cf. \cite[Proposition 4.2.5]{M-W}), and continuity of the complementary operator ${\mathcal V}_{\alpha ;0}:L_p(\partial \Omega ,{\mathbb R}^n)\to H_p^1(\partial \Omega ,{\mathbb R}^n)$ (cf. \cite[Theorem 3.4(b)]{K-L-W}) imply continuity of the first operator in \eqref{ss-s2}.
Continuity of the second operator in \eqref{ss-s2} and of the operators in \eqref{ds-s2} similarly follows from \cite[Propositions 4.2.7 - 4.2.10]{M-W} and \cite[Theorem 3.4(b)]{K-L-W}.
% and formula \eqref{2.37nt}.
{In addition, formula \eqref{7ms-sl-001} and the first relation in \eqref{estimate-Lp} yield the following equality
\begin{align}
({\mathcal V}_\alpha {\bf g})_j({\bf x})=\int _{\partial \Omega }{\mathcal G}_{jk}^\alpha ({\bf x}-{\bf y})g_{0;k}({\bf y})d\sigma _{\bf y}-\sum _{k=1}^n\sum_{r,\ell =1}^n{\rm{p.v.}}\int_{\partial \Omega }\left(\partial _{\tau _{r\ell }}\left({\mathcal G}_{jk}^\alpha ({\bf x}-{\bf y})\right)\right)g_{r\ell ;k}({\bf y})d\sigma _{{\bf y}}\ \ \mbox{a.a}\ {\bf x}\in \partial \Omega ,
\end{align}
for any ${\bf g}\in H_p^{-1}(\partial \Omega ,{\mathbb R}^n)$ (cf., e.g., \cite[(4.69)]{M-W} for $\alpha =0$).
Then the continuity of the first operator in \eqref{ss-s2-1} immediately follows (see also \cite[Proposition 4.2.5 (iii)]{M-W} for $\alpha =0$). Continuity of the Stokes double layer operator ${\bf K}:L_p({\partial \Omega },{\mathbb R}^n)\to L_p({\partial \Omega },{\mathbb R}^n)$ (cf., e.g., \cite[Corllary 4.2.4]{M-W}) and the continuity of the reminder operator ${\bf K}_{\alpha }-{\bf K}:L_p({\partial \Omega },{\mathbb R}^n)\to L_p({\partial \Omega },{\mathbb R}^n)$ (see \cite[Theorem 3.4 (b)]{K-L-W}) show the continuity of the second operator in \eqref{ss-s2-1}.} Continuity of the traces and conormal derivatives of the layer potentials involved in \eqref{68s}-\eqref{70aaaas} and hence continuity of the boundary operators \eqref{ss-s2-sm}, \eqref{ds-s2-sm} immediately follow from the mapping properties of the layer potentials in item (ii) and Lemmas \ref{trace-lemma-Besov}, \ref{lem 1.6}.

Finally, the jump {relations given by the first equalities} in \eqref{68s}-\eqref{70aaaas} follow from formulas \eqref{68}-\eqref{70aaaa}, together with the density of the embeddings $H^1_p(\partial {\Omega},{\mathbb R}^n)\hookrightarrow B_{p,q}^{s}(\partial {\Omega},{\mathbb R}^n)$ and $L_p(\partial {\Omega},{\mathbb R}^n)\hookrightarrow B_{p,q}^{s-1}(\partial {\Omega},{\mathbb R}^n)$, {and equivalence results in Theorems \ref{trace-equivalence-L}(i) and \ref{2.13}(i) for traces and conormal derivatives}.
\hfill\end{proof}

Let us mention the following useful result.
\begin{lemma}
\label{L3.6}
Let ${\Omega}_{+}\subset {\mathbb R}^n$ {$(n\geq 3)$} be a bounded Lipschitz domain with connected boundary $\partial {\Omega}$ and let $\Omega _{-}:={\mathbb R}^n\setminus \overline{\Omega }_+$.
\begin{itemize}
\item[$(i)$] If $p\in (1,\infty )$, $\alpha \in (0,\infty )$,
${\bf g}\in L_p(\partial {\Omega},{\mathbb R}^n)$ and ${\bf h}\in H^1_p(\partial {\Omega},{\mathbb R}^n)$, then
%the following relations hold a.e. on $\partial\Omega $
\begin{align}
\label{68-n1}
&{\gamma}_{\pm}({\bf V}_{\alpha }{\bf g})=({\bf V}_{\alpha }{\bf g})^\pm_{\rm nt}
\in H_{p;\boldsymbol \nu}^{1}(\partial {\Omega},\mathbb{R}^{n}),
%{\quad\forall\ {\bf g}\in H^{-1}_p(\partial {\Omega},{\mathbb R}^n);}
\\
\label{68-s1-n1}
&{\gamma}_{\pm}({\bf W}_{\alpha }{\bf h})=({\bf W}_{\alpha }{\bf h})^\pm_{\rm nt}
\in H_{p;\boldsymbol \nu}^{1}(\partial {\Omega},\mathbb{R}^{n}),
%{\quad\forall\ {\bf h}\in L_p(\partial {\Omega},{\mathbb R}^n);}
\\
\label{70aaa-n1}
&\mathbf t_\alpha^{\pm}\left({\bf V}_{\alpha }{\bf g},{\mathcal Q}^{s}{\bf g}\right)
=\mathbf t_{\rm nt}^{\pm}\left({\bf V}_{\alpha }{\bf g},\mathcal Q^{s}{\bf g}\right)
\in L_p(\partial {\Omega}, \mathbb{R}^{n}),
%{\quad\forall\ {\bf g}\in L_p(\partial {\Omega},{\mathbb R}^n);}
\\
\label{70aaaa-n1}
&\mathbf t_\alpha^{\pm}\big({\bf W}_{\alpha }{\bf h},{\mathcal Q}_{\alpha }^d{\bf h}\big)
=\mathbf t_{\rm nt}^{\pm}\big({\bf W}_{\alpha }{\bf h},\mathcal Q_{\alpha}^d{\bf h}\big)
\in L_p(\partial {\Omega}, \mathbb{R}^{n})
%{\quad\forall\ {\bf h}\in H^1_p(\partial {\Omega},{\mathbb R}^n).}
\end{align}
with the corresponding norm estimates.
\item[$(ii)$] If $p,q\in (1,\infty )$, $s\in(0,1)$, $\alpha \in (0,\infty )$,
$\mathbf g\in B_{p,q}^{s-1}(\partial {\Omega},{\mathbb R}^n)$ and
$\mathbf h\in B_{p,q}^{s}(\partial {\Omega},{\mathbb R}^n)$, then
\begin{align}
\label{68-n1B}
&{\gamma}_{\pm}({\bf V}_{\alpha }{\bf g})=({\bf V}_{\alpha }{\bf g})^\pm_{\rm nt}
\in B^{s}_{p,q;\boldsymbol \nu}(\partial {\Omega},\mathbb{R}^{n}),
%{\quad\forall\ {\bf g}\in H^{-1}_p(\partial {\Omega},{\mathbb R}^n);}
\\
\label{68-s1-n1B}
&{\gamma}_{\pm}({\bf W}_{\alpha }{\bf h})=({\bf W}_{\alpha }{\bf h})^\pm_{\rm nt}
\in B^{s}_{p,q;\boldsymbol \nu}(\partial {\Omega},\mathbb{R}^{n})
\end{align}
with the corresponding norm estimates.
\end{itemize}
\end{lemma}
\begin{proof}
Let first ${\bf g}\in L_p(\partial {\Omega},{\mathbb R}^n)$ and  ${\bf h}\in H^1_p(\partial {\Omega},{\mathbb R}^n)$, $p\in (1,\infty )$.
Then, according to Lemma \ref{nontangential-sl}(ii,v), the right hand sides of the equalities in \eqref{68-n1}-\eqref{70aaaa-n1}
exist almost everywhere on $\partial \Omega $ {in the sense of non-tangential limit},
while Theorem \ref{layer-potential-properties}(i)
yields that
$\left({\bf V}_{\alpha }{\bf g},{\mathcal Q}^{s}{\bf g}\right),
\big({\bf W}_{\alpha }{\bf h},{\mathcal Q}_{\alpha }^d{\bf h}\big)\in
\mathfrak{B}_{p,p^*;{\rm div}}^{1+\frac{1}{p},t}(\Omega_+,\mathcal{L}_{\alpha})$
and
$\left({\bf V}_{\alpha }{\bf g},{\mathcal Q}^{s}{\bf g}\right),
\big({\bf W}_{\alpha }{\bf h},{\mathcal Q}_{\alpha }^d{\bf h}\big)\in
\mathfrak{B}_{p,p^*;\rm div, loc}^{1+\frac{1}{p},t}(\Omega_-,\mathcal{L}_{\alpha})$
for any  $t\ge -\frac{1}{p'}$.
Moreover, Theorem \ref{layer-potential-properties} (iii) and the divergence theorem applied to the single layer potentials ${\bf V}_{\alpha }{\bf g}$ and ${\bf W}_{\alpha }{\bf h}$ in the domain $\Omega _+$ yield that
$({\bf V}_{\alpha }{\bf g})^\pm_{\rm nt}\in H_{p;\boldsymbol \nu}^{1}(\partial {\Omega},\mathbb{R}^{n}), {
\bf t}_{\rm{nt}}^\pm\left({\bf V}_{\alpha }{\bf g},{\mathcal Q}^{s}{\bf g}\right)
\in L_p(\partial {\Omega}, \mathbb{R}^{n})$, for any ${\bf g}\in L_p(\partial \Omega ,{\mathbb R}^n)$,
while $({\bf W}_{\alpha }{\bf h})^\pm_{\rm nt}\in H_{p;\boldsymbol \nu}^{1}(\partial {\Omega},\mathbb{R}^{n}),
{\bf t}_{\rm{nt}}^\pm\left({\bf W}_{\alpha }{\bf g},{\mathcal Q}^{d}{\bf g}\right)\in L_p(\partial {\Omega}, \mathbb{R}^{n})$, for any ${\bf h}\in H_p^1(\partial \Omega ,{\mathbb R}^n)$, with the corresponding norm estimates.
Hence Theorems \ref{trace-equivalence-L}(i) and \ref{trace-equivalence}(ii) along with Remark~\ref{2.13loc} imply relations \eqref{68-n1}-\eqref{70aaaa-n1}.

For $p,q\in (1,\infty )$ and $s\in(0,1)$, we have
$\mathbf g\in B_{p,q}^{s-1}(\partial {\Omega},{\mathbb R}^n)\subset H_{p}^{-1}(\partial {\Omega},{\mathbb R}^n)$,
$\mathbf h\in B_{p,q}^{s}(\partial {\Omega},{\mathbb R}^n)\subset L_{p}(\partial {\Omega},{\mathbb R}^n)$ and,
according to Lemma \ref{nontangential-sl}(iii,iv), the right hand sides of the equalities in  \eqref{68-n1B} and \eqref{68-s1-n1B}
exist almost everywhere on $\partial \Omega $,
while Theorem \ref{layer-potential-properties}(ii)
yields that
${\bf V}_{\alpha }{\bf g}, {\bf W}_{\alpha }{\bf h}\in {B}_{p,q;{\rm div}}^{s+\frac{1}{p}}(\Omega_+)$.
Hence Theorem \ref{trace-equivalence-L}(i) implies relations \eqref{68-n1B} and \eqref{68-s1-n1B}.
\hfill\end{proof}

We will further need the following integral representation (the third Green identity) for the homogeneous Brinkman system solution.
%for velocity and pressure.
\begin{lemma}
\label{Green-r-f-s}
Let ${\Omega}_{+}\subset {\mathbb R}^n$ {$(n\geq 3)$} be a bounded Lipschitz domain with connected boundary $\partial {\Omega}$ and let $\Omega _{-}:={\mathbb R}^n\setminus \overline{\Omega }_+$.
Let $\alpha \in (0,\infty )$, $p,q\in (1,\infty )$ and $s\in (0,1)$. If the
the pair $({\bf u},\pi )$ satisfies the system
\begin{equation}
\label{Brinkman-homogeneous-s}
\triangle {\bf u}-\alpha {\bf u}-\nabla \pi ={\bf 0},\ \ {\rm{div}}\ {\bf u}= 0 \ \mbox{ in } \ {\Omega}_+ \end{equation}
and
$({\bf u},\pi )\in H_{p}^{s+\frac{1}{p}}(\Omega _+,{\mathbb R}^n)\times H_{p}^{s-1-\frac{1}{p}}(\Omega _+)$, or
$({\bf u},\pi )\in B_{p,q}^{s+\frac{1}{p}}(\Omega _+,{\mathbb R}^n)\times B_{p,q}^{s-1-\frac{1}{p}}(\Omega _+)$,
then
\begin{align}
\label{Green-r-f-Brinkman-s}
&{\bf u}({\bf x})={\bf V}_{\alpha }\left(\mathbf t_\alpha^+({\bf u},\pi )\right)({\bf x})
-{\bf W}_{\alpha }\left(\gamma_+{\bf u}\right)({\bf x}),\ \forall\ {\bf x}\in {\Omega}_+.
%\\
%\label{Green-r-f-Brinkman-p}
%&\pi({\bf x})={\mathcal Q}^s\left(\mathbf t_\alpha^+({\bf u},\pi )\right)({\bf x})
%-{\mathcal Q}_{\alpha}^d\left(\gamma_+{\bf u}\right)({\bf x}).
\end{align}
\end{lemma}
\begin{proof}
Let $B(\boldsymbol{y},\epsilon)\subset \Omega$ be a ball of a radius $\epsilon$ around a point $\boldsymbol{y}\in \Omega_+$
and let ${\bf G}_{k}^{\alpha}({\bf x})=\left({\mathcal G}_{k1}^{\alpha }({\bf x}),\ldots ,{\mathcal G}_{kn}^{\alpha }({\bf x})\right)$, $k=1,\ldots ,n$, where $(\mathcal{G}^{\alpha},\Pi )$ is the fundamental solution of the Brinkman system in ${\mathbb R}^n$ (see \eqref{E41} and \eqref{E41-new}).
Applying the second Green identity \eqref{Green-formula2} in the domain $\Omega_+\setminus B(\boldsymbol{y},\epsilon)$ to
$(\mathbf{u},\pi)$ and to the fundamental solution $({\bf G}_{k}^{\alpha}(\cdot-\boldsymbol{y}),\Pi_{k})(\cdot-\boldsymbol{y})$  and taking the limit as $\epsilon\to 0$, we obtain \eqref{Green-r-f-Brinkman-s}.
%Similarly, applying the first Green identity \eqref{Green formula} in the domain $\Omega\setminus B(\boldsymbol{y},\epsilon)$ to $(\mathbf{u},\pi)$ and to the fundamental solution for pressure, $\Pi(\cdot-\boldsymbol{y})$, for $\mathbf{w}$, and taking the limit as $\epsilon\to 0$, we obtain \eqref{Green-r-f-Brinkman-p}.
\end{proof}

Next, we show the counterpart of the integral representation formula  \eqref{Green-r-f-Brinkman-s} written in terms of the non-tangential trace and conormal derivative.
\begin{lemma}
\label{Green-r-f}
Let ${\Omega}_+\subset \mathbb{R}^{n}$ $(n \geq 3)$ be a bounded Lipschitz domain with connected boundary $\partial {\Omega}$. %and
%Let $n\geq 3$ and
Let $\alpha >0$ and $p\in (1,\infty )$ be given constants. {Assume that ${{M}({\bf u})},{M}(\nabla {\bf u}), {M}(\pi)\in L_p(\partial {\Omega})$, there exist the non-tangential limits of ${\bf u}$, $\nabla {\bf u}$ and $\pi $ at almost all points of the boundary $\partial \Omega $, and that the pair $({\bf u},\pi )$ satisfies the homogeneous Brinkman system}
\begin{equation}
\label{Brinkman-homogeneous}
\triangle {\bf u}-\alpha {\bf u}-\nabla \pi ={\bf 0},\ \ {\rm{div}}\ {\bf u}= 0 \ \mbox{ in } \ {\Omega}_+.
\end{equation}
%{and if there exist $\gamma _+{\bf u},t_{\rm{nt}}^+({\bf u},\pi )$ almost everywhere on $\partial \Omega $ and $\gamma _+{\bf u}\in H^1_p(\partial {\Omega}'{\mathbb R}^n)$ and $t_{\rm{nt}}^+({\bf u},\pi )\in L_p(\partial {\Omega})$,}
{Then} %there exists $\varepsilon>0$ such that for all $p\in {\mathcal R}_{0}(n,\varepsilon )\cup(2,\infty)$,}
{${\bf u}$ satisfies also the following integral representation formula}
\begin{align}
\label{Green-r-f-Brinkman}
{\bf u}({\bf x})={\bf V}_{\alpha }\left(\mathbf t_{\rm{nt}}^+({\bf u},\pi )\right)({\bf x})-{\bf W}_{\alpha }\left({\bf u}_{\rm{nt}}^+\right)({\bf x}),\quad \forall\ { \bf x}\in {\Omega}_+.
\end{align}
\end{lemma}
\begin{proof}
{We use arguments similar to the ones in \cite[Proposition 4.4.1]{M-W} for the Stokes system.
% in a bounded Lipschitz domain.
In the case of a smooth bounded domain $\Omega _0\subset {\mathbb R}^n$
{and for ${\bf u}\in C^2(\overline\Omega_+,{\mathbb R}^n)$, $\pi\in C^1(\overline\Omega_+)$},
formula \eqref{Green-r-f-Brinkman} follows easily from the integration by parts, cf. e.g. \eqref{Green-r-f-Brinkman-s}.}
\comment{{In the case of a bounded Lipschitz domain ${\Omega}_+\subset \mathbb{R}^{n}$, we first remark that ${\bf u}\in C^\infty(\Omega_+,{\mathbb R}^n)$, $\pi\in C^\infty(\Omega_+)$ {and} $\mathbf u^+_{\rm nt}\in H_{p;\boldsymbol \nu}^{1}(\partial {\Omega},\mathbb{R}^{n})$ and $\mathbf t_{\rm{nt}}^{+}({\bf u},\pi)$ by Theorem~\ref{M-H}.}}
{Now consider a sequence of sub-domains $\left\{\Omega _j \right\}_{j \geq 1}$ in $\Omega _{+}$ that contain the
point ${\bf x}\in \Omega _{+}$ and converges to $\Omega _{+}$ in the sense of {Lemma \ref{2.13D}}.
Then formula  \eqref{Green-r-f-Brinkman} holds for each of the domains $\Omega _j $ and by the Lebesgue Dominated Convergence Theorem (applied again after the change of variable as in {Lemma \ref{2.13D}} that reduces the integral over $\partial \Omega _j$ to an integral over $\partial \Omega$) letting $j \to \infty $, we obtain \eqref{Green-r-f-Brinkman}  for the Lipschitz domain $\Omega _+$ as well.}
\hfill\end{proof}

\section{Invertibility of related integral operators}

\begin{lemma}
\label{L3.1}
Let ${\Omega}_+\subset \mathbb{R}^{n}$ $(n \geq 3)$ be a bounded Lipschitz domain with connected boundary $\partial {\Omega}$. %and
%Let $n\geq 3$ and
Let $\alpha \in (0,\infty )$ and $0\le s\le 1$.
%be given constants.
Then the following operators are isomorphisms,
\begin{align}
%\label{F-0-1}
%\frac{1}{2}\mathbb{I} + {\bf K}^{*}_{\alpha}&:L_2(\partial {\Omega},{\mathbb R}^n)\to L_2(\partial {\Omega},{\mathbb R}^n),\\
\label{F-0-s}
\frac{1}{2}\mathbb{I} + {\bf K}^{*}_{\alpha }&:H^{-s}_2(\partial {\Omega},{\mathbb R}^n)\to H^{-s}_2(\partial {\Omega},{\mathbb R}^n),\\
%\label{F-0-sns1*}
%\frac{1}{2}\mathbb{I} + {\bf K}^*_{\alpha }&:H^{-1}_2(\partial {\Omega},{\mathbb R}^n)\to H^{-1}_2(\partial {\Omega},{\mathbb R}^n)\\
%\label{F-0-1ns}
%\frac{1}{2}\mathbb{I} + {\bf K}_{\alpha }&:L_2(\partial {\Omega},{\mathbb R}^n)\to L_2(\partial {\Omega},{\mathbb R}^n),\\
\label{F-0-sns}
\frac{1}{2}\mathbb{I} + {\bf K}_{\alpha }&:H^{s}_2(\partial {\Omega},{\mathbb R}^n)\to H^{s}_2(\partial {\Omega},{\mathbb R}^n).
%\\
%\label{F-0-sns1}
%\frac{1}{2}\mathbb{I} + {\bf K}_{\alpha }&:H^1_2(\partial {\Omega},{\mathbb R}^n)\to H^1_2(\partial {\Omega},{\mathbb R}^n).
\end{align}
\end{lemma}
\begin{proof}
Isomorphism property of operator \eqref{F-0-s} for $s=0$ follows from
\cite[Proposition 7.1]{Med-CVEE-16} (see also \cite[Lemma 5.1]{Shen}). By duality this also implies the isomorphism property of operator \eqref{F-0-sns} for $s=0$.

Let us now remark that for $\alpha=0$ and $0<s\le 1$, operator \eqref{F-0-sns} is a Fredholm operator with index zero (cf., e.g., \cite[Proposition 10.5.3 and Theorem 5.3.6]{M-W}), while the operator
%\begin{equation}
%\label{comp1}
$
{\bf K}_{\alpha;0}:={\bf K}_{\alpha }-{\bf K}:H^s_2(\partial {\Omega},{\mathbb R}^n)\to
H^s_2(\partial {\Omega},{\mathbb R}^n)
$
%\end{equation}
is compact (cf., e.g., \cite[Theorem 3.4]{K-L-W}), implying that for $\alpha >0$ and $0<s\le 1$, \eqref{F-0-sns} is a Fredholm operator with index zero as well.
Then by Lemma \ref{Lem2 Fredholm} and the invertibility property of operator \eqref{F-0-sns} for $s=0$ we obtain the equalities
\begin{align}
\label{kernel-s1}
&{\rm{Ker}}\left\{\frac{1}{2}\mathbb{I} + {\bf K}_{\alpha }:H^s_2(\partial {\Omega},{\mathbb R}^n)\to H^s_2(\partial {\Omega},{\mathbb R}^n)\right\}
={\rm{Ker}}\left\{\frac{1}{2}\mathbb{I} + {\bf K}_{\alpha }:H^{0}_2(\partial {\Omega},{\mathbb R}^n)\to H^{0}_2(\partial {\Omega},{\mathbb R}^n)\right\}=\{{\bf 0}\},\quad 0<s\le 1,
\end{align}
which show {invertibility and hence isomorphism property} of operator \eqref{F-0-sns} for $\alpha >0$ and $0<s\le 1$ as well.
A duality argument implies that operator \eqref{F-0-s} is also an isomorphism whenever $\alpha >0$ and $0<s\le 1$.
\hfill\end{proof}

We will often need the following two intervals,
\begin{align}
\label{cases}
&\mathcal R_0(n,\varepsilon)= \left(\frac{2(n-1)}{n+1}-\varepsilon ,2+\varepsilon \right)\cap (1,+\infty ),
\quad %\mbox{and}\quad
\mathcal R_1(n,\varepsilon)=  \begin{cases}(2-\varepsilon,+\infty )& \mbox{ if }\ n=3,\\
\left(2-\varepsilon,\frac{2(n-1)}{n-3}+\varepsilon\right) & \mbox{ if }\ n>3
\end{cases},
\end{align}
which are particular cases of a more general interval
\begin{align}
\label{cases-s}
&{\mathcal R_\theta(n,\varepsilon)
= \begin{cases}(2-\varepsilon,+\infty )& \mbox{ if }\ n=3 \mbox{ and } \theta=1,\\
\left(\frac{2(n-1)}{n+1-2\theta}-\varepsilon,\frac{2(n-1)}{n-1-2\theta}+\varepsilon\right)\cap (1,+\infty ) & \mbox{ if }\ n>3 {\mbox{ or } 0\le\theta< 1}
\end{cases}\ .}
\end{align}
{\begin{lemma}
\label{L3.1p}
Let
${\Omega}_+\subset \mathbb{R}^{n}$ $(n \geq 3)$ be a bounded Lipschitz domain with connected boundary $\partial {\Omega}$. Let $\alpha \in (0,\infty )$.
%be a constant.
Then there exists $\varepsilon=\varepsilon(\partial \Omega )>0$ such that for any
$p\in\mathcal R_0(n, \varepsilon)$ and ${p'}\in\mathcal R_1(n, \varepsilon)$, see \eqref{cases},
the following operators are {isomorphisms,}
\begin{align}
\label{Lp-isom*}
\frac{1}{2}{\mathbb I} + {\bf K}^{*}_{\alpha }&:L_p(\partial {\Omega},{\mathbb R}^n)\to L_p(\partial {\Omega},{\mathbb R}^n),\\
\label{F-0-snsp*}
\frac{1}{2}{\mathbb I} + {\bf K}^*_{\alpha }&:H^{-1}_{p'}(\partial {\Omega},{\mathbb R}^n)\to H^{-1}_{p'} (\partial {\Omega},{\mathbb R}^n),\\
\label{Lp-isom-}
\frac{1}{2}{\mathbb I} + {\bf K}_{\alpha }&:L_{p'}(\partial {\Omega},{\mathbb R}^n)\to L_{p'}(\partial {\Omega},{\mathbb R}^n),\\
\label{F-0-snsp}
\frac{1}{2}{\mathbb I} + {\bf K}_{\alpha }&:H^1_p(\partial {\Omega},{\mathbb R}^n)\to H^1_p(\partial {\Omega},{\mathbb R}^n).
\end{align}
If $\Omega_+$ is of class $C^1$, then the above invertibility properties hold for all $p,{p'}\in (1,\infty )$.
\end{lemma}}
\begin{proof}
By \cite[Theorem 9.1.11]{M-W} there exists a parameter $\varepsilon>0$ such that for any $p\in\mathcal R_0(n, \varepsilon)$,
\begin{align}
\label{decomp-1}
\frac{1}{2}\mathbb{I}+{\bf K}^{*}:L_p(\partial {\Omega}, \mathbb{R}^{n})\to L_p(\partial {\Omega}, \mathbb{R}^{n})
\end{align}
is a Fredholm operator with index zero. Then compactness of the operator
%\begin{align*}
%\label{comp-Lp}
${\bf K}^{*}_{\alpha ;0}:={\bf K}^{*}_{\alpha }-{\bf K}^{*}:L_p(\partial {\Omega},{\mathbb{R}^{n})}\to L_p(\partial {\Omega}, \mathbb{R}^{n})$
%\end{align*}
for any $p\in (1,\infty )$ (see \cite[Theorem 3.4(b)]{K-L-W}), imply that operator \eqref{Lp-isom*} is Fredholm with index zero as well, for any $p\in\mathcal R_0(n, \varepsilon)$.
In addition, a density argument based on {Lemma \ref{Lem2 Fredholm}} and the invertibility property of operator {\eqref{F-0-s} in the case $s=0$}, show that operator \eqref{Lp-isom*} is an isomorphism {for $p=2$ and hence} for any $p\in\mathcal R_0(n,\varepsilon)$.

Similarly, by \cite[Theorem 9.1.3]{M-W} there exists a parameter {(for the sake of brevity, we use the same notation as above)} $\varepsilon>0$ such that for any
%{$p\in \left(\frac{2(n-1)}{n+1}-\varepsilon,2+\varepsilon\right)\cap (1,+\infty )$}
$p\in\mathcal R_0(n, \varepsilon)$
the operator
\begin{align}
\label{decomp-1}
\frac{1}{2}\mathbb{I} + {\bf K}&:
H^1_p(\partial {\Omega},{\mathbb R}^n)\to H^1_p(\partial {\Omega},{\mathbb R}^n)
\end{align}
is Fredholm with index zero. Then compactness of the complementary operator
%\begin{align}
%\label{comp-Lp}
${\bf K}_{\alpha ;0}:={\bf K}_{\alpha }-{\bf K}:H^1_p(\partial {\Omega},{\mathbb R}^n)\to H^1_p(\partial {\Omega},{\mathbb R}^n)$
%\end{align}
for any $p\in (1,\infty )$ (see \cite[Theorem 3.4(b)]{K-L-W}), implies that operator \eqref{F-0-snsp} is Fredholm with index zero as well, for any $p\in \mathcal R_0(n,\varepsilon)$.
In addition, a density argument based on {Lemma \ref{Lem2 Fredholm}} and the invertibility property {for operator \eqref{F-0-sns} in the case $s=1$}, show that operator \eqref{F-0-snsp} is an isomorphism {for $p=2$ and hence} for any %{$p\in \left(\frac{2(n-1)}{n+1}-\varepsilon,2+\varepsilon\right)\cap (1,+\infty )$}.
$p\in\mathcal R_0(n, \varepsilon)$.

{Isomorphism} property of operators \eqref{F-0-snsp*} and \eqref{Lp-isom-} then follow by duality and isomorphism property of operators \eqref{F-0-snsp} and  \eqref{Lp-isom*}, respectively, for ${p'}= \frac{p}{p-1}$.

{If $\Omega_+$ is of class $C^1$, then operator \eqref{decomp-1} is Fredholm with index zero for any $p\in (1,\infty )$, cf., e.g., \cite[Remark 3.1]{Russo-Tartaglione-2}, and the the rest of the proof holds true for any $p, q\in (1,\infty )$.}
\hfill\end{proof}

Lemmas \ref{L3.1p}, \ref{complex-interpolation} and \ref{Interp-Isom} ({ii}) {and an interpolation argument} {(provided by the complex and real interpolation theory)} imply the following assertion.
\begin{corollary}
\label{L3.1ps}
Let
${\Omega}_+\subset \mathbb{R}^{n}$ $(n \geq 3)$ be a bounded Lipschitz domain with connected boundary $\partial {\Omega}$, and
$\alpha \in (0,\infty )$.
% be a constant.
Then there exists $\varepsilon=\varepsilon(\partial \Omega )>0$ such that for any
$p\in \mathcal R_s(n,\varepsilon)$ and ${p'}\in \mathcal R_{1-s}(n,\varepsilon)$, cf. \eqref{cases-s},
the following operators are isomorphisms
\begin{align}
\label{F-0-snsp-s}
&{\frac{1}{2}{\mathbb I} + {\bf K}_{\alpha }:H^s_{p'}(\partial {\Omega},{\mathbb R}^n)\to
H^s_{p'}(\partial {\Omega},{\mathbb R}^n), \quad s\in[0,1]},\\
\label{F-0-sq-s}
&{\frac{1}{2}{\mathbb I} + {\bf K}^*_{\alpha }:H^{-s}_p(\partial {\Omega},{\mathbb R}^n)\to
H^{-s}_p(\partial {\Omega},{\mathbb R}^n),\quad s\in[0,1]},\\
\label{F-0-snsp-sB}
&{\frac{1}{2}{\mathbb I} + {\bf K}_{\alpha }:B^s_{p',q}(\partial {\Omega},{\mathbb R}^n)\to B^s_{p',q}(\partial {\Omega},{\mathbb R}^n), \quad s\in(0,1),\ q\in(1,\infty )},\\
\label{F-0-sq-sB}
&{\frac{1}{2}{\mathbb I} + {\bf K}^*_{\alpha }:B^{-s}_{p,q}(\partial {\Omega},{\mathbb R}^n)\to B^{-s}_{p,q}(\partial {\Omega},{\mathbb R}^n),\quad s\in(0,1),\ q\in(1,\infty )}.
\end{align}
{If $\Omega_+$ is of class $C^1$, then the properties hold for all $p,{p'}\in (1,\infty )$.}
\end{corollary}
\comment{
\begin{lemma}
\label{isom-D-N}
%Let ${\Omega}_+\subset \mathbb{R}^{n}$, $n\geq 3$, be a bounded Lipschitz domain with connected boundary $\partial {\Omega}$. %and
{Let $\alpha >0$.}
%be a given constant.
Then there exists a number $\varepsilon =\varepsilon (\partial \Omega )>0$ such that for any $p\in\mathcal R_0(n, \varepsilon)$ and $q\in\mathcal R_1(n, \varepsilon)$ $($cf. \eqref{cases}$)$,
%$p'=\frac{p}{p-1}$,
\begin{align}
\label{F0-1}
&-\frac{1}{2}{\mathbb I}+{\bf K}_{\alpha }:L_q(\partial \Omega ,{\mathbb R}^n)\to L_q(\partial \Omega ,{\mathbb R}^n),\\
\label{F0-2}
&-\frac{1}{2}{\mathbb I}+{\bf K}_{\alpha }^*:L_p(\partial \Omega ,{\mathbb R}^n)\to L_p(\partial \Omega ,{\mathbb R}^n)
\end{align}
are Fredholm operators with index zero with one-dimensional null-spaces, the null-space of operator \eqref{F0-2} is equal to ${\mathbb R}\boldsymbol \nu$, while the {operator}
\begin{align}
\label{Lp-isom-1*}
&{-\frac{1}{2}{\mathbb I}+{\bf K}_{\alpha }:L_q(\partial \Omega ,{\mathbb R}^n)\to L_{q;\boldsymbol \nu}(\partial \Omega ,{\mathbb R}^n)}
\end{align}
{is surjective and has} a continuous right inverse.

{If $\Omega_+$ is of class $C^1$, then the properties hold for all $p,q\in (1,\infty )$.}
\end{lemma}
\begin{proof}
{By \cite[Proposition 7.2]{Med-CVEE-16}, \eqref{Lp-isom-1*} is a Fredholm operator with index zero  for $q=2$ (cf. also \cite[Lemma 5.2]{Shen} for a differently defined double layer potential operator).}
According to \cite[Corollary 9.1.12]{M-W} there exists $\varepsilon =\varepsilon (\partial \Omega )>0$ such that for any $q\in\mathcal R_1(n, \varepsilon)$, the operator
\begin{align}
\label{Lp-isom-3*}
-\frac{1}{2}\mathbb{I}+{\bf K}:L_q(\partial {\Omega},\mathbb{R}^{n})\to L_q(\partial {\Omega}, \mathbb{R}^{n})
\end{align}
is Fredholm with index zero. Then compactness of the complementary operator
%\begin{align*}
%\label{comp-Lp}
${\bf K}_{\alpha ;0}:={\bf K}_{\alpha }-{\bf K}:L_q(\partial {\Omega},{\mathbb{R}^{n})}\to L_q(\partial {\Omega}, \mathbb{R}^{n})$
%\end{align*}
for any $q\in (1,\infty )$ (see \cite[Theorem 3.4(b)]{K-L-W}), and a duality argument imply that operators \eqref{F0-1} and \eqref{F0-2} are Fredholm with index zero as well for any $q\in\mathcal R_1(n, \varepsilon)$ and
$p=\frac{q}{q-1}$ (i.e, $p\in\mathcal R_0(n, \varepsilon)$).
In addition, in view of \cite[the proof of Proposition 7.2]{Med-CVEE-16} and the property that
$(-\frac{1}{2}{\mathbb I}+{\bf K}^*_{\alpha })\nu
={\bf t}_{\alpha }^{-}({\bf V}_{\alpha }\nu ,{\mathcal Q}^{s}\nu )={\bf 0}$,
{the null space of the operator} $-\frac{1}{2}{\mathbb I}+{\bf K}_{\alpha }^*:L_{2}(\partial \Omega ,{\mathbb R}^n)\to L_{2}(\partial \Omega ,{\mathbb R}^n)$ is
%{either \{0\} or}
equal to ${\mathbb R}\boldsymbol \nu$ {(see also \cite[Lemma 5.2]{Shen})}.
Then
%a density argument based on
{Lemma \ref{Lem2 Fredholm}} implies that the null space of the operator \eqref{F0-2} is
%{either \{0\} or}
equal to ${\mathbb R}\boldsymbol \nu$ for any $p$ as in \eqref{cases}, i.e., the dimension of the null space of operator \eqref{F0-2} is equal to one. Therefore, the range of this Fredholm operator of index zero has has codimension one in $L_p(\partial \Omega ,{\mathbb R}^n)$.
%Moreover, the divergence theorem yields that $\left(-\frac{1}{2}{\mathbb I}+{\bf K}_{\alpha }\right){\boldsymbol{\phi }}\in L_{p;\boldsymbol \nu }({\partial \Omega },{\mathbb R}^n)$ for any ${\boldsymbol{\phi }}\in L_{p;\boldsymbol \nu}({\partial \Omega },{\mathbb R}^n)$.
On the other hand, for any ${\bf h}\in L_q(\partial \Omega ,{\mathbb R}^n)$, the double layer potential ${\bf W}_{\alpha }{\bf h}$ satisfies the equation ${\rm{div}}\, {\bf W}_{\alpha }{\bf h}=0$ in $\Omega _{+}$, and then, by the divergence theorem and the trace formulas \eqref{68-s1}, we deduce that $\left(-\frac{1}{2}{\mathbb I}+{\bf K}_{\alpha }\right){\bf h}\subseteq L_{q;\boldsymbol\nu }(\partial \Omega ,{\mathbb R}^n)$.
%i.e., the codimension of the range of operator \eqref{F0-1} is at least one.
Consequently, the range of operator
%the Fredholm operator of index zero
\eqref{F0-1} is
%the subspace
$L_{q;\boldsymbol\nu }(\partial \Omega ,{\mathbb R}^n)$,
%of codimension one in $L_p(\partial \Omega ,{\mathbb R}^n)$,
and hence the operator \eqref{Lp-isom-1*} is surjective.
%{, i.e., }. Therefore, {because {\eqref{F0-1}}
%$-\frac{1}{2}{\mathbb I}+{\bf K}_{\alpha }:L_p(\partial \Omega ,{\mathbb R}^n)\to L_p(\partial \Omega ,{\mathbb R}^n)$
%is a Fredholm operator of index zero, this implies that its range} is the subspace $L_{p;\nu }(\partial \Omega ,{\mathbb R}^n)$ of codimension one in $L_p(\partial \Omega ,{\mathbb R}^n)$, i.e., operator \eqref{Lp-isom-1*} is surjective.
Since its {null space} is a {one-dimensional} closed subspace of the space $L_q(\partial \Omega ,{\mathbb R}^n)$, operator \eqref{Lp-isom-1*} has a linear continuous right inverse. %Moreover, the operator in \eqref{Lp-isom-1*} is injective.
If the domain $\Omega _{+}$ is of class $C^1$, the property follows easily from the compactness of the operator ${\bf K}_{\alpha }$ on $L_q(\partial \Omega ,{\mathbb R}^n)$, for any $q\in (1,\infty )$ (see, e.g., \cite[p. 1691]{Med-CVEE-16}).
\hfill\end{proof}

Since ${\mathbb R}\boldsymbol \nu $ is {complementary to} $L_{p;\boldsymbol \nu}(\partial \Omega ,{\mathbb R}^n)$ in $L_{p}(\partial \Omega ,{\mathbb R}^n)$, we obtain the following result (cf.
\cite[Corollary 9.1.12]{M-W} for $\alpha =0$).
} %\comment end
Next we show the following invertibility result {(see also \cite[Proposition 7.2]{Med-CVEE-16} in the case $p=2$ and $s=0$)}.
\begin{lemma}
\label{isom-p}
Let ${\Omega}_{+}\subset {\mathbb R}^n$ {$(n\geq 3)$} be a bounded Lipschitz domain with connected boundary $\partial {\Omega}$ and let $\Omega _{-}:={\mathbb R}^n\setminus \overline{\Omega }_+$.
Let $\alpha \in (0,\infty )$.
% be a given constant.
Then there exists a number $\varepsilon =\varepsilon (\partial \Omega )>0$ such that the operators
\begin{align}
\label{Lp-isom-1star}
&-\frac{1}{2}{\mathbb I}+{\bf K}_{\alpha }:L_{{p'};\boldsymbol \nu}(\partial \Omega ,{\mathbb R}^n)\to
L_{{p'};\boldsymbol \nu}(\partial \Omega ,{\mathbb R}^n),\\
\label{Lp-isom-2*}
&{-\frac{1}{2}{\mathbb I}+{\bf K}_{\alpha }^*:L_p(\partial \Omega ,{\mathbb R}^n)/{{\mathbb R}\boldsymbol \nu}\to L_p(\partial \Omega ,{\mathbb R}^n)/{{\mathbb R}\boldsymbol \nu}},\\
\label{Lp-isom-1-C1}
&-\frac{1}{2}{\mathbb I}+{\bf K}_{\alpha }:H^1_{p;\nu }(\partial \Omega ,{\mathbb R}^n)\to H^1_{p;\nu }(\partial \Omega ,{\mathbb R}^n),\\
\label{Lp-isom-1-C1*}
&-\frac{1}{2}{\mathbb I}+{\bf K}_{\alpha }^* :H^{-1}_{p'}(\partial \Omega ,{\mathbb R}^n)/{{\mathbb R}\boldsymbol \nu}\to H^{-1}_{p'}(\partial \Omega ,{\mathbb R}^n)/{{\mathbb R}\boldsymbol \nu}
\end{align}
are isomorphisms for all $p\in\mathcal R_0(n, \varepsilon)$ and ${p'}\in\mathcal R_1(n, \varepsilon)$ $($cf. \eqref{cases}$)$.

If the domain $\Omega $ is of class $C^1$, {the above properties hold for all $p,{p'}\in (1,\infty )$.}
%\begin{align}
%\label{Lp-isom-C1}
%&{-\frac{1}{2}{\mathbb I}+{\bf K}_{\alpha }:L_{p;\nu }(\partial \Omega ,{\mathbb R}^n)\to L_{p;\nu }(\partial \Omega ,{\mathbb R}^n)}\\
%\label{Lp-isom-1-C1}
%&{-\frac{1}{2}{\mathbb I}+{\bf K}_{\alpha }:H^1_{p;\nu }(\partial \Omega ,{\mathbb R}^n)\to H^1_{p;\nu }(\partial \Omega ,{\mathbb R}^n)}
%\end{align}
%{are isomorphisms for any $p\in (1,\infty )$.}
\end{lemma}
\begin{proof}
In the case $\alpha =0$, operator \eqref{Lp-isom-1star} is an isomorphism  (cf. \cite[Corollary 9.1.12]{M-W}), and hence a Fredholm operator with index zero for any ${p'}\in\mathcal R_1(n, \varepsilon)$.
Moreover, the operator ${\bf K}_{\alpha }-{\bf K}$ is compact on the space $L_{p'}(\partial \Omega ,{\mathbb R}^n)$ (see \cite[Theorem 3.4(b)]{K-L-W}), and its range is a subset of $L_{{p'};\boldsymbol \nu}(\partial \Omega ,{\mathbb R}^n)$. Indeed, by using the formula
\begin{align*}
\left({\bf K}_{\alpha }-{\bf K}\right){\bf h}=\left(-\frac{1}{2}{\mathbb I}+{\bf K}_{\alpha }\right){\bf h}-\left(-\frac{1}{2}{\mathbb I}+{\bf K}\right){\bf h}=\gamma _{+}{\bf W}_{\alpha }{\bf h}-\gamma _{+}{\bf W}{\bf h},
\end{align*}
the equations ${\rm{div}}\, {\bf W}_{\alpha }{\bf h}=0$ and ${\rm{div}}\, {\bf W}{\bf h}=0$ in $\Omega _{+}$, and then, the divergence theorem and the trace formulas \eqref{68-s1}, we deduce that
%$$\int_{\partial \Omega }\left(\left({\bf K}_{\alpha }-{\bf K}\right){\bf h}\right)\cdot \boldsymbol\nu d\sigma =0,$$ i.e.,
$\left({\bf K}_{\alpha }-{\bf K}\right){\bf h}\in L_{{p'};\boldsymbol \nu}(\partial \Omega ,{\mathbb R}^n)$ for any ${\bf h}\in L_{{p'};\boldsymbol \nu}(\partial \Omega ,{\mathbb R}^n)$.
Therefore, the operator ${\bf K}_{\alpha }-{\bf K}:L_{{p'};\boldsymbol \nu}(\partial \Omega ,{\mathbb R}^n)\to L_{{p'};\boldsymbol \nu}(\partial \Omega ,{\mathbb R}^n)$ is compact, and then operator \eqref{Lp-isom-1star} is Fredholm with index zero for any {${p'}\in\mathcal R_1(n,\varepsilon)$.
%and hence, by duality, operator \eqref{Lp-isom-2*} is also Fredholm with index zero for any  $p\in\mathcal R_0(n, \varepsilon)$.
On the other hand, by a similar reasoning (cf., e.g., \cite[Theorem 9.1.3]{M-W} and {\cite[Theorem 3.4 (b)]{K-L-W})}, operator \eqref{Lp-isom-1-C1}
%and \eqref{Lp-isom-1-C1*}
is Fredholm with index zero as well, for any $p\in\mathcal R_0(n, \varepsilon)$.

We show now that operators \eqref{Lp-isom-1star} and \eqref{Lp-isom-1-C1} are also injective.
Let us start from operator \eqref{Lp-isom-1-C1} with $p=2$.}
Let ${\bf h}_0\in H^1_{2;\boldsymbol \nu}(\partial \Omega ,{\mathbb R}^n)$ be such that $\left(-\frac{1}{2}{\mathbb I}+{\bf K}_{\alpha }\right){\bf h}={\bf 0}$. Thus, $\gamma _{+}{\bf W}_{\alpha }{\bf h}_0={\bf 0}$, and by applying the Green formula \eqref{Green formula} to the double layer velocity and pressure potentials ${\bf W}_{\alpha }{\bf h}_0$ and ${\mathcal Q}_{\alpha }^d{\bf h}_0$ in $\Omega _{+}$, we deduce that ${\bf W}_{\alpha }{\bf h}_0={\bf 0}$ and ${\mathcal Q}_{\alpha }^d{\bf h}_0=c_0\in {\mathbb R}$ in $\Omega _{+}$. According to formula \eqref{70aaaa}, we obtain that
${\mathbf t_{\rm nt}^{-}}\big({\bf W}_{\alpha }{\bf h}_0,{\mathcal Q}_{\alpha }^d{\bf h}_0\big)
={\mathbf t_{\rm nt}^{+}}\big({\bf W}_{\alpha }{\bf h}_0,{\mathcal Q}_{\alpha }^d{\bf h}_0\big)
=-c_0\boldsymbol\nu $, and then the relation $\gamma _{-}{\bf W}_{\alpha }{\bf h}_0={\bf h}_0\in H^1_{2;\boldsymbol \nu}(\partial \Omega ,{\mathbb R}^n)$ shows that $\langle {\mathbf t_{\rm nt}^{-}}\big({\bf W}_{\alpha }{\bf h}_0,{\mathcal Q}_{\alpha }^d{\bf h}_0\big),\gamma _{-}{\bf W}_{\alpha }{\bf h}_0\rangle _{\partial \Omega }=0$. Finally, the relations ${\bf W}_{\alpha }{\bf h}_0({\bf x})=O(|{\bf x}|^{-n})$ and ${\mathcal Q}^d{\bf h}_0=O(|{\bf x}|^{1-n})$ as $|{\bf x}|\to \infty $ (see, e.g., \cite[Lemma 2.12, (2.76)]{24}), and the Green formula \eqref{Green formula} applied to ${\bf W}_{\alpha }{\bf h}_0$ and ${\mathcal Q}_{\alpha }^d{\bf h}_0$ in $\Omega _{-}$ imply that ${\bf W}_{\alpha }{\bf h}_0={\bf 0}$ and ${\mathcal Q}_{\alpha }^d{\bf h}_0={\bf 0}$ in $\Omega _{-}$. Then the trace formula \eqref{68-s1} yields that ${\bf h}_0={\bf 0}$.
Consequently, {operator \eqref{Lp-isom-1-C1} with $p=2$  is injective.
Then Lemma \ref{Lem2 Fredholm} implies that operator \eqref{Lp-isom-1star} with ${p'}=2$ is injective as well. Applying Lemma \ref{Lem2 Fredholm} again, we now obtain that operator \eqref{Lp-isom-1-C1} with
$p\in\mathcal R_0(n, \varepsilon)$ and operator \eqref{Lp-isom-1star} with ${p'}\in\mathcal R_1(n, \varepsilon)$ are injective, and according to their Fredholm property, these operators are also isomorphisms.
Operators \eqref{Lp-isom-2*} and \eqref{Lp-isom-1-C1*} are then isomorphisms by duality.

If $\Omega $ is of $C^1$ class, then for all $p,{p'}\in (1,\infty )$ operators \eqref{Lp-isom-1star} and  \eqref{Lp-isom-2*} are Fredholm with index zero due to compactness of the operators $\bf K$ and ${\bf K}^*$ on the corresponding spaces (cf., e.g., \cite[Eq. (3.51) in the proof of Proposition 3.5]{D-M}), and \cite[Theorem 3.4 (b)]{K-L-W}. Then the previous paragraph implies that operators \eqref{Lp-isom-1star}-\eqref{Lp-isom-1-C1*} are isomorphisms for $p,{p'}\in (1,\infty )$.}
\hfill\end{proof}

Lemmas \ref{isom-p}, \ref{complex-interpolation} and \ref{Interp-Isom}(ii) {by interpolation} imply the following result {(see also \cite[Proposition 7.2]{Med-CVEE-16} for $p=2$ and $s=0$)}.
\begin{corollary}
\label{L3.1psW}
Let ${\Omega}_{+}\subset {\mathbb R}^n$ {$(n\geq 3)$} be a bounded Lipschitz domain with connected boundary $\partial {\Omega}$ and let $\Omega _{-}:={\mathbb R}^n\setminus \overline{\Omega }_+$.
Let $\alpha \in (0,\infty )$.
Then there exists $\varepsilon=\varepsilon(\partial \Omega )>0$ such that {for any
$p\in \mathcal R_s(n,\varepsilon)$ and ${p'}\in \mathcal R_{1-s}(n,\varepsilon)$ $($cf. \eqref{cases-s}$)$,} the following operators are isomorphisms,
\begin{align}
\label{B-d-l-s}
&{-\frac{1}{2}{\mathbb I}+{\bf K}_{\alpha }:H^s_{{p'};\nu }(\partial \Omega ,{\mathbb R}^n)\to H^s_{{p'}; \nu }(\partial \Omega ,{\mathbb R}^n), \ s\in[0,1]},\\
\label{B-d-l-s*}
&{-\frac{1}{2}{\mathbb I}+{\bf K}_{\alpha }^* :H^{-s}_p(\partial \Omega ,{\mathbb R}^n)/{{\mathbb R}\boldsymbol \nu}\to H^{-s}_p(\partial \Omega ,{\mathbb R}^n)/{{\mathbb R}\boldsymbol \nu}, \ s\in[0,1]},\\
\label{B-d-l-sB}
&{-\frac{1}{2}{\mathbb I}+{\bf K}_{\alpha }:B^s_{p',q;\nu }(\partial \Omega ,{\mathbb R}^n)\to B^s_{p',q;\nu }(\partial \Omega ,{\mathbb R}^n), \quad s\in(0,1),\ q\in(1,\infty)},\\
\label{B-d-l-s*B}
&{-\frac{1}{2}{\mathbb I}+{\bf K}_{\alpha }^*:B^{-s}_{p,q}(\partial\Omega,{\mathbb R}^n)/{{\mathbb R}\boldsymbol\nu}\to B^{-s}_{p,q}(\partial\Omega,{\mathbb R}^n)/{{\mathbb R}\boldsymbol \nu}, \ s\in(0,1),\ q\in(1,\infty)}.
\end{align}
If $\Omega_+$ is of class $C^1$, then the properties hold for all $p,{p'}\in (1,\infty )$.
\end{corollary}

In the case $\alpha =0$, the result, corresponding to the next one, has been obtained in
\cite[Theorem 9.1.4, Corollary 9.1.5]{M-W} (see also \cite[Theorem 6.1]{M-T}).
\begin{lemma}
\label{isom-sl}
Let ${\Omega}_{+}\subset {\mathbb R}^n$ {$(n\geq 3)$} be a bounded Lipschitz domain with connected boundary $\partial {\Omega}$ and let $\Omega _{-}:={\mathbb R}^n\setminus \overline{\Omega }_+$.
Let $\alpha \in (0,\infty )$.
Then there exists a number $\varepsilon >0$ such that for any $p\in\mathcal R_0(n, \varepsilon)$ and ${p'} \in\mathcal R_1(n, \varepsilon)$, see \eqref{cases}, the following Brinkman single layer potential operators are isomorphisms
\begin{align}
\label{s-l-r}
&{\mathcal V}_{\alpha }:L_p(\partial \Omega ,\mathbb{R}^{n})/{\mathbb R}\nu \to H_{p;\nu}^1(\partial {\Omega},\mathbb{R}^{n}),\\
%\end{align}
%\begin{align}
\label{B-s-l-adj}
&{\mathcal V}_{\alpha }:H_{p'}^{-1}(\partial \Omega ,\mathbb{R}^{n})/{\mathbb R}\boldsymbol\nu \to L_{{ p'};\nu}(\partial {\Omega},\mathbb{R}^{n}).
\end{align}
{{If $\Omega_+$ is of class $C^{1}$,} then the above invertibility properties hold for all $p,{p'}\in (1,\infty )$.}
\end{lemma}
\begin{proof}
First, we note that for any ${\bf f}\in L_p(\partial \Omega ,\mathbb{R}^{n})$ the inclusion ${\mathcal V}_{\alpha }{\bf f}\in H_p^1(\partial {\Omega},\mathbb{R}^{n})$ follows by  Theorem~\ref{layer-potential-properties}(iii). Moreover, the inclusion ${\mathcal V}_{\alpha }{\bf f}\in H_{p;\nu}^1(\partial {\Omega},\mathbb{R}^{n})$ follows from the equation ${\rm{div}}\, {\bf V}_{\alpha }{\bf f}=0$ in $\Omega _{+}$, the divergence theorem and relation \eqref{68s}.
On the other hand, there exists a number $\varepsilon >0$ such that the Stokes single layer potential operator
$$
{\mathcal V}:L_p(\partial \Omega ,\mathbb{R}^{n})/{\mathbb R}\boldsymbol \nu \to H_{p;\nu}^1(\partial {\Omega},\mathbb{R}^{n})
$$
is an isomorphism for any $p\in\mathcal R_0(n, \varepsilon)$ (cf. \cite[Theorem 9.1.4]{M-W}), which implies that
${\mathcal V}:L_p(\partial \Omega ,\mathbb{R}^{n})\to H_p^1(\partial {\Omega},\mathbb{R}^{n})$ is a Fredholm operator with index zero for the same range of $p$. Thus, the Brinkman single layer potential operator
\begin{align}
\label{B-s-l}
{\mathcal V}_{\alpha }:L_p(\partial \Omega ,\mathbb{R}^{n})\to H_p^1(\partial {\Omega},\mathbb{R}^{n})
\end{align}
is a Fredholm operator of index zero for any $p\in\mathcal R_0(n, \varepsilon)$, as follows from the equality ${\mathcal V}_{\alpha }={\mathcal V}+{\mathcal V}_{\alpha ;0}$, where %${\mathcal V}:L_p(\partial \Omega ,\mathbb{R}^{n})\to H_{p}^1(\partial {\Omega},\mathbb{R}^{n})$ is a Fredholm operator with index zero and the operator
${\mathcal V}_{\alpha ;0}:={\mathcal V}_{\alpha }-{\mathcal V}:L_p(\partial \Omega ,\mathbb{R}^{n})\to H_{p}^1(\partial {\Omega},\mathbb{R}^{n})$ is a compact operator (cf. \cite[Lemma 3.1]{K-L-W}). Then by
%a density argument and
Lemma \ref{Lem2 Fredholm}, we obtain the equality
\begin{align}
\label{s-l-b-1}
{\rm{Ker}}\left\{{\mathcal V}_{\alpha }:L_p(\partial \Omega ,\mathbb{R}^{n})\to H_{p}^1(\partial {\Omega},\mathbb{R}^{n})\right\}={\rm{Ker}}\left\{{\mathcal V}_{\alpha }:L_2(\partial \Omega ,\mathbb{R}^{n})\to H_{2}^1(\partial {\Omega},\mathbb{R}^{n})\right\},
\end{align}
for each $p\in\mathcal R_0(n, \varepsilon)$.

Moreover, by considering a density $\boldsymbol \varphi _0\in L_2(\partial \Omega ,\mathbb{R}^{n})$ such that ${\mathcal V}_{\alpha }\boldsymbol \varphi _0={\bf 0}$ on $\partial \Omega $, by applying the Green identity \eqref{Green formula} to the single layer velocity and pressure potentials ${\bf u}_0={\bf V}_{\alpha }\boldsymbol \varphi _0$ and $\pi _0={\mathcal Q}^s\boldsymbol \varphi _0$, and by using Theorem \ref{layer-potential-properties}, we deduce that ${\bf u}_0={\bf 0}$ and $\pi _0=c_0\in {\mathbb R}$ in $\Omega _{+}$. In addition, the behavior at infinity of the single layer potentials, ${\bf u}_0({\bf x})=O(|{\bf x}|^{-n})$, $\boldsymbol\sigma ({\bf u}_0,\pi _0)({\bf x})=O(|{\bf x}|^{1-n})$ as $|{\bf x}|\to \infty $ (see, e.g., {\cite[Section 4]{Med-CVEE-16}}), yields that the Green identity \eqref{Green formula} applies also to the fields ${\bf u}_0$ and $\pi _0$ in the exterior domain ${\Omega }_{-}$ and yields ${\bf u}_0={\bf 0}$, $\pi _0=0$ in $\Omega _{-}$.
Then by formulas \eqref{70aaa} $\boldsymbol \varphi _0=c_0\boldsymbol\nu$. On the other hand, the divergence theorem and the second equation in \eqref{2.2.1} imply that
$\left({\bf V}_{\alpha }\boldsymbol\nu \right)_j(x)
=\displaystyle\int_{\Omega_{+}}\dfrac{\partial {\mathcal G}^{\alpha }_{jk}(x-y)}{\partial y_k}dy=0,$
and accordingly that ${\mathcal V}_{\alpha }\boldsymbol \nu={\bf 0}$. Thus, we obtain the equality
\begin{align*}
%\label{s-l-b-2}
{\rm{Ker}}\left\{{\mathcal V}_{\alpha }:
L_2(\partial \Omega ,\mathbb{R}^{n})\to H_{2}^1(\partial{\Omega},\mathbb{R}^{n})\right\} ={\mathbb R}\boldsymbol\nu .
\end{align*}
%and hence, in view of \eqref{s-l-b-1}, the kernel of the operator ${\mathcal V}_{\alpha }:L_q(\partial \Omega ,\mathbb{R}^{n})\to H_{p}^1(\partial {\Omega},\mathbb{R}^{n})$ is ${\mathbb R}\boldsymbol\nu $ (a one dimensional closed subspace of the space $L_q(\partial \Omega ,\mathbb{R}^{n})$).
Therefore, by \eqref{s-l-b-1} the codimension of the range of the operator
${\mathcal V}_{\alpha }:L_p(\partial \Omega ,{\mathbb R}^n)\to H_{p}^1(\partial \Omega ,{\mathbb R}^n)$
is equal to one.
Moreover,
${\rm{Range}}\left({\mathcal V}_{\alpha ;\partial \Omega }\right)\subseteq H_{p;\boldsymbol\nu }^1(\partial \Omega ,{\mathbb R}^n)$,
as follows from the divergence theorem and the second equation in \eqref{2.2.1}.
Since $H_{p;\boldsymbol\nu }^1(\partial \Omega ,{\mathbb R}^n)$ is a subspace of codimension one in $H_{p}^1(\partial \Omega ,{\mathbb R}^n)$, we conclude that the range of the operator
${\mathcal V}_{\alpha }:L_p(\partial \Omega ,{\mathbb R}^n)\to H_{p}^1(\partial \Omega ,{\mathbb R}^n)$
is just $H_{p;\boldsymbol\nu }^1(\partial \Omega _j,{\mathbb R}^n)$.
Then the Fundamental quotient theorem for linear continuous maps implies
${\mathcal V}_{\alpha }:L_p(\partial \Omega ,\mathbb{R}^{n})/{\mathbb R}\boldsymbol\nu \to H_{p;\nu}^1(\partial {\Omega},\mathbb{R}^{n})$
is an isomorphism for any $p\in\mathcal R_0(n, \varepsilon)$, as asserted.

Since the operator
${\mathcal V}_{\alpha }$
%:L_2(\partial \Omega ,\mathbb{R}^{n})\to H_{2}^1(\partial{\Omega},\mathbb{R}^{n})$
is self-adjoint, duality shows that operator \eqref{B-s-l-adj} is also an isomorphism for any $q\in (1,\infty )$ such that $q=\frac{p}{p-1}$. Note that for the same range of $q$, the Stokes single layer potential operator
${\mathcal V}:H^{-1}_q(\partial \Omega ,\mathbb{R}^{n})/{\mathbb R}\boldsymbol \nu \to
L_{q;\nu}^1(\partial {\Omega},\mathbb{R}^{n})$
is an isomorphism as well (see \cite[Corollary 9.1.5]{M-W} for $\alpha =0$).

{{If $\Omega_+$ is of class $C^{1}$}, then the operator ${\mathcal V}:H_{q}^{-1}(\partial \Omega ,\mathbb{R}^{n})\to L_{p}(\partial {\Omega},\mathbb{R}^{n})$ is Fredholm with index zero for any $q\in (1,\infty )$ {(cf., e.g., \cite[Remark 3.1]{Russo-Tartaglione-2}; {see also \cite[Proposition 4.1]{H-M-T1}})}. By duality, we deduce that operator \eqref{B-s-l} is Fredholm with index zero as well for any $p\in (1,\infty )$ whenever $\alpha=0$. In view of \cite[Theorem 3.4]{K-L-W}, the complementary operator ${\mathcal V}_{\alpha }-{\mathcal V}:L_p(\partial \Omega ,\mathbb{R}^{n})\to H_{p}^1(\partial {\Omega},\mathbb{R}^{n})$ is compact (even in the case of a Lipschitz domain). Therefore, the operator ${\mathcal V}_{\alpha }: L_p(\partial \Omega ,\mathbb{R}^{n})\to H_{p}^1(\partial {\Omega},\mathbb{R}^{n})$ is Fredholm with index zero for any $p\in (1,\infty )$. Then the rest of the proof holds true for any $p,q\in (1,\infty )$.}
\hfill\end{proof}

Lemmas \ref{isom-sl}, \ref{complex-interpolation} and \ref{Interp-Isom}(ii) {and an interpolation argument} imply the following assertion {(see also \cite[Remark 3.1]{Russo-Tartaglione-2} in the case of a $C^1$ domain)}.
\begin{corollary}
\label{L3.1psV}
Let ${\Omega}_{+}\subset {\mathbb R}^n$ {$(n\geq 3)$} be a bounded Lipschitz domain with connected boundary $\partial {\Omega}$ and let $\Omega _{-}:={\mathbb R}^n\setminus \overline{\Omega }_+$.
Let $\alpha \in (0,\infty )$
%, $s\in[0,1]$
{and} $p\in \mathcal R_s(n,\epsilon)$, see \eqref{cases-s}.
Then there exists $\varepsilon=\varepsilon(\partial \Omega )>0$ such that
the following {operators are isomorphisms,}
\begin{align}
\label{B-s-l-adj-s}
{\mathcal V}_{\alpha }&:H_p^{-s}(\partial \Omega ,\mathbb{R}^{n})/{\mathbb R}\boldsymbol\nu \to H^{1-s}_{p;\nu}(\partial {\Omega},\mathbb{R}^{n}), \ s\in[0,1],\\
\label{B-s-l-adj-B}
{\mathcal V}_{\alpha }&:B_{p,q}^{-s}(\partial \Omega ,\mathbb{R}^{n})/{\mathbb R}\boldsymbol\nu \to B^{1-s}_{p,q;\nu}(\partial {\Omega},\mathbb{R}^{n}), \ s\in(0,1),\ q\in(1,\infty ).
\end{align}
{{If $\Omega_+$ is of class $C^{1}$,} then the property holds for any $p\in (1,\infty )$.} %and $s\in (0,1)$.} %{\bl *** Is this OK? ***}
\end{corollary}

\section{The Dirichlet and Neumann problems for the Brinkman system}
\subsection{\bf The Dirichlet problem for the Brinkman system}
Let us consider the Dirichlet problem for the homogeneous Brinkman system,
\begin{align}
\label{Brinkman-homogeneous1}
&\triangle {\bf u}-\alpha {\bf u}-\nabla \pi ={\bf 0},\ \ {\rm{div}}\ {\bf u}= 0 \ \mbox{ in } \ {\Omega}_+,\\
\label{DirCond}
&{\bf u}^+_{\rm nt}={\bf h}_0 \ \mbox{ on } \ {\partial\Omega},
\end{align}
and show the following assertion (cf. {\cite[Theorem 5.5]{Shen} for $p=2$ and the boundary data in the space $L_{2;\boldsymbol \nu}(\partial {\Omega}, \mathbb{R}^{n})$}; for $\alpha =0$ see also
\cite[Corollary 9.1.5, Theorems 9.1.4, 9.2.2 and 9.2.5]{M-W} and \cite[Theorem 7.1]{M-T}).
{The Dirichlet boundary condition \eqref{DirCond} is understood in the sense of non-tangential limit at almost all points of $\partial \Omega $.}
\begin{theorem}
\label{M-H-D}
Let $\Omega _{+}\subset \mathbb{R}^{n}$ $($$n\geq 3$$)$ be a bounded Lipschitz domain {with connected boundary $\partial \Omega _+$}. Let $\alpha \in (0,\infty )$, $p\in (1,\infty )$, and  $p^*:=\max \{p,2\}$.
\begin{itemize}
\item[$(i)$]
Let $\mathbf h_0\in H_{p;\boldsymbol \nu}^1(\partial {\Omega},\mathbb{R}^{n})$. Then there exists $\varepsilon=\varepsilon(\partial\Omega)>0$ such that for any $p\in\mathcal R_0(n, \varepsilon)$,
the Dirichlet problem \eqref{Brinkman-homogeneous1}-\eqref{DirCond} has a solution $({\bf u},\pi )$ such that ${M(\textbf{u})},M(\nabla\textbf{u}),M(\pi)\in L_p(\partial {\Omega})$ {{and there exist the non-tangential} limits of ${\bf u}$, $\nabla {\bf u}$ and $\pi $ at almost all points of the boundary $\partial \Omega $}.
Moreover, there exists a constant $C=C(\partial \Omega ,p,\alpha )>0$ such that
\begin{align}
\label{MpiLp}
&\|M({\bf u})\|_{L_p(\partial {\Omega})}+\|M(\nabla{\bf u})\|_{L_p(\partial {\Omega})}
+\|M(\pi)\|_{{L_p(\partial {\Omega})}}\le C\| \mathbf h_0\|_{H^1_p(\partial {\Omega},\mathbb{R}^{n})},\\
\label{MpiLptr}
&\|{\bf u}^+_{\rm nt}\|_{L_p(\partial {\Omega})}+\|\nabla{\bf u}^+_{\rm nt}\|_{L_p(\partial {\Omega})}
+\|\pi^+_{\rm nt}\|_{{L_p(\partial {\Omega})}}\le C\| \mathbf h_0\|_{H^1_p(\partial {\Omega},\mathbb{R}^{n})}.
\end{align}
In addition,
$\textbf{u}\in B_{p,p^*}^{1+\frac{1}{p}}({\Omega}_+,\mathbb{R}^{n})$, $\pi \in B_{p,p^*}^{\frac{1}{p}}({\Omega}_+)$ and
$$
\|{\bf u}\|_{B_{p,p^*}^{1+\frac{1}{p}}({\Omega}_+,\mathbb{R}^{n})}+\|\pi\|_{{B_{p,p^*}^{\frac{1}{p}}({\Omega}_+)}}
\le C\| \mathbf h_0\|_{H^1_p(\partial {\Omega},\mathbb{R}^{n})}.
$$
\item[$(ii)$]
Let $ \mathbf h_0\in L_{p;\boldsymbol \nu}(\partial {\Omega},\mathbb{R}^{n})$. Then there exists $\varepsilon=\varepsilon(\partial\Omega)>0$ such that for any $p\in\mathcal R_1(n, \varepsilon)$
the {Dirichlet} problem \eqref{Brinkman-homogeneous1}-\eqref{DirCond} has a solution $({\bf u},\pi )$ {such that ${M}(\textbf{u})\in L_p(\partial {\Omega})$}.
Moreover, there exists a constant $C>0$ such that
\begin{align}\label{MuLp}
\|M({\bf u})\|_{L_p(\partial {\Omega})}\le C\| \mathbf h_0\|_{L_p(\partial {\Omega},\mathbb{R}^{n})}.
\end{align}
In addition,
{$\textbf{u}\in B_{p,p^*}^{\frac{1}{p}}({\Omega}_+,\mathbb{R}^{n})$}
%$\pi \in B_{p,p^*}^{\frac{1}{p}-1}({\Omega}_+)$ and}
and
%{\bl *** Why is it not OK to mention the space of solution?}
$${\|{\bf u}\|_{B_{p,p^*}^{\frac{1}{p}}({\Omega}_+,\mathbb{R}^{n})}
%+\|\pi\|_{B_{p,p^*}^{\frac{1}{p}-1}({\Omega}_+)/\mathbb R}
\le C\| \mathbf h_0\|_{L_p(\partial {\Omega},\mathbb{R}^{n})}}.$$
\item[$(iii)$]
{Let $0<s<1$} and $\mathbf h_0\in H_{p;\nu}^s(\partial {\Omega},\mathbb{R}^{n})$. Then there exists $\varepsilon=\varepsilon(\partial\Omega)>0$ such that for any
$p\in\mathcal R_{1-s}(n,\epsilon)$ $($cf. \eqref{cases-s}$)$,
the Dirichlet problem \eqref{Brinkman-homogeneous1}-\eqref{DirCond} {$($where the Dirichlet condition \eqref{DirCond} is considered in the Gagliardo trace sense$)$} has a solution
%unique up to an arbitrary additive constant for the pressure $\pi$, and
$\textbf{u}\in B_{p,p^*}^{s+\frac{1}{p}}({\Omega}_+,\mathbb{R}^{n})$, $\pi \in B_{p,p^*}^{s+\frac{1}{p}-1}({\Omega}_+)$, and there exists a constant $C>0$ such that
$$
\|{\bf u}\|_{B_{p,p^*}^{s+\frac{1}{p}}({\Omega}_+,\mathbb{R}^{n})}
+\|\pi\|_{{B_{p,p^*}^{s+\frac{1}{p}-1}({\Omega}_+)}}
\le C\| \mathbf h_0\|_{H^s_p(\partial {\Omega},\mathbb{R}^{n})}.
$$
%{\bl and estimate \eqref{MuLp} holds. If $s=1$ then estimate \eqref{MpiLp} holds as well. *** Is this statement OK? ***}
\end{itemize}
%In all cases, (i)-(iii),
In each of the cases $(i)$, $(ii)$ and $(iii)$, % with $0<s\leq 1$,
the solution %of the problem
is unique up to an arbitrary additive constant for the pressure $\pi$, {and can be expressed in terms of the following double layer velocity and pressure potentials}
\begin{align}
\label{4.3}
{{\bf u}={\bf W}_{\alpha }\left(\left(-\frac{1}{2}{\bf I}+{\bf K}_{\alpha }\right)^{-1}{\bf h}_0\right),\
\pi =\mathcal Q_\alpha^d\left(\left(-\frac{1}{2}{\bf I}+{\bf K}_{\alpha }\right)^{-1}{\bf h}_0\right)
\mbox{ in } \Omega_+ \,.}
\end{align}
\end{theorem}
\begin{proof}
According to Lemmas \ref{nontangential-sl}, {\ref{isom-p} and Theorem \ref{layer-potential-properties}(iii),}
% and Corollary \ref{L3.1psW},
{the functions given by \eqref{4.3}} provide a solution of the Dirichlet problem \eqref{Brinkman-homogeneous1}-\eqref{DirCond}, {which}
% in view of Lemma \ref{nontangential-sl} {and Theorem \ref{layer-potential-properties}},
satisfies the corresponding norm estimates mentioned in items $(i)-(ii)$.
%This, in turn, implies also item (iii) for {$s=0$ and} $s=1$.
%
{For} $0<s<1$ in item (iii), we have by Corollary \ref{L3.1psW} that
$\left(-\frac{1}{2}{\bf I}+{\bf K}_{\alpha }\right)^{-1}{\bf h}_0\in
H_{p}^{s}(\partial {\Omega},\mathbb{R}^{n})\hookrightarrow B_{p,p^*}^{s}(\partial {\Omega},\mathbb{R}^{n})$
with corresponding norm estimates, which by \eqref{7ms-dl-0}, \eqref{ds-s1} and \eqref{68-s1} proves the desired solution properties.

{We will now prove uniqueness of the solution of the Dirichlet problem \eqref{Brinkman-homogeneous1}-\eqref{DirCond} satisfying the conditions in item $(ii)$,}
%which also implies uniqueness in {case (i)},
%can be proved
by modifying arguments in the proofs of \cite[Theorem 5.5.4]{M-W} and \cite[Theorem 7.1]{M-T}.
%
%Indeed,
{Let} $({\bf u}^0,\pi ^0)$ be a solution of the homogeneous version of the Dirichlet problem \eqref{Brinkman-homogeneous1}-\eqref{DirCond} {such that $M({\bf u}^0)\in{L_p(\partial {\Omega})}$ and ${\bf u}_0$ satisfies the homogeneous boundary condition in the sense of non-tangential limit at almost all points of the boundary $\partial \Omega $}.
Let ${\bf x}_0\in \Omega _{+}$ and let $\{\Omega _j\}_{j\geq 1}$ be a sequence of $C^\infty$ sub-domains in $\Omega _{+}$ that contain ${\bf x}_0$ and converge to $\Omega _{+}$ in the sense described in {Lemma \ref{2.13D}}.
Let ${\bf G}_{k}^{\alpha }({\bf x})=\left({\mathcal G}_{k1}^{\alpha }({\bf x}),\ldots ,{\mathcal G}_{kn}^{\alpha }({\bf x})\right)$, $k=1,\ldots ,n$, where $(\mathcal{G}^{\alpha},\Pi )$ is the fundamental solution of the Brinkman system in ${\mathbb R}^n$ (see \eqref{E41} and \eqref{E41-new}). %In view of the existence result from item $(i)$,
{Then}
%(for ${\bf x}_0\in \Omega _+$ and)
for each $\Omega _j$ and any $k=1,\ldots ,n$, the functions ${\bf v}^j$ and $q^j$ given by
\begin{align}
\label{Green-1-j}
{\bf v}^j_{{\bf x}_0}={\bf W}^j_{\alpha }\left({\bf h}'^{(j)}\right),\
q^j_{{\bf x}_0}={\mathcal Q}_{\alpha }^{j;d}\left({\bf h}'^{(j)}\right) \mbox{ in } {\mathbb R}^n\setminus \partial \Omega _j,\ \ %\mbox{ and }
{\bf h}'^{(j)}=\left(-\frac{1}{2}{\bf I}+{\bf K}^j_{\alpha }\right)^{-1}({\bf G}^{\alpha }_k({\bf x}_0-\cdot)|_{\partial\Omega_j}),
%\mbox{ in } \Omega_j.
\end{align}
satisfy the system
\begin{align}
\label{Brinkman-homogeneous1-j}
\left\{
\begin{array}{l}
\triangle {\bf v}^j_{{\bf x}_0}-\alpha {\bf v}^j_{{\bf x}_0}-\nabla q^j_{{\bf x}_0}={\bf 0},\
{\rm{div}}\, {\bf v}^j_{{\bf x}_0}=0 \mbox{ in } {\Omega _j},\\
({\bf v}^j_{{\bf x}_0})^+_{\rm nt}={\bf G}^{\alpha }_k({\bf x}_0,\cdot )|_{\partial \Omega _j}.
\end{array}
\right.
\end{align}
%{\rd and $M({\bf v}^j_{{\bf x}_0}),M(\nabla {\bf v}^j_{{\bf x}_0}),M(q^j_{{\bf x}_0})\in L_p(\partial \Omega _j)$. Note that the boundary condition in \eqref{Brinkman-homogeneous1-j} is satisfied at almost all points of $\partial \Omega _j$ in the sense of non-tangential limit}.
{Here} {${\bf W}^j_{\alpha }:={\bf W}_{\alpha ;\partial \Omega _j}$ and ${\mathcal Q}_{\alpha }^{j;d}:={\mathcal Q}_{\alpha ;\partial \Omega _j}^{d}$ are the double layer velocity and pressure potential operators corresponding to $\partial \Omega _j$, while ${\bf K}^j_{\alpha }:H^1_{p'}(\partial \Omega _j,{\mathbb R}^n)\to H^1_{p'}(\partial \Omega _j,{\mathbb R}^n)$ is the corresponding double layer integral operator.}
Indeed, ${\bf G}^{\alpha }_k({\bf x}_0-\cdot)|_{\partial\Omega_j}\in H_{p;\boldsymbol \nu^{(j)}}^1(\partial \Omega _j,{\mathbb R}^n)$ and, in view of Lemma \ref{isom-p}, the operator
%\begin{align}
$-\frac{1}{2}{\bf I}+{\bf K}^j_{\alpha }:H^1_{p';\boldsymbol\nu^{(j)}}(\partial \Omega _j,{\mathbb R}^n)\to H^1_{p';\boldsymbol\nu^{(j)}}(\partial \Omega _j,{\mathbb R}^n)$
%\end{align}
is an isomorphism for any $p'\in (1,\infty )$ {since $\Omega_j$ is a smooth domain}.

Note that the operator $-\frac{1}{2}{\bf I}+{\bf K}_{\alpha }:H_{p';\boldsymbol \nu}^1(\partial \Omega ,{\mathbb R}^n)\to H_{p';\boldsymbol \nu}^1(\partial \Omega ,{\mathbb R}^n)$ is an isomorphism for any $p'\in {\mathcal R}_{0}(n,\varepsilon )$ {(see Lemma \ref{isom-p}), i.e., for any $p'$ such that $\frac{1}{p'}=1-\frac{1}{p}$, where $p\in {\mathcal R}_1(n,\varepsilon )$}.
%Moreover, ${\bf T}_{\alpha ;D}:={\bf W}_{\alpha }\circ \left(-\frac{1}{2}{\bf I}+{\bf K}_{\alpha }\right)^{-1}\circ \gamma _{+}$ is the solution operator corresponding to the Dirichlet problem for the Brinkman system in $\Omega _+$, while ${\bf T}_{\alpha ;D;j}:={\bf W}^j_{\alpha }\circ \left(-\frac{1}{2}{\bf I}+{\bf K}^j_{\alpha }\right)^{-1}\circ \gamma _{+}$ is the solution operator for the Dirichlet problem in $\Omega _j$.
After performing a change of variable as in {Lemma \ref{2.13D}}, the operator $-\frac{1}{2}{\bf I}+{\bf K}^j_{\alpha }$ defined on $\partial \Omega _j$ can be identified with an operator ${\mathcal T}_\alpha^j$ acting on functions defined on $\partial \Omega $.
Then, employing the arguments, e.g., similar to those in the last paragraph in p.116 in \cite{M-W}, which are based on \cite[Lemmas 11.9.13 and 11.12.2]{M-W}, and taking into account \cite[Proposition 1]{Med-AAM} (see also  \cite[Theorems 3.8 (iv) and 4.15]{Fa-Ke-Ve}), one can show that the sequence of operators ${\mathcal T}_\alpha^j$ converges to the operator ${\mathcal T}_\alpha :=-\frac{1}{2}{\bf I}+{\bf K}_{\alpha }$ in the operator norm {and} the sequence of the inverses of the operators ${\mathcal T}_\alpha^j$ converges to the inverse of the operator ${\mathcal T}_\alpha$ in the operator norm.
%of the space ${\mathcal L}\left(H^1_p(\partial \Omega ,{\mathbb R}^n),H^1_p(\partial \Omega ,{\mathbb R}^n)\right)$,
%Since {\bl we have approximated the domain $\Omega _+$ with a sequence of smooth domains $\Omega _j$ with uniform Lipschitz characters from inside,
{Hence the operator norms
$\|\left(-\frac{1}{2}{\mathbb I}+{\bf K}^j_{\alpha}\right)^{-1}\|_{H^1_{p'}(\partial \Omega _j,{\mathbb R}^n)}$
are bounded uniformly in $j$, implying that there exist some constants $C_0, C'_0$ depending only on $p$, $n$, $\alpha$} and the Lipschitz character of $\Omega _+$ (thus, $C_0$ does not depend on $j$) such that
\begin{align}
\label{j-ell-1v'}
\|{{\bf h}'^{(j)}}\|_{H^1_{p'}(\partial \Omega _j,{\mathbb R}^n)}
\leq C_0\|{\bf G}^{\alpha }_k({\bf x}_0,\cdot )\|_{H_{p'}^1(\partial \Omega_j ,{\mathbb R}^n)}
\le C'_0(\|M({\bf G}^{\alpha}_k({\bf x}_0,\cdot))\|_{L_{p'}(\partial \Omega)}
+\|M(\nabla {\bf G}^{\alpha}_k({\bf x}_0,\cdot))\|_{L_{p'}(\partial \Omega)}),
\end{align}
{where the non-tangential maximal operator $M$ is considered with respect to a regular family of cones truncated at a height smaller than the distance from ${\bf x}_0$ to $\partial \Omega $}
{(cf. \cite[Theorem 1.12]{Verchota1984}, see also Lemma \ref{2.13D})}.
Further, by considering the change of variable ${\bf y}_j:=\Phi_j({\bf y})$ as in {Lemma \ref{2.13D}}, the double-layer potential representations \eqref{Green-1-j} become
\begin{align}
\label{dl-sol-2-1}
{v}^j_{{\bf x}_0;\ell}(\mathbf x)
&=\int_{\partial {\Omega _j}}S^{\alpha}_{i\ell s}({\bf y}_j,{\bf x}) \nu _{s}
({\bf y}_j)h^{'(j)}_{i}({\bf y}_j)d\sigma _{{\bf y}_j}
=\int_{\partial {\Omega}} S^{\alpha}_{i\ell s} (\Phi_j({\bf y}),{\bf x})\nu _{s}
(\Phi_j({\bf y}))H'^{(j)}_{i}({\bf y})d\sigma _{{\bf y}},\\
\label{dl-sol-2-2}
q^j_{{\bf x}_0}(\mathbf x)&
=\int_{\partial {\Omega _j}} \Lambda^{\alpha }_{is}({\bf y}_j,{\bf x})\nu _{s}({\bf y}_j)h^{'(j)}_{i}({\bf y}_j)d\sigma _{{\bf y}_j}
=\int_{\partial {\Omega}}\Lambda^{\alpha }_{is} (\Phi_j({\bf y}),{\bf x})\nu _{s}(\Phi_j({\bf y}))H'^{(j)}_{i}({\bf y})d\sigma _{{\bf y}},
 \ \forall \ {\bf x}\in \Omega _j,
\end{align}
where
%\begin{align}
$
{\bf H}'^{(j)}({\bf y}):={{\bf h}'^{(j)}}(\Phi_j({\bf y}))\omega _j({\bf y}),
$
${\bf y}\in \partial \Omega ,
$
%\end{align}
${\bf y}^{(j)}=(y_1^{(j)},\ldots ,y_n^{(j)})$, ${\bf h}'^{(j)}=(h^{'(j)}_{1},\ldots ,h^{'(j)}_{n})$, ${\bf H}'^{(j)}=({H'}_1^{(j)},\ldots ,{H'}_n^{(j)})$, and $\omega _j$ is the Jacobian of $\Phi_j:\partial \Omega \to \partial \Omega _j$.

In view of \eqref{j-ell-1v'} and of the uniform boundedness of $\{\omega _j\}_{j\geq 1}$, there exists {a constant $C_1>0$} (which depends only on $p$, $n$ and the Lipschitz character of $\Omega_+$) such that
\begin{align}
\label{e2-1}
%\int _{\partial \Omega }|{\bf H}'^{(j)}({\bf y})|^{p'}d\sigma _{{\bf y}}
%\leq c_1\int_{\partial \Omega _j}|{\bf G}^{\alpha }_k({\bf x}_0,{\bf y}_j)|^{p'}d\sigma _{{\bf y}_j}
%=c_1\int_{\partial \Omega }|{\bf G}^{\alpha }_k({\bf x}_0,\Phi_j({\bf y}))|^{p'}\omega _j({\bf y})d\sigma _{{\bf y}}
%\leq c_2\int_{\partial \Omega }|{\bf G}^{\alpha }_k({\bf x}_0,{\bf y})|^{p'}d\sigma _{{\bf y}},\ \forall \ j\geq 1.\\
\|{\bf H}'^{(j)}\|_{H^1_{p'}(\partial \Omega,{\mathbb R}^n)}
\le C_1\|{\bf h}'^{(j)}\|_{H^1_{p'}(\partial \Omega _j,{\mathbb R}^n)}
\leq C'_0C_1
%\|{\bf G}^{\alpha }_k({\bf x}_0,\cdot )\|_{H_{p'}^1(\partial \Omega_j ,{\mathbb R}^n)}\\
%\leq C_2
%%\|{\bf G}^{\alpha }_k({\bf x}_0,\cdot )\|_{H_{p'}^1(\partial \Omega_,{\mathbb R}^n)},\
(\|M({\bf G}^{\alpha}_k({\bf x}_0,\cdot))\|_{L_{p'}(\partial \Omega)}
+\|M(\nabla {\bf G}^{\alpha}_k({\bf x}_0,\cdot))\|_{L_{p'}(\partial \Omega)})
,\
\forall \ j\geq 1.
\end{align}
Hence $\{{\bf H}'^{(j)}\}_{j\geq 1}$ is a bounded sequence in $H_{p'}^1(\partial \Omega ,{\mathbb R}^n)$, and, thus, there {exists} a subsequence, {still} denoted as the sequence, and a function
${\bf H}'\in H_{p'}^1(\partial \Omega ,{\mathbb R}^n)$, such that ${\bf H}'^{(j)}\to {\bf H}'$ weakly in
$H_{p'}^1(\partial \Omega ,{\mathbb R}^n)$.
By this property and letting $j\to \infty $ in \eqref{dl-sol-2-1}-\eqref{dl-sol-2-2}, we obtain
%\begin{align}
%\label{u-dl}
$
{\bf v}^j_{{\bf x}_0}({\bf x})\to{\bf v}_{{\bf x}_0}({\bf x})={\bf W}_{\alpha }{\bf H}'({\bf x}),\
q^j_{{\bf x}_0}({\bf x})\to q_{{\bf x}_0}({\bf x})={\mathcal Q}_{\alpha }^{d}{\bf H}'({\bf x})
$
%\end{align}
pointwise for any ${\bf x}\in\Omega _+$.
Moreover, {in view of Lemma \ref{nontangential-sl} {(where the constants depend only on the Lipschitz character of $\Omega _+$), applied to $\partial\Omega_j$, and \eqref{j-ell-1v'},} we obtain the inequality}
%{\mg [***SM: One still have to prove that this implies \eqref{j-ell-1v}. Does it follow from Calderon-Zygmund theory and how (cf. e.g. \cite[p.57]{M-W})?]}
\comment{
According to Lemma \ref{nontangential-sl} (i),({iv}) there exists the non-tangential limit ${\bf u}_{{\rm nt}}^+=({\bf W}_{\alpha }{\bf H}')_{{\rm nt}}^+$ of ${\bf u}$ at almost all points of $\partial \Omega $,
and  by estimates \eqref{7ms-dl-0} and \eqref{e1}, we obtain that
%${\bf u}_{{\rm nt}}^+\in L_p(\partial \Omega ,{\mathbb R}^n)$.
%Indeed, the inequality
\begin{align}
\label{4.15-1}
\|{\bf u}_{{\rm nt}}^+\|_{L_p(\partial \Omega ,{\mathbb R}^n)}=\|({\bf W}_{\alpha }{\bf H}')_{{\rm nt}}^+\|_{L_p(\partial \Omega ,{\mathbb R}^n)}
\leq {C_0\|{\bf H}'\|_{L_p(\partial \Omega ,{\mathbb R}^n)}}
{\leq C_0\lim \inf_{j\to \infty }\|{\bf H}'^{(j)}\|_{L_p(\partial \Omega ,{\mathbb R}^n)}}
%\leq C_0\|\sup_{j\ge 1}|{\bf H}'^{(j)}|\,\|_{L_p(\partial \Omega )}
\leq C_1\|M({\bf u})\|_{L_p(\partial \Omega )}.
\end{align}
%and estimate \eqref{7ms-dl-0} imply the desired property.
Moreover, the divergence theorem shows that ${\mathbf u^+_{\rm nt}}=({\bf W}_{\alpha }{\bf H}')_{{\rm nt}}^+\in L_{p;\nu}(\partial {\Omega},\mathbb{R}^{n})$.
Estimate \eqref{4.9} is provided by the representations ${\bf u}={\bf W}_{\alpha }{\bf H}'$, $\pi=\mathcal Q^d_{\alpha }{\bf H}'$, and by the continuity of the operators in \eqref{fr3} and by estimates \eqref{4.15-1}.
} %\comment end
\begin{align}
\label{j-ell-1v}
\|M(\nabla {\bf v}^j_{\bf x_0})\|_{L_{p'}(\partial \Omega _{j})}
+\|M(q^{j}_{{\bf x}_0})\|_{L_{p'}(\partial \Omega _{j})}
\leq C_3\|{{\bf h}'^{(j)}}\|
\leq C'_0 C_3\left(\|M({\bf G}^{\alpha}_k({\bf x}_0,\cdot))\|_{L_{p'}(\partial \Omega)}
+\|M(\nabla {\bf G}^{\alpha}_k({\bf x}_0,\cdot))\|_{L_{p'}(\partial \Omega)}\right),
\end{align}
{with a constant $C_3$ depending only on $p$, $n$ and the Lipschitz character of $\Omega _+$.}
\comment{
%Old version
\begin{align}
\label{Green-1-j}
{\bf v}^j({\bf x}_0,\cdot)={\bf W}^j_{\alpha }\left(\left(-\frac{1}{2}{\bf I}+{\bf K}^j_{\alpha }\right)^{-1}({\bf G}^{\alpha }_k({\bf x}_0-\cdot)|_{\partial\Omega_j})\right),\
q^j({\bf x}_0,\cdot)={\mathcal Q}_{\alpha }^{j;d}\left(\left(-\frac{1}{2}{\bf I}+{\bf K}^j_{\alpha }\right)^{-1}({\bf G}^{\alpha }_k({\bf x}_0-\cdot)|_{\partial\Omega_j})\right) \mbox{ in } {\mathbb R}^n\setminus \partial \Omega _j,\\
%\mbox{ in } \Omega_j.
\end{align}
satisfy the system
\begin{align}
\label{Brinkman-homogeneous1-j}
\left\{
\begin{array}{l}
\triangle {\bf v}^j_{{\bf x}_0}-\alpha {\bf v}^j_{{\bf x}_0}-\nabla q^j_{{\bf x}_0}={\bf 0},\
{\rm{div}}\, {\bf v}^j_{{\bf x}_0}=0 \mbox{ in } {\Omega _j},\\
({\bf v}^j_{{\bf x}_0})^+_{\rm nt}={\bf G}^{\alpha }_k({\bf x}_0,\cdot )|_{\partial \Omega _j},
\end{array}
\right.
\end{align}
where {${\bf W}^j_{\alpha }:={\bf W}_{\alpha ;\partial \Omega _j}$ and ${\mathcal Q}_{\alpha }^{j;d}:={\mathcal Q}_{\alpha ;\partial \Omega _j}^{d}$ are the double layer velocity and pressure potential operators corresponding to $\partial \Omega _j$, while ${\bf K}^j_{\alpha }:H^1_{p'}(\partial \Omega _j,{\mathbb R}^n)\to H^1_{p'}(\partial \Omega _j,{\mathbb R}^n)$ is the corresponding double layer integral operator.}
Indeed, ${\bf G}^{\alpha }_k({\bf x}_0-\cdot)|_{\partial\Omega_j}\in H_{p;\boldsymbol \nu^{(j)}}^1(\partial \Omega _j,{\mathbb R}^n)$ and, in view of Lemma \ref{isom-p}, the operator
%\begin{align}
$-\frac{1}{2}{\bf I}+{\bf K}^j_{\alpha }:H^1_{p';\boldsymbol\nu^{(j)}}(\partial \Omega _j,{\mathbb R}^n)\to H^1_{p';\boldsymbol\nu^{(j)}}(\partial \Omega _j,{\mathbb R}^n)$
%\end{align}
is an isomorphism for any $p'\in (1,\infty )$.

{Note that the operator $-\frac{1}{2}{\bf I}+{\bf K}_{\alpha }:H_{p';\boldsymbol \nu}^1(\partial \Omega ,{\mathbb R}^n)\to H_{p';\boldsymbol \nu}^1(\partial \Omega ,{\mathbb R}^n)$ is an isomorphism for any $p'\in {\mathcal R}_{0}(n,\varepsilon )$, where $\frac{1}{p'}=1-\frac{1}{p}$, and hence, for any $p\in {\mathcal R}_1(n,\varepsilon )$ (see Lemma \ref{isom-p}).
%Moreover, ${\bf T}_{\alpha ;D}:={\bf W}_{\alpha }\circ \left(-\frac{1}{2}{\bf I}+{\bf K}_{\alpha }\right)^{-1}\circ \gamma _{+}$ is the solution operator corresponding to the Dirichlet problem for the Brinkman system in $\Omega _+$, while ${\bf T}_{\alpha ;D;j}:={\bf W}^j_{\alpha }\circ \left(-\frac{1}{2}{\bf I}+{\bf K}^j_{\alpha }\right)^{-1}\circ \gamma _{+}$ is the solution operator for the Dirichlet problem in $\Omega _j$.
After performing a change of variable as in {Lemma \ref{2.13D}}, the operator $-\frac{1}{2}{\bf I}+{\bf K}^j_{\alpha }$ defined on $\partial \Omega _j$ can be identified with an operator ${\mathcal T}_\alpha^j$ acting on functions defined on $\partial \Omega $.
Then, employing the arguments, e.g., similar to those in the last paragraph in p.116 in \cite{M-W}, which are based on \cite[Lemmas 11.9.13 and 11.12.2]{M-W}, and taking into account \cite[Proposition 1]{Med-AAM} (see also  \cite[Theorems 3.8 (iv) and 4.15]{Fa-Ke-Ve}), one can show that the sequence of operators ${\mathcal T}_\alpha^j$ converges to the operator ${\mathcal T}_\alpha :=-\frac{1}{2}{\bf I}+{\bf K}_{\alpha }$ in the operator norm, the sequence of the inverses of the operators ${\mathcal T}_\alpha^j$ converges to the inverse of the operator ${\mathcal T}_\alpha$ in the operator norm,
%of the space ${\mathcal L}\left(H^1_p(\partial \Omega ,{\mathbb R}^n),H^1_p(\partial \Omega ,{\mathbb R}^n)\right)$,
and hence there exists a constant $C'$ depending only on $p$, $n$ and the Lipschitz character of $\Omega _+$ (thus, not depending on $j$) such that
}
\begin{align}
\label{j-ell-1v}
\|M(\nabla {\bf v}^j_{\bf x_0})\|_{L_{p'}(\partial \Omega _{j})}
+\|M(q^{j}_{{\bf x}_0})\|_{L_{p'}(\partial \Omega _{j})}
\leq C'\|{\bf G}^{\alpha }_k({\bf x}_0,\cdot )\|_{H_{p'}^1(\partial \Omega ,{\mathbb R}^n)}.
\end{align}
} %\comment end

In addition, the pair $\left({\bf G}^{\alpha ;j}_k({\bf x}_0,\cdot),\pi ^{j}_{k}({\bf x}_0,\cdot)\right)$ given by
\begin{align}
\label{Green-j}
{\bf G}^{\alpha ;j}_k({\bf x}_0,\cdot):={\bf G}^\alpha _k({\bf x}_0-\cdot)-{\bf v}^j_{{\bf x}_0},\ \
%={\bf G}^\alpha _k({\bf x}_0-\cdot)-{\bf W}^j_{\alpha }\left(\left(-\frac{1}{2}{\bf I}+{\bf K}^j_{\alpha }\right)^{-1}({\bf G}^{\alpha }_k({\bf x}_0-\cdot)|_{\partial\Omega_j})\right),}\\
\pi ^{j}_{k}({\bf x}_0,\cdot):=\Pi_k({\bf x}_0-\cdot)-q^{j}_{{\bf x}_0}
%=\Pi_k({\bf x}_0-\cdot)-{\mathcal Q}_{\alpha }^{j;d}\left(\left(-\frac{1}{2}{\bf I}+{\bf K}^j_{\alpha }\right)^{-1}({\bf G}^{\alpha }_k({\bf x}_0-\cdot)|_{\partial\Omega_j})\right)},
\end{align}
defines the Green function of the Brinkman system in $\Omega _j$ and its corresponding pressure vector, i.e., it satisfies for each ${\bf x}_0\in \Omega _j$ the following relations
\begin{align}
\label{Brinkman-Green-j}
\left\{
\begin{array}{l}
-\nabla \pi_k^j({\bf x_0},{\bf y})+\triangle {\bf G}^{\alpha ;j}_k({\bf x}_0,{\bf y})-\alpha {\bf G}^{\alpha ;j}_k({\bf x}_0,{\bf y})=-\delta_{{\bf y}}({\bf x}_0){\bf I},\\
{\rm{div}}_{\bf y}{\bf G}^{\alpha ;j}_k({\bf x}_0,{\bf y})=0\ \mbox{ in }\ {\Omega _j},\\
{\bf G}^{\alpha ;j}_k({\bf x}_0,{\bf y})={\bf 0},\ {\bf y}\in \partial \Omega _j.
\end{array}
\right.
\end{align}
\comment{
Therefore, $\left({\bf G}^{\alpha ;j}_k\right)_{k=1,\ldots ,n}$ is the Green function of the Brinkman system in $\Omega _j$. Moreover, ${\bf G}^{\alpha ;\Omega _+}_k$ and $\pi _{k}$ given by
\begin{align}
\label{Brinkman-Green-Omega}
&{\bf G}^{\alpha ;\Omega _+}_k({\bf x}_0,\cdot):={\bf G}^\alpha _k({\bf x}_0-\cdot)-{\bf W}_{\alpha }\left(\left(-\frac{1}{2}{\bf I}+{\bf K}_{\alpha }\right)^{-1}({\bf G}^{\alpha }_k({\bf x}_0-\cdot)|_{\partial\Omega })\right),\\
&{\pi _{k}({\bf x}_0,\cdot):=\Pi_k({\bf x}_0-\cdot)-{\mathcal Q}_{\alpha }^{d}\left(\left(-\frac{1}{2}{\bf I}+{\bf K}_{\alpha }\right)^{-1}({\bf G}^{\alpha }_k({\bf x}_0-\cdot)|_{\partial\Omega })\right)}
\end{align}
are the Green function of the Brinkman system in the Lipschitz domain $\Omega _+$ and the corresponding pressure term.
Note that
\begin{align}
\label{Brinkman-Green-Omega-j}
{\bf G}^{\alpha ;j}_k({\bf x}_0,\cdot)={\bf T}_{\alpha ;j}\left({\bf G}^{\alpha }_k({\bf x}_0-\cdot)\right),\
{\bf G}^{\alpha }_k({\bf x}_0,\cdot)={\bf T}_{\alpha }\left({\bf G}^{\alpha }_k({\bf x}_0-\cdot)\right),
\end{align}
where
\begin{align}
{\bf T}_{\alpha ;j}:={\mathbb I}-{\bf W}^j_{\alpha }\circ \left(-\frac{1}{2}{\bf I}+{\bf K}^j_{\alpha }\right)^{-1}\circ \gamma _{+},\ {\bf T}_{\alpha }:={\mathbb I}-{\bf W}_{\alpha }\circ \left(-\frac{1}{2}{\bf I}+{\bf K}_{\alpha }\right)^{-1}\circ \gamma _{+}.
\end{align}
} %\comment end
{Hence,} for each $\Omega_j$ and any $k=1,\ldots ,n$, we obtain the relations
\begin{align}
%\label{j-ell-0}
%&{\rm{div}}\, {\bf G}^{\alpha ;j}_k({\bf x}_0,\cdot)={\bf 0} \mbox{ in } \Omega _j,\
%{\bf G}^{\alpha ;j}_k({\bf x}_0,\cdot)=0 \mbox{ on } \partial\Omega_j,\\
\label{j-ell}
&\left\langle \triangle {\bf G}^{\alpha ;j}_k({\bf x}_0,\cdot)-\alpha {\bf G}^{\alpha ;j}_k({\bf x}_0,\cdot)-\nabla \pi _k^j({\bf x}_0,\cdot),{\bf u}^0\right\rangle _{\Omega _j}=u_k^0({\bf x}_0).
\end{align}
Then by \eqref{Brinkman-Green-j} and \eqref{j-ell} we obtain that
\begin{align}
\label{j-ell-2}
u_k^0({\bf x}_0)
=\int_{\partial \Omega _j}{\bf t}^{\mathrm c+}({\bf G}^{\alpha ;j}_k({\bf x}_0,\cdot),\pi _k^j({\bf x}_0,\cdot))\cdot {\bf u}^0d\sigma _j.
%=\int_{\partial \Omega _{j}}[{\bf t}^{\mathrm c+}({\bf G}_{k}^{\alpha},\boldsymbol\Pi)-{\bf t}^{\mathrm c+}(\mathbf v^j,q^j)]\cdot {\bf u}^0d\sigma.
\end{align}
\comment{
Note that the operator $-\frac{1}{2}{\bf I}+{\bf K}_{\alpha }:H_{p';\boldsymbol \nu}^1(\partial \Omega ,{\mathbb R}^n)\to H_{p';\boldsymbol \nu}^1(\partial \Omega ,{\mathbb R}^n)$ is an isomorphism for any $p'\in {\mathcal R}_{0}(n,\varepsilon )$, where $\frac{1}{p'}=1-\frac{1}{p}$, and hence, for any $p\in {\mathcal R}_1(n,\varepsilon )$ (see Lemma \ref{isom-p}).
%Moreover, ${\bf T}_{\alpha ;D}:={\bf W}_{\alpha }\circ \left(-\frac{1}{2}{\bf I}+{\bf K}_{\alpha }\right)^{-1}\circ \gamma _{+}$ is the solution operator corresponding to the Dirichlet problem for the Brinkman system in $\Omega _+$, while ${\bf T}_{\alpha ;D;j}:={\bf W}^j_{\alpha }\circ \left(-\frac{1}{2}{\bf I}+{\bf K}^j_{\alpha }\right)^{-1}\circ \gamma _{+}$ is the solution operator for the Dirichlet problem in $\Omega _j$.
After performing a change of variable as in {Lemma \ref{2.13D}}, the operator $-\frac{1}{2}{\bf I}+{\bf K}^j_{\alpha }$ defined on $\partial \Omega _j$ can be identified with an operator ${\mathcal T}_\alpha^j$ acting on functions defined on $\partial \Omega $.
Then, employing the arguments, e.g., similar to those in the last paragraph in p.116 in \cite{M-W}, which are based on \cite[Lemmas 11.9.13 and 11.12.2]{M-W}, {and taking into account \cite[Proposition 1]{Med-AAM} (see also  \cite[Theorems 3.8 (iv) and 4.15]{Fa-Ke-Ve})}, one can show that the sequence of operators ${\mathcal T}_\alpha^j$ converges to the operator ${\mathcal T}_\alpha :=-\frac{1}{2}{\bf I}+{\bf K}_{\alpha }$ in the operator norm, the sequence of the inverses of the operators ${\mathcal T}_\alpha^j$ converges to the inverse of the operator ${\mathcal T}_\alpha$ in the operator norm,
%of the space ${\mathcal L}\left(H^1_p(\partial \Omega ,{\mathbb R}^n),H^1_p(\partial \Omega ,{\mathbb R}^n)\right)$,
and hence
} %\comment end
By \eqref{Green-j} and \eqref{j-ell-1v},
there exists a constant $C$ depending only on $\alpha$, $p$, $n$ and the Lipschitz character of $\Omega _+$
%(thus, not depending on $j$)}
such that
\begin{align*}
%\label{j-ell-1}
\|M(\nabla {\bf G}^{\alpha ;j}_k({\bf x}_0,\cdot))\|_{L_{p'}(\partial \Omega _{j})}
+\|M(\pi ^{j}_{k}({\bf x}_0,\cdot))\|_{L_{p'}(\partial \Omega _{j})}
\leq C
%\|{\bf G}^{\alpha }_k({\bf x}_0,\cdot )\|_{H_{p'}^1(\partial \Omega ,{\mathbb R}^n)},
(\|M({\bf G}^{\alpha}_k({\bf x}_0,\cdot))\|_{L_{p'}(\partial \Omega)}
+\|M(\nabla {\bf G}^{\alpha}_k({\bf x}_0,\cdot))\|_{L_{p'}(\partial \Omega)}),
\end{align*}
%({cf.} also {\bl \cite[Theorems 3.8 (iv) and 4.15]{Fa-Ke-Ve} and} \cite[Proposition 1]{Med-AAM}).
%The functions ${\bf G}_{k}^{\alpha;j}({\bf x}_0,\cdot)$ and {\bl $\pi ^{j}_{k}({\bf x}_0,\cdot)$}
%%are smooth in $\mathbb R^n\setminus \{{\bf x}_0\}$ %the functions $\mathbf v^j$ and $q^j$
%%and
%satisfy {estimate} \eqref{j-ell-1},
{Since} also $M({\bf u}^0)\in L_p(\partial \Omega )$ and $({\bf u}^0)^+_{\rm nt}={\bf 0}$ on $\partial \Omega ${, then} the Lebesgue Dominated Convergence Theorem (applied again after the change of variable as in {Lemma \ref{2.13D}} that reduces the integral over $\partial \Omega _j$ to an integral over $\partial \Omega$)
%and hence the sequence $\left\{{\bf t}^{\mathrm c+}\left({\bf G}^{\alpha ;j}_k({\bf x}_0,\cdot),\pi _k^j({\bf x}_0,\cdot)\right)\right\}_{j}$ can be replaced by a sequence of functions defined on $\partial \Omega $, bounded by a uniform constant related to the constant in \eqref{j-ell-1})}
implies that the right hand side in \eqref{j-ell-2} tends to zero as $\partial\Omega_j$ tends to $\partial\Omega$ and hence  $u_k^0({\bf x}_0)=0$. Because ${\bf x}_0$ is an arbitrary point in $\Omega _{+}$, we conclude that ${\bf u}^0={\bf 0}$ in $\Omega _{+}$, and by the first equation in \eqref{Brinkman-homogeneous1},
%$\pi^0=0$ (up to an additive constant pressure),
{$\pi^0$ is a constant pressure,}
as asserted.
{This completes the proof of the uniqueness in item (ii).
%, and hence in item (i) since its solutions are also solutions for item (ii).
}

%\comment
{
Let us show also the uniqueness result {for} item $(i)$. To do so, assume that $({\bf u}_0,\pi _0)$ is a solution of the homogeneous version of the Dirichlet problem \eqref{Brinkman-homogeneous1} {such that $M({\bf u}_0),{M}(\nabla\textbf{u}_0),{M}(\pi _0)\in L_p(\partial {\Omega})$, there exist the non-tangential limits of ${\bf u}_0$, $\nabla {\bf u}_0$ and $\pi _0$ at almost all points of the boundary $\partial \Omega $, and ${\bf u}_0$ satisfies the homogeneous Dirichlet boundary condition in the sense of non-tangential limit at almost all points of $\partial \Omega $}. Then the Green representation formula
%\begin{align*}
${\bf u}_0={\bf V}_{\alpha }\left(\mathbf t_{\rm{nt}}^+({\bf u}_0,\pi _0)\right)-{\bf W}_{\alpha }\left({\bf u}^{0+}_{\rm nt}\right) \mbox{ in } \Omega _{+}$
%\end{align*}
(cf. Lemma \ref{Green-r-f})
reduces to ${\bf u}_0={\bf V}_{\alpha }\left(\mathbf t_{\rm{nt}}^+({\bf u}_0,\pi _0)\right)$ in $\Omega _{+}$, and, by considering the non-tangential trace, we obtain that
${\mathcal V}_{\alpha }\left(\mathbf t_{\rm{nt}}^+({\bf u}_0,\pi _0)\right)={\bf 0}$ on $\partial \Omega $. Thus, $\mathbf t_{\rm{nt}}^+({\bf u}_0,\pi _0)\in {\mathbb R}\boldsymbol \nu $ (see Lemma \ref{isom-sl}), and hence ${\bf u}_0={\bf 0}$ in $\Omega _+$, while the Brinkman equation \eqref{Brinkman-homogeneous1} shows that $\pi^0=0$ in $\Omega _+$ (up to an additive constant pressure). This completes the proof of the statement in item $(i)$.
} %\comment end

Next we show for $s\in (0,1)$ the uniqueness of a solution to the Dirichlet problem \eqref{Brinkman-homogeneous1}-\eqref{DirCond}, in the hypothesis of item $(iii)$. %, i.e., {\bl without assuming in advance that $M({\bf u}^0)\in L_p(\partial \Omega )$.}
To this end, let $({\bf u}^0,\pi ^0)\in B_{p,p^*}^{s+\frac{1}{p}}(\Omega _+,{\mathbb R}^n)\times B_{p,p^*}^{s+\frac{1}{p}}(\Omega _+)$ denote a solution of the homogeneous version of the Dirichlet problem \eqref{Brinkman-homogeneous1}-\eqref{DirCond}.
By Lemmas \ref{trace-lemma-Besov}, \ref{lem 1.6} and Theorem~\ref{trace-equivalence-L} we obtain that
$\gamma _{+}{\bf u}^0={\bf u}^{0+}_{\rm nt}=0$ %${\bf u}^{0+}_{\rm nt}=0$
%\subset L_{p}(\partial\Omega,{\mathbb R}^n)$
and ${\bf t}_{\alpha }^+({\bf u}^0,\pi ^0)\in B_{p,p^*}^{s-1}(\partial\Omega,{\mathbb R}^n)$.
%\subset H_{p}^{s-1-\epsilon}(\partial\Omega,{\mathbb R}^n)$, $\forall\, \epsilon>0$.
%\subset H_{p}^{-1}(\partial\Omega,{\mathbb R}^n)$ for $s\in (0,1)$.
Then for $s\in (0,1)$, the Green representation formula \eqref{Green-r-f-Brinkman-s} applied to the pair $({\bf u}^0,\pi ^0)$ implies that ${\gamma_+}{\bf V}_{\alpha }\left({\mathbf t}_\alpha^+({\bf u},\pi )\right)={\bf 0}$ on $\partial \Omega $.
Hence by {\eqref{68s} and \eqref{B-s-l-adj-B}} we obtain that ${\bf t}_\alpha^+({\bf u},\pi )\in {\mathbb R}\boldsymbol \nu $.
Since ${\bf V}_{\alpha }\boldsymbol \nu ={\bf 0}$ in $\Omega _{+}$, we deduce that ${\bf u}_0={\bf 0}$ in $\Omega _{+}$, and by the Brinkman equation \eqref{Brinkman-homogeneous1}
%{\bn$\pi^0$ is a constant pressure.}
$\pi^0=0$ (up to an additive constant).
%In the case $s=1$, we have $\gamma _{+}{\bf u}^0={\bf u}^{0+}_{\rm nt}=0$ and ${\bf t}_{\alpha }^+({\bf u}^0,\pi ^0)\in B_{p,p^*}^{s'-1}(\partial\Omega,{\mathbb R}^n)$,
%\subset H_{p}^{-1}(\partial\Omega,{\mathbb R}^n)$, for any $s'\in (0,1)$.
%Then the proof of uniqueness in the case $s=1$ reduces to that in the case $s\in (0,1)$.
\hfill\end{proof}

%\begin{remark}
Note that for $p=2$, Theorem \ref{M-H-D} (ii) has been obtained by Z. Shen in \cite[Theorem 5.5]{Shen} by using another double layer potential approach.
%\end{remark}

The following regularity result has been obtained in \cite[Theorem 4.3.1]{M-W} and \cite[Theorem 7.1]{M-T} in the case of the Stokes system (i.e., for $\alpha =0$). We prove a similar result in the case of the Brinkman system (i.e., for $\alpha > 0$) by using the main ideas of the proof of \cite[Theorem 7.1]{M-T} (see also \cite[(2.95), Remark V p. 37]{M-M-T}, \cite[Theorem 2]{Choe-Kim}, \cite[Lemma 3.3]{K-L-W2}, \cite{Med-CVEE}).
%{\bl *** Do we delete the next result according to the suggestion of the referee? ***}
\begin{theorem}
\label{M-H}
Let ${\Omega}_+\subset \mathbb{R}^{n}$ be a bounded Lipschitz domain {with connected boundary $\partial \Omega $}.
Let $\alpha \geq 0$, $p\in (1,\infty )$ and $p^*:=\max \{p,2\}$. Assume that a pair $({\bf u},\pi)$ satisfies the homogeneous Brinkman system \eqref{Brinkman-homogeneous1}. {Then the following properties hold.}
\begin{itemize}
\item[$(i)$]
{There exists $\varepsilon=\varepsilon(\partial \Omega )>0$ such that for any $p\in (2-\varepsilon, \infty)$, the condition ${M}(\textbf{u})\in L_p(\partial {\Omega})$
%{, $p\in {\mathcal R}_{1}(n,\varepsilon )$,}
implies that} there exists the non-tangential limit of ${\bf u}$ almost everywhere on $\partial \Omega $ and $\mathbf u^+_{\rm nt}\in L_{p;\nu}(\partial {\Omega},\mathbb{R}^{n})$. Moreover,
\begin{align}
\label{4.9}
\|{\mathbf u^+_{\rm nt}}\|_{L_p(\partial \Omega ,{\mathbb R}^n)}\leq C_1\|M({\bf u})\|_{L_p(\partial \Omega )},\ \ \|{\bf u}\|_{B_{p,p^*}^{\frac{1}{p}}({\Omega}_+,\mathbb{R}^{n})}%+\|\pi\|_{B_{p,p^*}^{\frac{1}{p}-1}({\Omega}_+)}
\le C'_1\|M({\bf u})\|_{L_p(\partial \Omega )},
\end{align}
with some constants $C_1\equiv C_1(\partial\Omega,p,\alpha)>0$, $C'_1\equiv C'_1(\partial\Omega,p,\alpha)>0$.
\item[$(ii)$]
{There exists $\varepsilon=\varepsilon(\partial \Omega )>0$ such that for any $p\in {\mathcal R}_{0}(n,\varepsilon )\cup(2,\infty)$, the assumption ${{M}({\bf u})},{M}(\nabla {\bf u}),{M}(\pi)\in L_p(\partial {\Omega})$} implies that there exist the non-tangential limits of ${\bf u},\nabla {\bf u},\pi $ almost everywhere on $\partial \Omega $, and that
%and
%\label{non-tangential1}
${\mathbf u^+_{\rm nt}}\in {H_{p;\boldsymbol \nu}^{1}(\partial {\Omega},\mathbb{R}^{n})}$ and
$\mathbf t_{\rm{nt}}^{+}({\bf u},\pi)
%={\boldsymbol{\sigma}^+_{\rm nt}({\bf u},\pi)}\,\boldsymbol \nu
\in L_p(\partial {\Omega}, \mathbb{R}^{n})$. In addition, there exist some constants $C_2\equiv C_2(\partial\Omega,p,\alpha)>0$, $C'_2\equiv C'_2(\partial\Omega,p,\alpha)>0$ such that
%\end{align*}
\begin{align}
\label{n-p}
\|{\mathbf u^+_{\rm nt}}\|_{H_p^1(\partial \Omega ,{\mathbb R}^n)}
+\|\mathbf t_{\rm nt}^{+}({\bf u},\pi)\|_{L_p(\partial \Omega ,{\mathbb R}^n)}
&\leq C_2\left(\|M({\bf u})\|_{L_p(\partial \Omega )}+\|M(\nabla {\bf u})\|_{L_p(\partial \Omega )}+
\|M(\pi )\|_{L_p(\partial \Omega )}\right),\\
\|{\bf u}\|_{B_{p,p^*}^{1+\frac{1}{p}}({\Omega}_+,\mathbb{R}^{n})}+\|\pi\|_{B_{p,p^*}^{\frac{1}{p}}({\Omega}_+)}
&\le C'_2\left(\|M({\bf u})\|_{L_p(\partial \Omega )}+\|M(\nabla {\bf u})\|_{L_p(\partial \Omega )}+\|M(\pi )\|_{L_p(\partial \Omega )}\right).
\end{align}
\end{itemize}
\end{theorem}
\begin{proof}
{(i) We will use arguments similar to the ones in the proof of \cite[Lemma 8]{Choe-Kim}. First,
let $\{\Omega _j\}_{j\geq 1}$ be a sequence of sub-domains in $\Omega _{+}$ that converge to $\Omega _{+}$ in the sense described in {Lemma \ref{2.13D}}, with the corresponding notations $\Phi_j$, $\boldsymbol \nu^{(j)}$ and  $\omega_j$ also introduced there.
{Due to ellipticity of the homogeneous Brinkman system in $\Omega_+$, we have
$({\bf u},\pi)\in C^{\infty }(\Omega_+,{\mathbb R}^n)\times C^{\infty }(\Omega_+)$.
Now, let ${\bf h}^{(j)}:={\bf u}|_{\partial \Omega _j}$.
Then $({\bf u}_j,\pi _j):=({\bf u}|_{\overline \Omega _j},\pi |_{\overline \Omega _j})$} satisfies the homogeneous Brinkman system in $\Omega _j$ and the Dirichlet boundary condition ${\bf u}_j|_{\partial \Omega _j}={\bf h}^{(j)}$ on $\partial \Omega _j$, where ${\bf h}^{(j)}\in L_{p;\boldsymbol \nu^{(j)}}(\partial \Omega _j,{\mathbb R}^n)$.
The solution of such a problem is unique, up to an additive constant for the pressure (see, e.g., Theorem \ref{M-H-D}).
\comment{
as in \cite[page 80]{Necas2012}, \cite[Theorerm 1.12]{Verchota1984} (see also the proof of Theorem~\ref{2.13}(ii) in this paper), consider a sequence of $C^{\infty }$-domains $\left\{\Omega _j\right\}\subset \Omega _{+}$ approximating $\Omega _{+}$, such that their Lipschitz characters are uniformly controlled by the Lipschitz character of $\Omega _+$. Moreover, there exist some diffeomorphisms $\Phi_j:\partial \Omega \to \partial \Omega _j$ such that
\begin{align}
&\Phi_j({\bf x})\in {\mathfrak D}_{+}({\bf x}),\ \forall \ j\geq 1,\ \forall \ {\bf x}\in \partial \Omega ,\\
&\lim_{j\to \infty }|\Phi_j({\bf x})-{\bf x}|=0 \mbox{ uniformly in } {\bf x}\in \partial \Omega ,\\
&\lim_{j\to \infty }\boldsymbol \nu^{(j)}(\Phi_j({\bf x}))=\boldsymbol \nu ({\bf x})\ \mbox{ for a.e. }\ {\bf x}\in \partial \Omega ,
\end{align}
where ${\mathfrak D}_{+}({\bf x})$ is the non-tangential approach cone with vertex at ${\bf x}$, and $\boldsymbol \nu^{(j)}$ is the outward unit normal to $\partial \Omega _j$. We can assume that the Lipschitz constants of $\Phi_j$ and $\Phi_j^{-1}$ are uniformly bounded in $j$.

In addition, there exist some positive functions $\omega_j:\partial \Omega \to {\mathbb R}$ (the Jacobian related to $\Phi_j$, $j\in {\mathbb N}$) bounded away from zero and infinity uniformly in $j$, such that, for any measurable set $A\subset \partial \Omega $, $\int _A\omega _jd\sigma =\int _{\Phi_j(A)}d\sigma _j$. Moreover, $\lim _{j\to \infty }\omega _j=1$ a.e. on $\partial \Omega $ and in every space $L_p(\partial \Omega )$, $p\in (1,\infty )$ (cf. \cite[Theorem 1.12(iv)]{Verchota1984}).
} %\comment end

According to {Lemma \ref{isom-p}} applied to the smooth domain $\Omega _j$, such a solution can be expressed in terms of the double layer potential
${\bf u}_j={\bf W}_{\alpha ;\partial \Omega _j}{{\bf h}'^{(j)}}$, $\pi _j
={\mathcal Q}_{\alpha ;\partial \Omega _j}^d{{\bf h}'^{(j)}}$,
with a density ${{\bf h}'^{(j)}}\in L_{p;\boldsymbol \nu^{(j)}}(\partial \Omega _j,{\mathbb R}^n)$ satisfying the equation
{$\left(-\frac{1}{2}{\mathbb I}+{\bf K}^j_{\alpha}\right){{\bf h}'^{(j)}}={\bf h}^{(j)}$,
where ${\bf K}^j_{\alpha}:={\bf K}_{\alpha ;\partial \Omega _j}$ is associated (as in \eqref{68-s1s}) with the double layer potential $W_{\alpha ;\partial \Omega _j}$ defined on
$L_{p;\boldsymbol \nu^{(j)}}(\partial \Omega _j,{\mathbb R}^n)$,}
and, in view of Lemma \ref{isom-p}, the operator
$-\frac{1}{2}{\bf I}+{\bf K}^j_{\alpha}:L_{p;\boldsymbol\nu^{(j)}}(\partial \Omega _j,{\mathbb R}^n)\to L_{p;\boldsymbol\nu^{(j)}}(\partial \Omega _j,{\mathbb R}^n)$
is an isomorphism for any $p\in (1,\infty )$.

Note that the operator
$-\frac{1}{2}{\bf I}+{\bf K}_{\alpha }:L_{p;\boldsymbol \nu}(\partial \Omega ,{\mathbb R}^n)
\to L_{p;\boldsymbol \nu}(\partial \Omega ,{\mathbb R}^n)$
is an isomorphism for any
$p\in {\mathcal R}_{1}(n,\varepsilon)$ (see Lemma \ref{isom-p}).
%Moreover, ${\bf T}_{\alpha ;D}:={\bf W}_{\alpha }\circ \left(-\frac{1}{2}{\bf I}+{\bf K}_{\alpha }\right)^{-1}\circ \gamma _{+}$ is the solution operator corresponding to the Dirichlet problem for the Brinkman system in $\Omega _+$, while ${\bf T}_{\alpha ;D;j}:={\bf W}^j_{\alpha }\circ \left(-\frac{1}{2}{\bf I}+{\bf K}^j_{\alpha }\right)^{-1}\circ \gamma _{+}$ is the solution operator for the Dirichlet problem in $\Omega _j$.
After performing a change of variable as in {Lemma \ref{2.13D}}, the operator $-\frac{1}{2}{\bf I}+{\bf K}^j_{\alpha }$ defined on $\partial \Omega _j$ can be identified with an operator ${\mathcal T}_\alpha^j$ acting on functions defined on $\partial \Omega $.
Then, employing the arguments, e.g., similar to those in the last paragraph in p.116 in \cite{M-W}, which are based on \cite[Lemmas 11.9.13 and 11.12.2]{M-W}, and taking into account \cite[Proposition 1]{Med-AAM} (see also  \cite[Theorems 3.8 (iv) and 4.15]{Fa-Ke-Ve}), one can show that the sequence of operators ${\mathcal T}_\alpha^j$ converges to the operator ${\mathcal T}_\alpha :=-\frac{1}{2}{\bf I}+{\bf K}_{\alpha }$ in the operator norm {and} the sequence of the inverses of the operators ${\mathcal T}_\alpha^j$ converges to the inverse of the operator ${\mathcal T}_\alpha$ in the operator norm for $p\in {\mathcal R}_{1}(n,\varepsilon)$.
%of the space ${\mathcal L}\left(H^1_p(\partial \Omega ,{\mathbb R}^n),H^1_p(\partial \Omega ,{\mathbb R}^n)\right)$,
Hence, if $p\in {\mathcal R}_{1}(n,\varepsilon)$, {the operator norms $\|\left(-\frac{1}{2}{\mathbb I}+{\bf K}^j_{\alpha}\right)^{-1}\|_{L_p(\partial \Omega _j,{\mathbb R}^n)}$ are bounded uniformly in $j$, implying that} there exists a constant $c_0$ depending only on $p$, $n$, $\alpha$, and the Lipschitz character of $\Omega _+$ (thus, not depending on $j$) such that
%and there exists a constant $c_0>0$ such that
\begin{align}
\label{e1}
\|{{\bf h}'^{(j)}}\|^p_{L_p(\partial \Omega _j,{\mathbb R}^n)}
&\leq c_0\|{\bf h}^{(j)}\|^p_{L_p(\partial \Omega _j,{\mathbb R}^n)}
=c_0\|{\bf u}\|^p_{L_p(\partial \Omega _j,{\mathbb R}^n)}\nonumber\\
&= c_0\int_{\partial \Omega _j}|{\bf u}({\bf y}_j)|^pd\sigma _{{\bf y}_j}
=c_0\int_{\partial \Omega }|{\bf u}(\Phi_j({\bf y}))|^p\omega _j({\bf y})d\sigma _{{\bf y}}
\leq c_1\int_{\partial \Omega }|M({\bf u}({\bf y}))|^pd\sigma _{{\bf y}}
= c_1\|M({\bf u})\|^p_{L_p(\partial \Omega,{\mathbb R}^n)},
\end{align}
Recall that {we have approximated the domain $\Omega _+$ with a sequence of smooth domains $\Omega _j$ with uniform Lipschitz characters from inside, }
{and we have employed here the change of variable ${\bf y}_j:=\Phi_j({\bf y})$,
${\bf y}\in \partial \Omega,$
${\bf y}_j\in \partial \Omega_j,$
and $\omega _j$ is the Jacobian of $\Phi_j:\partial \Omega \to \partial \Omega _j$  (cf. Lemma \ref{2.13D}).
Hence the constants $c_0$ and $c_1$ depend only on $p$, $n$, $\alpha$, and the Lipschitz character of $\Omega _+$.}
%Due to the choice of the sequence $\left\{\Omega _j\right\}_{j\geq 1}$, $c_0$ depends only on the Lipschitz character of $\Omega _+$.
\comment{By employing again arguments similar to those in the last paragraph in p.116 in \cite{M-W} (which are essentially based on \cite[Lemma 11.9.3]{M-W}, inequality (11.196) and on the singular behaviour of the kernel of the Brinkman double layer potential), we deduce that {the operator norms of $-\frac{1}{2}{\mathbb I}+{\bf K}_{\alpha ;\partial \Omega _j}$ and its inverse are uniformly bounded in $j$ (see also \cite[p. 785]{Fa-Ke-Ve})}.
Then} %\comment end
%{The constant $c_0$ depends only on the Lipschitz character of $\Omega _+$, $n$ and $p$ but not on $j$ by the same arguments as in the proof for the constant $C_0$ in \eqref{j-ell-1v}}.
%(see also the proof of inequality $(1)$ in \cite{Choe-Kim}).}

%{***SM: It looks like the limitation $p\in {\mathcal R}_{1}(n,\varepsilon )$ should be placed in the theorem hypotheses. Check other sources.}

\comment{\bl Note that the solutions $({\bf u}_j,\pi _j)$ of the sequence of domains $\Omega _j$ approximating $\Omega _+$ satisfy the inequality
\begin{align}
\label{e1-1}
\|({\bf u}_j)_{{\rm nt}}^+\|_{L_p(\partial \Omega _j,{\mathbb R}^n)}\leq \|M({\bf u}_j)\|_{L_p(\partial \Omega _j)}\leq c_0^*\|{\bf h}^{(j)}\|_{L_p(\partial \Omega _j,{\mathbb R}^n)}
\end{align}
with a uniform constant $c_0^*$ (see, e.g., \cite[Theorem 3.9 (iii)]{Fa-Ke-Ve} in the case $\alpha =0$ and $p=2$). In addition, the following estimate holds
\begin{align}
\label{e1-2}
\|{\bf h}'_j\|_{L_p(\partial \Omega _j,{\mathbb R}^n)}\leq c_1^*\left\|\left(-\frac{1}{2}{\mathbb I}+{\bf K}_{\alpha ;\partial \Omega _j}\right){\bf h}'_j\right\|_{L_p(\partial \Omega _j,{\mathbb R}^n)}
\end{align}
with a uniform constant $c_1^*$ (cf., e.g., \cite[(5.17)]{Shen} for $p=2$). Then inequalities \eqref{e1-1} and \eqref{e1-2} lead to inequality \eqref{e1} (see also the proof of inequality $(1)$ in \cite{Choe-Kim}).}
%Due to the choice of the sequence $\left\{\Omega _j\right\}_{j\geq 1}$, $c_0$ depends only on the Lipschitz character of $\Omega _+$.}
%{[***SM: Why $c_0$ does not depend on $j$? One can try using the arguments similar to Lemma 11.12.2 in {M-W}. But one steel needs to redo the arguments to the Brinkman double layer potential.]}

{Further,} the double-layer potential ${\bf W}_{\alpha ;\partial \Omega _j}{{\bf h}'^{(j)}}$ becomes
\begin{align}
\label{dl-sol-2}
u_\ell ({\bf x})&
=\int_{\partial {\Omega _j}} S^{\alpha}_{i\ell s}({\bf y}_j,{\bf x})\nu _{s}({\bf y}_j)h'^{(j)}_{i}({\bf y}_j)d\sigma _{{\bf y}_j}
=\int_{\partial {\Omega}} S^{\alpha}_{i\ell s} (\Phi_j({\bf y}),{\bf x})\nu _{s}(\Phi_j({\bf y}))H'^{(j)}_{i}({\bf y})d\sigma _{{\bf y}},\ \ \forall \ {\bf x}\in \Omega _j,
\end{align}
where
%\begin{align}
$
{{\bf H}'}^{(j)}({\bf y}):={{\bf h}'^{(j)}}(\Phi_j({\bf y}))\omega _j({\bf y}),
$
%$
%{\bf y}\in \partial \Omega ,
%$
%\end{align}
%${\bf y}_j=(y_1^{(j)},\ldots ,y_n^{(j)})$,
${\bf h}'^{(j)}=(h^{'(j)}_{1},\ldots ,h^{'(j)}_{n})$, ${\bf H}'^{(j)}=({H'}_1^{(j)},\ldots ,{H'}_n^{(j)})$.
%and $\omega _j$ is the Jacobian of $\Phi_j:\partial \Omega \to \partial \Omega _j$.

In view of \eqref{e1} and of the {uniform} boundedness of $\{\omega _j\}_{j\geq 1}$, there exist some constants $c_1,c_2>0$ (which depend only on $\Omega _+$ and $p$) such that
\begin{align}
\label{e2}
\int _{\partial \Omega }|{\bf H}'^{(j)}({\bf y})|^pd\sigma _{{\bf y}}
\leq c_2\int_{\partial \Omega _j}|{\bf u}({\bf y}_j)|^pd\sigma _{{\bf y}_j}
%=c_1\int_{\partial \Omega }|{\bf u}(\Phi_j({\bf y}))|^p\omega _j({\bf y})d\sigma _{{\bf y}}
\leq c'_2\int_{\partial \Omega }|M({\bf u}({\bf y}))|^pd\sigma _{{\bf y}},\ \forall \ j\geq 1.
\end{align}
Hence $\{{\bf H}'^{(j)}\}_{j\geq 1}$ is a bounded sequence in $L_p(\partial \Omega ,{\mathbb R}^n)$, and, thus, there {exists a subsequence, still} denoted as the sequence, and a function ${\bf H}'\in L_p(\partial \Omega ,{\mathbb R}^n)$, such that ${\bf H}'^{(j)}\to {\bf H}'$ weakly in $L_p(\partial \Omega ,{\mathbb R}^n)$.
By this property and letting $j\to \infty $ in \eqref{dl-sol-2}, we obtain
%\begin{align}
%\label{u-dl}
$
{\bf u}={\bf W}_{\alpha }{\bf H}' \mbox{ in } \Omega _+.
$
%\end{align}
According to Lemma \ref{nontangential-sl}{(i,iv),} there exists the non-tangential limit ${\bf u}_{{\rm nt}}^+=({\bf W}_{\alpha }{\bf H}')_{{\rm nt}}^+$ of ${\bf u}$ at almost all points of $\partial \Omega $,
and by estimates \eqref{7ms-dl-0} and \eqref{e2}, we obtain that
%${\bf u}_{{\rm nt}}^+\in L_p(\partial \Omega ,{\mathbb R}^n)$.
%Indeed, the inequality
\begin{align}
\label{4.15}
\|{\bf u}_{{\rm nt}}^+\|_{L_p(\partial \Omega ,{\mathbb R}^n)}=\|({\bf W}_{\alpha }{\bf H}')_{{\rm nt}}^+\|_{L_p(\partial \Omega ,{\mathbb R}^n)}
\leq {c_3\|{\bf H}'\|_{L_p(\partial \Omega ,{\mathbb R}^n)}}
{\leq c_3\lim \inf_{j\to \infty }\|{\bf H}'^{(j)}\|_{L_p(\partial \Omega ,{\mathbb R}^n)}}
%\leq C_0\|\sup_{j\ge 1}|{\bf H}'^{(j)}|\,\|_{L_p(\partial \Omega )}
\leq c_4\|M({\bf u})\|_{L_p(\partial \Omega )},
\end{align}
{where the constants $c_3,c_4>0$ do not depend on $j$.}
%and estimate \eqref{7ms-dl-0} imply the desired property.
Moreover, the divergence theorem shows that ${\mathbf u^+_{\rm nt}}=({\bf W}_{\alpha }{\bf H}')_{{\rm nt}}^+\in L_{p;\nu}(\partial {\Omega},\mathbb{R}^{n})$.}
Estimate \eqref{4.9} is provided by the representation
${\bf u}={\bf W}_{\alpha }{\bf H}'$,
%$\pi=\mathcal Q^d_{\alpha }{\bf H}'+C_0$,
%where $C_0$ is an arbitrary constant,
by continuity of {operator \eqref{ds-1-0}}, and by estimates \eqref{4.15}.
This completes the proof of item (i) for {any} $p\in {\mathcal R}_{1}(n,\varepsilon )$.

{Let us now consider item (i) for any $p>2$ (not covered yet when $n>3$). Note that inclusions
$2\in {\mathcal R}_{1}(n,\varepsilon )$ and $L_p(\partial {\Omega})\subset L_2(\partial {\Omega})$ particularly imply that for such $p$ there exist  non-tangential limits of ${\bf u}$ almost everywhere on $\partial \Omega $. Implementing now, e.g., \cite[Proposition 3.29]{M-M1} completes the proof for any $p>2$.}

%{According to an extension of Coifman, McIntosh and Meyer result \cite{Co-Me} to singular integral operators of Calder\'{o}n-Zygmund type on $L_p$ spaces, we deduce that there exists a function ${\bf g}:=({\bf W}_{\alpha }{\bf H}')_{{\rm nt}}^+\in L_p(\partial \Omega ,{\mathbb R}^n)$ such that ${\bf u}\to {\bf g}$ non-tangentially a.e. on $\partial \Omega $. Moreover, the divergence theorem shows that ${\mathbf u^+_{\rm nt}}={\bf g}\in L_{p;\nu}(\partial {\Omega},\mathbb{R}^{n})$.}

(ii) Now assume that ${\bf u}$ and $\pi $ satisfy the Brinkman system and that {${{M}({\bf u})},{M}(\nabla {\bf u}), {M}(\pi)\in L_p(\partial {\Omega})$}. As in the proof of {item} (i), we consider again a sequence of smooth domains $\left\{\Omega _j\right\}_{j\in {\mathbb N}}$, such that $\overline\Omega _j \subseteq \Omega _+$ and $\Omega _j\to \Omega _+$ as $j\to \infty $.

{As we already mentioned, $({\bf u}_j,\pi _j):=({\bf u}|_{\overline \Omega _j},\pi |_{\overline \Omega _j})\in C^{\infty }(\overline\Omega _j,{\mathbb R}^n)\times C^{\infty }(\overline\Omega _j)$.
Thus, ${\bf h}^{(j)}:={\bf u}|_{\partial \Omega _j}\in C^{\infty }(\partial\Omega _j,\mathbb R^n)
\subset H_p^1(\partial \Omega _j,\mathbb R^n)$ and ${\bf h}^{(j)}\in L_{p;\boldsymbol \nu^{(j)}}(\partial\Omega _j,{\mathbb R}^n)$, for any $j\in {\mathbb N}$.
Then the pair $({\bf u}_j,\pi _j)\in C^{\infty }({\Omega }_j,{\mathbb R}^n)\times C^{\infty }({\Omega }_j)$ satisfies the Brinkman system in $\Omega _j$ with the Dirichlet boundary condition ${\bf u}_j|_{\partial \Omega _j}={\bf h}^{(j)}\in H^1_{p;\boldsymbol \nu^{(j)} }(\partial \Omega _j,{\mathbb R}^n)$.}
The solution of such a problem is unique up to an additive constant pressure (see Theorem \ref{M-H-D}(i)) and can be expressed in terms of a double layer potential as in item (i), but now with a density in
$H_{p;\boldsymbol \nu^{(j)}}^{1}(\partial \Omega _j,{\mathbb R}^n)$.
%{\rd or in terms of a single layer potential ${\bf u}_j={\bf V}_{\alpha ;\partial \Omega _j}{{\bf g}_j}$, $\pi _j={\mathcal Q}_{\alpha ;\partial \Omega _j}^s{{\bf g}_j}$ with a density {\bl ${\bf g}_j\in L_p(\partial \Omega _j,{\mathbb R}^n)$ (see \cite[Therem 4.15]{Fa-Ke-Ve} in the case $\alpha =0$)}}.
%We next employ the double layer potential approach. By Lemma~\ref{isom-p} the operator
\comment{
\begin{align}
\label{operator-j}
-\frac{1}{2}{\mathbb I}+{\bf K}_{\alpha ;\partial \Omega _j}:H^1_{p;\boldsymbol\nu^{(j)}}(\partial \Omega _j,{\mathbb R}^n)\to H^1_{p;\boldsymbol\nu^{(j)}}(\partial \Omega _j,{\mathbb R}^n)
\end{align}
is an isomorphism for any $p\in (1,\infty )$.
Thus, for ${\bf h}^{(j)}\in H_{p;\boldsymbol\nu^{(j)}}^1(\partial \Omega _j,{\mathbb R}^n)$ given, there exists a unique ${\bf g}_j\in H^1_{p;\boldsymbol\nu^{(j)}}(\partial \Omega _j,{\mathbb R}^n)$ such that
$\left(-\frac{1}{2}{\mathbb I}+{\bf K}_{\alpha ;\partial \Omega _j}\right){{\bf g}_j}={\bf h}^{(j)}$, and there exists a constant $c_0>0$ such that
\begin{align}
\label{e1-sl}
\|{\bf g}_j\|_{H^1_p(\partial \Omega _j,{\mathbb R}^n)}\leq \left\|\left(-\frac{1}{2}{\mathbb I}+{\bf K}_{\alpha ;\partial \Omega _j}\right)^{-1}\right\|_{{\mathcal L}(H^1_{p;\boldsymbol\nu^{(j)}}(\partial \Omega _j,{\mathbb R}^n),H^1_{p;\boldsymbol\nu^{(j)}}(\partial \Omega _j,{\mathbb R}^n))}\|{\bf h}^{(j)}\|_{H^1_p(\partial \Omega _j,{\mathbb R}^n)}
\leq c_0\|{\bf h}^{(j)}\|_{H^1_p(\partial \Omega _j,{\mathbb R}^n)}
=c_0\|{\bf u}\|_{H^1_p(\partial \Omega _j,{\mathbb R}^n)}.
\end{align}
In view of \cite[Lemma 11.12.2]{M-W}, the operator $-\frac{1}{2}{\mathbb I}+{\bf K}_{\alpha ;\partial \Omega _j}$ (after performing a change of variable that identifies this operator with one defined on $H^1_{p;\boldsymbol\nu }(\partial \Omega ,{\mathbb R}^n)$) converges to $-\frac{1}{2}{\mathbb I}+{\bf K}_{\alpha ;\partial \Omega }$ in the operator norm of the space ${\mathcal L}(H^1_{p;\boldsymbol\nu }(\partial \Omega ,{\mathbb R}^n),H^1_{p;\boldsymbol\nu }(\partial \Omega ,{\mathbb R}^n))$. Thus, the operator norms of $-\frac{1}{2}{\mathbb I}+{\bf K}_{\alpha ;\partial \Omega }$ and of its inverse are uniformly bounded in $j$ (see also \cite[p. 785]{Fa-Ke-Ve}). Hence the constant $c_0$ in \eqref{e1-sl} depends only on the Lipschitz character of $\Omega _+$ and $p$.
} %\comment end
Proceeding  similar to the proof of item (i), we prove item (ii).
\comment{Let us now prove item (ii) for $p>2$. Then, the results of item (ii) will hold particularly for
$p'=2\in {\mathcal R}_{0}(n,\varepsilon )$ by inclusion $L_p(\partial {\Omega})\subset L_2(\partial {\Omega})$, and especially, there exist the non-tangential limits of ${\bf u},\nabla {\bf u},\pi $ almost everywhere on $\partial \Omega $. Implementing now, e.g., \cite[Proposition 3.29]{M-M1} completes the proof for any $p>2$.
} %\comment end
\comment %
{By considering the change of variable ${\bf y}_j:=\Phi_j({\bf y})$, ${\bf y}_j=(y_1^{(j)},\ldots ,y_n^{(j)})$, the single-layer potential ${\bf V}_{\alpha ;\partial \Omega _j}{{\bf g}_j}$ becomes
\begin{align}
\label{sl-sol-2}
u_\ell ({\bf x})&=\int_{\partial {\Omega _j}}{\mathcal G}^{\alpha}_{i\ell} ({\bf y}_j,{\bf x})h^{(j)}_{i}({\bf y}_j)d\sigma _{{\bf y}_j}=\int_{\partial {\Omega}}{\mathcal G}^{\alpha}_{i\ell } (\Phi_j({\bf y}),{\bf x})H^{(j)}_{i}({\bf y})d\sigma _{{\bf y}},\ \ \forall \ {\bf x}\in \Omega _j.
\end{align}
where
%\begin{align}
${\bf H}'^{(j)}({\bf y}):={{\bf g}_j}(\Phi_j({\bf y}))\omega _j({\bf y}),\ {\bf y}\in \partial \Omega $, and ${\bf y}_j=(y_1^{(j)},\ldots ,y_n^{(j)})$, ${\bf H}'^{(j)}=({H'}_1^{(j)},\ldots ,{H'}_n^{(j)})$.
%\end{align}
%${H'}_i^{(j)}$ is the $i${\rm{th}} component of ${\bf H}'^{(j)}$, and $\omega _j$ is the Jacobian of $\Phi_j:\partial \Omega \to \partial \Omega _j$.

According to \eqref{e1-sl} and the uniformly boundedness of $\{\omega _j\}_{j\geq 1}$, there exist some constants $c_3,c_4>0$ (which depend only on $\Omega _+$ and $p$) such that
\begin{align}
\label{e2-sl}
\int _{\partial \Omega }|{\bf H}'^{(j)}({\bf y})|^pd\sigma _{{\bf y}}
\leq c_3\int_{\partial \Omega _j}|{\bf u}({\bf y}_j)|^pd\sigma _{{\bf y}_j}
\leq c_4\int_{\partial \Omega }|M({\bf u}({\bf y}))|^pd\sigma _{{\bf y}},\ \forall \ j\geq 1.
\end{align}
Therefore, $\{{\bf H}'^{(j)}\}_{j\geq 1}$ is a bounded sequence in $L_p(\partial \Omega ,{\mathbb R}^n)$, and hence there exist a subsequence, denoted as the sequence, and a function ${\bf H}'\in L_p(\partial \Omega _+,{\mathbb R}^n)$, such that ${\bf H}'^{(j)}\to {\bf H}'$ weakly in $L_p(\partial \Omega ,{\mathbb R}^n)$. By this property and letting $j\to \infty $ in \eqref{sl-sol-2}, we obtain
\begin{align}
\label{u-sl}
{\bf u}={\bf V}_{\alpha }{\bf H}'\ \mbox{ in }\ \Omega _{+}.
\end{align}
%where ${\bf H}'\in L_p(\partial \Omega ,{\mathbb R}^n)$.
Moreover, letting $j\to \infty $ in the representation formula for $\pi _j$, we obtain that $\pi ={\mathcal Q}_{\alpha ;\partial \Omega }^s{\bf H}'$ in $\Omega _+$. Consequently,
\begin{align}
\label{u-sl-pi}
{\bf u}={\bf V}_{\alpha }{\bf H}',\ \ \pi ={\mathcal Q}^s{\bf H}' \ \mbox{ in }\ \Omega _+.
\end{align}
{In view of Lemma \ref{nontangential-sl}(i),(ii) there exist the non-tangential limits ${\bf u}_{{\rm nt}}^+=({\bf V}_{\alpha }{\bf H}')_{{\rm nt}}^+$, $(\nabla {\bf u})_{{\rm nt}}^+=\left(\nabla {\bf V}_{\alpha }{\bf H}'\right)_{{\rm nt}}^+$ and $\pi _{{\rm nt}}^+=({\mathcal Q}_{\alpha }^s{\bf H}')_{{\rm nt}}^+$ of ${\bf u}$, $\nabla {\bf u}$ and $\pi $, respectively, at almost all points of $\partial \Omega $, and ${\bf u}_{{\rm nt}}^+\in H_p^1(\partial \Omega ,{\mathbb R}^n)$, $\pi _{{\rm nt}}^+\in L_p(\partial \Omega )$. In addition, the divergence theorem implies that ${\mathbf u^+_{\rm nt}}=({\bf V}_{\alpha }{\bf H}')_{{\rm nt}}^+\in H_{p;\boldsymbol \nu}^{1}(\partial {\Omega},\mathbb{R}^{n})$.}
Formula \eqref{u-sl-pi} and Lemma \ref{nontangential-sl}(i),(ii) show also that %the stress tensor associated with ${\bf u}$ and $\pi $ is given by
%\begin{align}
%\label{stress-u-p}
%(\boldsymbol \sigma ({\bf u},\pi ))_{j\ell }({\bf x})=\int _{\partial \Omega }S_{ji\ell }^{\alpha %}({\bf x},{\bf y})H_i({\bf y})d\sigma _{\bf y}, \ \ {\bf x}\in \Omega _{+}.
%\end{align}
%{Then \eqref{u-sl} and \eqref{7ms} show that
there exists the non-tangential limit of the stress tensor $\boldsymbol{\sigma }({\bf u},\pi )=-\pi {\mathbb I}+\nabla {\bf u}+(\nabla {\bf u})^\top$ at almost all points of $\partial \Omega $, and $\left(\boldsymbol \sigma ({\bf u},\pi )\right)_{\rm{nt}}^+\in L_p(\partial \Omega ,{\mathbb R}^{n\times n})$. In addition, we have that ${\mathbf t}_{\rm{nt}}^{+}({\bf u},\pi)={\boldsymbol{\sigma}^+_{\rm nt}({\bf u},\pi)}\,\boldsymbol \nu \in L_p(\partial {\Omega}, \mathbb{R}^{n})$.
Finally, formulas \eqref{u-sl-pi} and properties \eqref{ss-1} yield that ${\bf u}\in {B}_{p,p^*}^{1+\frac{1}{p}}({\Omega} _{+},{\mathbb R}^n)$ and $\pi \in {B}_{p,p^*}^{\frac{1}{p}}({\Omega}_{+})$, as asserted, while estimate \eqref{n-p} is immediate.
} %\comment end
\hfill\end{proof}

%{\bl *** If we delete Theorem 5.2 as the referee required, then the new sentence of Lemma 5.2 and the proof are as follows. ***}

\begin{remark}
\label{non-tangential-p}
{The {condition requiring the existence of the non-tangential limits of ${\bf u}$, $\nabla {\bf u}$ and $\pi $ at almost all points of the boundary $\partial \Omega $} in Lemma \ref{Green-r-f} {is} particularly satisfied if $p\in {\mathcal R}_{0}(n,\varepsilon )\cup(2,\infty)$ with $\varepsilon >0$ as in Theorem \ref{M-H}(ii).}
{Indeed, for such $p$, the condition is implied by the inclusions ${{M}({\bf u})},{M}(\nabla {\bf u}),{M}(\pi)\in L_p(\partial {\Omega})$ and by the Brinkman system \eqref{Brinkman-homogeneous}}.
%Then the integral representation formula \eqref{Green-r-f-Brinkman} follows.
\end{remark}

Having in view Theorem \ref{M-H-D}(iii), we are now able to consider the Poisson-Dirichlet problem for the Brinkman system,
\begin{equation}
\label{Poisson-Dirichlet-Brinkman-Lp}
\left\{
\begin{array}{l}
\triangle {\bf u} - \alpha {\bf u} - \nabla \pi =\mathbf f,\
{\rm{div}}\ {\bf u} = 0\ \mbox{ in }\  {\Omega}_+\\
\gamma_+{\bf u}={\mathbf h_{0}} \mbox{ on } \partial \Omega
\end{array}
\right.
\end{equation}
with the Dirichlet datum for the Gagliardo trace $\gamma_+{\bf u}$ {(see also \cite[Theorem 10.6.2]{M-W} for $\alpha =0$)}.
\begin{theorem}
\label{M-H-DG}
Let ${\Omega}_+\subset \mathbb{R}^{n}$ {$(n\geq 3)$} be a bounded Lipschitz domain with connected boundary $\partial {\Omega}$. {Let $\alpha \in (0,\infty )$ and $0<s\le 1$}.
Then there exists $\varepsilon=\varepsilon(\partial\Omega)>0$ such that for any $p\in\mathcal R_{1-s}(n,\epsilon)$ $($cf. \eqref{cases-s}$)$, the Dirichlet problem \eqref{Poisson-Dirichlet-Brinkman-Lp} with $\mathbf f\in L_p({\Omega}_+,{\mathbb R}^3)$ and $\mathbf h_0\in H_{p;\nu}^s(\partial {\Omega},\mathbb{R}^{n})$ has a solution
$(\textbf{u},\pi) \in  B_{p,p^*}^{s+\frac{1}{p}}({\Omega}_+,\mathbb{R}^{n})\times B_{p,p^*}^{s+\frac{1}{p}-1}({\Omega}_+)$,
which is unique up to an arbitrary additive constant for the pressure $\pi$, where $p^{*} = \max\{2, p\}$. In addition, there exists {a constant $C=C(s,p,\Omega _+)>0$} such that
$$
\|{\bf u}\|_{B_{p,p^*}^{s+\frac{1}{p}}({\Omega}_+,\mathbb{R}^{n})}
+\|\pi\|_{B_{p,p^*}^{s+\frac{1}{p}-1}({\Omega}_+)/\mathbb R}
\le C(\| \mathbf f\|_{L_p(\Omega_+,\mathbb{R}^n)}+\| \mathbf h_0\|_{H^s_p(\partial {\Omega},\mathbb{R}^{n})}).
$$
%and if, in addition, $s=1$, then
%$\|\mathbf t^+({\bf u},\pi )\|_{L_p(\partial\Omega,\mathbb{R}^{n})}\le {\cn C'}(\| \mathbf f\|_{L_p(\Omega_+,\mathbb{R}^n)}+\| \mathbf h_0\|_{H^s_p(\partial {\Omega},\mathbb{R}^{n})}).$

%Moreover, if $\mathbf f=\mathbf 0$ then there exists a constant ${\cn C_M}>0$ such that
%$\|M({\bf u})\|_{L_p(\partial {\Omega})}\le {\cn C_M}\| \mathbf h_0\|_{L_p(\partial {\Omega},\mathbb{R}^{n})},$
%and if, in addition, $s=1$, then
%\begin{equation}\label{MpiLpG}
%$\|M({\bf u})\|_{L_p(\partial {\Omega})}+\|M(\nabla{\bf u})\|_{L_p(\partial {\Omega})}
%+\|M(\pi)\|_{L_p(\partial {\Omega})/\mathbb R}\le {\cn C_M}\| \mathbf h_0\|_{H^1_p(\partial %{\Omega},\mathbb{R}^{n})}.$
%\end{equation}
\end{theorem}
\begin{proof}
If $\mathbf f=\mathbf 0$, the existence of a solution of the problem \eqref{Poisson-Dirichlet-Brinkman-Lp} for $0<s<1$
%given by the solution of the corresponding problem \eqref{Brinkman-homogeneous1}-\eqref{DirCond} with the non-tangential trace in the Dirichlet condition,
%whose existence
is implied by Theorem \ref{M-H-D}(iii) together with the asserted estimate, while for $s=1$ it follows from Theorems \ref{M-H-D} (i) and \ref{trace-equivalence-L} (iii).
%Here we relied also on the equivalence of traces, $\gamma_+{\bf u}={\bf u}^+_{\rm nt}$, implied by Theorems \ref{trace-equivalence-L}.

If $\mathbf f\not=\mathbf 0$, we will look for a solution of problem \eqref{Poisson-Dirichlet-Brinkman-Lp} in the form
\begin{align}
\label{Poisson-mixed-2-LpD}
{\bf u}={\mathbf N}_{\alpha ;{\Omega}_+}{\mathbf f}+{\bf v},\ \pi ={\mathcal Q}_{{\Omega}_+}{\mathbf f}+q,
\end{align}
where the Newtonian velocity and pressure potentials
${\mathbf N}_{\alpha ;{\Omega}_+}{\mathbf f}$ and ${\mathcal Q}_{\Omega_+}{\mathbf f}$ are defined by \eqref{Newtonian-1a-Lp}.
By  Remark~\ref{gammaN},
% they satisfy equations
\begin{align*}
%\label{Newtonian-2-Lp}
&\triangle {\mathbf N}_{\alpha ;{\Omega_+}}{\mathbf f}-\alpha {\mathbf N}_{\alpha ;{\Omega_+}}{\mathbf f}-\nabla {\mathcal Q}_{\Omega+}{\mathbf f}={\mathbf f},\
{\rm{div}}\ {\mathbf N}_{\alpha ;{\Omega_\pm}}{\mathbf f}=0\ \mbox{ in }\ {\Omega}_+,\\
&(\mathbf N_{\alpha;\Omega_+}{\mathbf f},{\mathcal Q}_{\Omega_+}{\mathbf f})\in {B_{p,p^*}^{2}({\Omega}_+,\mathbb{R}^{n})}\times { B_{p,p^*}^{1}({\Omega}_+)},
\ {\gamma}_{+}(\mathbf N_{\alpha;{\Omega}_+}{\mathbf f})\in H^1_{p;\boldsymbol\nu}(\partial {\Omega},{\mathbb R}^n),\ {\mathbf t}_\alpha ^+\left({\mathbf N}_{\alpha ;\Omega_+}{\mathbf f},{\mathcal Q}_{\Omega_\pm}{\mathbf f}\right)
\in L_p(\partial {\Omega},{\mathbb R}^{n}).
\end{align*}
Then problem \eqref{Poisson-Dirichlet-Brinkman-Lp} reduces to the one for the corresponding homogeneous Brinkman system,
\begin{equation}
\label{Newtonian-4-LpD}
\left\{
\begin{array}{l}
\triangle {\bf v} - \alpha {\bf v} - \nabla q= {\bf 0},\
{\rm{div}}\ {\bf v}= 0\ \mbox{ in }\  {\Omega}_+,\\
\gamma_+{\bf v}={\mathbf h_{00}},
\end{array}
\right.
\end{equation}
where
%\begin{align}
%\label{modif-LpD}
$
\mathbf h_{00}:=\mathbf h_0-\gamma_+\left({\mathbf N}_{\alpha ;{\Omega}_+}{\mathbf f}\right)\in H^s_{p;\boldsymbol\nu}(\partial\Omega,{\mathbb R}^n),
$
%\end{align}
already discussed in the first paragraph of the proof.
Therefore, there exists a
solution
$({\bf u},\pi)\in {B_{p,p^*}^{s+\frac{1}{p}}({\Omega}_+,\mathbb{R}^{n})}\times B_{p,p^*}^{s+\frac{1}{p}-1}({\Omega}_+)$
of the Poisson problem \eqref{Poisson-Dirichlet-Brinkman-Lp}, which satisfies the asserted estimate.

Let us prove the uniqueness of the solution to the Poisson problem  \eqref{Poisson-Dirichlet-Brinkman-Lp} for $0<s<1$. %(for $s=1$ uniqueness then follows by embedding).
To do so, we consider a solution
$({\bf u}^0,\pi ^0)\in B_{p,p^*}^{s+\frac{1}{p}}({\Omega},\mathbb{R}^{3})\times B_{p,p^*}^{s+\frac{1}{p}-1}({\Omega})$
of the homogeneous version of the problem \eqref{Poisson-Dirichlet-Brinkman-Lp}.
Let us take the trace of
the Green representation formula \eqref{Green-r-f-Brinkman-s} for $({\bf u}^0,\pi ^0)$. Since
$\gamma_+{\bf u}^0=\mathbf 0$, we obtain the equation
$$
{\mathcal V}_{\alpha}\left(\mathbf t_\alpha^+({\bf u^0},\pi^0 )\right)=\mathbf 0\ \mbox{ on } \partial\Omega,
$$
for $\mathbf t_\alpha^+({\bf u^0},\pi^0 )\in B_{p,p^*}^{s-1}(\partial\Omega)$, which by Corollary \ref{L3.1psV} has  a one-dimensional set of solutions, $\mathbf t_\alpha^+({\bf u^0},\pi^0 )=c\boldsymbol\nu $, where $c\in {\mathbb R}$.
Substituting this back into the Green representation formula \eqref{Green-r-f-Brinkman-s} we obtain
${\bf u^0}=c{\mathbf V}_{\alpha}\boldsymbol\nu=\bf 0$ in $\Omega _+$
(cf. the arguments in the proof of Lemma \ref{isom-sl}), and by the homogeneous Brinkman equation, $\pi^0$ is an arbitrary constant.
Finally, uniqueness for $0<s<1$ implies also uniqueness for $s=1$.
\hfill\end{proof}

%{\bl ***22/5/2017 SM: Add Theorem \ref{M-H-DG} version with $\mathbf h_0\in B^{s}_{p,p^*}(\partial {\Omega},{\mathbb R}^n)$?}

\subsection{\bf The Neumann problem for the Brinkman system}
\label{section2}
Using an argument similar to the one for the Robin boundary value problem for the Brinkman system in \cite{K-L-W2}, we obtain in this section the well-posedness of the Neumann problem for the linear Brinkman system,
\begin{equation}
\label{the homogeneous Neumann Brinkman system}
\left\{
\begin{array}{l}
\triangle {\textbf u} - \alpha {\textbf u} - \nabla \pi = 0,\
\mbox{ in } \ {\Omega}_+, \\
{\rm{div}}\ {\textbf u} = 0 \ \mbox{ in }\ {\Omega}_+, \\
\mathbf t^+_{\rm nt}(\textbf u,\pi)=\mathbf g_{0}\ \mbox{ on }\ \partial \Omega
%\in L_2(\partial {\Omega},\mathbb{R}^{n}).
% \\
%{M}(\nabla {\textbf u}),\ {M}({\textbf u}),\ {M}(\pi)\in L_2(\partial {\Omega}),
\end{array}
\right.
\end{equation}
in $L_p-$based Bessel potential and Besov spaces
%We extend the well-posedness result of the $L_2$-Neumann problem (see \cite[Theorem 5.3]{Shen}) to the case of $L_p$-based {Bessel potential} and Besov spaces, with {$p\in\mathcal R_0(\varepsilon)$
%\left(\frac{2(n-1)}{n+1}-\varepsilon , 2+\varepsilon \right)\cap (1,+\infty )$
for some $\varepsilon > 0$, and extend the results obtained in the case $p=2$ and for a conormal derivative given by
$\displaystyle\frac{\partial {\bf u}}{\partial n}:=
-\pi {\boldsymbol \nu }+ \displaystyle\frac{\partial {\bf u}}{\partial \boldsymbol \nu}$,
%where $({\bf u},\pi )$ satisfies the Brinkman system and $\boldsymbol \nu $ is the outward unit normal to $\Omega _{+}$, the
%well-posedness of the Neumann problem for the Brinkman system has been obtained
in \cite[Theorem 5.3]{Shen} (see also \cite[Theorem 5.5.2]{M-W} in the case $\alpha =0$). Note that the Neumann boundary condition in \eqref{the homogeneous Neumann Brinkman system} is understood in the sense of non-tangential limit almost everywhere on $\partial \Omega $.
%As in the proof of \cite[Theorem 4.1 and Theorem 4.2]{K-L-W2} and \cite[Theorem 5.3]{Shen}, the technical details rely on the mapping properties of single-layer potential and related operators for the Brinkman system.
%In Section~\ref{MP},  these results will be used to obtain well-posedness for the mixed Dirichlet-Neumann problem for the Brinkman system in $\Omega _{+}$.

%\subsection{\bf The Neumann problem for the Brinkman system in a bounded Lipschitz domain}
%As above, ${\Omega}_+\subset {\mathbb R}^n$ {$(n\geq 3)$} means a bounded Lipschitz domain with connected boundary $\partial {\Omega}$.
%{Let $H_{2;{\rm{div}}}^{\frac{3}{2}}({\Omega}, \mathbb{R}^{n})$ be the space of all divergence free vector fields in $H_{2}^{\frac{3}{2}}({\Omega}, \mathbb{R}^{n})$.}
%Next we

\comment%
{
Consider the following boundary value problem of Neumann type for the Brinkman system
\begin{equation}
\label{the homogeneous Neumann Brinkman system}
\left\{
\begin{array}{l}
\triangle {\textbf u} - \alpha {\textbf u} - \nabla \pi = 0,\
%\mbox{ in } \ {\Omega}_+, \\
{\rm{div}}\ {\textbf u} = 0 \ \mbox{ in }\ {\Omega}_+, \\
{\mathbf t^+_{\rm nt}}({\textbf u},\pi )={\mathbf g_{0}}\ \mbox{ on }\ \partial \Omega .
%\in L_2(\partial {\Omega},\mathbb{R}^{n}).
% \\
%{M}(\nabla {\textbf u}),\ {M}({\textbf u}),\ {M}(\pi)\in L_2(\partial {\Omega}),
\end{array}
\right.
\end{equation}
We will show that for ${\mathbf g_{0}}\in L_p(\partial \Omega ,\mathbb{R}^{n})$ given
{for some range of $p$ there exists a unique solution of problem \eqref{the homogeneous Neumann Brinkman system} such that ${M}({\textbf u}),\, {M}(\nabla {\textbf u}),\, {M}(\pi)\in L_p(\partial {\Omega})$. Moreover, then $({\bf u},\pi )\in B_{p,p^{*}}^{1+\frac{1}{p}}({\Omega}_+,\mathbb{R}^{n})\times  B_{p,p^{*}}^{\frac{1}{p}}({\Omega}_+)$.}
%and some choice of $p>1$, there exists a unique {$L_p$-solution} of problem \eqref{the homogeneous Neumann Brinkman system}, i.e., a unique pair $({\bf u},\pi )\in B_{p,p^{*}}^{1+\frac{1}{p}}({\Omega}_+,\mathbb{R}^{n})\times  B_{p,p^{*}}^{\frac{1}{p}}({\Omega}_+)$, which satisfies the Brinkman system everywhere in $\Omega _+$, the Neumann condition almost everywhere on $\partial \Omega $ in the sense of nontangential limit, and the conditions ${M}({\textbf u}), {M}(\nabla {\textbf u}),{M}(\pi)\in L_p(\partial {\Omega})$.}
} %\comment end

%We now extend the well-posedness result in Theorem \ref{Th2.1} to some $L_p$-based {Bessel potential} spaces, as follows.
\begin{theorem}
\label{Th2.1'}
Let ${\Omega}_+\subset \mathbb{R}^{n}$ {$(n\geq 3)$} be a bounded Lipschitz domain with connected boundary $\partial {\Omega}$. Let $\alpha \in (0,\infty )$. Then there {exists} $\epsilon>0$, such that for any $p\in\mathcal R_0(n,\epsilon)$ $($see \eqref{cases}$)$,
%$p\in \left(\frac{2(n-1)}{n+1}-\varepsilon,2+\varepsilon\right)\cap (1,+\infty )$,
and for any given datum ${\mathbf g_{0}}\in L_p(\partial {\Omega},\mathbb{R}^{n})$, the Neumann problem
\eqref{the homogeneous Neumann Brinkman system} has a unique solution
$({\textbf u},\pi)$ {such that ${M}({\textbf u}),{M}(\nabla {\textbf u}),{M}(\pi)\in L_p(\partial {\Omega})$.}
The solution can be represented by the single layer velocity and pressure potentials
\begin{align}
\label{SoNB4}
{{\bf u}={\bf V}_{\alpha }\left(\left(\frac{1}{2}\mathbb{I}+{\bf K}^{*}_{\alpha }\right)^{-1}{\mathbf g_{0}}\right), \ \pi=Q^{s}\left(\left(\frac{1}{2}\mathbb{I}+{\bf K}^{*}_{\alpha }\right)^{-1}{\mathbf g_{0}}\right).}
\end{align}
Moreover, {$(\textbf{u},\pi)\in  B_{p,p^*}^{1+\frac{1}{p}}({\Omega}_+,\mathbb{R}^{n})\times B_{p,p^*}^{\frac{1}{p}}({\Omega}_+)$,} and there exist some constants $C_M$, $C$ and $C'$ depending only on ${\Omega}_+$, $\alpha$, and $p$ such that
\begin{align}
\label{ineq-N-p}
&\|M(\nabla {\bf u})\|_{L_p(\partial \Omega )}+\|M({\bf u})\|_{L_p(\partial \Omega )}+\|M(\pi )\|_{L_p(\partial \Omega )}\leq C_M\|{\mathbf g_{0}}\|_{L_p(\partial {\Omega}, \mathbb{R}^{n})},\\
\label{continuity}
&\|{\bf u}\|_{B_{p,p^{*}}^{1+\frac{1}{p}}({\Omega}_+,\mathbb{R}^{n})} + \|\pi\|_{B_{p,p^{*}}^{\frac{1}{p}}({\Omega}_+)}\leq C\|{\mathbf g_{0}}\|_{L_p(\partial {\Omega}, \mathbb{R}^{n})},\\
&
\label{tr-continuityN}
\|\gamma_+{\bf u}\|_{H_{p}^1(\partial {\Omega},\mathbb{R}^{n})}
+\|\mathbf t^+_\alpha(\textbf u,\pi)\|_{L_{p}(\partial{\Omega},\mathbb{R}^n)}\le C'\|\mathbf g_0\|_{L_p(\partial{\Omega},\mathbb{R}^{n})}.
\end{align}
%where $p^{*} = \max\{2, p\}$.
\end{theorem}
\begin{proof}
{We use an} argument similar to that for \cite[Theorem 4.15]{Fa-Ke-Ve} $($see also \cite[Theorem 3.1, Proposition 3.3]{M-T}$)$.
By Lemma~\ref{L3.1p} there exists $\epsilon>0$ such that operator
$\frac{1}{2}{\mathbb I} + {\bf K}^{*}_{\alpha }:
L_p(\partial {\Omega},{\mathbb R}^n)\to L_p(\partial {\Omega},{\mathbb R}^n)$
is an isomorphism for $p\in\mathcal R_0(n, \epsilon)$.
Along with Lemma~\ref{nontangential-sl}, Theorem~\ref{layer-potential-properties} and Lemma~\ref{L3.6} this implies that representation
\eqref{SoNB4} gives a solution of problem \eqref{the homogeneous Neumann Brinkman system} {that belongs to the space $B_{p,p^{*}}^{1+\frac{1}{p}}({\Omega}_+,\mathbb{R}^{n})\times B_{p,p^{*}}^{\frac{1}{p}}({\Omega}_+)$} and satisfies estimates \eqref{ineq-N-p}-\eqref{tr-continuityN}.

In order to show the uniqueness assertion, we assume that $({\bf u}^0, \pi ^0)$ is a solution of the homogeneous version of \eqref{the homogeneous Neumann Brinkman system} such that ${M}({\textbf u}^0),{M}(\nabla {\textbf u}^0),{M}(\pi)^0\in L_p(\partial {\Omega})$ {and satisfies the Neumann condition almost everywhere on $\partial \Omega $ in the sense of non-tangential limit}.
Then the Green representation formula \eqref{Green-r-f-Brinkman}
%(the third Green identity) for $({\bf u}^0, \pi ^0)$
gives,
% (see also \cite[(4.120)]{M-W} in the case $\alpha =0$),
\begin{align}
\label{dbrf}
{\bf u}^0={\bf V}_{\alpha }\left(\mathbf t^+_{\rm nt}({\bf u}^0,\pi ^0)\right)
-{\bf W}_{\alpha }\left({\bf u}^{0+}_{\rm nt}\right)=-{\bf W}_{\alpha }\left({\bf u}^{0+}_{\rm nt}\right)\
\mbox{ in }\ {\Omega}_+,
\end{align}
which, combined with formulas \eqref{68-s1}, leads to the boundary integral equation
\begin{align}
\label{equation-uniq}
\left(\frac{1}{2}\mathbb{I}+{\bf K}_{\alpha }\right){\bf u}^{0+}_{\rm nt}={\bf 0}\ \mbox{ on }\ \partial {\Omega}.
\end{align}
Here ${\bf u}^{0+}_{\rm nt}\in H_p^{1}(\partial{\Omega}, \mathbb{R}^{n})$ due to Lemma \ref{nontangential-sl}(i). %and {Lemma~\ref{M-H}(ii)}.
Then invertibility of operator \eqref{F-0-snsp} in Lemma~\ref{L3.1p} implies that ${\bf u}^{0+}_{\rm nt}={\bf 0}$ on $\partial\Omega$ and thus, by \eqref{dbrf}, ${\bf u}^0={\bf 0}$ in $\Omega_+$.
{Moreover, by the homogeneous Neumann condition satisfied by $({\bf u}^0,\pi ^0)$, we obtain that $\pi^0=0$ in ${\Omega}_+$. This concludes the proof of uniqueness of the solution of the Neumann problem \eqref{the homogeneous Neumann Brinkman system}, and hence the proof of the theorem.}
\hfill\end{proof}

\comment{Taking into account that $B_{2,2}^{\frac{3}{2}}({\Omega}_+,\mathbb{R}^{n})=H_{2}^{\frac{3}{2}}({\Omega}_+,\mathbb{R}^{n})$,
we obtain the following {important} particular case of Theorem~\ref{Th2.1'}.
\begin{theorem}
\label{Th2.1}
Let $\alpha >0$.
Then for any boundary datum ${\mathbf g_{0}}\in L_2(\partial {\Omega},\mathbb{R}^{n})$ the {Neumann} problem \eqref{the homogeneous Neumann Brinkman system} has
a unique solution $({\textbf u},\pi)$
%$({\textbf u},\pi)\in H_{2}^{\frac{3}{2}}({\Omega}_+,\mathbb{R}^{n})\times H_2^{\frac{1}{2}}({\Omega}_+)$
such that ${M}({\textbf u}),{M}(\nabla {\textbf u}),{M}(\pi)\in L_2(\partial {\Omega})$.
Moreover, the solution can be represented by the single layer velocity and pressure potentials \eqref{SoNB4}, and there exist constants $C'$ and $C$ depending only on ${\Omega}_+$ and $\alpha$ such that
\begin{align}
\label{ineq-N-2}
&\|M(\nabla {\bf u})\|_{L_2(\partial \Omega )}+\|M({\bf u})\|_{L_2(\partial \Omega )}
+\|M(\pi )\|_{L_2(\partial \Omega )}\leq C'\|{\mathbf g_{0}}\|_{L_2(\partial {\Omega}, \mathbb{R}^{n})},\\
\label{estimate}
&\|{\bf u}\|_{H_{2}^{\frac{3}{2}}({\Omega}_+, \mathbb{R}^{n})}+ \|\pi\|_{H_2^{\frac{1}{2}}({\Omega}_+,\mathbb{R}^{n})}\leq C\|{\mathbf g_{0}}\|_{L_2(\partial {\Omega}, \mathbb{R}^{n})},
\end{align}
\end{theorem}
\begin{proof}
First, we note that the uniqueness follows immediately from the Green identity \eqref{Green formula}. \comment{yields that
\begin{equation}
\label{oNB1}
2\left\langle \mathbb{E}({\bf u}^{(0)}),\mathbb{E}({\bf u}^{(0)})\right\rangle _{{\Omega}_+}+\alpha \left\langle {\bf u}^{(0)},{\bf u}^{(0)}\right\rangle _{{\Omega}_+}=\left\langle{\bf t}_{\alpha }^{+}({\bf u}^{(0)},\pi^{(0)}),{\gamma}\ {\bf u}^{(0)}\right\rangle _{\partial {\Omega}}=0,
\end{equation}
and hence that
${\bf u}^{(0)}={\bf 0}$ in ${\Omega}_+.$
Moreover, the homogeneous Neumann condition implies that $\pi^{(0)}=0$ in ${\Omega}_+$.}

On the other hand, the isomorphism property of {operator \eqref{F-0-s}} in Lemma~\ref{L3.1} and Lemma \ref{nontangential-sl} imply that the single layer velocity and pressure potentials ${\bf u}$ and $\pi $ given by \eqref{SoNB4} determine the unique solution of the Neumann problem \eqref{the homogeneous Neumann Brinkman system} in the space $H_{2}^{\frac{3}{2}}({\Omega}_+, \mathbb{R}^{n})\times H_2^{\frac{1}{2}}({\Omega}_+)$, which satisfies the conditions {${M}(\nabla \textbf{u}),{M}(\textbf{u}),{M}(\pi)\in L_2(\partial {\Omega})$} and inequality \eqref{ineq-N-2}. In addition, {continuity} of the operators in \eqref{SoNB4} leads to inequality \eqref{estimate}.
\hfill\end{proof}
} %\comment end
{Having in view Theorem \ref{Th2.1'}, we are now able to consider the Poisson-Neumann problem for the Brinkman system,
\begin{equation}
\label{Poisson-Neumann-Brinkman-Lp}
\left\{
\begin{array}{l}
\triangle {\bf u} - \alpha {\bf u} - \nabla \pi =\mathbf f,\
{\rm{div}}\ {\bf u} = 0\ \mbox{ in }\  {\Omega}_+\\
\mathbf t_\alpha^+({\bf u},\pi)={\mathbf g_{0}}\ \mbox{ on }\ \partial \Omega
\end{array}
\right.
\end{equation}
with the Neumann datum for the canonical conormal derivative $\mathbf t_\alpha^+({\bf u},\pi)$ {(see also \cite[Theorem 10.6.4]{M-T} for $\alpha =0$)}.
\begin{theorem}
\label{M-H-NG}
Let ${\Omega}_+\subset \mathbb{R}^{n}$ {$(n\geq 3)$} be a bounded Lipschitz domain with connected boundary $\partial {\Omega}$. {Let $\alpha \in (0,\infty )$}. Then there exists $\varepsilon=\varepsilon(\partial\Omega)>0$ such that for any
$p\in\mathcal R_0(n,\epsilon)$ $($cf. \eqref{cases}$)$, the Neumann problem \eqref{Poisson-Neumann-Brinkman-Lp} with $\mathbf f\in L_p({\Omega}_+,{\mathbb R}^3)$ and $\mathbf g_0\in L_{p}(\partial {\Omega},\mathbb{R}^{n})$ has a unique solution
$(\textbf{u},\pi)\in B_{p,p^*}^{1+\frac{1}{p}}({\Omega}_+,\mathbb{R}^{n})\times B_{p,p^*}^{\frac{1}{p}}({\Omega}_+)$, where $p^{*} = \max\{2, p\}$. In addition, there exists a constant $C=C(p,\Omega _+)>0$ such that
\begin{align*}
&\|{\bf u}\|_{B_{p,p^*}^{1+\frac{1}{p}}({\Omega}_+,\mathbb{R}^{n})}
+\|\pi\|_{B_{p,p^*}^{\frac{1}{p}}({\Omega}_+)}\le C(\|\mathbf f\|_{L_p(\Omega_+,\mathbb{R}^{n})}+\| \mathbf g_0\|_{L_p(\partial{\Omega},\mathbb{R}^{n})}),\\
&\|\gamma_+{\bf u}\|_{H_{p}^1(\partial {\Omega},\mathbb{R}^{n})}\le C(\|\mathbf f\|_{L_p(\Omega_+,\mathbb{R}^{n})}+ \mathbf g_0\|_{L_p(\partial{\Omega},\mathbb{R}^{n})}).
\end{align*}
Moreover, if $\mathbf f=\mathbf 0$, then {${M}({\textbf u}), {M}(\nabla {\textbf u}),{M}(\pi)\in L_p(\partial {\Omega})$} and there exists a constant $C_M>0$ such that
$$
\|M({\bf u})\|_{L_p(\partial {\Omega})}+\|M(\nabla{\bf u})\|_{L_p(\partial {\Omega})}
+\|M(\pi)\|_{L_p(\partial {\Omega})}\le C_M\| \mathbf g_0\|_{L_p(\partial {\Omega},\mathbb{R}^{n})}.
$$
\end{theorem}
\begin{proof}
If $\mathbf f=\mathbf 0$, there exists a solution of problem \eqref{Poisson-Neumann-Brinkman-Lp} given by the solution of the corresponding problem \eqref{the homogeneous Neumann Brinkman system} with the non-tangential conormal derivative in the Neumann condition, whose existence is provided by Theorem \ref{Th2.1'} together with the asserted estimate. Here we rely also on the equivalence of the conormal derivatives, ${\bf t}_\alpha^+({\bf u},\pi)\!=\!{\bf t}^+_{\rm nt}({\textbf u},\pi )$, due to Theorem \ref{2.13}.

If $\mathbf f\not=\mathbf 0$, we will look for a solution of problem  \eqref{Poisson-Neumann-Brinkman-Lp} in the form
\begin{align}
\label{Poisson-mixed-2-LpD}
{\bf u}={\mathbf N}_{\alpha ;{\Omega}_+}{\mathbf f}+{\bf v},\ \pi ={\mathcal Q}_{{\Omega}_+}{\mathbf f}+q,
\end{align}
where the Newtonian velocity and pressure potentials
${\mathbf N}_{\alpha ;{\Omega}_+}{\mathbf f}$ and ${\mathcal Q}_{\Omega_+}{\mathbf f}$ are defined by \eqref{Newtonian-1a-Lp}. According to Remark~\ref{gammaN}, we obtain the relations
% they satisfy equations
\begin{align*}
%\label{Newtonian-2-Lp}
&\triangle {\mathbf N}_{\alpha ;{\Omega_+}}{\mathbf f}-\alpha {\mathbf N}_{\alpha ;{\Omega_+}}{\mathbf f}-\nabla {\mathcal Q}_{\Omega+}{\mathbf f}={\mathbf f},\
{\rm{div}}\ {\mathbf N}_{\alpha ;{\Omega_\pm}}{\mathbf f}=0\ \mbox{ in }\ {\Omega}_+,\\
&(\mathbf N_{\alpha;\Omega_+}{\mathbf f},{\mathcal Q}_{\Omega_+}{\mathbf f})\in {B_{p,p^*}^{2}({\Omega}_+,\mathbb{R}^{n})}\times B_{p,p^*}^{1}({\Omega}_+),
\ {\gamma}_{+}(\mathbf N_{\alpha;{\Omega}_+}{\mathbf f})\in H^1_{p;\boldsymbol\nu}(\partial {\Omega},{\mathbb R}^n),
\ {\mathbf t}^+\left({\mathbf N}_{\alpha ;\Omega_+}{\mathbf f},
{\mathcal Q}_{\Omega_\pm}{\mathbf f}\right)
\in L_p(\partial {\Omega},{\mathbb R}^{n}).
\end{align*}
Then problem
\eqref{Poisson-Neumann-Brinkman-Lp} reduces to the
problem for the corresponding homogeneous Brinkman system,
\begin{equation}
\label{Newtonian-4-LpD}
\left\{
\begin{array}{l}
\triangle {\bf v} - \alpha {\bf v} - \nabla q= {\bf 0},\
{\rm{div}}\ {\bf v}= 0\ \mbox{ in }\  {\Omega}_+,\\
\mathbf t_\alpha^+({\bf u},\pi)={\mathbf g_{00}}\ \mbox{ on }\ \partial \Omega ,
\end{array}
\right.
\end{equation}
where
$
\mathbf g_{00}:=\mathbf g_0-{\mathbf t}_\alpha^+\left(\mathbf N_{\alpha ;\Omega_\pm}{\mathbf f}_\pm,
{\mathcal Q}_{\Omega_\pm}{\mathbf f}_\pm\right)\in L_{p}(\partial\Omega,{\mathbb R}^n),
$
%\end{align}
already discussed in the first paragraph of the proof.
Therefore, there exists a
solution
$({\bf u},\pi)\in {B_{p,p^*}^{1+\frac{1}{p}}({\Omega}_+,\mathbb{R}^{n})}\times B_{p,p^*}^{\frac{1}{p}}({\Omega}_+)$
of the Poisson problem \eqref{Poisson-Neumann-Brinkman-Lp}, which satisfies all the asserted estimates.

Let us prove uniqueness of the solution to the Poisson problem  \eqref{Poisson-Neumann-Brinkman-Lp}.
Indeed, let us consider a solution
$({\bf u}^0,\pi^0)\in B_{p,p^*}^{1+\frac{1}{p}}({\Omega},\mathbb{R}^{3})\times B_{p,p^*}^{\frac{1}{p}}({\Omega})$
of the homogeneous version of problem \eqref{Poisson-Neumann-Brinkman-Lp}.
Let us take the trace of the Green representation formula \eqref{Green-r-f-Brinkman-s} for $({\bf u}^0,\pi^0)$, considered for any $s\in (0,1)$.
Since $\mathbf t_\alpha^+({\bf u},\pi)=\mathbf 0$, we obtain the equation
$$
\gamma_+{\bf u}^0=\frac{1}{2}\gamma_+{\bf u}^0-{\bf K}_{\alpha }\gamma_+{\bf u}^0\ \mbox{ on } \partial\Omega,
$$
with the unknown $\gamma_+{\bf u}^0\in B_{p,p^*}^{s}(\partial\Omega ,{\mathbb R}^n)$, which, by Corollary \ref{L3.1ps}, has only the trivial solution.
Substituting this back to the Green representation formula \eqref{Green-r-f-Brinkman-s} we obtain
${\bf u^0}={\bf 0}$  in $\Omega _+$.
Then the Brinkman system implies $\pi^0=c\in {\mathbb R}$, and taking again into account that $\mathbf t_\alpha^+({\bf u},\pi)=\mathbf 0$, we obtain $\pi^0=0$ in $\Omega _+$, as asserted.
\hfill\end{proof}

%***22/5/2017 SM: Add Theorem \ref{M-H-NG} versions with $\mathbf g_0\in H^{s-1}_{p}(\partial {\Omega},{\mathbb R}^n)$ and $\mathbf g_0\in B^{s-1}_{p,p^*}(\partial {\Omega},{\mathbb R}^n)$, $s\in(0,1)$?}}

\section{The mixed Dirichlet-Neumann problem for the Brinkman system}
\label{MP}
In this section we show the well-posedness of the mixed {{\em Dirichlet-Neumann} boundary value problem} for the Brinkman system
\begin{align}
\label{mixed homogeneous Brinkman system}
\left\{
\begin{array}{lll}
\triangle {\bf u} - \alpha {\bf u} - \nabla \pi = {\bf 0},\
{\rm{div}}\ {\bf u} = 0\ \mbox{ in }\  {\Omega}_+,\\
%\label{MD}
\mathbf u^+_{\rm nt}|_{S_D}={\mathbf h_{0}},
%\in H_2^1(S_{D},\mathbb{R}^{n})
\\
%\label{MN}
\mathbf t^+_{\rm nt}({\bf u},\pi )|_{S_N}=\mathbf g_{0},
\end{array}\right.
\end{align}
{on a bounded, {\it creased} Lipschitz domain ${\Omega}_+\subset \mathbb{R}^{n}$ $(n\geq 3)$ with connected boundary $\partial \Omega $,} which is decomposed into two disjoint admissible patches $S_{D}$ and $S_{N}$ (see Definition \ref{creased domain}), $\cdot |_{S_D}$ is the operator of restriction from $H_p^s(\partial {\Omega},\mathbb{R}^{n})$ to $H_p^s(S_{D},\mathbb{R}^{n})$, and $\cdot |_{S_N}$ is defined similarly.
We {show} that for ${\mathbf h_{0}}\in H_p^1(S_{D},\mathbb{R}^{n})$ and ${\mathbf g_{0}}\in L_p(S_{N},\mathbb{R}^{n})$ given and for some range of $p$, there exists a unique
{solution  $({\bf u},\pi)$
of the mixed problem \eqref{mixed homogeneous Brinkman system},
such that ${M}({\textbf u}),{M}(\nabla {\textbf u}),{M}(\pi)\in L_p(\partial {\Omega})$},
%{\rd {\it $L_p$-solution}} ,
%, i.e., a unique pair $({\bf u},\pi )\in C^{\infty }(\Omega _+,{\mathbb R}^n)\times C^{\infty }(\Omega _+)$ {\rd {\it such that ${\bf u}$ and $\pi $ satisfy the Brinkman system in $\Omega _+$, {${M}({\textbf u}),{M}(\nabla {\textbf u}),{M}(\pi)\in L_p(\partial {\Omega})$,
%there exist the non-tangential limits of ${\bf u}$, $\nabla {\bf u}$ and $\pi $ at almost all points of $\partial \Omega $,
and the Dirichlet and Neumann boundary conditions in \eqref{mixed homogeneous Brinkman system} are satisfied in the sense of non-tangential {limits} at almost all points of $S_D$ and $S_N$, respectively}.} Moreover, we will show that $({\bf u},\pi )\in B_{p,p^{*}}^{1+\frac{1}{p}}({\Omega}_+,\mathbb{R}^{n})\times  B_{p,p^{*}}^{\frac{1}{p}}({\Omega}_+)$.

{We consider also a counterpart mixed problem
\begin{align}
\label{mixed homogeneous Brinkman systemG}
\left\{
\begin{array}{lll}
\triangle {\bf u} - \alpha {\bf u} - \nabla \pi = {\bf 0},\
{\rm{div}}\ {\bf u} = 0\ \mbox{ in }\  {\Omega}_+\\
%\label{MDG}
\gamma_+\mathbf u|_{S_D}={\mathbf h_{0}},
\\
%\label{MNG}
\mathbf t^+_{\alpha}({\bf u},\pi )|_{S_N}=\mathbf g_{0},
\end{array}\right.
\end{align}
where, unlike the mixed problem setting \eqref{mixed homogeneous Brinkman system}, the trace is considered in the Gagliardo sense and the conormal derivative in the canonical sense.
We will show that for ${\mathbf h_{0}}\in H_p^1(S_{D},\mathbb{R}^{n})$ and ${\mathbf g_{0}}\in L_p(S_{N},\mathbb{R}^{n})$ given and for some range of $p$, there exists a unique solution $({\bf u},\pi )\in B_{p,p^{*}}^{1+\frac{1}{p}}({\Omega}_+,\mathbb{R}^{n})\times  B_{p,p^{*}}^{\frac{1}{p}}({\Omega}_+)$ of problem \eqref{mixed homogeneous Brinkman systemG}.
Moreover, we will obtain that ${M}({\textbf u}),\, {M}(\nabla {\textbf u}),\, {M}(\pi)\in L_p(\partial {\Omega})$.

The corresponding mixed problems for the Poisson-Brinkman system, i.e., with non-zero right hand side of the Brinkman system, will be also considered.}
%}

\subsection{\bf Creased Lipschitz domains}
\label{Creased}
%Next we present the definition of admissible patch and of creased Lipschitz domain as in, e.g., \cite{B-M-M-W, M-M}.
Next, we recall the definition of admissible patch (cf., e.g., \cite[Definition 2.1]{M-M}, \cite{B-M-M-W}).
\begin{defi}
\label{Lipschity-domain-1}
Let ${\Omega} \subset \mathbb{R}^{n}$ $(n\geq 3)$ be a Lipschitz domain. Let $S$ be an open set of $\partial {\Omega}$, such that for any $x_{0}\in \partial S$ there exists a new orthogonal system obtained from the original one by a rigid motion with $x_{0}$ as the origin and with the property that one can find a cube $Q = Q_{1}\times Q_{2}\times \cdots \times Q_{n}\subset \mathbb{R}^{n}$ centered at $0$ and two Lipschitz functions
\begin{align}
\left\{\begin{array}{ll}
\Phi :Q':= Q_{1}\times \ldots \times Q_{n-1}\to Q_{n}\,, & \Phi(0)=0,\\
\Psi :Q'':= Q_{2}\times \ldots \times Q_{n-1}\to Q_{1}\,, & \Psi(0)=0,
\end{array}\right.
\end{align}
such that
\begin{align}\label{e3}
&S\cap Q=\left\{(x', \Phi(x')) : \ x' \in Q', \ \Psi(x'') \le x_{1}\right\},\nonumber\\
&\left(\partial {\Omega\setminus} \overline{S}\right)\cap Q=\left\{(x',\Phi(x')):\ x' \in Q', \ \Psi(x'')\geq x_{1}\right\},\\
&\partial S \cap Q=\left\{\left(\Psi(x''),x'',\Phi(\Psi(x''), x'')\right): x''\in Q''\right\}.\nonumber
\end{align}
Such a set $S$ is called an \textit{admissible patch} of $\partial {\Omega}$.
\end{defi}
%\end{definition}
Definition \ref{Lipschity-domain-1} shows that if $S\subset \partial {\Omega}$ is an admissible patch then $\partial {\Omega}\setminus \overline{S}$ is also an admissible patch (cf., e.g., \cite{M-M}).
Next, we recall the definition of a creased Lipschitz graph domain (cf. \cite[Definition 2.2]{M-M}).
%\begin{definition}
\begin{defi}
\label{special creased domain}
Let ${\Omega}\subset \mathbb{R}^{n}$ $(n\geq 3)$ be an open, connected set. Suppose that $S_{D},S_{N}\subset \partial {\Omega}$ are two non-empty, disjoint admissible patches such that $\overline{S_{D}}\cap \overline{S_{N}}=\partial {S_{D}}=\partial {S_{N}}$ and $\overline{S_{D}}\cup \overline{S_{N}}=\partial {\Omega}$.
The set ${\Omega} $ is a {\it creased Lipschitz graph domain} if the following conditions are satisfied:
\begin{enumerate}
\item[$(a)$]
There exists a Lipschitz function $\phi :\mathbb{R}^{n-1} \rightarrow \mathbb{R}$ such that $${\Omega}=\left\{(x',x_{n}) \in \mathbb{R}^{n}:x_{n}>\phi(x')\right\}.$$
\item[$(b)$]
There exists a Lipschitz function $\Psi :\mathbb{R}^{n-2} \rightarrow \mathbb{R}$ such that
\begin{align}
&S_{N}= \{(x_{1},x",x_{n}) \in \mathbb{R}^{n} : x_{1}> \Psi(x") \} \cap \partial {\Omega},\\
&S_{D}= \{(x_{1},x",x_{n}) \in \mathbb{R}^{n} : x_{1}< \Psi(x") \} \cap \partial {\Omega}.
\end{align}
\item[$(c)$]
There exist some constants $\delta_{D}, \delta_{N} \ge 0$, $\delta_{D} + \delta_{N}>0$ with the property that
\begin{equation}
\frac{\partial \phi}{\partial x_{1} } \ge \delta_{N}\ \mbox{ a.e. on }\ S_{N}, \
\frac{\partial \phi}{\partial x_{1} } \le -\delta_{D}\ \mbox{ a.e. on }\ S_{D}.
\end{equation}
\end{enumerate}
\end{defi}

%\end{definition}
Let us now refer to a creased bounded Lipschitz domain (cf. \cite[Definition 2.3]{M-M}).
%\begin{definition}
\begin{defi}
%\end{definition}
\label{creased domain}
Assume that ${\Omega}\subset {\mathbb R}^n$ is a bounded Lipschitz domain with connected boundary $\partial \Omega$, and that $S_{D},S_{N}\subset \partial {\Omega}$ are two non-empty, disjoint admissible patches such that $\overline{S_{D}}\cap \overline{S_{N}}=\partial {S_{D}}=\partial {S_{N}}$ and $\overline{S_{D}}\cup \overline{S_{N}}=\partial {\Omega}$. Then ${\Omega}$ is {\it creased} if
\begin{itemize}
\item[$(a)$]
There exist $m\in {\mathbb N}$, $a>0$ and $z_i\in \partial {\Omega}$, $i=1,\ldots,m$, such that $\partial {\Omega\subset}\cup_{i=1}^mB_a(z_i)$, where $B_a(z_i)$ is the ball of radius $a$ and center at $z_i$.
\item[$(b)$]
For any point $z_i$, $i=1,\ldots ,m$, there exist a coordinate system $\{x_1,\ldots ,x_n\}$ with origin at $z_i$ and a Lipschitz function $\phi _i$ from ${\mathbb R}^{n-1}$ to ${\mathbb R}$ such that the set ${\Omega}_i:=\{(x',x_n)\in {\mathbb R}^n:x_n >\phi _i(x')\}$, whose boundary $\partial {\Omega}_i$ admits the decomposition $\partial {\Omega}_i=\overline{S_{D_i}}\cup \overline{S_{N_i}}$, is a creased Lipschitz graph domain in the sense of Definition $\ref{special creased domain}$, and
\begin{equation}
\label{creased-Lip}
{\Omega}\cap B_{2a}(z_i)={\Omega}_i\cap B_{2a}(z_i),\ S_D\cap B_{2a}(z_i)=S_{D_i}\cap B_{2a}(z_i),\ S_{N}\cap B_{2a}(z_i)=S_{N_i}\cap B_{2a}(z_i).
\end{equation}
\end{itemize}
\end{defi}
%\end{definition}
The geometric meaning of {Definitions \ref{special creased domain} and \ref{creased domain}} is that $S_D$ and $S_N$ are separated by a Lipschitz interface ($\overline{S_D}\cap \overline{S_N}$ is a {\it creased collision} manifold for $\mathfrak D$) and that $S_D$ and $S_N$ meet at an angle which is strictly less than $\pi $ (cf., e.g., \cite{Br,M-M}). A main property of a (bounded or {graph}) creased Lipschitz domain is the existence of a function $\boldsymbol{\varphi}\in C^{\infty }({\overline{\Omega}},{\mathbb R}^n)$ and of a constant $\delta >0$ such that
\begin{align}
\label{creased-sign}
\boldsymbol{\varphi}\cdot \boldsymbol\nu >\delta \mbox{ a.e. on } S_N,\ \ \boldsymbol{\varphi}\cdot \boldsymbol\nu <-\delta \mbox{ a.e. on } S_D,
\end{align}
i.e., the scalar product $\boldsymbol{\varphi}\cdot \boldsymbol\nu $, between $\boldsymbol{\varphi}$ and the unit normal $\boldsymbol\nu $, changes the sign when moving from $S_D$ to $S_N$ (cf., e.g., \cite[(1.122)]{B-M}, \cite[(2.2)]{B-M-M-W}). For such a domain, Brown \cite{Br} showed that the mixed problem for the Laplace equation has a unique solution whose gradient belongs to $L_2(\partial \mathfrak D)$ when the Dirichlet datum belongs to $H^1_2(S_D)$ and the Neumann datum to $L_2(S_N)$. For the same class of domains, well-posedness of the mixed problem for the Laplace equation in a range of $L_p-$based spaces has been obtained in \cite{M-M}.

\subsection{\bf Mixed Dirichlet-Neumann problem for the Brinkman system with {boundary} data in $L_2$-{based} spaces}
{Mitrea and Mitrea in \cite{M-M} have proved sharp well-posedness results for the Poisson problem for the Laplace operator with mixed boundary conditions of Dirichlet and Neumann type on bounded creased Lipschitz {domains in ${\mathbb R}^n$ ($n\geq 3$),} whose boundaries satisfy a geometric condition, and with data in Sobolev and Besov spaces.
Brown et al. in \cite[Theorem 1.1]{B-M-M-W} have obtained the well-posedness result for the mixed Dirichlet-Neumann problem for the Stokes system with boundary data in $L_2$-based spaces on creased Lipschitz domains in ${\mathbb R}^n$ $(n\geq 3)$, by reducing such a boundary value problem to the analysis of a boundary integral equation (see also the references therein).}
Well-posedness of the mixed {\em Dirichlet-Robin} problem for the Brinkman system in a creased Lipschitz domain with boundary data in $L_2$-{based} spaces has been recently proved in \cite[Theorem 6.1]{K-L-W2}.
Using the main ideas of that proof,  we show in this section well-posedness of the mixed {{\em Dirichlet-Neumann} boundary value problem} for the Brinkman system in $L_2$-based {Bessel potential} spaces defined on a bounded, creased Lipschitz domain ${\Omega}_+$.
% in $\mathbb{R}^{n}$ $(n\geq 3)$.
\comment{
Therefore, we assume that $S_{D}$ and $S_{N}$ are two disjoint admissible patches of $\partial {\Omega}$ (see Definition \ref{creased domain}), and consider the following mixed boundary value problem
\begin{equation}
\label{mixed homogeneous Brinkman system}
\left\{
\begin{array}{l}
\triangle {\bf u} - \alpha {\bf u} - \nabla \pi = {\bf 0},\
{\rm{div}}\ {\bf u} = 0\ \mbox{ in }\  {\Omega}_+\\
\left({\gamma}_{+}{\bf u}\right)|_{S_D}={\mathbf h_{0}}\in H_2^1(S_{D},\mathbb{R}^{n}) \\
\left(\mathbf t_{\alpha}^{+}({\bf u},\pi )\right)|_{S_N}=\mathbf g_{0}\in L_2(S_{N},\mathbb{R}^{n}), \\
{M}(\nabla {\bf u}),\ {M}({\bf u}),\ {M}(\pi)\in L_2(\partial {\Omega}),
\end{array}
\right.
\end{equation}
where $\cdot |_{S_D}$ is the operator of restriction from $H_2^1(\partial {\Omega},\mathbb{R}^{n})$ to $H_2^1(S_{D},\mathbb{R}^{n})$, and $\cdot |_{S_N}$ is the operator of restriction from $L_2(\partial {\Omega},\mathbb{R}^{n})$ to $L_2(S_{N},\mathbb{R}^{n})$. In addition, the trace and conormal derivative operators in \eqref{mixed homogeneous Brinkman system} are considered in the sense of non-tangential limits.
} %\comment end
%In order to show the following result we use the main ideas of the proof of \cite[Theorem 6.1]{K-L-W2}.
\begin{theorem}
\label{Th2.6}
Assume that ${\Omega}_+\subset \mathbb{R}^{n}$ $(n\geq 3)$
is a bounded, creased Lipschitz domain with connected boundary $\partial {\Omega}$,
%, and that $\partial {\Omega}$ is
which is {decomposed} into two disjoint admissible patches $S_{D}$ and $S_{N}$. Then the mixed
%Dirichlet-Neumann boundary value
problem \eqref{mixed homogeneous Brinkman system} with given data
$({\mathbf h_{0}},\mathbf g_{0})\in H_2^1(S_{D},\mathbb{R}^{n})\times L_2(S_{N},\mathbb{R}^{n})$ has a unique {solution $({\bf u},\pi)$ such that ${M}({\textbf u}),{M}(\nabla {\textbf u}),{M}(\pi)\in L_2(\partial {\Omega})$}.
%, {\rd such that there exist the non-tangential limits of ${\bf u}$, $\nabla {\bf u}$ and $\pi $ at almost all points of $\partial \Omega $, ${M}({\textbf u}),{M}(\nabla {\textbf u}),{M}(\pi)\in L_2(\partial {\Omega})$, and ${\bf u}$ and $\pi $ satisfy the Dirichlet and Neumann boundary conditions in the sense of non-tangential limit at almost all points of $S_D$ and $S_N$, respectively.}
%for the Brinkman system is well-posed, i.e.
Moreover,
%\begin{equation}
$
({\bf u},\pi)\in H_{2}^{\frac{3}{2}}({\Omega}_+,\mathbb{R}^{n})\times H_2^{\frac{1}{2}}({\Omega}_+),
$
%\end{equation}
and there exist some constants $C_M$ and $C$ depending only on $S_{D}$, $S_{N}$ and $\alpha$ such that
%for some constant $C\equiv C(\alpha, S_{D},S_{N})>0$ we have}
\begin{align}
\label{ineq-M-2}
&\|M(\nabla {\bf u})\|_{L_2(\partial \Omega )}+\|M({\bf u})\|_{L_2(\partial \Omega )}
+\|M(\pi )\|_{L_2(\partial \Omega )}\leq C_M\left(\|{\mathbf h_{0}}\|_{H_2^1(S_{D},\mathbb{R}^{n})}+ \|\mathbf g_{0}\|_{L_2(S_{N}, \mathbb{R}^{n})}\right),\\
\label{estimate-mixed}
&\|{\bf u}\|_{H_{2}^{\frac{3}{2}}({\Omega}_+, \mathbb{R}^{n})} + \|\pi\|_{H_2^{\frac{1}{2}}({\Omega}_+)}\leq C\left(\|{\mathbf h_{0}}\|_{H_2^1(S_{D},\mathbb{R}^{n})}+ \|\mathbf g_{0}\|_{L_2(S_{N}, \mathbb{R}^{n})}\right).
\end{align}
\end{theorem}
\begin{proof}
First, we note that if a couple $({\bf u},\pi)$ satisfies the Brinkman system \eqref{mixed homogeneous Brinkman system} and the conditions ${M}({\textbf u}),{M}(\nabla {\textbf u}),{M}(\pi)\in L_2(\partial {\Omega})$, then, taking into account that $B_{2,2}^{\frac{3}{2}}({\Omega}_+,\mathbb{R}^{n})=H_{2}^{\frac{3}{2}}({\Omega}_+,\mathbb{R}^{n})$, $B_{2,2}^{\frac{1}{2}}({\Omega}_+)=H_{2}^{\frac{1}{2}}({\Omega}_+)$,
Theorem~\ref{M-H}(ii) implies
that $({\bf u},\pi)\in \mathfrak{H}_{2,\rm div}^{\frac{3}{2},t}({\Omega},\mathcal{L}_{\alpha})$ for any
$t\ge -\frac{1}{2}$, while $\gamma_+{\bf u}={\bf u}^+_{\rm nt}$ and $\mathbf t_{\alpha}^{+}({\bf u},\pi)=\mathbf t^+_{\rm nt}({\bf u},\pi)$ by Theorems \ref{trace-equivalence-L} and \ref{2.13}.

Let us show that the mixed boundary value problem \eqref{mixed homogeneous Brinkman system} has at most one $L_2$-solution.
Indeed, if a couple $({\bf u}^{(0)},\pi^{(0)})$ satisfies the homogeneous problem associated to \eqref{mixed homogeneous Brinkman system}, and moreover $({\bf u}^{(0)},\pi^{(0)})\in \mathfrak{H}_{2,\rm div}^{\frac{3}{2},0}({\Omega},\mathcal{L}_{\alpha})$, then by the first Green identity \eqref{Green formula}, we obtain the equality
\begin{equation}
\label{DN-unique}
\left\langle \mathbf t_{\alpha}^{+}({\bf u}^{(0)},\pi^{(0)}),{\gamma}_{+}{\bf u}^{(0)}\right \rangle_{\partial {\Omega}}=2\left\langle \mathbb{E}({\bf u}^{(0)}),\mathbb{E}({\bf u}^{(0)}) \right\rangle_{{\Omega}_+} + \alpha \left\langle {\bf u}^{(0)},{\bf u}^{(0)}\right\rangle_{{\Omega}_+},
\end{equation}
where the left-hand side vanishes, due to
%the decomposition of $\partial {\Omega}$ into the admissible patches $S_{D}$ and $S_{N}$ and
the homogeneous boundary conditions satisfied by $\gamma_+{\bf u}^{(0)}={\bf u}^{(0)+}_{\rm nt}$ and
$\mathbf t_{\alpha}^{+}({\bf u}^{(0)},\pi^{(0)})=\mathbf t^+_{\rm nt}({\bf u}^{(0)},\pi^{(0)})$ on $S_{D}$ and $S_{N}$, respectively. Then by \eqref{DN-unique} we immediately obtain that ${\bf u}^{(0)}={\bf 0}$ and $\pi^{(0)} = 0$ in ${\Omega}_+$.

Next, we consider the operator
\begin{align}
\label{isomorphism-D-N}
{\mathcal S}_{\alpha }:L_2(\partial {\Omega},{\mathbb R}^n)\to H^1_2(S_D,{\mathbb R}^n)\times
L_2(S_N,{\mathbb R}^n),\ {\mathcal S}_{\alpha }\Psi :=\left(\left({\mathcal V}_{\alpha }\boldsymbol \Psi \right)\big|_{{S_D}},\left(\left(\frac{1}{2}{\mathbb I}+{{\bf K}}_{\alpha }^*\right)\boldsymbol \Psi \right)\Big|_{{S_N}}\right)
\end{align}
(cf. \cite[(6.6)-(6.8)]{K-L-W2}), and show that this is an isomorphism, {which will yield} the
well-posedness of the mixed problem \eqref{mixed homogeneous Brinkman system}. To this {end}, we note that
${\mathcal S}_{\alpha }$ can be written {as}
%\begin{equation}
%\label{isomorphism-D-R-1}
$
{\mathcal S}_{\alpha }={\mathcal S}_{0}+{\mathcal S}_{\alpha ;0},
$
%\end{equation}
where
\begin{align}
\label{isomorphism-D-N-2}
&{\mathcal S}_{0}:L_2(\partial {\Omega},{\mathbb R}^n)\!\to \!H^1_2(S_D,{\mathbb R}^n)\!\times \!L_2(S_N,{\mathbb R}^n),\,
{\mathcal S}_{0}\Psi :=\left(\!\left({\mathcal V}_{0}\boldsymbol \Psi
\right)\big|_{{S_D}},\left(\left(\frac{1}{2}{\mathbb I}\!+\!{{\bf K}}_{0}^*\right)\boldsymbol \Psi \right)\Big|_{{S_N}}\right),\\
\label{isomorphism-D-N-2c}
&{\mathcal S}_{\alpha ;0}:L_2(\partial {\Omega},{\mathbb R}^n)\!\to \!H^1_2(S_D,{\mathbb R}^n)\!\times\!L_2(S_N,{\mathbb R}^n),\, {\mathcal S}_{\alpha ;0}\boldsymbol \Psi \!:=\!\left(\!\left({\mathcal V}_{\alpha ;0}\Psi \!\right)\big|_{{S_D}},\left({{\bf K}}_{\alpha ;0}^*\boldsymbol \Psi \right)\big|_{{S_N}}\right).
\end{align}
{Here} $\mathcal{V}_{\alpha; 0}:L_2(\partial {\Omega},{\mathbb R}^n)\to H^1_2(\partial {\Omega},{\mathbb R}^n)$ and ${\bf K}^{*}_{\alpha;0}:L_2(\partial {\Omega},{\mathbb R}^n)\to L_2(\partial {\Omega},{\mathbb R}^n)$ are the complementary layer potential operators defined {as}
\begin{equation}
\label{compl-adj}
\mathcal{V}_{\alpha; 0}\boldsymbol \Psi := \mathcal{V}_{\alpha }\boldsymbol \Psi - \mathcal{V}_{0}\boldsymbol \Psi \ \mbox{ and }\ {\bf K}^{*}_{\alpha ;0}\boldsymbol \Psi :={\bf K}^{*}_{\alpha }\Psi -{\bf K}_{0}^{*}\boldsymbol \Psi .
\end{equation}
{The} operator ${\mathcal S}_{0}$ defined in \eqref{isomorphism-D-N-2} is an isomorphism and this property is equivalent with the well-posedness result of the mixed Dirichlet-Neumann problem for the Stokes system on a creased Lipschitz domain with Dirichlet and Neumann boundary data in $L_2$-{based} spaces (cf. the proof of \cite[Theorem 6.3]{B-M-M-W}), when the BVP solution is looked for in the form of {the Stokes} single layer potential. In addition, the continuity of the restriction operators from $H^1_2(\partial {\Omega},{\mathbb R}^n)$ to $H^1_2(S_D,{\mathbb R}^n)$ and from $L_2(\partial {\Omega},{\mathbb R}^n)$ to $L_2(S_N,{\mathbb R}^n)$, respectively, as well as the compactness of the complementary operators in \eqref{compl-adj} (cf. \cite[Theorem 3.4]{K-L-W}) imply that the operator ${\mathcal S}_{\alpha ;0}$ in \eqref{isomorphism-D-N-2c} is compact as well. Therefore, the operator ${\mathcal S}_{\alpha }$ in \eqref{isomorphism-D-N} is Fredholm with index zero. %Arguments similar to those in \cite[Theorem 6.1]{K-L-W1}, which are based on the uniqueness result of the mixed problem \eqref{mixed homogeneous Brinkman system}- and on the uniqueness result of the exterior Dirichlet problem for the Brinkman system with data in $H^1_2(\partial\Omega,{\mathbb R}^n)$, together with the jump relations \eqref{70aaa} imply that
This operator is also injective. Indeed, if $\boldsymbol \Psi^{(0)}\in L_2(\partial \Omega ,{\mathbb R}^n)$ satisfies the equation ${\mathcal S}_{\alpha }\boldsymbol \Psi ^{(0)}=0$ then the single layer velocity and pressure potentials ${\bf u}^{(0)}:={\bf V}_{\alpha }\boldsymbol \Psi ^{(0)}$ and $\pi ^{(0)}:={\mathcal Q}^s\boldsymbol \Psi ^{(0)}$ will determine a solution of the homogeneous mixed problem associated to \eqref{mixed homogeneous Brinkman system}, such that $({\bf u}^{(0)},\pi ^{(0)})\in H_{2}^{\frac{3}{2}}({\Omega}_+,\mathbb{R}^{n})\times  H_{2}^{\frac{1}{2}}({\Omega}_+)$ and ${M}({\textbf u}^{(0)}),{M}(\nabla {\textbf u}^{(0)}),{M}(\pi ^{(0)})\in L_2(\partial {\Omega})$.
Then
%by using the Green formula \eqref{Green formula} we deduce that
${\bf u}^{(0)}={\bf 0}$ and $\pi ^{(0)}=0$ in $\Omega _{+}${, as shown above}.
{Consequently,} ${\bf t}_{\rm{nt}}^+({\bf u}^{(0)},\pi ^{(0)})
%=\left(-\pi ^{(0)}+\nabla {\bf u}^{(0)}+(\nabla {\bf u}^{(0)})^\top\right)_{\rm{nt}}^+{\boldsymbol \nu}
={\bf 0}$ a.e. on $\partial \Omega $, which, in view of \eqref{70aaa}, can be written as
\begin{align*}
\left(\frac{1}{2}{\mathbb I}+{\bf K}_{\alpha }^*\right)\boldsymbol \Psi ^{(0)}={\bf 0}.
\end{align*}
Moreover, the invertibility of the operator $\frac{1}{2}\mathbb{I}+{\bf K}^{*}_{\alpha}:L_2(\partial {\Omega},{\mathbb R}^n)\to L_2(\partial {\Omega},{\mathbb R}^n)$ (see {Lemma \ref{L3.1p}}) shows that $\boldsymbol \Psi ^{(0)}={\bf 0}$. Consequently, operator \eqref{isomorphism-D-N} is an isomorphism, as asserted.
Then the fields
\begin{equation}
\label{solution-D-N}
{\bf u}={\bf V}_{\alpha
}\left({\mathcal S}_{\alpha }^{-1}({\mathbf h_{0}},\mathbf g_{0})\right),\ \
\pi =\mathcal Q^s\left({\mathcal S}_{\alpha }^{-1}({\mathbf h_{0}},\mathbf g_{ 0})\right)
\end{equation}
determine the unique solution of the mixed Dirichlet-Neumann problem \eqref{mixed homogeneous Brinkman system}. According to Lemma \ref{nontangential-sl}, {Theorem} \ref{layer-potential-properties} and \eqref{solution-D-N}, the solution belongs to the space $H_{2}^{\frac{3}{2}}({\Omega}_+,\mathbb{R}^{n})\times H_2^{\frac{1}{2}}({\Omega}_+)$ and satisfies the {estimate \eqref{ineq-M-2} with some constant $C_M>0$ depending on $S_{D}$, $S_{N}$ and $\alpha$}, as well as
%(see also \cite[Theorem 3.1]{M-T} for $\alpha =0$)
estimate \eqref{estimate-mixed} with the constant
$C=\left(\|{\bf V}_{\alpha }\|+\|{Q}^s\|\right)\|\mathcal{S}^{-1}_{\alpha}\|$. %provided by the operator norms of the operators ${\bf V}_{\alpha }$, ${Q}^s$ and $\mathcal{S}^{-1}_{\alpha}$ mentioned above.
\hfill\end{proof}

\subsection{\bf Mixed Dirichlet-Neumann problem for the Brinkman system with data in $L_p$-spaces}
Next, we extend the {results established in Theorem \ref{Th2.6}, to $L_p$-based
spaces with $p$ in some neighborhood of $2$, for the mixed Dirichlet-Neumann problem for the Brinkman system \eqref{mixed homogeneous Brinkman system}, with the boundary data $(\textbf{h}_0,\textbf{g}_0)\in{H_p^{1}(S_{D},\mathbb{R}^{n})}\times L_p(S_{N},\mathbb{R}^{n})$.}
%\begin{equation}
%\label{mixed homogeneous Brinkman system}
%\left\{
%\begin{array}{l}
%\triangle {\bf u} - \alpha {\bf u}- \nabla \pi = 0,\ {\rm{div}}\ {\bf u} = 0\ \mbox{ in }\ {\Omega},\\
%\left({\gamma}_+{\bf u}\right)|_{S_D}={\mathbf h_{0}}\in {H_p^{1}(S_{D},\mathbb{R}^{n}),}\\
%\left(\mathbf t_{\alpha}^{+}({\bf u},\pi )\right)|_{S_N}=\mathbf g_{0}\in L_p(S_{N},\mathbb{R}^{n})
%,\\
%{M}(\nabla {\bf u}),\ {M}({\bf u}),\ {M}(\pi)\in L_p(\partial {\Omega})
%\end{array}
%\right.
%\end{equation}
%will be determined
%{such that ${M}({\textbf u}),{M}(\nabla {\textbf u}),{M}(\pi)\in L_p(\partial {\Omega})$,
%and will hence also satisfy the inclusion}
%\begin{equation}
{We will obtain the well-posedness result in the space $B_{p,p^{*}}^{1+\frac{1}{p}}({\Omega}_+,\mathbb{R}^{n})\times B_{p,p^{*}}^{\frac{1}{p}}({\Omega}_+)$,
%\end{equation}
where $p^{*}=\max\{2, p\}$.}
%(cf., e.g., \cite[Chapter 3, Remark V]{M-M-T}).

We further need the space
\begin{equation}
\label{space-L0}
\widetilde H_p^0(S_{0},\mathbb{R}^{n}):=\left\{{\bf \Phi }\in L_p(\partial {\Omega},\mathbb{R}^{n}) : {\rm{supp}}\ {\bf \Phi }\subseteq \overline{S_{0}}\right\},\quad S_0\subset \partial {\Omega}.
\end{equation}

\subsection*{$\bullet$ {\bf The Neumann-to-Dirichlet {operator} for the Brinkman system}}
As in the work \cite{M-M}, devoted to the mixed Dirichlet-Neumann problem for the Laplace equation in a creased Lipschitz domain, we consider the Neumann-to-Dirichlet operator $\Upsilon_{{\rm nt};\alpha}$,
which associates to ${\mathbf g}\in L_p(\partial {\Omega}, \mathbb{R}^{n})$, the restriction {of the non-tangential trace ${\bf u}^+_{\rm nt}$ to the patch $S_{D}$}, where $({\bf u},\pi )$ is the unique { $L_p$-solution} of the Neumann problem {\eqref{the homogeneous Neumann Brinkman system} for the Brinkman system with the non-tangential conormal derivative ${\mathbf g}$}. {Thus, $({\bf u},\pi )$ satisfies the Neumann condition almost everywhere on $\partial \Omega $ in the sense of non-tangential limit, as well as the conditions ${M}({\textbf u}),{M}(\nabla {\textbf u}),{M}(\pi)\in L_p(\partial {\Omega})$}, and
\begin{equation}
\label{Neumann-to-Dirichlet operator}
\Upsilon_{{\rm nt};\alpha}{\mathbf {g}}={\bf u}^+_{\rm nt}|_{S_{D}}.
\end{equation}
{Similarly, we consider the Neumann-to-Dirichlet operator $\Upsilon_{\alpha}$,
which associates to ${\mathbf g}\in L_p(\partial {\Omega}, \mathbb{R}^{n})$, the restriction of the trace $\gamma_+{\bf u}$ to the patch $S_{D}$, where $({\bf u},\pi )$ is the unique solution of the Neumann problem  \eqref{Poisson-Neumann-Brinkman-Lp} for the Brinkman system with $\mathbf f=\mathbf 0$ and the canonical conormal derivative ${\mathbf g}$, i.e.,}
\begin{equation}
\label{Neumann-to-Dirichlet operatorG}
{\Upsilon}_{\alpha}{\mathbf {g}}=\gamma_+{\bf u}|_{S_{D}}.
\end{equation}
A way to extend the well-posedness result in Theorem \ref{Th2.6} to $L_p$-based {spaces} is to show the invertibility of the Neumann-to-Dirichlet operator $\Upsilon_{{\rm nt};\alpha}$%{and} $\Upsilon_{\alpha}$
on such spaces. An intermediary step to obtain this property is given by the following result.
\comment{
\begin{theorem}[Old version]
\label{Th2.5}
Assume that {${\Omega}_+$}
%\subset \mathbb{R}^{n}$ $(n\geq 3)$
is a bounded, creased Lipschitz domain with connected boundary $\partial {\Omega}$, and that $\partial {\Omega}$ is decomposed into two disjoint admissible patches $S_{D}$ and $S_{N}$. Then there exists a number $\varepsilon>0$ such that for any
$p\in\mathcal R_0(n, \varepsilon)$, see \eqref{cases},
%$p\in \left(\frac{2(n-1)}{n+1}-\varepsilon,2+\varepsilon\right)\cap (1,+\infty )$,
the mixed Dirichlet-Neumann boundary value problem for the Brinkman system \eqref{mixed homogeneous Brinkman system} is well-posed in the space $B_{p,p^{*}}^{1+\frac{1}{p}}({\Omega}_+,\mathbb{R}^{n})\times  B_{p,p^{*}}^{\frac{1}{p}}({\Omega}_+)$, where $p^{*}=\max\{2, p\}$, if and only if the operator
\begin{equation}
\label{Robin-to-Dirichlet operator1}
{\Upsilon}_{\alpha}: {\widetilde H}_p^0(S_{D},\mathbb{R}^{n})\to H_p^{1}(S_{D},\mathbb{R}^{n})
\end{equation}
is an isomorphism.
\end{theorem}
\begin{proof}
{By Theorem \ref{Th2.1'}, there exists $\varepsilon>0$ such that for
$p\in \left(\frac{2(n-1)}{n+1}-\varepsilon_{0},2+\varepsilon_{0} \right)\cap (1,+\infty )$
the Neumann problem for the Brinkman system \eqref{the homogeneous Neumann Brinkman system} has a unique solution given by \eqref{SoNB4}. Then} we deduce that the operator {${\Upsilon}_{\alpha}: {\widetilde H}_p^0(S_{D},\mathbb{R}^{n})\to H_p^{1}(S_{D},\mathbb{R}^{n})$} given by \eqref{Neumann-to-Dirichlet operator} has the expression
\begin{equation}
\label{construction Y}
{\Upsilon}_{\alpha} ={\left({\mathcal V}_{\alpha } \circ \left(\frac{1}{2}\mathbb{I} + {\bf K}^{*}_{\alpha }\right)^{-1}\right)\Bigg|_{S_{D}}},
\end{equation}
and is continuous, due to the continuity of both operators in the right-hand side of \eqref{construction Y}.

\begin{itemize}
\item[(i)]
First, we assume that problem \eqref{mixed homogeneous Brinkman system} is well-posed and show the invertibility of the operator ${\Upsilon}_{\alpha}$.
\end{itemize}
In order to prove the injectivity property of this operator, we consider a density ${\mathbf g^0}\in {\widetilde H}_p^0(S_D,\mathbb{R}^{n})$,
such that ${\Upsilon}_{\alpha}{\mathbf g^0}={\bf 0}$.
In addition, in view of \eqref{Neumann-to-Dirichlet operator}, we also have that ${\Upsilon}_{\alpha}{\mathbf g^0}=({\gamma}_{+}{\bf u}^{0})|_{S_{D}}$, where $({\bf u}^{0}, \pi^{0})$ is the unique {$L_p$-solution} of the Neumann problem \eqref{the homogeneous Neumann Brinkman system} for the homogeneous Brinkman system with boundary datum {${\mathbf g^0}\in {\widetilde H}_p^0(S_D,\mathbb{R}^{n})$} on $\partial {\Omega}$.
Therefore,
\begin{equation}
\label{0}
{\Upsilon}_{\alpha}{\mathbf g^0}=({\gamma}_{+}{\bf u}^{0})|_{S_{D}}={\bf 0},
\end{equation}
and
\begin{equation}
\label{the homogeneous Neumann Brinkman system2}
\left\{
\begin{array}{l}
\triangle {\bf u}^{0} - \alpha {\bf u}^{0}- \nabla \pi ^{0}={\bf 0},\
{\rm{div}}\ {\bf u}^{0}= 0\ \mbox{ in }\ {\Omega}_+, \\
\mathbf t_{\alpha}^{+}({\bf u}^{0},\pi ^{0})= {\mathbf g^0}\ \mbox{ on }\ \partial {\Omega}.
%\\
%{M}(\nabla {\bf u}),\ {M}({\bf u}),\ {M}(\pi )\in L_p(\partial {\Omega}).
\end{array}
\right.
\end{equation}
{Due to the relations
\begin{equation}
\label{N-D1}
({\gamma}_{+}{\bf u}^{0})|_{S_{D}}={\bf 0}\ \mbox{ on }\ S_{D}, \ \ \left(\mathbf t_{\alpha}^{+}({\bf u}^{0},\pi^{0})\right)|_{S_N}={\bf 0}\ \mbox{ on }\ S_{N},
\end{equation}
where the second of them follows from the assumption that ${\mathbf g^0}\in {\widetilde H}_p^0(S_D,\mathbb{R}^{n})$, and hence that ${\mathbf g^0}={\bf 0}$ on $S_N$, and by using the assumed well-posedness of the mixed Dirichlet-Neumann problem \eqref{mixed homogeneous Brinkman system} (with zero boundary {data} \eqref{N-D1}), we deduce that ${\bf u}^{0}={\bf 0}$ and $\pi^{0} = 0$ in ${\Omega}_+$.} Thus, $\mathbf t_{\alpha}^{+}({\bf u}^{0},\pi^{0})={\bf 0}$ on $\partial {\Omega}$, and hence ${\mathbf g^0}={\bf 0}$. {This} implies that the operator ${\Upsilon}_{\alpha}$ is injective.

We show that the operator ${\Upsilon}_{\alpha}$ is also surjective. Let ${\mathbf h_0}\in H_p^{1}(S_{D},\mathbb{R}^{n})$ be given.
Then, due to the assumed well posedness of the mixed Dirichlet-Neumann problem \eqref{mixed homogeneous Brinkman system}, there  exists a unique solution, $({\bf u}_{0},\pi _{0})$, of this problem with the Dirichlet datum ${\mathbf h_0}$ on $S_D$ and with the particular Neumann datum $\mathbf g_{0}\equiv {\bf 0}$ on $S_N$.
In particular, we deduce that the field ${\mathbf g_{0}}:=\mathbf t_{\alpha}^{+}({\bf u}_{0},\pi _{ 0})\in L_p(\partial {\Omega},\mathbb{R}^{n})$
%satisfies the condition ${\mathbf g_0}|_{S_{N}}={\bf 0}$.
belongs to ${\widetilde H}_p^0(S_{D}, \mathbb{R}^{n})$, due to definition \eqref{space-L0}.
In addition, the uniqueness result in Theorem \ref{Th2.1'} shows that $({\bf u}_{0},\pi _{0})$ is the unique solution of the Neumann problem for the Brinkman system in ${\Omega}_+$ with the Neumann datum
${\mathbf g_{0}}\in{\widetilde H}_p^0(S_{D},\mathbb{R}^{n})\subset L_p(\partial {\Omega}, \mathbb{R}^{n})$.
Then by definition (\ref{Neumann-to-Dirichlet operator}) of the operator ${\Upsilon}_{\alpha}$, we obtain that ${\Upsilon}_{\alpha}{\mathbf g_{0}}=({\gamma}_{+}{\bf u}_{0})|_{S_{D}}$.
However, $({\gamma}_{+}{\bf u}_{0})|_{S_{D}}={\mathbf h_0}$, due to the choice of $({\bf u}_{0},\pi _{ 0})$ as the solution of the mixed Dirichlet-Neumann problem \eqref{mixed homogeneous Brinkman system} with boundary data $({\mathbf h_0},{\bf 0})$.
%Hence, ${\Upsilon}_{\alpha}{\mathbf g_0}= ({\gamma}_{+}{\bf u}_0)|_{S_{D}}=\mathbf h_0$.
{Consequently}, for a given ${\mathbf h_0}\in H_p^{1}(S_{D},\mathbb{R}^{n})$ there {exists}
%a unique
${\mathbf g_{0}}\in {\widetilde H}_p^0(S_{D},\mathbb{R}^{n})$ such that ${\Upsilon}_{\alpha}{\mathbf g_{0}}={\mathbf h_0}$. This shows that the operator ${\Upsilon}_{\alpha}$ is surjective, and thus, it is an isomorphism, as asserted.
\begin{itemize}
\item[(ii)]
Next, we show the converse result, i.e., that the invertibility of the operator ${\Upsilon}_{\alpha}$ implies the well-posedness of the mixed Dirichlet-Neumann problem \eqref{mixed homogeneous Brinkman system}.
\end{itemize}
Let us first show the uniqueness of the solution to problem \eqref{mixed homogeneous Brinkman system}.
To this {end}, we assume that $({\bf u}^{(0)},\pi^{(0)})$ is an $L_p$-solution of the homogeneous version of \eqref{mixed homogeneous Brinkman system}.
Hence, ${\mathbf g}^{(0)}:=\mathbf t_\alpha^{+}({\bf u}^{(0)},\pi^{(0)})\in {\widetilde H}_p^0(S_{D}, \mathbb{R}^{n})$ since $\left(\mathbf t_{\alpha}^+({\bf u}^{(0)},\pi^{(0)})\right)|_{S_{N}}={\bf 0}$,
%, which yields that {${\rm{supp}}\{\mathbf t_{\alpha}^{+}({\bf u}^{(0)},\pi^{(0)})\}\subseteq \overline{S_{D}}$}.
%In addition, we deduce
implying
that $({\bf u}^{(0)}, \pi^{(0)})$ is (by Theorem \ref{Th2.1'}) the unique $L_p$-solution of the Neumann problem for the Brinkman system with Neumann datum ${\mathbf g}^{(0)}$ on $\partial {\Omega}$.
Then by \eqref{Neumann-to-Dirichlet operator}, ${\Upsilon}_{\alpha}{\mathbf g}^{(0)}=({\gamma}_{+}{\bf u}^{(0)})|_{S_{D}}={\bf 0}$, where the last equality follows from the assumption that $({\bf u}^{(0)}, \pi^{(0)})$ satisfies the homogeneous version of \eqref{mixed homogeneous Brinkman system}.
The injectivity of ${\Upsilon}_{\alpha}$ implies that ${\mathbf g}^{(0)}={\bf 0}$, and hence we obtain that $\mathbf t_{\alpha}^{+}({\bf u}^{(0)},\pi^{(0)} )={\bf 0}$ on $\partial {\Omega}$. In addition, Theorem \ref{Th2.1'} implies that ${\bf u}^{0}= {\bf 0}$, $\pi ^{0} = 0$ in ${\Omega}_+$. This concludes the proof of uniqueness of the {$L_p$-solution} to the mixed problem \eqref{mixed homogeneous Brinkman system}.

To show the existence of an {$L_p$-solution} to the mixed problem \eqref{mixed homogeneous Brinkman system}, let us consider such a problem with arbitrary boundary data $({\mathbf h_{0}},\mathbf g_{0})\in H_p^{1}(S_{D}, \mathbb{R}^{n})\times L_p(S_{N}, \mathbb{R}^{n})$. Also let ${\bf G}\in L_p(\partial {\Omega}_+, \mathbb{R}^{n})$ be such that
\begin{align}
\label{G}
{\bf G}|_{S_{N}}=\mathbf g_{0}.
\end{align}
Then by Theorem \ref{Th2.1'} there exists a unique {$L_p$-solution} $({\bf v},q)$ of {the Neumann problem \eqref{the homogeneous Neumann Brinkman system}} with the Neumann datum ${\bf G}$. Note that ${\bf v}$ is expressed in terms of a single-layer potential with a density in the space $L_p(\partial \Omega ,{\mathbb R}^n)$, and hence $\gamma _{+}{\bf v}\in H_p^1(\partial \Omega ,{\mathbb R}^n)$ (see Lemma \ref{M-H}(ii)). On the other hand, the invertibility of the operator ${\Upsilon}_{\alpha}: {\widetilde H}_p^0(S_{D},\mathbb{R}^{n})\to H_p^{1}(S_{D}, \mathbb{R}^{n})$ assures that the equation
\begin{align}
{\Upsilon}_{\alpha}{\bf F}=\left({\gamma}_{+}{\bf v}|_{S_{D}}-{\mathbf h_{0}}\right)\in H_p^{1}(S_{D}, \mathbb{R}^{n})
\end{align}
has a unique solution ${\bf F}\in {\widetilde H}_p^0(S_{D},\mathbb{R}^{n})\subset L_p(\partial {\Omega},\mathbb{R}^{n})$. Next, let $({\bf w}, \tilde{q})$ be the unique {$L_p$-solution} of the Neumann problem \eqref{the homogeneous Neumann Brinkman system} with the Neumann datum ${\bf F}$. Also let
\begin{align}
\label{constr}
({\bf u},\pi):=({\bf v}-{\bf w},q-\tilde{q}).
\end{align}
Then we obtain the relations
\begin{align}
({\gamma}_{+}{\bf u})|_{S_{D}}&=({\gamma}_{+}{\bf v})|_{S_{D}}-({\gamma}_{+}{\bf w})|_{S_{D}}
%\nonumber\\
%&
=\left({\Upsilon}_{\alpha}{\bf F}+{\mathbf h_{0}}\right)-{\Upsilon}_{\alpha}{\bf F}
%\nonumber\\
%&
={\mathbf h_{0}}
\end{align}
and
\begin{align}
\left(\mathbf t_{\alpha}^{+}({\bf u},\pi)\right)|_{S_N}&
=\left(\mathbf t_{\alpha}^{+}({\bf v},q)\right)|_{S_N}-\left(\mathbf t_{\alpha}^{+}({\bf w},\tilde{q})\right)|_{S_N}
%\nonumber\\
%&
={\bf G}|_{S_N}-{\bf F}|_{S_N}
%\nonumber\\
%&
=\mathbf g_{0},
\end{align}
where the last equality follows from \eqref{G} and the inclusion {${\bf F}\in {\widetilde H}_p^0(S_{D},\mathbb{R}^{n})$.}
%the last one implying that ${\rm{supp}}\ {\bf F}\subseteq \overline{S_{D}}$, and hence that ${\bf F}|_{S_N}={\bf 0}$.
Moreover,
%$({\bf u},\pi)\in {H_{p}^{1+\frac{1}{p}}({\Omega},\mathbb{R}^{n})}\times H_p^{\frac{1}{p}}({\Omega})$
$({\bf u},\pi)\in B_{p,p^{*}}^{1+\frac{1}{p}}({\Omega}_+,\mathbb{R}^{n})\times  B_{p,p^{*}}^{\frac{1}{p}}({\Omega}_+)$
and ${M}({\textbf u}),{M}(\nabla {\textbf u}),{M}(\pi)\in L_p(\partial {\Omega})$, due to \eqref{constr} and the mapping properties of the pairs $({\bf v},q)$ and $({\bf w},\tilde{q})$
{given  by Theorem \ref{Th2.1'}.}
%each of them being expressed in terms of single-layer velocity and pressure potentials.
Consequently, the mixed Dirichlet-Neumann problem \eqref{mixed homogeneous Brinkman system} is well-posed.
\hfill\end{proof}
} %\comment end
\begin{lemma} %[New version]
\label{Th2.5}
Let ${\Omega}_+\subset {\mathbb R}^n$ $($$n\geq 3$$)$ be a bounded, creased Lipschitz domain with connected boundary $\partial {\Omega}$ which is
%, and that $\partial {\Omega}$ is
decomposed into two disjoint admissible patches $S_{D}$ and $S_{N}$. Let $\alpha \in (0,\infty )$. {Then there exists $\varepsilon=\varepsilon(\partial\Omega)>0$ such that for any $p\in\mathcal R_0(n, \varepsilon)$ the following properties hold.}
\begin{itemize}
\item[$(i)$]
{The operators $\Upsilon_{{\rm nt};\alpha}$ and $\Upsilon_{\alpha}$ coincide and are given by}
%{\rd The operator $\Upsilon_{{\rm nt};\alpha}$ is given by}
\begin{equation}
\label{construction Y}
{\Upsilon_{{\rm nt};\alpha}=}\Upsilon_{\alpha} ={\left({\mathcal V}_{\alpha } \circ \left(\frac{1}{2}\mathbb{I} + {\bf K}^{*}_{\alpha }\right)^{-1}\right)\Bigg|_{S_{D}}}.
%{\Upsilon_{{\rm nt};\alpha}}={\left({\mathcal V}_{\alpha } \circ \left(\frac{1}{2}\mathbb{I} + {\bf %K}^{*}_{\alpha }\right)^{-1}\right)\Bigg|_{S_{D}}}.
\end{equation}
\item[$(ii)$]
{The mixed Dirichlet-Neumann problem \eqref{mixed homogeneous Brinkman system} with given data
$({\mathbf h_{0}},\mathbf g_{0})\in H_p^1(S_{D},\mathbb{R}^{n})\times L_p(S_{N},\mathbb{R}^{n})$ has a unique {solution $({\bf u},\pi)$, such that ${M}({\textbf u}),{M}(\nabla {\textbf u}),{M}(\pi)\in L_p(\partial {\Omega})$,} %{\rd such that there exist the non-tangential limits of ${\bf u}$, $\nabla {\bf u}$ and $\pi $ at almost all points of $\partial \Omega $, ${M}({\textbf u}),{M}(\nabla {\textbf u}),{M}(\pi)\in L_p(\partial {\Omega})$ and ${\bf u}$ and $\pi $ satisfy the Dirichlet and Neumann boundary conditions almost everywhere in the sense of non-tangential limit},
if and only if the operator}
\begin{equation}
\label{Robin-to-Dirichlet operator1}
{\Upsilon}_{{\rm nt};\alpha}: {\widetilde H}_p^0(S_{D},\mathbb{R}^{n})\to H_p^{1}(S_{D},\mathbb{R}^{n})
\end{equation}
is an isomorphism.
\item[$(iii)$]
{The mixed Dirichlet-Neumann problem \eqref{mixed homogeneous Brinkman systemG} with given data
$({\mathbf h_{0}},\mathbf g_{0})\in H_p^1(S_{D},\mathbb{R}^{n})\times L_p(S_{N},\mathbb{R}^{n})$ has a unique solution {$(\textbf{u},\pi)\in B_{p,p^*}^{1+\frac{1}{p}}({\Omega}_+,\mathbb{R}^{n})\times B_{p,p^*}^{\frac{1}{p}}({\Omega}_+)$} if and only if the operator}
\begin{equation}
\label{Robin-to-Dirichlet operator1G}
{{\Upsilon}_{\alpha}: {\widetilde H}_p^0(S_{D},\mathbb{R}^{n})\to H_p^{1}(S_{D},\mathbb{R}^{n})}
\end{equation}
{is an isomorphism.}
\item[]
{Moreover, when the solution $({\bf u},\pi)$ in item $(ii)$ or $(iii)$ exists, then}
%{Moreover, when the solution $({\bf u},\pi)$ exists,
{{it belongs to the space $B_{p,p^{*}}^{1+\frac{1}{p}}({\Omega}_+,\mathbb{R}^{n})\times B_{p,p^{*}}^{\frac{1}{p}}({\Omega}_+)$ and} there exist some constants $C_M\equiv C_M(\alpha, p,S_{D},S_{N})$, $C\equiv C(\alpha, p,S_{D},S_{N})$ and $C'\equiv C'(\alpha, p,S_{D},S_{N})$ such that}
\begin{align}
\label{ineq-MY}
&\|M(\nabla {\bf u})\|_{L_p(\partial \Omega )}+\|M({\bf u})\|_{L_p(\partial \Omega )}
+\|M(\pi )\|_{L_p(\partial \Omega )}
\leq C_M\left(\|{\mathbf h_{0}}\|_{H_p^{1}(S_{D}, \mathbb{R}^{n})}
+\|\mathbf g_{0}\|_{L_p(S_{N},\mathbb{R}^{n})}\right),\\
\label{continuity-dataY}
&\|{\bf u}\|_{B_{p,p^{*}}^{1+\frac{1}{p}}({\Omega}_+,\mathbb{R}^{n})}+\|\pi\|_{ B_{p,p^{*}}^{\frac{1}{p}}({\Omega}_+)}
\leq C\left(\|{\mathbf h_{0}}\|_{H_p^{1}(S_{D}, \mathbb{R}^{n})}+\|\mathbf g_{0}\|_{L_p(S_{N},\mathbb{R}^{n})}\right),
\quad p^{*} = \max\{2, p\},\\
\label{tr-continuity}
&\|\gamma_+{\bf u}\|_{H_{p}^1(\partial\Omega,\mathbb{R}^n)}
+\|\mathbf t^+_{\alpha}({\bf u},\pi)\|_{L_p(\partial\Omega,\mathbb{R}^{n}))}
\le C'\left(\|{\mathbf h_{0}}\|_{H_p^{1}(S_{D},\mathbb{R}^{n})}+\|\mathbf g_{0}\|_{L_p(S_{N},\mathbb{R}^{n})}\right).
\end{align}
\end{itemize}
\end{lemma}
\begin{proof}
{(i) By Theorem \ref{Th2.1'}, there exists $\varepsilon=\varepsilon(\partial\Omega)>0$ such that for any $p\in\mathcal R_0(n, \varepsilon)$ the Neumann problem \eqref{the homogeneous Neumann Brinkman system} %and \eqref{Poisson-Neumann-Brinkman-Lp}, with $\mathbf f=\mathbf 0$,
has a unique solution, and it can be expressed in form} \eqref{SoNB4}.
Then {due to Theorem~\ref{layer-potential-properties} and Lemma~\ref{L3.6}} we deduce that the operator
%${\Upsilon}_{\alpha}: {\widetilde H}_p^0(S_{D},\mathbb{R}^{n})\to H_p^{1}(S_{D},\mathbb{R}^{n})$
\eqref{Neumann-to-Dirichlet operator} %and \eqref{Neumann-to-Dirichlet operatorG}
has the expression \eqref{construction Y}
and {is} continuous, due to the continuity of both operators in the right-hand side of \eqref{construction Y}.

%\begin{itemize}
%\item[(i)]
{(ii) First}, we assume that problem \eqref{mixed homogeneous Brinkman system} is well-posed and show the invertibility of {operator \eqref{Robin-to-Dirichlet operator1}}.
%\end{itemize}

In order to prove the injectivity property of this operator, we consider a function ${\mathbf g^0}\in {\widetilde H}_p^0(S_D,\mathbb{R}^{n})$, such that $\Upsilon_{{\rm nt};\alpha}{\mathbf g^0}={\bf 0}$.
{Denoting by $({\bf u}^{0}, \pi^{0})$} the unique {$L_p$-solution} of the Neumann problem \eqref{the homogeneous Neumann Brinkman system} for the homogeneous Brinkman system with boundary datum ${\mathbf g^0}\in {\widetilde H}_p^0(S_D,\mathbb{R}^{n})$ on $\partial {\Omega}${, in view of \eqref{Neumann-to-Dirichlet operator}, we have}
\begin{equation}
\label{0}
{\bf u}^+_{\rm nt}|_{S_{D}}=\Upsilon_{{\rm nt};\alpha}{\mathbf g^0}={\bf 0},
\end{equation}
and
\begin{equation}
\label{the homogeneous Neumann Brinkman system2}
\left\{
\begin{array}{l}
\triangle {\bf u}^{0} - \alpha {\bf u}^{0}- \nabla \pi ^{0}={\bf 0},\
{\rm{div}}\ {\bf u}^{0}= 0\ \mbox{ in }\ {\Omega}_+, \\
{\mathbf t^+_{\rm nt}}({\bf u}^{0},\pi ^{0})= {\mathbf g^0}\ \mbox{ on }\ \partial {\Omega}.
%\\
%{M}(\nabla {\bf u}),\ {M}({\bf u}),\ {M}(\pi )\in L_p(\partial {\Omega}).
\end{array}
\right.
\end{equation}
{In addition, $({\bf u}^0,\pi ^0)$ satisfies the conditions ${M}({\textbf u}^0),{M}(\nabla {\textbf u}^0),{M}(\pi ^0)\in L_p(\partial {\Omega})$, and the Neumann condition holds almost everywhere on $\partial \Omega $ in the sense of non-tangential limit.}

According to {relation \eqref{0} and the inclusion ${\mathbf g^0}\in {\widetilde H}_p^0(S_D,\mathbb{R}^{n})$, we have}
\begin{equation}
\label{N-D1}
{{\bf u}^{0+}_{\rm nt}}|_{S_{D}}={\bf 0}\ \mbox{ on }\ S_{D}, \ \
{\mathbf t^+_{\rm nt}({\bf u}^{0},\pi ^{0})}|_{S_N}={\bf 0}\ \mbox{ on }\ S_{N},
\end{equation}
%where the second of them follows from the assumption that ${\mathbf g^0}\in {\widetilde H}_p^0(S_D,\mathbb{R}^{n})$,
and hence
%that ${\mathbf g^0}={\bf 0}$ on $S_N$
%, and by using
by the assumed well-posedness of the mixed Dirichlet-Neumann problem
\eqref{mixed homogeneous Brinkman system},
%(with zero boundary data \eqref{N-D1}),
we deduce that
${\bf u}^{0}={\bf 0}$ and $\pi^{0} = 0$ in ${\Omega}_+$.
Thus, ${\mathbf g^0}=\mathbf t^+_{\rm nt}({\bf u}^{0},\pi^{0})={\bf 0}$ on $\partial {\Omega}$, which
implies that the operator ${\Upsilon}_{\alpha}$ is injective.

We show that the operator $\Upsilon_{{\rm nt};\alpha}$ is also surjective.
Due to the assumed well posedness of the mixed Dirichlet-Neumann problem \eqref{mixed homogeneous Brinkman system}, for any Dirichlet datum ${\mathbf h_0}\in H_p^{1}(S_{D},\mathbb{R}^{n})$ on $S_D$ and  the Neumann datum $\mathbf g_{0}\equiv {\bf 0}$ on $S_N$, there exists a unique {$L_p$-solution}, $({\bf u}_{0},\pi _{0})$, of this problem.
In particular, we deduce that the {vector field} ${\mathbf g^{0}}:={\mathbf t^+_{\rm nt}}({\bf u}_{0},\pi _{ 0})\in L_p(\partial {\Omega},\mathbb{R}^{n})$
%satisfies the condition ${\mathbf g_0}|_{S_{N}}={\bf 0}$.
belongs to ${\widetilde H}_p^0(S_{D}, \mathbb{R}^{n})$, due to definition \eqref{space-L0}.
In addition, the uniqueness result in Theorem \ref{Th2.1'} shows that $({\bf u}_{0},\pi _{0})$ is the unique solution of the Neumann problem for the Brinkman system in ${\Omega}_+$ with the Neumann datum
${\mathbf g^{0}}\in{\widetilde H}_p^0(S_{D},\mathbb{R}^{n})\subset L_p(\partial {\Omega}, \mathbb{R}^{n})$.
Then by definition (\ref{Neumann-to-Dirichlet operator}) of the operator $\Upsilon_{{\rm nt};\alpha}$, we obtain that $\Upsilon_{{\rm nt};\alpha}{\mathbf g_{0}}={{\bf u}^{+}_{0,\rm nt}}|_{S_{D}}
={\mathbf h_0}.$
Consequently, for a given ${\mathbf h_0}\in H_p^{1}(S_{D},\mathbb{R}^{n})$ there exists
${\mathbf g_{0}}\in {\widetilde H}_p^0(S_{D},\mathbb{R}^{n})$ such that $\Upsilon_{{\rm nt};\alpha}{\mathbf g_{0}}={\mathbf h_0}$. This shows that the operator $\Upsilon_{{\rm nt};\alpha}$ is surjective, and thus, it is an isomorphism, as asserted.

%\begin{itemize}
%\item[(ii)]
Next, we show the converse result, i.e., that the invertibility of the operator $\Upsilon_{{\rm nt};\alpha}$ implies the well-posedness of the mixed Dirichlet-Neumann problem \eqref{mixed homogeneous Brinkman system}.
%\end{itemize}
Let us first show {uniqueness} of the solution to problem \eqref{mixed homogeneous Brinkman system}.
To this {end}, we assume that $({\bf u}^{(0)},\pi^{(0)})$ is an $L_p$-solution of the homogeneous version of \eqref{mixed homogeneous Brinkman system}.
Hence, ${\mathbf g}^{(0)}:={\mathbf t^+_{\rm nt}}({\bf u}^{(0)},\pi^{(0)})\in {\widetilde H}_p^0(S_{D}, \mathbb{R}^{n})$ since ${\mathbf t^+_{\rm nt}}({\bf u}^{(0)},\pi^{(0)})|_{S_{N}}={\bf 0}$,
%, which yields that {${\rm{supp}}\{\mathbf t_{\alpha}^{+}({\bf u}^{(0)},\pi^{(0)})\}\subseteq \overline{S_{D}}$}.
%In addition, we deduce
implying
that $({\bf u}^{(0)}, \pi^{(0)})$ is (by Theorem \ref{Th2.1'}) the unique {solution} of the Neumann problem for the Brinkman system with Neumann datum ${\mathbf g}^{(0)}$ on $\partial {\Omega}$.
Then by \eqref{Neumann-to-Dirichlet operator},
$\Upsilon_{{\rm nt};\alpha}{\mathbf g}^{(0)}={{\bf u}^{(0)+}_{\rm nt}}|_{S_{D}}={\bf 0}$,
%where the last equality follows from the assumption that $({\bf u}^{(0)}, \pi^{(0)})$ satisfies the homogeneous version of \eqref{mixed homogeneous Brinkman system}.
and injectivity of $\Upsilon_{{\rm nt};\alpha}$ implies that ${\mathbf g}^{(0)}={\bf 0}$.
Hence ${\mathbf t^+_{\rm nt}}({\bf u}^{(0)},\pi^{(0)} )={\bf 0}$ on $\partial {\Omega}$ and Theorem \ref{Th2.1'} implies that ${\bf u}^{0}= {\bf 0}$, $\pi ^{0} = 0$ in ${\Omega}_+$. This concludes the proof of uniqueness of the solution to the mixed problem \eqref{mixed homogeneous Brinkman system}.

To show existence of an {$L_p$-solution} to the mixed problem \eqref{mixed homogeneous Brinkman system}, let us consider such a problem with arbitrary boundary data $({\mathbf h_{0}},\mathbf g_{0})\in H_p^{1}(S_{D}, \mathbb{R}^{n})\times L_p(S_{N}, \mathbb{R}^{n})$. Also let ${\bf G}\in \widetilde H^0_p(S_N, \mathbb{R}^{n})$ be such that
\begin{align}
\label{G}
{\bf G}|_{S_{N}}=\mathbf g_{0}.
\end{align}
Then by Theorem \ref{Th2.1'} there exists a unique {$L_p$-solution} $({\bf v},q)$ of {the Neumann problem \eqref{the homogeneous Neumann Brinkman system}} with the Neumann datum ${\bf G}$, such that {there exist the non-tangential limits of ${\bf u}$, $\nabla {\bf u}$, $\pi $ at almost all points of $\partial \Omega $, ${M}({\textbf v}),{M}(\nabla {\textbf v}),{M}(q)\in L_2(\partial {\Omega})$, and satisfies the Neumann boundary condition in the sense of non-tangential limit at almost all points of $\partial \Omega $.} Note that ${\bf v}$ can be expressed in terms of a single-layer potential with a density in the space $L_p(\partial \Omega ,{\mathbb R}^n)$, and hence ${\mathbf v^+_{\rm nt}}\in H_p^1(\partial \Omega ,{\mathbb R}^n)$ (see {Lemma \ref{L3.6}}).

On the other hand, the invertibility of the operator $\Upsilon_{{\rm nt};\alpha}: {\widetilde H}_p^0(S_{D},\mathbb{R}^{n})\to H_p^{1}(S_{D}, \mathbb{R}^{n})$ assures that the equation
\begin{align}
\Upsilon_{{\rm nt};\alpha}{\bf g^0}=\left({{\mathbf h_{0}}-{{\bf v}^{+}_{\rm nt}}|_{S_{D}}}\right)\in H_p^{1}(S_{D}, \mathbb{R}^{n})
\end{align}
has a unique solution ${\bf g^0}\in {\widetilde H}_p^0(S_{D},\mathbb{R}^{n})\subset L_p(\partial {\Omega},\mathbb{R}^{n})$. Next, let $({\bf u^0}, \pi^0)$ be the unique {$L_p$-solution} of the Neumann problem \eqref{the homogeneous Neumann Brinkman system} with the Neumann datum ${\bf g^0}$. Also let
\begin{align}
\label{constr}
({\bf u},\pi):=({\bf v}+{\bf u^0},q+\pi^0).
\end{align}
Then we obtain the relations
\begin{align}
{\bf u}_{\rm nt}^+|_{S_{D}}&={\bf v}_{\rm nt}^+|_{S_{D}}+{\bf u}_{\rm nt}^{0+}|_{S_{D}}
%\nonumber\\
%&
=\left({\mathbf h_{0}}-\Upsilon_{{\rm nt};\alpha}{\bf g^0}\right)+\Upsilon_{{\rm nt};\alpha}{\bf g^0}
%\nonumber\\
%&
={\mathbf h_{0}},\\
%\end{align}
%and
%\begin{align}
{\mathbf t^+_{\rm nt}}({\bf u},\pi)|_{S_N}&
={\mathbf t^+_{\rm nt}}({\bf v},q)|_{S_N}+{\mathbf t^+_{\rm nt}}({\bf u^0},\pi^0)|_{S_N}
%\nonumber\\
%&
={\bf G}|_{S_N}+{\bf g^0}|_{S_N}
%\nonumber\\
%&
=\mathbf g_{0},
\end{align}
where the last equality follows from \eqref{G} and the inclusion {${\bf g^0}\in {\widetilde H}_p^0(S_{D},\mathbb{R}^{n})$.}
Moreover,
%${M}({\textbf u}),{M}(\nabla {\textbf u}),{M}(\pi)\in L_p(\partial {\Omega})$,
{the estimates \eqref{ineq-MY} and \eqref{continuity-dataY} corresponding to item (ii) are}
due to \eqref{constr} and the mapping properties of the pairs $({\bf v},q)$ and $({\bf u^0},{\pi^0})$
given  by Theorem \ref{Th2.1'}.
Consequently, the mixed Dirichlet-Neumann problem \eqref{mixed homogeneous Brinkman system} is well-posed and estimates \eqref{ineq-MY}-\eqref{tr-continuity} hold true.

{The proof for item (iii) of the lemma and estimates \eqref{ineq-MY}-\eqref{tr-continuity} follow from similar arguments as those for item (ii), by refering to Theorems \ref{M-H-DG} and \ref{M-H-NG} instead of Theorems \ref{M-H-D} and \ref{Th2.1'}.}
\hfill\end{proof}

By combining Theorem \ref{Th2.6} and {Lemma} \ref{Th2.5}, we are now able to obtain the well-posedness results for the mixed Dirichlet-Neumann problem \eqref{mixed homogeneous Brinkman system} %{and \eqref{mixed homogeneous Brinkman systemG}}
with boundary data in $L_p$-based {Bessel potential} spaces and with $p$ in a neighborhood of $2$, which is the main result of this section. {Recall that $p^{*}=\max\{2,p\}$.}
\begin{theorem}
\label{Th2.7}
Assume that ${\Omega}_+ \subset \mathbb{R}^{n}$ $(n\geq 3)$ is a bounded, creased Lipschitz domain with connected boundary $\partial {\Omega}$ which is decomposed into two disjoint admissible patches $S_{D}$ and $S_{N}$. {Then there exists a number {$\varepsilon >0$ such that for any $p\in (2-\varepsilon , 2+\varepsilon )$} and for all given data $({\mathbf h_{0}},\mathbf g_{0})\in H_p^{1}(S_{D},\mathbb{R}^{n})\times L_p(S_{N},\mathbb{R}^{n})$ the following properties hold.}
\begin{itemize}
\item[$(i)$]
{The mixed} Dirichlet-Neumann problem for the Brinkman system \eqref{mixed homogeneous Brinkman system} has a unique {solution $({\bf u},\pi)$, such that ${M}({\textbf u}),{M}(\nabla {\textbf u}),{M}(\pi)\in L_p(\partial {\Omega})$}.
%{\rd such that ${M}({\textbf u}),{M}(\nabla {\textbf u}),{M}(\pi)\in L_2(\partial {\Omega})$, there exist the non-tangential limits of ${\bf u}$, $\nabla {\bf u}$ and $\pi $ at almost all points of $\partial \Omega $, and the Dirichlet and Neumann boundary conditions are satisfied in the sense of non-tangential limit at almost all points of $S_D$ and $S_N$, respectively.}
Moreover, {$({\bf u},\pi)\in B_{p,p^{*}}^{1+\frac{1}{p}}({\Omega}_+,\mathbb{R}^{n})\times B_{p,p^{*}}^{\frac{1}{p}}({\Omega}_+)$, and there exist some constants $C_M\equiv C_M(\alpha, p,S_{D},S_{N})>0$, $C\equiv C(\alpha, p,S_{D},S_{N})>0$ and $C'\equiv C'(\alpha, p,S_{D},S_{N})>0$ such that}
%is well-posed in
%$B_{p,p^{*}}^{1+\frac{1}{p}}({\Omega}_+,\mathbb{R}^{n})\times B_{p,p^{*}}^{\frac{1}{p}}({\Omega}_+),$
%where $p^{*} = \max\{2, p\}$,
%which means that \eqref{mixed homogeneous Brinkman system} has a unique solution
%\begin{equation}
%\label{solution-mixed}
%$
%({\bf u},\pi)\in B_{p,p^{*}}^{1+\frac{1}{p}}({\Omega}_+,\mathbb{R}^{n})\times  B_{p,p^{*}}^{\frac{1}{p}}({\Omega}_+),
%$
%\end{equation}
%where $p^{*} = \max\{2, p\}$,
%which satisfies {the inclusions ${M}({\textbf u}),{M}(\nabla {\textbf u}),{M}(\pi)\in L_p(\partial {\Omega})$} and the inequality
\begin{align}
\label{ineq-M}
&\|M(\nabla {\bf u})\|_{L_p(\partial \Omega )}+\|M({\bf u})\|_{L_p(\partial \Omega )}
+\|M(\pi )\|_{L_p(\partial \Omega )}
\leq C_M\left(\|{\mathbf h_{0}}\|_{H_p^{1}(S_{D}, \mathbb{R}^{n})}
+\|\mathbf g_{0}\|_{L_p(S_{N},\mathbb{R}^{n})}\right),\\
\label{continuity-data}
&\|{\bf u}\|_{B_{p,p^{*}}^{1+\frac{1}{p}}({\Omega}_+,\mathbb{R}^{n})}+\|\pi\|_{ B_{p,p^{*}}^{\frac{1}{p}}({\Omega}_+)}
\leq C\left(\|{\mathbf h_{0}}\|_{H_p^{1}(S_{D}, \mathbb{R}^{n})}+\|\mathbf g_{0}\|_{L_p(S_{N},\mathbb{R}^{n})}\right),\\
\label{tr-continuity3}
&\|\gamma_+{\bf u}\|_{H_{p}^1(\partial\Omega,\mathbb{R}^n)}+\|\mathbf t_\alpha ^+({\bf u},\pi)\|_{L_p(\partial\Omega,\mathbb{R}^{n}))}
\le C'\left(\|{\mathbf h_{0}}\|_{H_p^{1}(S_{D},\mathbb{R}^{n})}+\|\mathbf g_{0}\|_{L_p(S_{N},\mathbb{R}^{n})}\right).
\end{align}
\item[$(ii)$]
The mixed Dirichlet-Neumann problem for the Brinkman system
\eqref{mixed homogeneous Brinkman systemG} has a unique solution
$({\bf u},\pi)\in B_{p,p^{*}}^{1+\frac{1}{p}}({\Omega}_+,\mathbb{R}^{n})\times B_{p,p^{*}}^{\frac{1}{p}}({\Omega}_+).$
Moreover, the solution satisfies estimates \eqref{ineq-M}-\eqref{tr-continuity3}.
\end{itemize}
\end{theorem}
\begin{proof}
(i) By Theorem \ref{Th2.6} the mixed Dirichlet-Neumann problem \eqref{mixed homogeneous Brinkman system} is well-posed for $p=2$.
Then by Lemma \ref{Th2.5} (ii) and Theorem~\ref{Th2.1'} for $p=2$, the operator $\Upsilon_{{\rm nt};\alpha}:{\widetilde H}_2^0(S_{D},\mathbb{R}^{n})\to H_2^{1}(S_{D},\mathbb{R}^{n})$ is an isomorphism.
Moreover, by Lemma \ref{complex-interpolation}, the sets $\{{\widetilde H}_p^0(S_{D},\mathbb{R}^{n})\}_{p\geq 1}$ and $\{H_p^{1}(S_{D},\mathbb{R}^{n})\}_{p\geq 1}$ are complex interpolation scales.
%(cf., e.g., \cite[Proposition 4.2]{M-M}).
Then by the stability of the invertibility property given in Lemma \ref{Prop1.4}, there exists a {number $\varepsilon _1>0$}, such that the operator $\Upsilon_{{\rm nt};\alpha}:{\widetilde H}_p^0(S_{D},\mathbb{R}^{n})\to H_p^{1}(S_D, \mathbb{R}^{n})$ is an isomorphism as well, for any {$p\in (2-\varepsilon _1, 2+\varepsilon _1)$}.
Finally, by choosing the parameter ${\varepsilon }:=\min \{{\epsilon },\varepsilon _1\}>0$, where $\epsilon$ is the parameter in Theorem \ref{Th2.1'}, and by using again Lemma~\ref{Th2.5} (ii), we deduce the well-posedness result of the mixed Dirichlet-Neumann problem \eqref{mixed homogeneous Brinkman system}
and estimates \eqref{ineq-M}-\eqref{tr-continuity3}, whenever $p\in (2-\varepsilon ,2+\varepsilon )$.

(ii) Let $\varepsilon$ be as in the proof of item (i). Let $p\in (2-\varepsilon ,2+\varepsilon )$. Then Lemma~\ref{Th2.5} (i) implies that $\Upsilon_\alpha=\Upsilon_{{\rm nt};\alpha}$, and hence
$\Upsilon_\alpha:{\widetilde H}_p^0(S_{D},\mathbb{R}^{n})\to H_p^{1}(S_D, \mathbb{R}^{n})$
is an isomorphism, and by Lemma~\ref{Th2.5} (ii) the mixed Dirichlet-Neumann problem \eqref{mixed homogeneous Brinkman systemG} is well posed
and estimates \eqref{ineq-M}-\eqref{tr-continuity3} hold.
\hfill\end{proof}

\begin{rem}
\label{solution-Lp}
Under the conditions of Theorem $\ref{Th2.7}$, the {solution $({\bf u},\pi)$ of the mixed Dirichlet-Neumann  problem \eqref{mixed homogeneous Brinkman system}} %and \eqref{mixed homogeneous Brinkman systemG}}
can be expressed by the single layer velocity and pressure potentials
\begin{equation}
\label{solution-D-N-Lp}
{\bf u}={\bf V}_{\alpha
}\left({\mathcal S}_{\alpha }^{-1}({\mathbf h_{0}},\mathbf g_{0})\right),\ \
\pi =\mathcal Q^s_{\partial\Omega}\left({\mathcal S}_{\alpha }^{-1}({\mathbf h_{0}},\mathbf g_{ 0})\right),
\end{equation}
where the operator
\begin{align}
\label{isomorphism-D-N-Lp}
{\mathcal S}_{\alpha }:L_p(\partial {\Omega},{\mathbb R}^n)\to H^1_p(S_D,{\mathbb R}^n)\times L_p(S_N,{\mathbb R}^n),\ {\mathcal S}_{\alpha }\boldsymbol\Psi :=\left(\left({\mathcal V}_{\alpha }\boldsymbol\Psi
\right)\big|_{{S_D}},\left(\left(\frac{1}{2}{\mathbb I}+{{\bf K}}_{\alpha }^*\right)\boldsymbol\Psi \right)\Big|_{{S_N}}\right)
\end{align}
is an isomorphism.
Indeed, as shown in the proof of Theorem~$\ref{Th2.6}$, the operator ${\mathcal S}_{\alpha }:L_2(\partial {\Omega},{\mathbb R}^n)\to H^1_2(S_D,{\mathbb R}^n)\times L_2(S_N,{\mathbb R}^n)$ is an {isomorphism}, and then, {by using again Lemma \ref{complex-interpolation} and Lemma $\ref{Prop1.4}$}, we can extend the isomorphism property of the operator \eqref{isomorphism-D-N-Lp} to $L_p$-spaces, with $p$ in a neighborhood of $2$, which can be chosen to coincide with that in Theorem $\ref{Th2.7}$.
\end{rem}

\subsection{\bf Poisson problem of mixed Dirichlet-Neumann type for the Brinkman system with data in $L_p$-based spaces}
Having in view Theorem \ref{Th2.7}, we are now able {to consider} the well-posedness of the following Poisson problem of mixed Dirichlet-Neumann type for the Brinkman system in a creased Lipschitz domain ${\Omega}_+$, with data in some $L_p$-based spaces,
\begin{equation}
\label{Poisson-mixed-Brinkman-Lp}
\left\{
\begin{array}{l}
{\triangle {\bf u} - \alpha {\bf u} - \nabla \pi ={\mathbf f}\in L_p({\Omega}_+,{\mathbb R}^3)},\
{\rm{div}}\ {\bf u} = 0\ \mbox{ in }\  {\Omega}_+\\
{{\gamma}_{+}{\bf u}}|_{S_D}={\mathbf h_{0}}\in H_p^{1}(S_{D},\mathbb{R}^{n}) \\
{\mathbf t_{\alpha}^{+}({\bf u},\pi )}|_{S_N}=\mathbf g_{0}\in L_p(S_{N},\mathbb{R}^{n}).
%{\mathcal N}(\nabla {\bf u}),\ {\mathcal N}({\bf u}),\ {\mathcal N}(\pi)\in L_2(\partial {\Omega}).
\end{array}
\right.
\end{equation}
\begin{remark}
\label{def-Poisson-mixed-DN}
{(i) {By a \it solution} of the Poisson problem of mixed Dirichlet-Neumann type \eqref{Poisson-mixed-Brinkman-Lp} we mean a pair $({\bf u},\pi)\in {B_{p,p^*}^{1+\frac{1}{p}}({\Omega}_+,\mathbb{R}^{n})}\times {B_{p,p^*}^{\frac{1}{p}}({\Omega}_+)}$, where $p^{*}=\max\{2, p\}$, which satisfies the non-homogeneous Brinkman system in $\Omega _+$, the Dirichlet boundary condition on $S_D$ in the Gagliardo trace sense, and the Neumann boundary condition on $S_N$ in the canonical sense described in Definition $\ref{lem 1.6D}$}.

{(ii) If a pair $({\bf u},\pi)\in {B_{p,p^*}^{1+\frac{1}{p}}({\Omega}_+,\mathbb{R}^{n})}\times {B_{p,p^*}^{\frac{1}{p}}({\Omega}_+)}$, $p\in (1,\infty)$, solves the non-homogeneous Brinkman system in the first line of \eqref{Poisson-mixed-Brinkman-Lp} with  ${\bf f}\in L_p(\partial \Omega ,{\mathbb R}^n)$, then
$({\bf u},\pi)\in \mathfrak{B}_{p,p^*,\rm div}^{1+\frac{1}{p},0}(\Omega_+;{\mathcal L}_{\alpha})$
by Definition~\ref{2.5}.
Hence, by Lemma~\ref{trace-lemma-Besov}, Definition~\ref{lem 1.6D}, Lemma~\ref{lem 1.6} and the embeddings $B_{p,p^*}^{1+\frac{1}{p}}({\Omega}_+,\mathbb{R}^{n})\hookrightarrow B_{p,p^*}^{s+\frac{1}{p}}({\Omega}_+,\mathbb{R}^{n})$,
$\mathfrak{B}_{p,p^*,\rm div}^{1+\frac{1}{p},0}(\Omega_+;{\mathcal L}_{\alpha})\hookrightarrow
\mathfrak{B}_{p,p^*,\rm div}^{s+\frac{1}{p},-\frac{1}{p'}}(\Omega_+;{\mathcal L}_{\alpha})$, for any $0<s<1$,
the trace $\gamma_{+}{\bf u}$ and canonical conormal derivative $\mathbf t_{\alpha}^{+}({\bf u},\pi)$ are well defined and belong to $B_{p,p^*}^{s}(\partial \Omega ,\mathbb{R}^{n})$ and  $B_{p,p^*}^{s-1}(\partial \Omega ,\mathbb{R}^{n})$, respectively.
Thus, the boundary conditions in \eqref{Poisson-mixed-Brinkman-Lp} are well defined as well.
In what follows, we show that the sharper inclusions, $\gamma_{+}{\bf u}\in H_{p}^{1}(\partial \Omega ,\mathbb{R}^{n})$ and $\mathbf t_{\alpha}^{+}({\bf u},\pi)\in L_p(\partial \Omega ,\mathbb{R}^{n})$, hold if the spaces of the given boundary data in the boundary conditions are those mentioned in \eqref{Poisson-mixed-Brinkman-Lp}.}
\comment{
Then the embeddings ${B_{p,p^*}^{1+\frac{1}{p}}({\Omega}_+,\mathbb{R}^{n})}\times {B_{p,p^*}^{\frac{1}{p}}({\Omega}_+)}\hookrightarrow {B_{p,p^*}^{s+\frac{1}{p}}({\Omega}_+,\mathbb{R}^{n})}\times {B_{p,p^*}^{\frac{1}{p}-s}({\Omega}_+)}$ and $L_p(\partial \Omega ,{\mathbb R}^n)\hookrightarrow B_{p,p^*}^{s-1}({\Omega}_+,\mathbb{R}^{n})$
show that for any $({\bf u},\pi)\in {B_{p,p^*}^{1+\frac{1}{p}}({\Omega}_+,\mathbb{R}^{n})}\times {B_{p,p^*}^{\frac{1}{p}}({\Omega}_+)}$, the Gagliardo trace ${\gamma}_{+}{\bf u}$ belongs to the space $B_{p,p^*}^{s}(\partial \Omega ,\mathbb{R}^{n})$, while the corresponding canonical conormal derivative $\mathbf t_{\alpha}^{+}({\bf u},\pi )$ belongs to the Besov space $B_{p,p^*}^{s-1}(\partial \Omega ,\mathbb{R}^{n})$.
Thus, the boundary conditions in the mixed problem \eqref{Poisson-mixed-Brinkman-Lp} are considered in the narrow spaces $H_p^{1}(S_{D},\mathbb{R}^{n})\hookrightarrow B_{p,p^*}^{s}(S_D, \mathbb{R}^{n})$ and $L_p(S_{N},\mathbb{R}^{n})\hookrightarrow B_{p,p^*}^{s-1}(S_N,\mathbb{R}^{n})$, respectively.}
\end{remark}
%{\bl We next show that the Poisson problem \eqref{Poisson-mixed-Brinkman-Lp} has a unique solution represented by a combination between a Brinkman Newtonian potential with the density ${\bf f}$ and a Brinkman single layer potential with a density in the space $L_p(\partial \Omega ,{\mathbb R}^n)$.}
\begin{theorem}
\label{Poisson-mixed-Lp}
Assume that ${\Omega}_+\subset \mathbb{R}^{n}$ $(n\geq 3)$
is a bounded, creased Lipschitz domain with connected boundary $\partial {\Omega}$, and that $\partial {\Omega}$ is decomposed into two disjoint admissible patches $S_{D}$ and $S_{N}$. Then there exists a number {$\varepsilon >0$ such that for any $p\in (2-\varepsilon ,2+\varepsilon )$} and for all given data
$({{\mathbf f}},{\mathbf h_{0}},\mathbf g_{0})\in
L_p({\Omega}_+,{\mathbb R}^n)\times H_p^{1}(S_{D},\mathbb{R}^{n})\times L_p(S_{N},\mathbb{R}^{n})$
the Poisson problem of mixed Dirichlet-Neumann type \eqref{Poisson-mixed-Brinkman-Lp} has a solution
$({\bf u},\pi)\in B_{p,p^*}^{1+\frac{1}{p}}({\Omega}_+,\mathbb{R}^{n})\times B_{p,p^*}^{\frac{1}{p}}({\Omega}_+)$
%$($in the sense described above$)$
that can be represented in the form
\begin{align}
\label{Poisson-mixed-DN}
{{\bf u}={\mathbf N}_{\alpha ;{\Omega}_+}{\mathbf f}+{\bf V}_{\alpha
}\left({\mathcal S}_{\alpha }^{-1}({\mathbf h_{00}},\mathbf g_{00})\right),\ \pi ={\mathcal Q}_{{\Omega}_+}{\mathbf f}+\mathcal Q^s_{\partial\Omega}\left({\mathcal S}_{\alpha }^{-1}({\mathbf h_{00}},\mathbf g_{00})\right),}
\end{align}
{where ${\mathcal S}_{\alpha }:L_p(\partial {\Omega},{\mathbb R}^n)\to H^1_p(S_D,{\mathbb R}^n)\times L_p(S_N,{\mathbb R}^n)$ is the isomorphism defined in \eqref{isomorphism-D-N-Lp}, and}
\begin{align}
\label{modif-Lp}
{{\mathbf h_{00}}:={\mathbf h_{0}}-{\gamma_+}\left({\mathbf N}_{\alpha ;{\Omega}_+}{{\mathbf f}}\right)|_{S_D}\in H^1_p(S_D,{\mathbb R}^n),\ \
{\mathbf g_{00}}:=\mathbf g_{0}-{\mathbf t_{\alpha}^{+}}\left({\mathbf N}_{\alpha ;{\Omega}_+}{{\mathbf f}},{\mathcal Q}_{\alpha ;{\Omega}_+}{{\mathbf f}}\right)|_{S_N}\in L_p(S_N,{\mathbb R}^n).}
\end{align}
Moreover,
%\begin{equation}
%\label{Poisson-mixed-1-Lp}
%$({\bf u},\pi)\in {B_{p,p^*}^{1+\frac{1}{p}}({\Omega}_+,\mathbb{R}^{n})}\times {B_{p,p^*}^{\frac{1}{p}}({\Omega}_+)},$
%\end{equation}
%{\rd which satisfies the Dirichlet and Neumann boundary conditions in the sense of the non-tangential limit.}
{the solution $({\bf u},\pi)$
is unique in the space} ${B_{p,p^*}^{1+\frac{1}{p}}({\Omega}_+,\mathbb{R}^{n})}\times {B_{p,p^*}^{\frac{1}{p}}({\Omega}_+)},$ and there exist some constants\\ $C\equiv C(\alpha, p,S_{D},S_{N})>0$ and $C'\equiv C'(\alpha, p,S_{D},S_{N})>0$ such that the following inequalities hold
%Moreover, the solution satisfies the inequalities
\begin{align}
\label{continuity-data-p}
\|{\bf u}\|_{B_{p,p^{*}}^{1+\frac{1}{p}}({\Omega}_+,\mathbb{R}^{n})}+\|\pi\|_{ B_{p,p^{*}}^{\frac{1}{p}}({\Omega}_+)}
&\leq C\left(\mathbf f\|_{L_p(\Omega_+,\mathbb{R}^{n})}+ \|{\mathbf h_{0}}\|_{H_p^{1}(S_{D}, \mathbb{R}^{n})} +\|\mathbf g_{0}\|_{L_p(S_{N},\mathbb{R}^{n})}\right),
\\
\label{cond-p}
\|\gamma_+{\bf u}\|_{H_{p}^1(\partial\Omega,\mathbb{R}^n)}+\|\mathbf t_\alpha ^+({\bf u},\pi)\|_{L_p(\partial\Omega,\mathbb{R}^{n})}
&\le C'\left(\mathbf f\|_{L_p(\Omega_+,\mathbb{R}^{n})}+\|{\mathbf h_{0}}\|_{H_p^{1}(S_{D},\mathbb{R}^{n})}+\|\mathbf g_{0}\|_{L_p(S_{N},\mathbb{R}^{n})}\right).
\end{align}
%where the trace and the conormal derivative are considered in the sense of non-tangential limit.
In addition, there exists a linear continuous operator
\begin{align}
{\mathcal A}_p:{L_p({\Omega}_+,\mathbb{R}^{n})}\times H_p^{1}(S_{D},\mathbb{R}^{n})\times L_p(S_{N},\mathbb{R}^{n})\to {B_{p,p^*}^{1+\frac{1}{p}}({\Omega}_+,\mathbb{R}^{n})}\times { B_{p,p^*}^{\frac{1}{p}}({\Omega}_+)}\nonumber
\end{align}
delivering this solution{, i.e., $\mathcal A_p({{\mathbf f}},{\mathbf h_{0}},\mathbf g_{0})=({\bf u},\pi)$}.
\end{theorem}
\begin{proof}
Let $\varepsilon >0$ as in Theorem \ref{Th2.7}, and let $p\in (2-\varepsilon ,2+\varepsilon )$.
We will look for a solution of problem \eqref{Poisson-mixed-Brinkman-Lp} in the form
\begin{align}
\label{Poisson-mixed-2-Lp}
{\bf u}={\mathbf N}_{\alpha ;{\Omega}_+}{{\mathbf f}}+{\bf v},\ \pi ={\mathcal Q}_{{\Omega}_+}{{\mathbf f}}+q,
\end{align}
where the Newtonian velocity and pressure potentials
${\mathbf N}_{\alpha ;{\Omega}_+}{{\mathbf f}}$ and ${\mathcal Q}_{\Omega_+}{{\mathbf f}}$ are defined by \eqref{Newtonian-1a-Lp}.
{By properties \eqref{Newtonian-1b-Lpm}-\eqref{Newtonian-1p},
Corollary~\ref{C2.16} and Remark~\ref{gammaN}}, we obtain that
\begin{align}
\label{Newtonian-2-Lp}
&\triangle {\mathbf N}_{\alpha ;{\Omega_+}}{{\mathbf f}}-\alpha {\mathbf N}_{\alpha ;{\Omega_+}}{{\mathbf f}}-\nabla {\mathcal Q}_{\Omega_+}{{\mathbf f}}={{\mathbf f}},\
{\rm{div}}\ {\mathbf N}_{\alpha ;{\Omega_\pm}}{{\mathbf f}}=0\ \mbox{ in }\ {\Omega}_+,\\
\label{Poisson-mixed-1-LpN}
&{(\mathbf N_{\alpha;\Omega_+}{{\mathbf f}},{\mathcal Q}_{\Omega_+}{{\mathbf f}})\in {H_p^2(\Omega _+,{\mathbb R}^n)\times H_p^1(\Omega _+)}\hookrightarrow {B_{p,p^*}^{1+\frac{1}{p}}({\Omega}_+,\mathbb{R}^{n})}\times B_{p,p^*}^{\frac{1}{p}}({\Omega}_+),}\\
\ \
%\label{cond-p}
&{{\gamma}_{+}\mathbf N_{\alpha;{\Omega}_+}{{\mathbf f}}\in H^1_p(\partial {\Omega},{\mathbb R}^n),\ \
{\mathbf t}_\alpha^{+}(\mathbf N_{\alpha;{\Omega}_+}{{\mathbf f}},{\mathcal Q}_{\Omega_+}{{\mathbf f}})
\in L_p(\partial {\Omega},{\mathbb R}^n),}
\end{align}
{where $\gamma _+$ is the Gagliardo trace operator from $H_p^2(\Omega _+,{\mathbb R}^n)$ to $H^1_p(\partial {\Omega},{\mathbb R}^n)$.} %, and ${\mathbf t}_\alpha^{+}$ is the conormal derivative operator in the sense of Definition \eqref{lem 1.6D}.}
Then
%by \eqref{Poisson-mixed-2-Lp}, \eqref{Newtonian-2-Lp} and \eqref{Newtonian-3-Lp}
the mixed Poisson problem
%of mixed type
\eqref{Poisson-mixed-Brinkman-Lp} reduces to the mixed
%Dirichlet-Neumann
problem for the corresponding homogeneous system,
\begin{equation}
\label{Newtonian-4-Lp}
\left\{
\begin{array}{l}
\triangle {\bf v} - \alpha {\bf v} - \nabla q= {\bf 0},\
{\rm{div}}\ {\bf v}= 0\ \mbox{ in }\  {\Omega}_+,\\
{\gamma_+{\bf v}}|_{S_D}={\mathbf h_{00}}\in H_p^{1}(S_{D},\mathbb{R}^{n}), \\
{\mathbf t_{\alpha}^+({\bf v},q)}|_{S_N}={\mathbf g_{00}}\in L_p(S_{N},\mathbb{R}^{n}),
%\\
%{M}(\nabla {\bf v}),\ {M}({\bf v}),\ {M}(q)\in L_p(\partial {\Omega}),
\end{array}
\right.
\end{equation}
where ${\mathbf h_{00}}\in H^1_p(S_D,{\mathbb R}^n)$ and ${\mathbf g_{00}}\in L_p(S_N,{\mathbb R}^n)$ are given by \eqref{modif-Lp}, {and these inclusions follow from \eqref{Poisson-mixed-1-LpN}.}
\comment{
\begin{align}
\label{modif-Lp}
{\mathbf h_{00}}:={\mathbf h_{0}}-{\gamma_+}\left({\mathbf N}_{\alpha ;{\Omega}_+}{{\mathbf f}}\right)|_{S_D}\in H^1_p(S_D,{\mathbb R}^n),\ \
{\mathbf g_{00}}:=\mathbf g_{0}-{\mathbf t_{\alpha}^{+}}\left({\mathbf N}_{\alpha ;{\Omega}_+}{{\mathbf f}},{\mathcal Q}_{\alpha ;{\Omega}_+}{{\mathbf f}}\right)|_{S_N}\in L_p(S_N,{\mathbb R}^n),
\end{align}}

By Theorem \ref{Th2.7}{(ii)},
\comment{According to Lemma \ref{L3.6} the single layer velocity and pressure potentials
\begin{equation}
\label{solution-D-N-Lp0}
{{\bf v}={\bf V}_{\alpha }\left({\mathcal S}_{\alpha }^{-1}({\mathbf h_{00}},\mathbf g_{00})\right),\ \
q=\mathcal Q^s_{\partial\Omega}\left({\mathcal S}_{\alpha }^{-1}({\mathbf h_{00}},\mathbf g_{00})\right),}
\end{equation}
where ${\mathcal S}_{\alpha }$ from $L_p(\partial {\Omega},{\mathbb R}^n)$ to $H^1_p(S_D,{\mathbb R}^n)\times L_p(S_N,{\mathbb R}^n)$ is the isomorphism defined by \eqref{isomorphism-D-N-Lp}, determine a solution of
problem \eqref{Newtonian-4-Lp} in the sense described in Remark \ref{def-Poisson-mixed-DN}.
Moreover, in view of Theorem \ref{layer-potential-properties} (i) and Lemma \ref{L3.6}, the pair
$({\bf v},q)$ belongs to the space ${B_{p,p^*}^{1+\frac{1}{p}}({\Omega}_+,\mathbb{R}^{n})}\times {B_{p,p^*}^{\frac{1}{p}}({\Omega}_+)}$, is the unique solution of the problem \eqref{Newtonian-4-Lp} in this space, and satisfies the following
%the embeddings ${M}({\bf v}),M(\nabla {\bf v}),{M}(q)\in L_p(\partial {\Omega})$ and
estimates
} %\comment end
there exists a unique solution $({\bf v},q)\in{B_{p,p^*}^{1+\frac{1}{p}}({\Omega}_+,\mathbb{R}^{n})}\times {B_{p,p^*}^{\frac{1}{p}}({\Omega}_+)}$ of problem \eqref{Newtonian-4-Lp}, and it satisfies the following estimates
\begin{align}
\label{estimate-mixed-Poisson-1-Lp}
&\|{\bf v}\|_{{B_{p,p^*}^{1+\frac{1}{p}}({\Omega}_+,\mathbb{R}^{n})}}+ \|q\|_{{ B_{p,p^*}^{\frac{1}{p}}({\Omega}_+)}}\leq c\left(\|{\mathbf h_{00}}\|_{H_p^{1}(S_{D},\mathbb{R}^{ n})}+\|\mathbf g_{00}\|_{L_p(S_{N}, \mathbb{R}^{n})}\right),\\
\label{tr-continuity4}
&\|\gamma_+{\bf v}\|_{H_{p}^1(\partial\Omega,\mathbb{R}^n)}+\|\mathbf t_\alpha ^+({\bf v},q)\|_{L_p(\partial\Omega,\mathbb{R}^{n}))}
\le c'\left(\|{\mathbf h_{00}}\|_{H_p^{1}(S_{D},\mathbb{R}^{n})}+\|\mathbf g_{00}\|_{L_p(S_{N},\mathbb{R}^{n})}\right),
\end{align}
with some constants $c\equiv c(\alpha, p,S_{D},S_{N})>0$ and $c'\equiv c'(\alpha, p,S_{D},S_{N})>0$.

According to Lemma \ref{L3.6} the single layer velocity and pressure potentials
\begin{equation}
\label{solution-D-N-Lp0}
{{\bf v}={\bf V}_{\alpha }\left({\mathcal S}_{\alpha }^{-1}({\mathbf h_{00}},\mathbf g_{00})\right),\ \
q=\mathcal Q^s_{\partial\Omega}\left({\mathcal S}_{\alpha }^{-1}({\mathbf h_{00}},\mathbf g_{00})\right),}
\end{equation}
where ${\mathcal S}_{\alpha }:L_p(\partial {\Omega},{\mathbb R}^n)\to H^1_p(S_D,{\mathbb R}^n)\times L_p(S_N,{\mathbb R}^n)$ is the isomorphism defined by \eqref{isomorphism-D-N-Lp}, determine the unique solution of problem \eqref{Newtonian-4-Lp}. Moreover, in view of Theorem \ref{layer-potential-properties} (i) and Lemma \ref{L3.6}, the pair
$({\bf v},q)$ given by \eqref{solution-D-N-Lp0} belongs indeed to the space ${B_{p,p^*}^{1+\frac{1}{p}}({\Omega}_+,\mathbb{R}^{n})}\times {B_{p,p^*}^{\frac{1}{p}}({\Omega}_+)}$,

Therefore, there exists a solution $({\bf u},\pi)\in {B_{p,p^*}^{1+\frac{1}{p}}({\Omega}_+,\mathbb{R}^{n})}\times  B_{p,p^*}^{\frac{1}{p}}({\Omega}_+)$
of the {mixed Poisson problem \eqref{Poisson-mixed-Brinkman-Lp}, which is given by representation \eqref{Poisson-mixed-DN} and satisfies estimates \eqref{continuity-data-p} and \eqref{cond-p}. The uniquness result of such a solution follows from Theorem \ref{Th2.7} (ii).
Moreover, linearity and continuity of the Newtonian potential operators  \eqref{Newtonian-1}, \eqref{Newtonian-1p} and estimate \eqref{estimate-mixed-Poisson-1-Lp}
imply the continuity of the operator ${\mathcal{A}_p}$ delivering such a solution.}
\hfill\end{proof}

\section{Mixed Dirichlet-Neumann problem for the semilinear Darcy-Forchheimer-Brinkman system in Besov spaces}
\label{section3}
Next we consider the mixed Dirichlet-Neumann problem for the semilinear Darcy-Forchheimer-Brinkman system
\begin{equation}
\label{D-F-B system}
%\left\{
%\begin{array}{l}
\triangle {\bf u} - \alpha {\bf u} - {\beta |{\bf u}|{\bf u}}-\nabla \pi ={\bf f},\ \
{\rm{div}}\ {\bf u} = 0\ \mbox{ in }\  {\Omega}_+.
%\end{array}
%\right.
\end{equation}
Such a nonlinear system describes flows in porous media saturated with viscous incompressible fluids (see, e.g., \cite[p.17]{Ni-Be}), and the constants $\alpha ,\beta >0$ are related by the physical properties of such a porous medium, as they describe the viscosity and the convection of the fluid flow.

Due to some embedding results that play a main role in our {arguments, we will restrict our analysis in this section to the cases $n=3$}.

A numerical study of a mixed Dirichlet-Neumann problem for system \eqref{D-F-B system} in the particular case of a two-dimensional square cavity driven by a moving wall has been obtained in \cite{Robert}. Well-posedness and numerical results for an extended nonlinear system, called the Darcy-Forchheimer-Brinkman system, where both semilinear and nonlinear terms $|{\bf u}|{\bf u}$ and $({\bf u} \cdot \nabla ){\bf u}$ are involved, have been obtained in \cite{G-K-W}, and boundary value problems of Robin type for the Darcy-Forchheimer-Brinkman system with data in $L_2$-based {Bessel potential (Sobolev)} spaces have been studied in \cite{K-L-W3, K-L-W2}.

In what follows, we extend an existence and uniqueness result obtained in \cite[Theorem 7.1]{K-L-W2} for the mixed problem \eqref{Darcy-Brinkamn problem} with the given data in $L_2$-based Sobolev spaces, to the case of $L_p$-based {Bessel potential} spaces, i.e., when the given boundary data $({\mathbf h_{ 0}},\mathbf g_{0})$ belong to the space $H_p^{1}(S_{D},\mathbb{R}^{n}) \times L_p(S_{N},\mathbb{R}^{n})$, with {$p\in \left(2-\varepsilon ,2+\varepsilon \right)$}, and the parameter $\varepsilon >0$ as in Theorem \ref{Poisson-mixed-Lp}. In addition, the given data should be sufficiently small in a sense that will be specified below.
\begin{theorem}
\label{Theorem 3.2}
Assume that ${\Omega}_+\subset \mathbb{R}^{3}$ is a bounded creased Lipschitz domain with connected boundary $\partial {\Omega}$, and that $\partial {\Omega}$ is decomposed into two disjoint admissible patches $S_{D}$ and $S_{N}$. Let $\alpha ,\beta >0$ be given constants.
Then there exists a number $\varepsilon >0$ such that for any $p\in \left(2-\varepsilon ,2+\varepsilon \right)$ {and $p^{*} = \max\{2, p\}$,} there exist two constants {$\zeta _p\equiv \zeta _p({\Omega}_+,\alpha,\beta, p)>0$ and $\eta _p\equiv \eta _p({\Omega}_+,\alpha,\beta, p)>0$} with the property that for all given data
$({\mathbf f,}{\mathbf h_{0}},\mathbf g_{0})\in
{L_p(\Omega_+,\mathbb{R}^3)\times}H_p^{1}(S_{D},\mathbb{R}^{3})\times L_p(S_{N},\mathbb{R}^{3})$
satisfying the condition
\begin{equation}
\label{inequality p2 1}
\|{\mathbf h_{0}}\|_{H_p^{1}(S_{D},\mathbb{R}^{3})}+\|\mathbf g_{0}\|_{L_p(S_{N},\mathbb{R}^{3})}
{+\|\mathbf f\|_{L_p(\Omega_+,\mathbb{R}^3)}}
\leq {\zeta _p},
\end{equation}
the mixed Dirichlet-Neumann problem for the semilinear Darcy-Forchheimer-Brinkman system
\begin{equation}
\label{Darcy-Brinkamn problem}
\left\{
\begin{array}{l}
\triangle {\bf u} - \alpha {\bf u} - {\beta |{\bf u}|{\bf u}} - \nabla \pi ={\bf f},\ {\rm{div}}\ {\bf u}=0\ \mbox{ in }\ {\Omega}_+,\\
{\gamma}_{+}{\bf u}|_{S_D}={\mathbf h_{0}}\ \mbox{ on }\ S_D\\
\mathbf t_{\alpha}^{+}({\bf u},\pi )|_{S_N}=\mathbf g_{0}\ \mbox{ on }\ S_N
\end{array}
\right.
\end{equation}
has a unique solution $({\bf u},\pi)\in {B_{p,p^*}^{1+\frac{1}{p}}({\Omega}_+,\mathbb{R}^{n})\times B_{p,p^*}^{\frac{1}{p}}({\Omega}_+)}$, which satisfies the inequality
\begin{equation}
\label{inequality p2 2-Lp}
\|{\bf u}\|_{{B_{p,p^*}^{1+\frac{1}{p}}({\Omega}_+,\mathbb{R}^{n})}}\leq {\eta _p}.
\end{equation}
{Moreover, ${{\gamma}_{+}{\bf u}\in H^1_p(\partial {\Omega},{\mathbb R}^{n})},\ {{\mathbf t}_\alpha^{+}({\bf u},\pi )\in L_p(\partial {\Omega},{\mathbb R}^{n})}$
and} the solution depends continuously on the given data, which means that there exists {some constants $C_*\equiv C_*(\Omega_+,\alpha,\beta,p)>0$ and $C_*'\equiv C_*(\Omega_+,\alpha,\beta,p)>0$} such that
\begin{align}
\label{estimate-D-B-F-new1-new1-D-Rn2-Lp}
\|{\bf u}\|_{{B_{p,p^*}^{1+\frac{1}{p}}({\Omega}_+,\mathbb{R}^{n})}}+\|\pi \|_{{ B_{p,p^*}^{\frac{1}{p}}({\Omega}_+)}}
&\leq C_*\left({\|\mathbf f\|_{L_p(\Omega_+,\mathbb{R}^{n})}+}
\|{\mathbf h_{ 0}}\|_{H_p^{1}(S_D,\mathbb{R}^{n})}
+\|\mathbf g_{0}\|_{L_p(S_N,\mathbb{R}^{n})}\right),\\
\label{cond-pDF}
\|\gamma_+{\bf u}\|_{H_{p}^1(\partial\Omega,\mathbb{R}^n)}
+\|\mathbf t^+_\alpha({\bf u},\pi)\|_{L_p(\partial\Omega,\mathbb{R}^{n}))}
&\le C_*'\left(\|\mathbf f\|_{L_p(\Omega_+,\mathbb{R}^{n})}
+\|{\mathbf h_{0}}\|_{H_p^{1}(S_{D},\mathbb{R}^{n})}+\|\mathbf g_{0}\|_{L_p(S_{N},\mathbb{R}^{n})}\right).
\end{align}
\end{theorem}
\begin{proof}
We use {the arguments similar to those}
%developed
in the proof of \cite[Theorem 5.2]{K-L-M-W} devoted to transmission problems with Lipschitz interface in ${\mathbb R}^{n}$ for the Stokes and Darcy-Forchheimer-Brinkman systems {in $L_2-$based Sobolev spaces}.

\comment{
Let $\varepsilon >0$ be as in Theorem \ref{Poisson-mixed-Lp} and such that $2-\varepsilon \geq \frac{3}{2}$.

(i) If {$p\in [2,2+\varepsilon )$} then $p^*=\max\{p,2\}=p$, and by \eqref{Z2} and the Sobolev embedding theorem, we obtain the inclusions
\begin{align*}
%\label{sm}
{B_{p,p^*}^{1+\frac{1}{p}}({\Omega}_+,\mathbb{R}^{n})}=B_{p,p}^{1+\frac{1}{p}}({\Omega}_+,\mathbb{R}^{n}) =W_{p}^{1+\frac{1}{p}}({\Omega}_+,\mathbb{R}^{n})\hookrightarrow W^1_p({\Omega}_+,\mathbb{R}^{n}) \hookrightarrow L_{\frac{3p}{3-p}}({\Omega}_+,\mathbb{R}^{n}).%{ \hookrightarrowL_{2p}({\Omega}_+,\mathbb{R}^{n})},
\end{align*}
Note that the embedding $B_{p,p}^{1+\frac{1}{p}}({\Omega}_+,\mathbb{R}^{n})\hookrightarrow L_{\frac{3p}{3-p}}({\Omega}_+,\mathbb{R}^{n})$ follows also directly from \eqref{embed-2}.

(ii) If {$p\in \left(2-\varepsilon ,2\right)$} then $p^*=\max\{p,2\}=2$ and by \eqref{Z1}, the first embedding in \eqref{embed-3}, the second embedding in \eqref{embed-4}, and by \eqref{embed-5}, we obtain the following inclusions
\begin{align*}
%\label{sm-1}
{B_{p,p^*}^{1+\frac{1}{p}}({\Omega}_+,\mathbb{R}^{n})}=
B_{p,2}^{1+\frac{1}{p}}({\Omega}_+,\mathbb{R}^{n})\hookrightarrow
B_{p,\infty }^{1+\frac{1}{p}}({\Omega}_+,\mathbb{R}^{n})\hookrightarrow H^1_p({\Omega}_+,\mathbb{R}^{n}) =W^1_p({\Omega}_+,\mathbb{R}^{n}){\hookrightarrow
L_{\frac{3p}{3-p}}({\Omega}_+,\mathbb{R}^{n})}.%\hookrightarrow L_{{2p}}({\Omega}_+,\mathbb{R}^{n})}.
\end{align*}
In addition, $L_{\frac{3p}{3-p}}({\Omega}_+,\mathbb{R}^{n})\hookrightarrow L_{2p}({\Omega}_+,\mathbb{R}^{n})$, as follows from the assumption $p>2-\varepsilon $ and the choice of $\varepsilon $ such that $2-\varepsilon \geq \frac{3}{2}$.

Hence, for any {$p\in \left(2-\varepsilon ,2+\varepsilon \right)$} {(and $\varepsilon $ such that $2-\varepsilon \geq \frac{3}{2}$)}, we obtain the inclusions
\begin{align}
\label{sm-1}
{B_{p,p^*}^{1+\frac{1}{p}}({\Omega}_+,\mathbb{R}^{n})}\hookrightarrow L_{\frac{3p}{3-p}}({\Omega}_+,\mathbb{R}^{n})\hookrightarrow L_{2p}({\Omega}_+,\mathbb{R}^{n}).
\end{align}
} %\comment end

{According  to \eqref{embed-2} and the second formula in \eqref{embed-3}, for $n\le 5$ and $p>3/2$, we obtain the following continuous embeddings,}
\begin{align}
\label{sm-1}
B_{p,p^*}^{1+\frac{1}{p}}({\Omega}_+,\mathbb{R}^{n})
\hookrightarrow B^0_{2p,\min\{2p,(2p)'\}}({\Omega}_+,\mathbb{R}^{n})
\hookrightarrow H^0_{2p}({\Omega}_+,\mathbb{R}^{n})
= L_{2p}({\Omega}_+,\mathbb{R}^{n}).
\end{align}

Now, by \eqref{sm-1} and the H\"{o}lder inequality we obtain the estimates
\begin{align}
\label{3.0.3-Lp}
\|\ |{\bf v}|{\bf w}\ \|_{L_p({\Omega}_+,\mathbb{R}^{n})}
&\leq \|{\bf v}\|_{L_{2p}({\Omega}_+,\mathbb{R}^{n})}\|{\bf w}\|_{L_{2p}({\Omega}_+,\mathbb{R}^{n})}
%\nonumber\\
%&\leq c_0'\|{\bf v}\|_{L_{\frac{3p}{3-p}}({\Omega}_+,\mathbb{R}^{n})}\|{\bf w}\|_{L_{\frac{3p}{3-p}}({\Omega}_+,\mathbb{R}^{n})}\nonumber\\
%&
{\leq} c_1'\|{\bf v}\|_{B_{p,p^*}^{1+\frac{1}{p}}({\Omega}_+,\mathbb{R}^{n})}\|{\bf w}\| _{ B_{p,p^*}^{1+\frac{1}{p}}({\Omega}_+,\mathbb{R}^{n})},\ \forall \ {\bf v},{\bf w}\in { B_{p,p^*}^{1+\frac{1}{p}}({\Omega}_+,\mathbb{R}^{n}),}
\end{align}
with some constants $c_{k}'\equiv c_{k}'({\Omega}_+,p)>0$, $k=0,1$, implying that
%. Then by \eqref{3.0.3-Lp}, we deduce that
%\begin{equation}
%\label{3.0.5-Lp}
$
|{\bf v}|{\bf w}\in  L_p({\Omega}_+,\mathbb{R}^{n}), \ \forall \ {\bf v},\, {\bf w}\in B_{p,p^*}^{1+\frac{1}{p}}({\Omega}_+,\mathbb{R}^{n}).
$
%\end{equation}

Next, for a given fixed ${\mathbf v}\in B_{p,p}^{1+\frac{1}{p}}({\Omega}_+,\mathbb{R}^{n})$, we consider the linear Poisson problem of mixed type for the Brinkman system
\begin{equation}
\label{Newtonian-D-B-F-new2-D-Rn-Lp}
\left\{\begin{array}{lll}
\triangle {{\bf v}^0}-\alpha {{\bf v}^0}-\nabla {\pi^0}=
{\mathbf f+}\beta |{\bf v}|{\bf v}\ \mbox{ in }\ {\Omega}_+,\\
{\gamma}_{+}{{\bf v}^0}|_{S_D}
={\mathbf h_{0}}\in H^1_p(S_D,{\mathbb R}^{n}),\\
\mathbf t_{\alpha}^{+}\left({{\bf v}^0},{\pi^0}\right)|_{S_N}=\mathbf g_{0}\in L_p(S_N,{\mathbb R}^{n}),
%{{\mathcal N}(\nabla {\bf v}^0),\ {\mathcal N}({\bf v}^0),\ {\mathcal N}(\pi^0)\in L_2(\partial {\Omega}_+)}
\end{array}\right.
\end{equation}
with the unknown fields $({\bf v}^0,\pi^0)\in { B_{p,p^*}^{1+\frac{1}{p}}({\Omega}_+,\mathbb{R}^{n})\times B_{p,p^*}^{\frac{1}{p}}({\Omega}_+)}$.
%, such that
%\begin{align}
%\label{add}
%{{{\gamma}_{+}{\bf v}^0\in H^1_p(\partial {\Omega},{\mathbb R}^{n})},\ {{\mathbf t}_\alpha^{+}({\bf v}^0,\pi^0)\in L_p(\partial {\Omega},{\mathbb R}^{n})}.}
%\end{align}

{Let $2-\varepsilon<p<2+\varepsilon$, where $\varepsilon >0$ is as in Theorem \ref{Poisson-mixed-Lp} and such that $2-\varepsilon > \frac{3}{2}$.
Then by} Theorem \ref{Poisson-mixed-Lp}, problem \eqref{Newtonian-D-B-F-new2-D-Rn-Lp} with given data
$\left({\mathbf f+}\beta|{\bf v}|{\bf v},{\mathbf h_{0}},\mathbf g_{0}\right)
\in L_p({\Omega}_+,\mathbb{R}^{n})\times H^1_p(S_D,{\mathbb R}^{n})\times L_p(S_N,{\mathbb R}^{n})$
has a unique solution
\begin{align}
\label{solution-v0-Lp}
\left({\bf v}^0,\pi^0\right):=\left({\mathcal U}({\bf v}),{\mathcal P}({\bf v})\right)
={\mathcal A}_p\left({\mathbf f+}\beta |{\bf v}|{\bf v},\ {\mathbf h_{0}},\ \mathbf g_{0}\right)\in {\mathcal X}_p,
\end{align}
%satisfying conditions
%\begin{align}
%\label{add}
%{\gamma}_{+}{\bf v}^0\in H^1_p(\partial {\Omega},{\mathbb R}^{n}),\ {{\mathbf t}_\alpha^{+}({\bf v}^0,\pi^0)\in L_p(\partial {\Omega},{\mathbb R}^{n})},
%\end{align}
{where} the linear and continuous operator ${\mathcal A}_p:\mathcal Y_p\to\mathcal X_p$ has been defined in Theorem \ref{Poisson-mixed-Lp}, and
\begin{equation}
\label{SaNot-Lp}
\mathcal{X}_p:={{B_{p,p^*}^{1+\frac{1}{p}}({\Omega}_+,\mathbb{R}^{n})}\times B_{p,p^*}^{\frac{1}{p}}({\Omega}_+)},\ \
\mathcal{Y}_p:={L_p({\Omega}_+,\mathbb{R}^{n})}\times H_p^{1}(S_{D},\mathbb{R}^{n})\times L_p(S_{N}, \mathbb{R}^{n}).
\end{equation}
Hence, for fixed data
%$({\mathbf h_{0}},\mathbf g_{0})\in H^1_p(S_D,{\mathbb R}^{n})\times L_p(S_D,{\mathbb R}^{n})$,
$\left(\mathbf f,{\mathbf h_{0}},\mathbf g_{0}\right)
\in L_p({\Omega}_+,\mathbb{R}^{n})\times H^1_p(S_D,{\mathbb R}^{n})\times L_p(S_N,{\mathbb R}^{n})$,
the nonlinear operators
\begin{align}
\label{Newtonian-D-B-F-new3n}
&({\mathcal U},{\mathcal P}):{{B_{p,p^*}^{1+\frac{1}{p}}({\Omega}_+,\mathbb{R}^{n})}}\to {\mathcal X}_p
\end{align}
defined in \eqref{solution-v0-Lp}, are continuous and bounded,
%i.e., there exists a
%{and denoting
%\footnote{${\mathcal L}(\mathcal Y_p,\mathcal X_p):=\left\{T:\mathcal Y_p\to \mathcal X_p:T \mbox{ is a linear and continuous operator}\right\}$.}
%$C=\|{\mathcal A}_p\|_{{\mathcal L}(\mathcal Y_p,\mathcal X_p)}$,
we obtain,
%such that
\begin{align}
\label{estimate-D-B-F-new1-new1-D-Rn-Lp}
\big\|\big({\mathcal U}({\bf w}),{\mathcal P}({\bf w})\big)\big\|_{{\mathcal X}_p}
&\leq C\|\left({\mathbf f+}\beta |{\bf w}|{\bf w},{\mathbf h_{0}},\mathbf g_{0}\right)\|_{\mathcal Y_p} \nonumber\\
&\leq C\left({\|\left(\mathbf f,{\mathbf h_{0}},\mathbf g_{0}\right)\|_{L_p({\Omega}_+,\mathbb{R}^{n}) \times H^1_p(S_D,{\mathbb R}^{n})\times L_p(S_N,{\mathbb R}^{n})}}
+\beta \|\ |{\bf w}|{\bf w}\ \|_{L_p({\Omega}_+,{\mathbb R}^{n})}\right)\nonumber\\
&\leq C{\|\left(\mathbf f,{\mathbf h_{0}},\mathbf g_{0}\right)\|_{\mathcal Y_p}}
+CC_2\|{\bf w}\|^2_{B_{p,p^*}^{1+\frac{1}{p}}({\Omega}_+,\mathbb{R}^{n})},
\ \ \forall\ {\bf w}\in {B_{p,p^*}^{1+\frac{1}{p}}({\Omega}_+,\mathbb{R}^{n})},
\end{align}
\begin{align}
\label{cond-pDF2}
\|\gamma_+{\mathcal U}({\bf w})\|_{H_{p}^1(\partial\Omega,\mathbb{R}^n)}
+\|\mathbf t^+_\alpha\big({\mathcal U}({\bf w}),{\mathcal P}({\bf w})\big)\|_{L_p(\partial\Omega,\mathbb{R}^{n}))}
&\le C'\|\left(\mathbf f,{\mathbf h_{0}},\mathbf g_{0}\right)\|_{\mathcal Y_p}
+C'C_2\|{\bf w}\|^2_{B_{p,p^*}^{1+\frac{1}{p}}({\Omega}_+,\mathbb{R}^{n})}.
\end{align}
where $C_2:=c_1'\beta >0$, and $c_1'\equiv c_1'({\Omega}_+,p)>0$ is the constant that appears in inequality \eqref{3.0.3-Lp}, and $C$ can be taken as $C=\|{\mathcal A}_p\|_{{\mathcal L}(\mathcal Y_p,\mathcal X_p)}$.
In addition, in view of \eqref{solution-v0-Lp} and due to the definition of ${\mathcal A}_p$, we obtain that
$\left({\bf v}^0,\pi^0\right)=({\mathcal U}({\bf v}),{\mathcal P}({\bf v}))$
%\begin{equation}
%\label{Newtonian-D-B-F-new4n-Lp}
%\left\{\begin{array}{lll}
%\triangle {\mathcal U}({\bf v})-\alpha {\mathcal U}({\bf v})-\nabla {\mathcal P}({\bf v})={\mathbf f+}\beta |{\bf v}|{\bf v}\in L_p({\Omega}_+,{\mathbb R}^n),\\
%\left({\gamma}_{+}\left({{\mathcal U}({\bf v})}\right)\right)|_{S_D}={\mathbf h_{0}}\in H^1_p(S_D,{\mathbb R}^{n}),\\
%\left(\mathbf t_{\alpha}^{+}\left({\mathcal U}({\bf v}),{\mathcal P}({\bf v})\right)\right)|_{S_N}=\mathbf g_{0}\in L_p(S_N,{\mathbb R}^{n}),
%{{\mathcal N}(\nabla U({\bf v})),\ {\mathcal N}(U({\bf v})),\ {\mathcal N}(P({\bf v}))\in L_2(\partial {\Omega}_+).}
%\end{array}\right.
%\end{equation}
%and $({\mathcal U}({\bf v}),{\mathcal P}({\bf v}))$
{and satisfy \eqref{Newtonian-D-B-F-new2-D-Rn-Lp}}.
% conditions \eqref{add}.
Therefore, if we show that the nonlinear operator ${\mathcal U}$ has a fixed point ${\bf u}\in { B_{p,p^*}^{1+\frac{1}{p}}({\Omega}_+,\mathbb{R}^{n})}$, i.e., such that ${\mathcal U}({\bf u})={\bf u}$, then ${\bf u}$ together with the pressure function $\pi ={\mathcal P}({\bf u})$ determine a solution of the nonlinear mixed problem \eqref{Darcy-Brinkamn problem} in the space ${\mathcal X}_p$.
In order to show the existence of such a fixed point, we introduce the constants
\begin{align}
\label{Newtonian-D-B-F-new9n-Lp}
&\zeta _p:=\frac{3}{16C_2C^2}>0,\ \ \eta _p:=\frac{1}{4C_2C}>0
\end{align}
(cf. \cite[Theorem 5.2]{K-L-M-W}) and the closed ball
\begin{align}
\label{gamma-0n-Lp}
\mathbf{B}_{\eta _p}:=\left\{{\bf w}\in {B_{p,p^*}^{1+\frac{1}{p}}({\Omega}_+,\mathbb{R}^{n})}:\|{\bf w}\|_{B_{p,p^*}^{1+\frac{1}{p}}({\Omega}_+,\mathbb{R}^{n})}\leq \eta _p\right\},
\end{align}
and assume that the given data satisfy the inequality
\begin{align}
\label{cond-small-sm-msn-Lp}
{\|\left(\mathbf f,{\mathbf h_{0}},\mathbf g_{0}\right)\|_{\mathcal Y_p}}\leq \zeta _p.
\end{align}
Then by (\ref{estimate-D-B-F-new1-new1-D-Rn-Lp}), \eqref{Newtonian-D-B-F-new9n-Lp}-(\ref{cond-small-sm-msn-Lp}) we deduce that
\begin{align}
\label{Newtonian-D-B-F-new9-new-crackn-Lp}
\|\left({{\mathcal U}}({\bf w}),{\mathcal P}({\bf v})\right)\|_{{\mathcal X}_p}\leq {\frac{1}{4C_2C}}= \eta _p,\ \forall \ {\bf w}\in \mathbf{B}_{\eta _p}.
\end{align}
%and hence that $\|{\mathcal U}({\bf w})\|_{{B_{p,{ p^*}}^{1+\frac{1}{p}}({\Omega}_+,\mathbb{R}^{n})}}\leq \eta _p$ for any ${\bf w}\in \mathbf{B}_{\eta _p}$.
Consequently, ${\mathcal U}$ maps $\mathbf{B}_{\eta _p}$ into $\mathbf{B}_{\eta _p}$.

Moreover, we now prove that ${\mathcal U}$ is a contraction on $\mathbf{B}_{\eta _p}$. Indeed, by using the expression of ${\mathcal U}$ given in (\ref{solution-v0-Lp}), the linearity and continuity of the operator ${\mathcal A}_p$, and inequality \eqref{3.0.3-Lp}, we obtain that
\begin{align}
\label{4.30n-Lp}
\|{\mathcal U}({\bf v})-{\mathcal U}({\bf w})\|_{B_{p,{ p^*}}^{1+\frac{1}{p}}({\Omega}_+,\mathbb{R}^{n})}
&\leq \|{\mathcal A}_p\left(\beta |{\bf v}|{\bf v}-\beta |{\bf w}|{\bf w},\ {\bf 0},\ {\bf 0}\right)\|_{B_{p,{p^*}}^{1+\frac{1}{p}}({\Omega}_+,\mathbb{R}^{n})}\nonumber\\
&\leq C\beta \|\ |{\bf v}|{\bf v}-|{\bf w}|{\bf w}\ \|_{L_p({\Omega}_+,{\mathbb R}^{n})}
=C\beta \|\ (|{\bf v}|-|{\bf w}|){\bf v}+|{\bf w}|({\bf v}-{\bf w})\ \|_{L_p({\Omega}_+,{\mathbb R}^{n})} \nonumber
\\
&\leq Cc_1'\beta \Big(\|{\bf v}\|_{B_{p,{p^*}}^{1+\frac{1}{p}}({\Omega}_+,\mathbb{R}^{n})}+\|{\bf w}\|_{B_{p,{p^*}}^{1+\frac{1}{p}}({\Omega}_+,\mathbb{R}^{n})}\Big)
\|{\bf v}-{\bf w}\|_{B_{p,{p^*}}^{1+\frac{1}{p}}({\Omega}_+,\mathbb{R}^{n})}\nonumber
\\
& \leq 2\eta _pCC_2\|{\bf v}-{\bf w}\|_{B_{p,{ p^*}}^{1+\frac{1}{p}}({\Omega}_+,\mathbb{R}^{n})} =\frac{1}{2}\|{\bf v}-{\bf w}\|_{B_{p,{p^*}}^{1+\frac{1}{p}}({\Omega}_+,\mathbb{R}^{n})},
\ \forall\ {\bf v},{\bf w}\in \mathbf{B}_{\eta _p},
\end{align}
%where $C=\|{\mathcal A}_p\|_{{\mathcal L}(\mathcal Y_p,\mathcal X_p)}$ and $C_2=c_1'\beta $
{see also \eqref{estimate-D-B-F-new1-new1-D-Rn-Lp}}.
%Hence, ${\mathcal U}:\mathbf{B}_{\eta _p}\to \mathbf{B}_{\eta _p}$ is a contraction, as desired.
Then the Banach-Caccioppoli fixed point theorem implies that there exists a unique fixed point ${\bf u}\in \mathbf{B}_{\eta _p}$ of ${\mathcal U}$, i.e., ${\mathcal U}({\bf u})={\bf u}$.
Moreover, ${\bf u}$ and the pressure function $\pi ={\mathcal P}({\bf u})$, given by \eqref{solution-v0-Lp}, determine a solution of the semilinear problem \eqref{Darcy-Brinkamn problem} in the space $B_{p,{p^*}}^{1+\frac{1}{p}}({\Omega}_+,\mathbb{R}^{n})\times B_{p,{ p^*}}^{\frac{1}{p}}({\Omega}_+)$.
%, {which satisfies the relations \eqref{add}}.
In addition, since the solution satisfies the condition ${\bf u}\in \mathbf{B}_{\eta }$, by inequality \eqref{estimate-D-B-F-new1-new1-D-Rn-Lp} we obtain the estimate
\begin{align}
\label{estimate-D-B-F-new1-new1-D-Rn1-Lp}
\|{\bf u}\|_{B_{p,{p^*}}^{1+\frac{1}{p}}({\Omega}_+,\mathbb{R}^{n})}+
\|\pi\|_{B_{p,{p^*}}^{\frac{1}{p}}({\Omega}_+)}
\leq C{\|\left(\mathbf f,{\mathbf h_{0}},\mathbf g_{0}\right)\|_{\mathcal Y_p}}
+\frac{1}{4}\|{\bf u}\|_{B_{p,{p^*}}^{1+\frac{1}{p}}({\Omega}_+,\mathbb{R}^{n})},
\end{align}
implying that
%\begin{align}
%\label{estimate-D-B-F-new1-new1-D-Rn1-new-Lp}
%\|{\bf u}\|_{B_{p,{p^*}}^{1+\frac{1}{p}}({\Omega}_+,\mathbb{R}^{n})}\leq
%\dfrac{4}{3}C{\|\left(\mathbf f,{\mathbf h_{0}},\mathbf g_{0}\right)\|_{\mathcal Y_p}}.
%\end{align}
%Substituting \eqref{estimate-D-B-F-new1-new1-D-Rn1-new-Lp} into the right-hand side of \eqref{estimate-D-B-F-new1-new1-D-Rn1-Lp}, we obtain the estimate
\begin{align}
\label{Newtonian-D-B-F-new4n-Lp}
\|{\bf u}\|_{B_{p,{p^*}}^{1+\frac{1}{p}}({\Omega}_+,\mathbb{R}^{n})}+
\|\pi\|_{B_{p,{p^*}}^{\frac{1}{p}}({\Omega}_+)}
\leq \frac{4}{3}C{\|\left(\mathbf f,{\mathbf h_{0}},\mathbf g_{0}\right)\|_{\mathcal Y_p}}, %\nonumber
\end{align}
which is just the inequality \eqref{estimate-D-B-F-new1-new1-D-Rn2-Lp} with the constant
$C_*=\dfrac{4}{3}C=\dfrac{4}{3}\|{\mathcal A}_p^{-1}\|_{{\mathcal L}(\mathcal Y_p,\mathcal X_p)}$.
Similarly, \eqref{cond-pDF2} and \eqref{Newtonian-D-B-F-new4n-Lp} lead to \eqref{cond-pDF} with the constant $C_*'=\dfrac{4}{3}C'$.

Next, we prove the uniqueness of the semilinear mixed problem \eqref{Darcy-Brinkamn problem} solution $\left({\bf u},\pi \right)\in {\mathcal X}_p$, that satisfies inequality \eqref{inequality p2 2-Lp}, when the given data satisfy conditions \eqref{inequality p2 1}.
Assume that $\left({\bf u}',\pi '\right)\in {\mathcal X}_p$ is another solution of problem \eqref{Darcy-Brinkamn problem}, which satisfies inequality \eqref{inequality p2 2-Lp}, implying ${\bf u}'\in \mathbf{B}_{\eta _p}$.
Then ${\mathcal U}({\bf u}')\in \mathbf{B}_{\eta _p},$ where $\left({\mathcal U}({\bf u}'),{\mathcal P}({\bf u}')\right)$ are given by (\ref{solution-v0-Lp}) and satisfy (\ref{Newtonian-D-B-F-new2-D-Rn-Lp}) with ${\bf v}$ replaced by ${\bf u}'$.
Then by \eqref{Darcy-Brinkamn problem} and (\ref{Newtonian-D-B-F-new4n-Lp}) (both written in terms of $\left({\bf u}',{\pi}'\right)$) we obtain the linear mixed Dirichlet-Neumann problem
\begin{equation}
\label{uniqueness-mixed-Ap}
\left\{\begin{array}{lll}
\triangle \left({\mathcal U}({\bf u}')-{\bf u}'\right)-\alpha \left({\mathcal U}({\bf u}')-{\bf u}'\right)-\nabla \left({{\mathcal P}({\bf u}')}-\pi '\right)={\bf 0}\ \mbox{ in }\ {\Omega}_+,\\
\left({\gamma}_{+}\left({\mathcal U}({\bf u}')-{\bf u}'\right)\right)|_{S_D}={\bf 0}\ \mbox{ on }\ S_D,\\
\left(\mathbf t_{\alpha}^{+}\left({\mathcal U}({\bf u}')-{\bf u}',{{\mathcal P}({\bf u}')}-\pi '\right)\right)|_{S_N}={\bf 0}\ \mbox{ on }\ S_N,
\end{array}\right.
\end{equation}
and {${{\gamma}_{+}\left({\mathcal U}({\bf u}')-{\bf u}'\right)\in H^1_p(\partial {\Omega}_+,{\mathbb R}^{n})},\ {{\mathbf t}_\alpha^{+}\left({\mathcal U}({\bf u}')-{\bf u}',{{\mathcal P}({\bf u}')}-\pi '\right)\in L_p(\partial {\Omega}_+,{\mathbb R}^{n})}$}. This problem has only the trivial solution in the space ${\mathcal X}_p$ (see Theorem \ref{Poisson-mixed-Lp}), i.e., ${\mathcal U}({\bf u}')={\bf u}'$, ${\mathcal P}({\bf u}')= \pi '$. Thus, ${\bf u}'$ is a fixed point of ${\mathcal U}$. Since ${\mathcal U}:\mathbf{B}_{\eta _p}\to \mathbf{B}_{\eta _p}$ is a contraction, it has a unique fixed point in $\mathbf{B}_{\eta _p}$, which has been already denoted by ${\bf u}$. Consequently, ${\bf u}'={\bf u}$, and, in addition, $\pi '=\pi $.
\end{proof}

\section*{Appendices}
\appendix
%{
%{\section*{Appendix A: Besov spaces in ${\mathbb R}^n$}
%\addtocounter{section}{1}
%\setcounter{section}{1}
%\setcounter{equation}{0}
%\setcounter{theorem}{0}
%}
\section{Besov spaces in ${\mathbb R}^n$}

Let ${\mu} =({\mu} _1,\ldots ,{\mu} _n)$ be an arbitrary multi-index in ${\mathbb Z}_{+}^n$ of length $|{\mu} |:={\mu} _1+\cdots +{\mu} _n$, and let
$\partial^{{\mu} }:=\dfrac{\partial^{|{\mu} |}}{\partial x_1^{{\mu} _1}\cdots \partial x_n^{{\mu}_n}}.$
Next we recall the definition of Besov spaces in ${\mathbb R}^n$ (cf., e.g., \cite[Section 11.1]{M-W}). By $\Xi$ one denotes the collection of all sets $\{\xi _j\}_{j=0}^{\infty}$ of Schwartz functions with the following property:
\begin{itemize}
\item[(i)] There are some constants $b,c,d> 0$ such that
\begin{equation}
{\rm{supp}}(\xi_{0})\subset \{x : |x|\leq b\},\
{\rm{supp}}(\xi_{j}) \subset \{x : 2^{j-1}c \leq |x| \le 2^{j+1}d \}, \ j = 1,2,\ldots
\end{equation}
\item[(ii)]
Let ${\mu}$ be an arbitrary multi-index in ${\mathbb R}^n$. Then there exists a constant $c_{\partial {\Omega}}>0$ such that
\begin{equation}
\sup_{x \in \mathbb{R}^{n}} \sup_{j \in \mathbb{N}} 2^{j|{\mu} |}|\partial^{{\mu}}\xi_{j}(x)|\leq  c_{\partial {\Omega} }.
\end{equation}
\item[(iii)] The following equality holds
\begin{equation}
\sum_{j=0}^{\infty} \xi_{j} (x) = 1,\ \forall \ x\in {\mathbb R}^n.
\end{equation}
\end{itemize}

Let $s\in \mathbb{R}$, $p,q\in (0, \infty)$. Then for a sequence $\{\xi _j\}_{j=0}^{\infty} \subset \Xi$, the Besov space $B^s_{p,q}(\mathbb{R}^{n})$ is defined by
\begin{equation}
B^s_{p,q}(\mathbb{R}^{n}):= \left\{f\in {\mathcal S}'(\mathbb{R}^{n}): \|f\|_{B^s_{p,q}(\mathbb{R}^{n})}:=\Big(\sum_{j = 1}^{\infty} \|2^{sj}\mathcal{F}^{-1}(\xi_{j} \mathcal{F}f)\|^{q}_{L_p(\mathbb{R}^{n})}\Big)^{\frac{1}{q}} < \infty \right\},
\end{equation}
where ${{\mathbf f}}$ is the Fourier transform and ${\mathcal S}'(\mathbb{R}^{n})$ denotes the space of tempered distributions in $\mathbb{R}^{n}$.
Note that the above definition of the Besov space $B^s_{p,q}(\mathbb{R}^{n})$ is independent of the choice of the set $\{\xi _j\}_{j=0}^{\infty} \subset \Xi$, which means that another sequence in $\Xi$ leads to the same space with an equivalent norm. In particular, for any $s\in {\mathbb R}$, the Besov space $B_{2,2}^s(\mathbb{R}^{n})$ coincides with the Sobolev space $H^s_2(\mathbb{R}^{n})$, i.e.,
$B_{2,2}^s(\mathbb{R}^{n})=H^s_2(\mathbb{R}^{n})$.
Moreover, denoting by {$W^{s}_{p}({\mathbb R}^n)$ the Sobolev-Slobodeckij spaces} {(defined in the classical way through their norms)}, we have the relations (see, e.g., \cite{Triebel}, \cite{Be})
\begin{align}
\label{Z2}
&{W^{s}_{p}({\mathbb R}^n)=B^{s}_{p,p}({\mathbb R}^n),\
%p\in (1,\infty),\
s\in {\mathbb R}\setminus {\mathbb Z}},\\
\label{Z1}
&{W_p^k({\mathbb R}^n)=H_p^k({\mathbb R}^n),\ \ k\in {\mathbb Z}}.
\end{align}

{Let $s_0,s_1\in {\mathbb R}$, $1<p_0{\le}p_1<\infty $ be such that
%\begin{align}
%\label{embed}
$
s_1-\dfrac{n}{p_1}<s_0-\dfrac{n}{p_0},
$
%\end{align}
and $0<q_0,q_1\leq \infty $. Then {the} embedding
\begin{align}
\label{embed-2}
B^{s_0}_{p_0,q_0}({\mathbb R}^n)\hookrightarrow B^{s_1}_{p_1,q_1}({\mathbb R}^n)
\end{align}
is continuous (cf. \cite[Theorem in Section 2.7.1 and  Proposition 2(ii) in Section 2.3.2]{Triebel},  {\cite[Remark 2 in Section 2.2.3]{Runst-Sickel1996}). Note that ${\mathbb R}^n$ in \eqref{embed-2} can be replaced by a domain $\Omega\in{\mathbb R}^n$.}

Let us also recall the following useful inclusions between Besov spaces and {Bessel potential} spaces. Assume that $1\leq q_1\leq q_2\leq \infty $, $1\leq p,q\leq \infty $ and $s_1<s<s_2$. Let $p'$ denote the conjugate exponent of $p$, i.e., $\frac{1}{p'}=1-\frac{1}{p}$. Then we have the {following continuous embeddings,}
\begin{align}
\label{embed-3}
&B^{s}_{p,q_1}({\mathbb R}^n)\hookrightarrow B^{s}_{p,q_2}({\mathbb R}^n),\
B^{s}_{p,\min \{p,p'\}}({\mathbb R}^n)\hookrightarrow H_p^s({\mathbb R}^n)
\hookrightarrow B^{s}_{p,\max\{p,p'\}}({\mathbb R}^n),\\
\label{embed-4}
&B^s_{2,2}({\mathbb R}^n)=H^s_2({\mathbb R}^n),\ B^{s_2}_{p,\infty }({\mathbb R}^n)\hookrightarrow H_p^s({\mathbb R}^n)\hookrightarrow B^{s_1}_{p,1}({\mathbb R}^n),
\end{align}
(cf., e.g., \cite[Chapter 6]{Bl}, \cite[(3.2)]{Toft}, \cite[(4.19)]{M-T}), which imply the continuity of the embedding
\begin{align}
\label{embed-5}
B^{s_2}_{p,q}({\mathbb R}^n)\hookrightarrow B^{s_1}_{p,q}({\mathbb R}^n).
\end{align}}
\noindent These {embeddings} hold also when ${\mathbb R}^n$ is replaced by a bounded Lipschitz domain (see \cite[Chapter 6]{Bl}, \cite[(8)]{Triebel-1}).

%Note that the scales of Bessel potential spaces can be also obtained by using the method of real interpolation (cf., e.g., \cite{Triebel}, see also \cite[Theorem 11.1]{M-W}). Indeed, if $1<p_0,p_1<\infty $, $s_0,s_1\in {\mathbb R}$, $s_0\neq s_1$, and $\theta \in (0,1)$, then
%\begin{align}
%\label{real-int}
%\left(H^{s_0}_{p_0}({\mathbb R}^n),H^{s_1}_{p_1}({\mathbb R}^n)\right)_{\theta ,q}=B^{s}_{p,q}({\mathbb %R}^n)
%\end{align}
%where $s=(1-\theta )s_0+\theta s_1$ and $\displaystyle\frac{1}{p}=\displaystyle\frac{1-\theta }{p_0}+\frac{\theta }{p_1}$.

The scales of Bessel potential and Besov spaces can be obtained by the method of complex interpolation. Indeed, if $s_0,s_1\in {\mathbb R}$, $s_0\neq s_1$, $p_0,p_1\in (1,+\infty )$, $q_0,q_1\in (1,+\infty )$ and $\theta \in (0,1)$, then
%the scales of
(cf., e.g., \cite{Triebel}, \cite[Theorem 11.1.2]{M-W}, \cite[Theorem 3.1]{Be}):
\begin{align}
\label{complex-int}
\left[H^{s_0}_{p_0}({\mathbb R}^n),H^{s_1}_{p_1}({\mathbb R}^n)\right]_{\theta }=H^{s}_{p}({\mathbb R}^n),\ \ \left[B^{s_0}_{p_0,q_0}({\mathbb R}^n),B^{s_1}_{p_1,q_1}({\mathbb R}^n)\right]_{\theta }=B^{s}_{p,q}({\mathbb R}^n),
\end{align}
where $s=(1-\theta )s_0+\theta s_1$, $\frac{1}{p}=\frac{1-\theta }{p_0}+\frac{\theta }{p_1}$ and $\frac{1}{q}=\frac{1-\theta }{q_0}+\frac{\theta }{q_1}$.

Moreover, the scale of Besov spaces can be also obtained by using the method of {real interpolation of Sobolev spaces}.
Indeed, for {$p,q\in (1,+\infty )$, $s_0\neq s_1$, and $\theta \in (0,1)$, we have the following real interpolation property
\begin{align}
\label{real-int}
{\left(H^{s_1}_{p}({\mathbb R}^n),H^{s_2}_{p}({\mathbb R}^n)\right)_{\theta ,q}=B^{s}_{p,q}({\mathbb R}^n,{\mathbb R}^n),}
\end{align}
where $s=(1-\theta )s_0+\theta s_1$ (cf., e.g., \cite[Theorem 14.1.5]{Agr}, \cite[p. 329]{Fa-Me-Mi}, \cite{J-K1}, \cite[(5.38)]{M-M}, \cite{Triebel}, \cite[Theorem 3.1]{Be}).}

Formulas \eqref{complex-int} and \eqref{real-int} remain true if ${\mathbb R}^n$ is replaced by a Lipschitz domain (cf., e.g., \cite[Theorem 3.2, Remark 3.3]{Be}).

For the following property we refer the reader to, e.g., {\cite[relations (3.11)
and Proposition 4.2]{M-M}}.
%\cite[Proposition 4.2]{M-M}.
\begin{lemma}
\label{complex-interpolation}
Let $\Omega \subset {\mathbb R}^n$ be a bounded Lipschitz domain. Let  $S\subset\partial\Omega$ be an admissible patch. If $p_0,p_1\in (1,\infty )$, $s_0,s_1\in [0,1]$ or $s_0,s_1\in [-1,0]$, and $\theta \in (0,1)$, then the following complex
{and real}
interpolation properties hold
\begin{align}
\label{complex-int-patch}
&[H_{p_0}^{s_0}(\partial\Omega),H_{p_1}^{s_1}(\partial\Omega)]_\theta =H_{p}^{s}(\partial\Omega),\quad
[H_{p_0}^{s_0}(S),H_{p_1}^{s_1}(S)]_\theta =H_{p}^{s}(S),\quad [\widetilde{H}_{p_0}^{s_0}(S),\widetilde{H}_{p_1}^{s_1}(S)]_\theta =\widetilde{H}_{p}^{s}(S)
,\\
\label{complex-int-patch-B}
&(H_{p_0}^{s_0}(\partial\Omega),H_{p_1}^{s_1}(\partial\Omega))_{\theta,q} =B_{p,q}^{s}(\partial\Omega),\quad
(H_{p_0}^{s_0}(S),H_{p_1}^{s_1}(S))_{\theta,q} =B_{p,q}^{s}(S),\quad [\widetilde{H}_{p_0}^{s_0}(S),\widetilde{H}_{p_1}^{s_1}(S)]_{\theta,q} =\widetilde{B}_{p,q}^{s}(S),
\end{align}
where $\frac{1}{p}=\frac{1-\theta }{p_0}+\frac{\theta }{p_1}$ and $s=(1-\theta )s_0+\theta s_1$.
{In \eqref{complex-int-patch-B} also $s_0\not=s_1$ and  $q\in (1,\infty]$.}
\end{lemma}

\section{Some general {assertions on} interpolation theory and continuous operators}
Let us consider two compatible couples of Banach spaces,  $X_{0}, X_{1}$ and $Y_{0}, Y_{1}$. Let $X_{\theta}$ and $Y_{\theta}$ be interpolation spaces with respect to $X_{0}, X_{1}$ and $Y_{0}, Y_{1}$, according to \cite[Definition 2.4.1]{Bl}.
If $A_j:X_{j} \to Y_{j}$, $j=0,1$ are linear continuous compatible operators
(i.e., $A_0|_{X_0\cap X_1}=A_1|_{X_0\cap X_1}$) then
they induce the operator $A_+:X_0+X_1 \to Y_0+Y_1$, such that $A_+x:=A_0x_0+A_1x_1$, for any $x\in X_0+X_1$, where $x=x_0+x_1$, $x_j\in X_j$, and $\|A_+\|\le\max(\|A_0\|,\|A_1\|)$, cf. \cite[Section 2.3, Eq. (3)]{Bl}.
Further, $X_\theta\subset X_0+X_1$ and the operator $A_\theta:=A_+|_{X_\theta}$ is linear and continuous.
In the following assertion we consider some cases when the interpolation preserves isomorphism properties of operators.
\begin{lemma}
\label{Interp-Isom}
Let $X_{0}, X_{1}$ and $Y_{0}, Y_{1}$ be two compatible couples of Banach spaces. Let $X_{\theta}$ and $Y_{\theta}$ be interpolation spaces with respect to $X_{0}, X_{1}$ and $Y_{0}, Y_{1}$.
Let $A_j:X_{j} \to Y_{j}$, $j=0,1$, be linear continuous compatible operators that  are isomorphisms.
Let $A_{\theta}:X_{\theta} \rightarrow Y_{\theta}$ be the operator induced by $A_j$.

\begin{itemize}
\item[$(i)$]
If the operators $R_j :Y_{j} \to X_{j}$, inverse to the operators $A_j:X_{j} \to Y_{j}$,  $j=0,1$, respectively, are compatible (i.e., $R_0|_{Y_0\cap Y_1}=R_1|_{Y_0\cap Y_1}$),
then $A_{\theta}:X_{\theta} \rightarrow Y_{\theta}$ is an isomorphism.

\item[$(ii)$]
If $X_0 \subset X_1$,
then $A_{\theta}:X_{\theta} \rightarrow Y_{\theta}$ is an isomorphism.

\item[$(iii)$]
If there exist linear subspaces $X_*\subset X_0\cap X_1$ and $Y_*\subset Y_0\cap Y_1$ such that $Y_*$ is dense in $Y_0\cap Y_1$ and the operator $A_*:=A_0|_{X_*}=A_1|_{X_*}:X_*\to Y_*$ is an isomorphism,
then $A_{\theta}:X_{\theta} \rightarrow Y_{\theta}$ is an isomorphism.
\end{itemize}
\end{lemma}
\begin{proof}
Let us prove item (i).
Since the inverse operators $R_j :Y_{j} \to X_{j}$ are compatible,  they induce a continuous operator $R_+ :Y_0+Y_1 \to X_0+X_1$, such that $R_+ y:=R_0 y_0+R_1 y_1$, for any $y\in Y_0+Y_1$, where $y=y_0+y_1$, $y_j\in Y_j$,  and continuous operator $R_\theta =R_+ |_{Y_\theta}:Y_\theta \to X_\theta$.
Let us show that the operator $R_\theta $ is inverse to $A_\theta$.
Indeed, any $x\in X_\theta$ can be represented as $x=x_0+x_1$, where  $x_j\in X_j$, and then
$$
R_\theta A_\theta x=R_+ A_+ x=R_+ A_+ (x_0+x_1)=R_+ (A_0 x_0+A_1 x_1)=R_0 A_0 x_0+R_1 A_1 x_1=
x_0+x_1=x.
$$
Similarly, any $y\in Y_\theta$ can be represented as $y=y_0+y_1$, where  $y_j\in Y_j$, and then
$$
A_\theta R_\theta  y=A_+ R_+ y=A_+R_+  (y_0+y_1)=A_+(R_0  y_0+R_1  y_1)=A_0R_0  y_0+A_1R_1  y_1=
y_0+y_1=y.
$$
This proves that $R_\theta :Y_\theta \to X_\theta$ is the operator inverse to $A_{\theta}:X_{\theta} \rightarrow Y_{\theta}$ and hence the latter one is an isomorphism.

To prove item (ii) we remark that the inclusion $X_0 \subset X_1$, the compatibility of the operators  $A_j:X_{j} \to Y_{j}$, $j=0,1$, and the invertibility of the operator $A_0:X_{0} \to Y_{0}$ imply that $Y_0\subset Y_1$. Then the invertibility of the operator $A_1:X_{1} \to Y_{1}$ implies $R_1|_{Y_0}=R_0$, i.e., the compatibility of the inverse operators to the operators $A_j:X_{j} \to Y_{j}$, $j=0,1$, which reduces item (ii) to item (i).

Let us prove item (iii).
If $A_j:X_{j} \to Y_{j}$, $j=0,1$, are isomorphisms then there exist continuous inverse operators $R_j :Y_{j} \to X_{j}$, $j=0,1$.
Let us prove that $R_j$ are compatible operators.
Let $R_*:Y_*\to X_*$ be the inverse to the operator $A_*:=A_0|_{X_*}=A_1|_{X_*}:X_*\to Y_*$.
Then for any $\psi\in Y_*$, there exists $\phi\in X_*$ such that $\psi=A_*\phi=A_0\phi=A_1\phi$.
Hence $R_*\psi=\phi=R_0\psi=R_1\psi$, i.e., $R_*=R_0|_{Y_*}=R_1|_{Y_*}$.

Due to the density of $Y_*$ in $Y_0\cap Y_1$, for any $y\in Y_0\cap Y_1$ there exists a sequence $\{{\psi}^i\}_{i=1}^\infty\subset Y_*$ converging to $y$ in $Y_0\cap Y_1$ and hence in $Y_0$ and in $Y_1$.
Then $R_*{\psi}^i\in X_*\subset X_0\cup X_1$ and due to continuity of the operators $R_j :Y_{j} \to X_{j}$, $j=0,1$,
$\lim_{i\to\infty}R_*{\psi}^i=\lim_{i\to\infty}R_j{\psi}^i=R_j y$ in $ X_{j}$ for $j=0,1$, which implies $R_1|_{Y_0\cap Y_1}=R_2|_{Y_0\cap Y_1}$, i.e., the inverse operators $R_j :Y_{j} \to X_{j}$, $j=0,1$ are compatible.

Employing now item (i) concludes the proof of item (iii).
\hfill\end{proof}

Note that  item (iii) of Lemma \ref{Interp-Isom} is available in \cite[Lemma 8.4]{Fa-Me-Mi} for the cases, when the image and domain spaces coincide, i.e, $X_j=Y_j$, under the additional assumptions that $X_*=Y_*$ is a Banach space.
}

Let us give the following useful result in the complex interpolation theory (cf., e.g., {\cite[Theorem 2.7, Corollary 2.8]{Cao}}
%, \cite[Theorem 11.9.24]{M-W}
{and} the references therein, see also \cite[Appendix B]{Lean}).
\begin{lemma}
\label{Prop1.4}
Let $X_{0}, X_{1}$ and $Y_{0}, Y_{1}$ be two compatible couples of Banach spaces and $A_j:X_{j} \to Y_{j}$, $j=0,1$, be two continuous compatible linear operators.
Let $X_{\theta}:= [X_{0}, X_{1}]_{\theta}$ and $Y_{\theta}:= [Y_{0}, Y_{1}]_{\theta}$ {denote the complex interpolation spaces of $X_{0}, X_{1}$ and $Y_{0}, Y_{1}$, respectively}, for each $\theta \in (0,1)$.
%Then the operator
%$A_{\theta}:X_{\theta} \rightarrow Y_{\theta},$
%induced by $A$ on $X_{\theta}$, is linear and continuous, for any $\theta \in (0,1)$.
%Moreover,
If there exists a number $\theta_{0}\in (0,1)$ such that $A_{\theta_0}:X_{\theta_0} \rightarrow Y_{\theta_0}$ is an isomorphism, then there exists $\varepsilon >0$ such that the operator $A_{\theta}: X_{\theta} \rightarrow Y_{\theta}$ is an isomorphism as well, for any $ \theta \in (\theta_{0}-\varepsilon, \theta_{0}+\varepsilon)$.
\end{lemma}
\begin{rem}
{The extension of Lemma \ref{Prop1.4} to the case of two compatible couples of quasi-Banach spaces, $X_{0}, X_{1}$ and $Y_{0}, Y_{1}$, such that $X_0+X_1$ and $Y_0+Y_1$ are analytically convex can be found in \cite[Theorem 11.9.24]{M-W} and the references therein. Note that any Banach space is analytically convex $($cf., e.g., \cite[p. 223]{M-W} $)$.}
\end{rem}
Finally, let us mention the following useful result (cf, e.g., \cite[Lemma 11.9.21]{M-W}).
\begin{lemma}
\label{Lem2 Fredholm}
Let $ X_{1}, X_{2} $ and $ Y_{1}, Y_{2} $, be Banach spaces such that the embeddings $ X_{1} \hookrightarrow X_{2} $ and $ Y_{1} \hookrightarrow Y_{2} $ are continuous, and also that the embedding $ Y_{1} \hookrightarrow Y_{2} $ has dense range. Assume that $T:X_{1}\rightarrow Y_{1}$  and  $T : X_{2} \rightarrow Y_{2} $ are Fredholm operators with the same index, ${\rm{ind}}\left(T : X_{1}\rightarrow Y_{1}\right)= {\rm{ind}}\left(T : X_{2} \rightarrow Y_{2}\right)$. Then
${\rm{Ker}}\{T : X_{1}\rightarrow Y_{1}\}={\rm{Ker}}\{T : X_{2} \rightarrow Y_{2}\}.$
\end{lemma}

\end{document}